\documentclass[hdr, twoside, final]{unswthesis}
\usepackage{mystyle} 
\usepackage{lmodern}
\usepackage{epigraph}
\numberwithin{thm}{section}
\newlength\longest

\makeatletter
\fancypagestyle{noHeading}{
        \fancyhead{}
        \renewcommand{\headrulewidth}{0pt}
}
\thesistitle{Trace Formulas\\ in Noncommutative Geometry}
\thesisschool{School of Mathematics and Statistics}
\thesisauthor{Eva-Maria Hekkelman}
\thesisZid{z5369973}
\thesisdegree{Doctor of Philosophy}
\thesisdate{June 2025}

\DeclareGraphicsExtensions{.pdf,.jpeg,.png}
\graphicspath{{2-intro/}{3-literature/}{4-techchapter/}{5-techchapter/}{6-techchapter/}{8-appendices/}}

\begin{document}
    
    \renewcommand{\bibname}{References}
    \frontmatter  
    \maketitle
    
    \chapter{Abstract}

Trace formulas appear in many forms in noncommutative geometry (NCG). In the first part of this thesis, we obtain results for asymptotic expansions of trace formulas like heat trace expansions by adapting the theory of Multiple Operator Integration to NCG. More broadly, this construction provides a natural language for operator integrals in NCG, which systematises and simplifies operator integral arguments throughout the literature. Towards this end, we construct a functional calculus for abstract pseudodifferential operators and generalise Peller's construction of multiple operator integrals to this abstract pseudodifferential calculus. In the process, we obtain a noncommutative Taylor formula.

In the second part of this thesis, we shift our attention to Dixmier trace formulas. First, we provide an approximation of the noncommutative integral for spectrally truncated spectral triples in the Connes--Van Suijlekom paradigm of operator system spectral triples. Our approximation has a close link to Quantum Ergodicity, which we will use to state an NCG analogue of the fundamental result that ergodic geodesic flow implies quantum ergodicity. Furthermore, we provide a Szeg\H{o} limit theorem in NCG. Next, we provide a Dixmier trace formula for the density of states, a measure originating in solid state physics that can be associated with an operator on a geometric space. We first provide this formula in the setting of discrete metric spaces, and then in the setting of manifolds of bounded geometry. The latter leads to a Dixmier trace formula for Roe's index on open manifolds.
    \chapter{Acknowledgement}
With heartfelt gratitude and admiration to my PhD supervisor Fedor Sukochev, and co-supervisors Edward McDonald and Dmitriy Zanin. You have not only the highest standards for research, but also the utmost confidence in your students to meet them. No encouragement is more effective. Thank you for your guidance, I will remember my time with you very fondly.

In particular, I would like to thank Edward McDonald for the many hours you set aside to keep an eye on my progress. 
You are an inspiration to me, and I am proud to call myself a student of yours.

I owe much to Teun van Nuland as well. Thank you for taking me along in a new research direction when I was looking for something fresh. 
I learned a lot from you, and it has been fantastic to work together.

A sincere thanks to Nigel Higson is in order, I am immensely grateful for your support. It has been inspirational to witness your joyful lectures. 

I would like to thank Nigel Higson, Marius Junge, Edward McDonald, Teun van Nuland, and Walter van Suijlekom for their hospitality during my visits abroad, as well as the staff at the Pennsylvania State University, Delft University of Technology, and the Radboud University Nijmegen.
Thank you to Nigel Higson, Fedor Sukochev, and the Australian Mathematical Society for providing funds for these travels.

For the work this thesis was based on, I am glad to acknowledge my co-authors Nurulla Azamov, Edward McDonald, Teun van Nuland, Fedor Sukochev, and Dmitriy Zanin, as well as helpful communications with Magnus Goffeng, Nigel Higson, Bruno Iochum, Eric Leichtnam, Qiaochu Ma, and Rapha\"el Ponge. Their contributions will be indicated in the relevant chapters as well.

For the brave task of proofreading this thesis, I thank and apologise to Fedor Sukochev, Edward McDonald, and Teun van Nuland. Of course, I claim full responsibility for any mistakes remaining. Furthermore, I am thankful to the referees who volunteered their time for reading and assessing this dissertation, as well as the anonymous referees of the preprints and journal articles on which this is all based.

I thoroughly enjoyed my trips down to the University of Wollongong for the OANCG seminars. Thank you to Adam Rennie, Aidan Sims, Angus Alexander, Ada Masters, Alexander Mundey, Abraham Ng and the others for the warm welcomes, the cold drinks, and of course the maths.

Many thanks to Thomas Scheckter for answering silly questions of mine, and for helping me navigate the scary world of academic paperwork. This PhD was partially supported by Australian Research Council grant FL170100052.

Cheers, Abirami Srikumar and Haley Stone, for helping with a bit of Python coding in the final stretch of work on this thesis.

Finally, I wish to thank my parents, my sibling, my grandmother, and my dear friends for your love and support.  
These past few years have been quite a ride, and I will never forget that you were there for me through it~all.

 {\setlength{\epigraphwidth}{0.45\textwidth}
  \epigraph{ someone will remember us\\ 
     \ \ \quad\quad \quad \quad \quad \quad \quad \quad \quad  I say \\
     \quad \ \ \quad  \quad \quad \quad \quad \quad \quad \quad even in another time}{Sappho, Fragment 147~\cite{Voigt}\\
     Trans. Anne Carson~\cite{Carson2002}}
    }


    \chapter{List of Publications}
\printbibliography[heading=none,keyword=ownphd]

    \tableofcontents
    
    \mainmatter
    \pagestyle{fancy}
        \fancyhf{}
        \fancyhead[LE]{\leftmark}
        \fancyhead[RO]{\rightmark}
        \fancyfoot[C]{\thepage}
        \renewcommand{\headrulewidth}{1pt}
        \setcounter{secnumdepth}{3}

    \chapter{Introduction}\label{Ch:Intro}
{\setlength{\epigraphwidth}{\widthof{Do not forget our motto ``work harder'', and act~accordingly.}}
\epigraph{Do not forget our motto ``work harder'', and act~accordingly.}{Fedor Sukochev}}
The thesis before you covers a variety of topics in spectral theory, tied together by Alain Connes' philosophy of noncommutative geometry~\cite{Connes1994}. These will be served in two parts: the first concerns pseudodifferential operators and multiple operator integrals based on the paper~\cite{MOOIs}, the second part revolves around Dixmier trace formulas related to Connes' integral formula based on the papers~\cite{AHMSZ, HekkelmanMcDonald2024, QE}. All four papers are joint works with Edward McDonald, and furthermore Nurulla Azamov, Teun van Nuland, Fedor Sukochev and Dmitriy Zanin feature as co-author on one paper each. 

Along the way we will additionally encounter heat trace expansions, quantum ergodicity, and a little index theory.
While this mix hopefully provides an entertaining read, as a consequence there is a lot of background material to discuss. I have made an attempt to summarise the five main themes of this dissertation in five sections. The experts and the brave 
could skip ahead as the mathematical content of each chapter is meant to stand alone --- with the exception of Chapter~\ref{Ch:MOIs} which depends on Chapter~\ref{Ch:FunctCalc}. In comparison with the papers this thesis draws from, the presentation here is made more accessible, with additional background information and some novel results.

\section{Noncommutative geometry}\label{S:NCG}
Over the past decades noncommutative geometry has developed into a rich field of mathematics, bringing together differential geometry, functional analysis, spectral theory, $K$-theory, representation theory, quantum field theory, and many more areas. It allows the study of `manifolds' on which coordinate functions do not necessarily commute, exotic spaces which appear in various areas of mathematics and physics. 
As any brief summary of this subject must, the overview here leaves out many interesting and important aspects of NCG. Recommended expositions which do more justice to the field can be found in~\cite{Connes1994, GVF2001, Suijlekom2025}. The book~\cite{EcksteinIochum2018} is an excellent reference for many of the analytical details that we will use. For a more algebraic introduction to noncommutative geometry, see~\cite{Khalkhali2013}.

\subsection{Spectral geometry} \label{SS:SpectralGeometry}
As is taught in kindergarten, operators associated with geometric spaces carry geometric information in their spectra which can be carefully extracted. Indeed, (hearing) children can tell that smaller objects tend to produce higher-pitched sounds, hollow objects sound different from solid ones, and can blindly distinguish between the sound of strings, drums, and blocks. What their ears are picking up is the spectrum of the Laplace operator. 

Modeling an object as a bounded domain (open connected set) in $d$-dimensional Euclidean space $\Omega \subseteq \R^d$, the sound it produces when lightly struck decomposes into frequencies which are determined by the eigenvalues of the Laplace operator $\Delta:= \sum_{j=1}^d\partial_{j}^2$ on $\Omega$ with Dirichlet boundary conditions, that is, they are the eigenvalues of those $\lambda \in \R_{\geq 0}$ for which there exists a solution of the Helmholtz equation
\[\begin{cases}
  -\Delta u(x) &= \lambda  u(x),\quad   x\in \Omega;\\
  u|_{\partial \Omega} &\equiv 0.
\end{cases}
\]
Writing $N(\lambda)$ for the number of such eigenvalues (counting multiplicities) less than $\lambda$, Weyl's law~\cite{Weyl1912, Chavel1984, Ivrii2016} gives that
\begin{equation}\label{eq:WeylLaw}
N(\lambda) \sim \frac{\omega_{d}}{d(2\pi)^{d}}\vol(\Omega) \lambda^{\frac{d}{2}}, \quad \lambda \to \infty,
\end{equation}
where $\omega_d$ is the volume of the unit sphere $\mathbb{S}^{d-1}\subseteq \mathbb{R}^d$, and $\sim$ means that the ratio converges to $1$. This indicates that in a perfect world, our ears can detect the dimension $d$ and the volume of an object. For $d \geq 2$ and with technical assumptions on $\Omega$  (smooth boundary, set of periodic billiards has measure zero), Duistermaat, Guillemin and Ivrii have shown~\cite{DuistermaatGuillemin1975,Ivrii1980, Ivrii2016} that even the $d-1$-dimensional volume of the boundary $\partial \Omega$, $\vol(\partial \Omega)$, can be `heard' as indicated by the Weyl law with error terms 
\begin{equation}\label{eq:WeylLaw2}
N(\lambda) \sim \frac{\omega_{d}}{d(2\pi)^{d}}\vol(\Omega) \lambda^{\frac{d}{2}}- \frac{\omega_{d-1}}{4(d-1)(2\pi)^{d-1}} \vol(\partial \Omega)  \lambda^{\frac{d-1}{2}} + o(\lambda^{\frac{d-1}{2}}), \quad \lambda \to \infty.
\end{equation}
This error term was already conjectured by Weyl himself~\cite{Weyl1913}. 
However, note that spheres satisfy neither these technical conditions nor this version of Weyl's law~\cite{Ivrii2016}. See Figure~\ref{F:WeylLaw} for an illustration of the effect of the error term in Weyl's law.

\begin{figure}[h!]
\centering
    \includegraphics[width=0.9\linewidth]{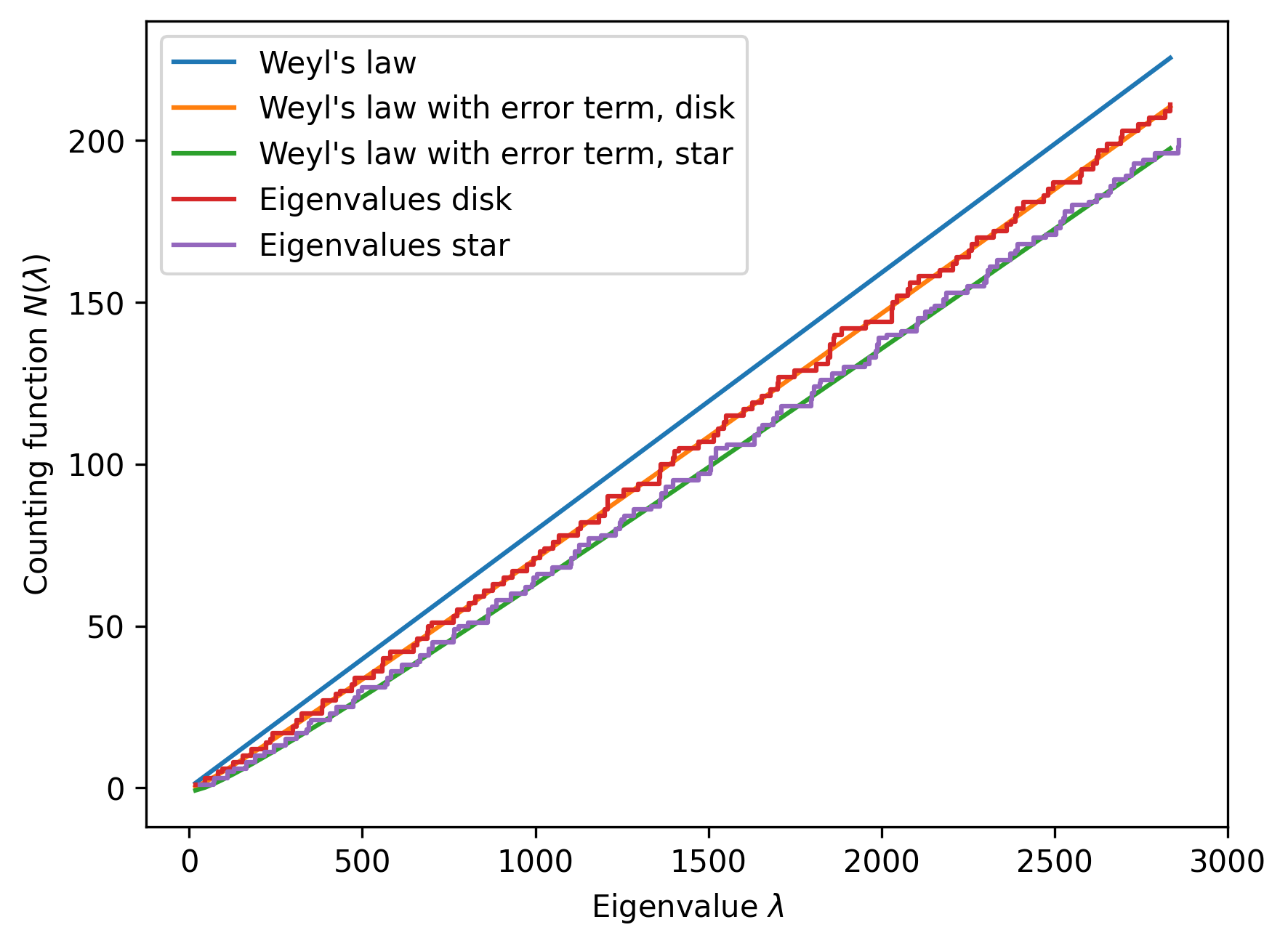}
    \caption{Weyl's law and the spectral counting function for the Laplacian on a disk and a 5-pointed star, with Dirichlet boundary conditions. Both domains have area equal to~$1$, making the first-order term in Weyl's law the same, $N(\lambda) \approx \frac{1}{4\pi} \lambda$, plotted in blue. The 5-pointed star has a longer perimeter than the disk, resulting in a larger deviation from Weyl's law~\eqref{eq:WeylLaw2}. These eigenvalues were generated in Python with the finite-element method, using FEniCSx~\cite{DOLFINx}.}
    \label{F:WeylLaw}
\end{figure}

More is true: even information about an object's curvature is hidden inside the spectrum of the Laplace operator. Taking now a closed (oriented) Riemannian manifold $(M, g)$ with Laplace--Beltrami operator $\Delta_g$, Minakshisundaram and Pleijel~\cite{MinakshisundaramPleijel1949} proved that the heat trace admits an asymptotic expansion
\[
\Tr(\exp(t \Delta_g)) \sim \sum_{k=0}^\infty a_k(\Delta_g) t^{\frac{k-d}{2}}, \quad t \to 0,
\]
the first coefficients of which can be given as~\cite{Berger1968,Gilkey1975}
\begin{align*}
    a_0(\Delta_g) &= (4 \pi)^{-\frac{d}{2}} \operatorname{Vol}(M);\\
    a_2(\Delta_g) &= -\frac{1}{6} (4 \pi)^{-\frac{d}{2}} \int_M R \, d\nu_g,
\end{align*}
where $\nu_g$ is the Riemannian volume form of $M$ and $R$ is the scalar curvature. The coefficients $a_k(\Delta_g)$ vanish for odd $k$, and the higher coeffients $a_k(\Delta_g)$ are integrals over $M$ of (complicated) expressions involving the metric of $M$~\cite{Gilkey1975}.

These facts raise the question how far this philosophy can be pushed. 
And indeed it is tradition in any text on spectral geometry to cite an influential essay by Mark Kac from 1966 with the title
``Can one hear the shape of a drum?''~\cite{KacDrum}. In this text, Kac asked whether a bounded domain in $\R^2$ can be determined up to isometries by the eigenvalues of the Laplace operator, and explored various aspects surrounding this question. For Riemannian manifolds the answer was known to be negative~\cite{Milnor1964} --- only much later two isospectral domains in $\R^2$ were identified~\cite{Gordon1992}, answering Kac's question in the negative, see Figure~\ref{F:drums}.
\begin{figure}[h]
\centering
    \includegraphics[width=0.9\linewidth]{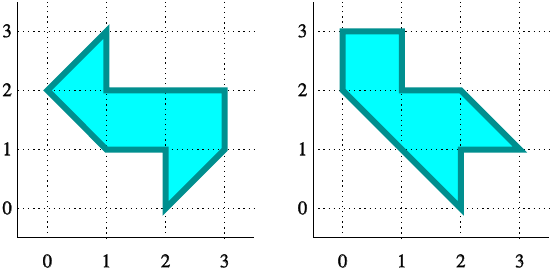}
    \caption{Two domains in $\mathbb{R}^2$ with spectrally indistinguishable Laplacians~\cite{Gordon1992}. This figure was created by Jitse Niesen and is in public domain.}
    \label{F:drums}
\end{figure}

In some sense, the Laplacian is not too far off from pinning down a compact Riemannian manifold, however. Let us say that the Riemannian manifolds $(M_1, g_1)$ and $(M_2, g_2)$ are isometric if there exists a diffeomorphism $\phi: M_1 \to M_2$ such that $\phi^* g_2 = g_1$, where $\phi^*$ denotes the pull-back via $\phi$. Arendt, Biegert and Ter Elst proved that the data $(L_2(M), \Delta_g)$ in combination with the cone of positive elements $L_2(M)_+$, that is, those $\phi \in L_2(M)$ corresponding to positive real-valued functions, determines the manifold up to isometries.

\begin{thm}[\cite{ArendtBiegert2012}]\label{T:ABTE}
    Let $(M_1, g_1)$ and $(M_2, g_2)$ be connected compact Riemannian manifolds, with corresponding Laplace--Beltrami operators $\Delta_1$ and $\Delta_2$. Then the following are equivalent:
    \begin{enumerate}
        \item the Riemannian manifolds $(M_1, g_1)$ and $(M_2, g_2)$ are isometric;
        \item there exists a bounded linear map $U: L_2(M_1) \to L_2(M_2)$ such that 
        \begin{gather*}
        \phi \in L_2(M_1)_+ \iff U\phi \in L_2(M_2)_+;\\
        U\Delta_1 = \Delta_2 U. \ \, \,
        \end{gather*}
    \end{enumerate}
\end{thm}

Still, the cone of positive elements $L_2(M)_+$ is in some sense not operator theoretical, and one can wonder if there is a clever and natural way to nail down a compact Riemannian manifold using only operators. 

And indeed there is an algebra of operators on $L_2(M)$ that we can add to the data $(L_2(M), \Delta_g)$ to do this job. It is the commutative $C^*$-algebra $C(M)$ of continuous complex-valued functions on $M$, which represents canonically as bounded operators on $L_2(M)$ via pointwise multiplication: for $f \in C(M)$ we write
\begin{align*}
    M_f: L_2(M) &\to L_2(M)\\
    g &\mapsto fg.
\end{align*}
Then the triple $(C(M), L_2(M), \Delta_g)$ uniquely determines the Riemannian manifold.

\begin{cor}
    Let $(M_1, g_1)$ and $(M_2, g_2)$ be connected compact Riemannian manifolds, with corresponding Laplace--Beltrami operators $\Delta_1$ and $\Delta_2$. 
    Then the following are equivalent:
    \begin{enumerate}
        \item the Riemannian manifolds $(M_1, g_1)$ and $(M_2, g_2)$ are isometric;
        \item there exists a unital $*$-isomorphism $\psi: C(M_1) \xrightarrow{
   \,\smash{\raisebox{-0.35ex}{\ensuremath{\scriptstyle\sim}}}\,} C(M_2)$ and
        a unitary operator $U: L_2(M_1) \to L_2(M_2)$ such that  
        \begin{align*}
            UM_f &= M_{\psi(f)}U, \quad f \in C(M_1)\\
            U\Delta_1 &= \Delta_2 U.
        \end{align*}
    \end{enumerate}
\end{cor}
\begin{proof}
    $(1) \Rightarrow (2)$ is trivial. For the reverse direction, suppose we have a unitary $U: L_2(M_1) \to L_2(M_2)$ as in (2). Because the manifolds $M_i$ are assumed to be compact, the kernel of $\Delta_i$ consists of constant functions. Since $U$ intertwines the Laplacians $\Delta_i$, we can therefore assume without loss of generality that $U 1_{L_2(M_1)} = 1_{L_2(M_2)}$. By Theorem~\ref{T:ABTE}, we need only check that
    \[
    \phi\in L_2(M_1)_+ \iff U\phi \in L_2(M_2)_+.
    \]
    Indeed, we have    
    \begin{align*}
        \phi \in L_2(M_1)_+ &\iff \langle M_f 1_{L_2(M_1)}, \phi\rangle_{L_2(M_1)}  \geq 0 \quad \forall f\in C(M_1)_+\\
        &\iff \langle M_{\psi(f)}1_{L_2(M_2)}, U\phi \rangle_{L_2(M_2)} \geq 0 \quad \forall f\in C(M_1)_+\\
        &\iff U\phi \in L_2(M_2)_+. \qedhere
    \end{align*}
\end{proof}

\subsection{Noncommutative geometry}
Comfortable with the idea that operators on a Hilbert space can encode geometrical data, we can now take the radical step that Alain Connes laid out and generalise geometry beyond the commutative world.

With Gelfand duality in mind (the correspondence between commutative $C^*$-algebras and locally compact Hausdorff spaces~\cite{Murphy1990}), the study of noncommutative $C^*$-algebras can be thought of as `noncommutative topology'. Similarly, the theory of noncommutative von Neumann algebras can be thought of as `noncommutative measure theory'. These analogies are more than amusing observations. Such noncommutative $C^*$-algebras and von Neumann algebras have proven their value as two of the most versatile objects in operator theory, appearing in countless areas of mathematics and physics. Topological and measure theoretic constructions and arguments can inform analogues in $C^*$-algebras and von Neumann algebras and vice versa.

As a concrete example, $K$-theory was developed by Atiyah and Hirzebruch~\cite{AtiyahHirzebruch1961} as a theory of groups constructed from vector bundles on topological spaces (now called topological $K$-theory). This was used with great success in a proof of the Atiyah--Singer index theorem~\cite{AtiyahSinger1968}. Soon after, $K$-theory was generalised to (noncommutative) $C^*$-algebras, where it became indispensable for operator algebraists, with landmark results being Elliott’s program for the classification of $C^*$-algebras~\cite{Elliott1976,Elliott1993} and Brown--Douglas--Fillmore's development of $K$-homology and the classification of extensions of $C^*$-algebras~\cite{BDF} (see~\cite{Wegge-Olsen1993,Blackadar1998,RordamLarsen2000} for books on $K$-theory).

In the wake of the developments and ideas surrounding $K$-theory and the Atiyah--Singer index theorem, Connes postulated his noncommutative differential geometry~\cite{Connes1982,Connes1985, Connes1994}. 
He saw that a noncommutative $*$-algebra of bounded operators, in combination with a self-adjoint operator satisfying certain properties, can be interpreted as a noncommutative generalisation of a (spin) manifold. These ideas and techniques are relevant for badly behaved spaces like the unitary dual of a finitely generated non-abelian discrete or Lie group, the leaf space of a foliated manifold, or the orbit space of group actions on manifolds, which are more easily described with noncommutative $*$-algebras than as a usual point-set topological space. 

Explicitly, a `noncommutative space' is described by a spectral triple. This has its origin in a paper on Kasparov's $KK$-theory by Baaj and Julg~\cite{BaajJulg1983} where spectral triples are introduced as a type of `unbounded $K$-cycle'. Spectral triples were later popularised by Alain Connes who recognised their potential as a characterisation of manifolds. 

\begin{defn}
    A unital spectral triple $(\Ac, \Hc, D)$ consists of:
    \begin{enumerate}
        \item a Hilbert space $\Hc$;
        \item a unital $*$-algebra $\Ac$ represented faithfully as bounded operators $\pi: \Ac \to B(\Hc)$, where $\pi(1) = 1_{B(\Hc)}$;
        \item a self-adjoint operator $D$ on $\Hc$ with compact resolvent,
    \end{enumerate}
    such that $\pi(a)\dom(D)\subseteq \dom(D)$ and $[D, \pi(a)]$ extends to a bounded operator for all $a \in \Ac$. We usually identify $a\in \Ac$ with its image $\pi(a) \in B(\Hc)$, and omit writing $\pi$ altogether.
    
    The spectral triple is called `even' if equipped with a $\Z_2$-grading $\gamma$ on $\Hc$ such that $D\gamma = -\gamma D$ and $a\gamma = \gamma a$ for all $a \in \Ac$. 
\end{defn}

The basic commutative example of a spectral triple has roots in the Atiyah--Singer index theorem. It is not dissimilar from the triple $(C(M), L_2(M), \Delta_g)$ that we arrived at in Section~\ref{SS:SpectralGeometry}. We simply replace $C(M)$ by $C^\infty(M)$, and $\Delta_g$ by the Dirac operator, an elliptic first-order differential operator which functions as a `square root' of $\Delta_g$. Such an operator does not exist on $L_2(M)$, but it does on the Hilbert space of square-integrable sections of the spinor bundle of so-called `spin manifolds'. An oriented Riemannian manifold is called \textit{spin} if there exists a complex Hermitian vector bundle $S \to M$ (the spinor bundle) such that $\mathrm{End}(S)$ is isomorphic to a certain Clifford algebra, and there exists a charge conjugation on $S$~\cite[Chapter~II]{LawsonMichelsohn1989}\cite[Chapter~9]{GVF2001}\cite[Chapter~4]{Suijlekom2025}. The precise details do not matter much for this thesis, what is important is that the triple then takes the form $(C^\infty(M), L_2(S), D_S)$, where $L_2(S)$ are the square-integrable sections of $S$ and $D_S$ is the Dirac operator associated with $S$. The terminology `even' spectral triple is motivated by the fact that for an even-dimensional spin manifold $M$, the  canonical spectral triple $(C^\infty(M), L_2(S), D_S)$ admits a natural grading $\gamma_M$ making it an even spectral triple~\cite[Chapter~9]{GVF2001}\cite[Chapter~4]{Suijlekom2025}.

The claim that $D_S$ functions as a square root of the Laplacian is motivated by the Lichnerowicz formula~\cite[Theorem~II.8.8]{LawsonMichelsohn1989}\cite[Theorem~9.16]{GVF2001}
\[
D_S^2 = -\Delta_S + \frac{1}{4}s,
\]
where $\Delta_S$ is the Laplacian associated with the spin bundle $S$ and $s$ is the scalar curvature of $M$. In line with this formula, for general spectral triples $(\Ac, \Hc, D)$, the operator $D^2$ is commonly interpreted as an analogue of a Laplace operator.

In light of Section~\ref{SS:SpectralGeometry} it is not surprising that the canonical spectral triple of a compact Riemannian spin manifold is a unique invariant for the spin manifold. However, for spectral triples we even have an explicit \textit{reconstruction} theorem~\cite{Connes2013}. From an abstract unital spectral triple $(\Ac, \Hc, D)$ with specific additional properties like $\Ac$ being commutative, one can construct a compact spin manifold $M$ which realises the spectral triple concretely as $(C^\infty(M), L_2(S), D_S)$. This highly non-trivial result requires recognising when the $*$\nobreakdash-algebra $\Ac$ is of the form $C^\infty(M)$ for a smooth manifold $M$, using only how $\Ac$ and $D$ are situated relative to each other as operators on the abstract Hilbert space $\Hc$. This puts solid ground under the claim that spectral triples can be considered noncommutative generalisations of (spin) manifolds.

Interest in NCG has not been limited to mathematics. In physics, noncommutative spaces have been used to explain features of the quantum Hall effect~\cite{BellissardVanElst1994} and topological insulators~\cite{Schulz-Baldes2016}, and there is extensive literature on applications of NCG in Quantum Field Theory, see e.g.~\cite{Szabo2003,ConnesMarcolli2008,Suijlekom2025} for overviews. The Standard Model of particle physics, minimally coupled to gravity, can be entirely derived through an action principle~\cite{ChamseddineConnes1996} from the description of the universe as a noncommutative space~\cite{ChamseddineConnes2007,Suijlekom2025}.
This unified geometrical treatment of both Quantum Field Theory and General Relativity makes for an attractive approach to quantum gravity, and extends beyond the Standard Model to, for example, the Pati--Salam model~\cite{ChamseddineConnes2013,ChamseddineConnes2015}.

Wrapping up this overview of NCG, we will now give two examples of noncommutative spaces which will appear throughout this dissertation. This exposition is taken from the joint paper of the author with Edward McDonald~\cite{QE}; for more details on the noncommutative torus (on which there is much literature) we refer to~\cite{HaLee2019,HaLee2019b} and~\cite[Section IV.12.3]{GVF2001}, for details on almost commutative manifolds to~\cite[Chapter~10]{Suijlekom2025}.

\begin{ex}
    Let $d\geq 2$ and let $\theta$ be a real $d \times d$ antisymmetric matrix. The noncommutative torus is the universal $C^*$-algebra $C(\mathbb{T}^d_\theta)$ generated by a family of unitary elements $\{u_n\}_{n\in \Z^d}$ subject to the relations
    \[
    u_n u_m = e^{\frac{i}{2}\langle n, \theta m\rangle}u_{n+m}, \quad n,m\in \Z^d.
    \]
    The functional
    \[
    \tau_\theta \bigg(\sum_{k \in \mathbb{Z}^d}c_k u_k \bigg):= c_0
    \]
    extends to a continuous faithful tracial state on $C(\mathbb{T}^d_\theta).$ 
    The smooth subspace $C^\infty(\mathbb{T}^d_\theta)$ is the subalgebra of $x \in C(\mathbb{T}^d_\theta)$ for which $\widehat{x}(k) = \tau_\theta(xu_k^*)$ is a rapidly decaying sequence on $\Z^d.$
    The Hilbert space in the GNS representation corresponding to $\tau_\theta$ is denoted $L_2(\mathbb{T}^d_\theta),$ and $\{u_n\}_{n\in \Z^d}$ is an orthonormal basis for $L_2(\mathbb{T}^d_\theta).$ The self-adjoint operators $D_j$, $j=1, \ldots, d$ on $L_2(\mathbb{T}^d_\theta)$ are defined on the basis by
    \[
    D_j u_k := k_j u_k,\quad k=(k_1,\ldots,k_d)\in \Z^d.
    \]
    The operator $D = \sum_{j=1}^d D_j \otimes \gamma_j$ on $L_2(\mathbb{T}^d_\theta)\otimes \mathbb{C}^{N_d}$, where $\gamma_j$ are standard Clifford matrices on $\mathbb{C}^{N_d}$ with $N_d = 2^{\lfloor\frac{d}{2}\rfloor}$, gives a spectral triple
    \[
    (C^\infty(\mathbb{T}^d_\theta),L_2(\mathbb{T}^d_\theta)\otimes \mathbb{C}^{N_d}, D ),
    \]
    where we represent $C^\infty(\mathbb{T}^d_\theta)$ as operators on $L_2(\mathbb{T}^d_\theta)\otimes \mathbb{C}^{N_d}$ by acting on the first component~\cite[Section 12.3]{GVF2001}.
    We write $\Delta := -\sum_{j=1}^d D_j^2$ as an operator on $L_2(\mathbb{T}^d_\theta)$, so that $|D|= \sqrt{-\Delta} \otimes 1_{\mathbb{C}^{N_d}}$. 
\end{ex}

\begin{ex}
    Given the canonical spectral triple $(C^\infty(M), L_2(S), D_M)$ of an even-dimensional spin manifold with natural grading $\gamma_M$, and a finite spectral triple $(\Ac_F, \Hc_F, D_F)$, meaning that $\Hc_F$ and $\Ac_F$ are finite-dimensional, the product spectral triple
    \[
    (C^\infty(M) \otimes \Ac_F, L_2(S) \otimes \Hc_F, D_M \otimes 1 + \gamma_M \otimes D_F).
    \]
    is called an almost-commutative manifold. 
\end{ex}

\section{Pseudodifferential operators}\label{S:PSDOs}
Pseudodifferential operators naturally emerged from the study of partial differential equations (PDEs)~\cite{KohnNirenberg1965,Hormander1967}, where they became a crucial tool in determining qualitative properties of solutions to PDEs. For us, they are indispensable in obtaining Weyl laws and in determining coefficients in heat trace expansions like those covered in Section~\ref{S:NCG}. See~\cite{,Taylor1981,Shubin2001,HormanderLPDOIII} for classic books on the subject.

Let us loosely sketch a path leading to the definition of a classical pseudodifferential operator, with no respect for technical details. 
Suppose we want to study a PDE on Euclidean space,
\[
     \sum_{|\alpha|\leq k} a_\alpha(x)\partial^\alpha u(x) = f(x), \quad x \in \R^d,
\]
for nice enough functions $f$ and $a_\alpha$. The operator $L := \sum_{|\alpha|\leq k} M_{a_\alpha}\partial^\alpha$, complicated as it is, is much easier to study after a Fourier transform. For the Fourier transform, we will use the convention
\[
\hat{u}(\xi) := \mathcal{F}(u)(\xi) := \frac{1}{(2\pi)^{\frac d2}}\int_{\R^d}e^{-i\langle x, \xi \rangle} u(x) \,dx, \quad \xi \in \R^d, u \in L_2(\R^d).
\]
In the frequency domain, the operator $L$ then corresponds to multiplying with the polynomial $p_L(x,\xi):= \sum_{|\alpha|\leq k} a_\alpha(x)(i\xi)^\alpha$,
\[
Lu(x) = \frac{1}{(2\pi)^{\frac{d}{2}}} \int_{\R^d} e^{i\langle x, \xi\rangle} p_{L}(x,\xi) \hat{u}(\xi) \, d\xi.
\]

In an ideal scenario, when $p_L(x,\xi) = p_L(\xi)$ is non-zero and independent of $x$, we can simply solve the PDE from before by putting
\[
u(x) = L^{-1}f(x) = \frac{1}{(2\pi)^{\frac{d}{2}}} \int_{\R^d} e^{i\langle x, \xi\rangle} \frac{1}{p_{L}(\xi)} \hat{f}(\xi) \, d\xi.
\]
Likewise, the heat equation 
\begin{align*}
    Lu(t,x) &= \partial_t u(t,x);\\
    u(0,x) &= f(x),
\end{align*}
can then be solved by putting
\[
u(t,x) = e^{tL}f(x) :=  \frac{1}{(2\pi)^{\frac{d}{2}}} \int_{\R^d} e^{i\langle x, \xi\rangle} e^{t p_{L}(\xi)} \hat{f}(\xi) \, d\xi.
\]
The idea behind pseudodifferential operators is that 
even though these operators $L^{-1}$ and $e^{tL}$ are not differential operators, the analytical properties of an operator of the form
\[
T_a u(x) = \frac{1}{(2\pi)^{\frac{d}{2}}} \int_{\R^d} e^{i\langle x, \xi\rangle} a(x,\xi) \hat{u}(\xi) \, d\xi,
\]
depend not so much on its symbol $a$ being a \textit{polynomial}, but rather on the qualitative properties of $a(x,\xi)$ like its asymptotic growth as $|\xi| \to \infty$. Studying properties of these more general operators, and deriving similar norm and regularity estimates as for differential operators, lets us conclude properties about solutions to PDEs, as the examples $L^{-1}$ and $e^{tL}$ above show.

With the theory of pseudodifferential operators, it can be made precise when and in what ways the operator $L^{-1}$ behaves like an order $-m$ `differential' operator, $e^{tL}$ an operator of order $-\infty$, and an operator like $(1-\Delta)^{\frac{s}{2}}$ for $s \in \R$ an operator of  order $s$.

The oldest and most `standard' theory of pseudodifferential operators is that of classical pseudodifferential operators on $\R^d$~\cite{KohnNirenberg1965,Hormander1967}. Its definition is quite technical; we follow~\cite[Chapter~II]{Taylor1981}, and write $\langle \xi \rangle := (1+|\xi|^2)^{\frac{1}{2}}$.

\begin{defn}\label{D:Symbols}
    The symbol class $S^m_{1,0}(\R^d)$, $m\in \R$, is defined as those smooth functions $a \in C^\infty(\R^d \times \R^d)$ such that for all multi-indices $\alpha, \beta \in \N^d$, there exists a constant $C_{\alpha,\beta}\geq 0$ such that 
    \[
    |\partial_x^\alpha \partial_\xi^\beta a(x,\xi)|\leq C_{\alpha,\beta}\langle\xi\rangle^{m-|\alpha|}, \quad x,\xi \in \R^d. 
    \]
    We say that $a \in S^{m}_{1,0}(\R^d)$ is a classical symbol, denoted $a\in S^m_{cl}(\R^d)$, if there exist smooth functions $a_{m-j} \in C^\infty(\R^d \times \R^d)$ for all $j\in \N$ which are homogeneous of order $m-j$,
    \[
    a_{m-j}(x, \lambda \xi) = \lambda^{m-j} a_{m-j}(x,\xi), \quad \lambda,|\xi|\geq1,
    \]
    with the property that
    \[
    a(x,\xi) \sim \sum_{j\geq 0} a_{m-j}(x,\xi),
    \]
    meaning that
    \[
    a(x,\xi)-\sum_{j=0}^Na_{m-j}(x,\xi) \in S^{m-N-1}_{1,0}(\R^d), \quad N\in \N.
    \]
\end{defn}

It is no coincidence that for the differential operator $L$ from before, if each coefficient $a_\alpha(x)$ is smooth and it and its derivatives are bounded, the symbol $p_L(x,\xi) = \sum_{|\alpha| \leq k} a_\alpha(x) (i\xi)^\alpha$ is in the symbol class $S^k_{1,0}(\R^d)$. It is furthermore a classical symbol, as we can write
\[
p_L(x,\xi) = \sum_{j=0}^k \sum_{|\alpha|=j}a_\alpha(x) (i\xi)^\alpha,
\]
where each component $\sum_{|\alpha|=j}a_\alpha(x) (i\xi)^\alpha$ is homogeneous of order $j$. 

\begin{defn}
    The Schwartz functions $\mathcal{S}(\R^d)$ are defined as those $f \in C^\infty(\R^d)$ for which all seminorms
    \[
    \|f\|_{\alpha, \beta}:= \sup_{x\in \R^d} |x^\alpha \partial_x^\beta f(x)|, \quad \alpha, \beta \in \N^d,
    \]
    are finite. 
\end{defn}

\begin{defn}\label{def:ClassicPSDOs}
    A pseudodifferential operator of order $m\in \R$, denoted $T\in \Psi^m(\R^d)$, is an operator $T : \mathcal{S}(\R^d) \to \mathcal{S}(\R^d)$ such that
\[
T f(x) = \frac{1}{(2\pi)^{\frac{d}{2}}} \int_{\R^d} e^{i\langle x, \xi\rangle} a(x,\xi) \hat{f}(\xi) \, d\xi, \quad f \in \mathcal{S}(\R^d),
\]
for some $a \in S^m_{1,0}(\R^d)$ --- see~\cite[Theorem~18.1.6]{HormanderLPDOIII} for a proof that $Tf \in \mathcal{S}(\R^d)$. If furthermore $a \in S^m_{cl}(\R^d)$, we write $T \in \Psi^m_{cl}(\R^d)$, and call $T$ a classical pseudodifferential operator. Finally, 
\[
\Psi^\infty(\R^d):= \bigcup_{m \in \R}\Psi^m(\R^d), \quad \Psi^{-\infty}(\R^d):=\bigcap_{m \in \R} \Psi^m(\R^d).
\]
Similarly we define $\Psi^\infty_{cl}(\R^d):= \bigcup_{m \in \R}\Psi_{cl}^m(\R^d)$. 
\end{defn}

With these constructions, in many ways the pseudodifferential operators $\Psi^\infty(\R^d)$ behave like differential operators, and $\Psi^\infty_{cl}(\R^d)$ mimics this behaviour even more closely. For $S \in \Psi^m(\R^d)$ and $T \in \Psi^{m'}(\R^d)$, we have
\[
S\circ T \in \Psi^{m+m'}(\R^d), \quad [S,T]  \in \Psi^{m+m'-1}(\R^d),
\]
as is the case for differential operators. Furthermore, we can define the Bessel-potential Sobolev spaces (see e.g.~\cite[Section~VII.9]{HormanderLPDOI} or~\cite[Section~2.3.3]{Triebel1978})
\[
\Hc^s:= \overline{\dom (1-\Delta)^{\frac{s}{2}}}^{\|\cdot\|_s}, \quad \|f\|_s:= \| (1-\Delta)^{\frac{s}{2}}f\|_{L_2(\R^d)}, \quad s \in \R,
\]
which form Hilbert spaces. To be very precise, we take $\Delta$ to be the Friedrich's extension of the Laplacian, and $\dom  (1-\Delta)^{\frac{s}{2}}$ indicates the domain of the corresponding operator on $L_2(\R^d)$ defined via functional calculus~\cite[Section~5.3]{Schmudgen2012}. For $s \geq 0$, this domain is complete in the norm $\| \cdot \|_s$,
but for $s < 0$ taking the completion in the norm $\| \cdot \|_s$ is necessary.
For positive integer $s$, the space $\Hc^s$ consists of the $s$ times weakly differentiable $L_2$-functions. For the differential operator $L = \sum_{|\alpha|\leq k} M_{a_\alpha}\partial^\alpha$ as above, if all coefficients $a_\alpha$ and their derivatives are uniformly bounded, then $L$ extends to a bounded operator
\[
L: \Hc^{s+k} \to \Hc^s, \quad s\in \R.
\]
Pseudodifferential operators have the same property: if $T \in \Psi^m(\R^d)$, then $T$ extends to a bounded operator
\[
T: \Hc^{s+m} \to \Hc^s, \quad s \in \R.
\]
Intuitively, $T$ therefore `removes' $m$ degrees of regularity from a function. Note, however, that $m \in \R$ need not be an integer --- or even positive.

All these objects can similarly be constructed on smooth (not necessarily Riemannian) manifolds~\cite[Section~I.4]{Shubin2001}\cite[Chapter~XVIII]{HormanderLPDOIII}. This takes some care, as we need to think about how to handle the Fourier transform on manifolds. Commonly, pseudodifferential operators are defined locally in coordinate patches via the definitions for $\R^d$ above, and stitched together to define global operators on the manifold. However, symbols can be defined intrinsically as functions on the cotangent bundle as well~\cite{Widom1978,Widom1980,Safarov1997}. 

The principal symbol of a pseudodifferential operator is of particular importance, as it can tell us a lot about the operator's (approximate) invertibility. For compact manifolds, we have a short exact sequence of vector spaces
\[
0 \to \Psi_{cl}^{m-1}(M) \to \Psi_{cl}^m(M) \xrightarrow{\sigma_0} C^\infty(S^*M)\to 0,
\]
where the map $\sigma_0$ is called the symbol map, and $S^*M \subseteq T^*M$ is the cotangent sphere. The manifold $M$ is not required to be equipped with a Riemannian metric in order for this to make sense: the cotangent sphere $S^*M$ can be defined independently of such a metric as the rays in the cotangent bundle $T^*M\setminus M$, which is $T^*M$ with the zero vectors removed. The intuitive explanation of this exact sequence is that an essentially homogeneous function $f \in C^\infty(T^*M)$ of order $m$, which satisfies locally (up to a compactly supported error),
\[
f(x,\lambda \xi) = \lambda^m f(x,\xi), \quad \lambda , |\xi|\geq 1,
\]
clearly defines a function on $S^*M$. The map $\sigma_0$ in the exact sequence selects the part of the symbol of the classical pseudodifferential operator that is essentially homogeneous of order $m$ and produces the corresponding function on $S^*M$.

Turning this around, such exact sequences can be used to \textit{define} principal symbols. For example, since on compact manifolds the operators $\Psi^{-1}(M)$ are compact, a map like $\sigma_0$ above, now mapping from order zero operators into the Calkin algebra $B(\Hc
)/K(\Hc)$, is the basis for a paradigm called the $C^*$-algebraic approach to the principal symbol~\cite{Cordes1979, Dao1, Dao2, Dao3}. 

Furthermore, it should be emphasised that there are by now many different pseudodifferential calculi, in different settings, for different purposes, all generalising differential operators in a different way. Fundamentally, these various pseudodifferential calculi are crafted for studying different classes and aspects of PDEs. On both Euclidean space and on manifolds, various symbol classes have been used as a base (variations of Definition~\ref{D:Symbols})~\cite{BealsFefferman1973,Bony2013,NicolaRodino2010}. Other directions of research include pseudodifferential calculi on the Heisenberg group~\cite{BahouriFermanianKammerer2012,FischerRuzhansky2014}, Lie groups~\cite{Melin1983,Taylor1984,MeladzeShubin1986,RuzhanskyTurunen2010}, Lie groupoids~\cite{MonthubertPierrot1997,NistorWeinsteinXu1999}, and types of manifolds like manifolds with boundaries~\cite{Melrose1991,Melrose1993,MazzeoMelrose1998}, Heisenberg manifolds~\cite{BealsGreiner1988, Taylor1984,Ponge2008}, and filtered manifolds~\cite{vanErpYuncken2019,AndroulidakisMohsen2022,Fermanian-KammererFischer2024}.

In this thesis, we will focus on a very general abstract pseudodifferential calculus which strips away any reference to a geometric space. It is the purely operator theoretical skeleton at the basis of any pseudodifferential calculus, and as such it is difficult to give accurate credits towards its invention. Elements of this theory even \textit{predate} the study of pseudodifferential operators via the symbol calculus by Kohn--Nirenberg~\cite{KohnNirenberg1965} and H\"ormander~\cite{Hormander1967}. The abstract Sobolev spaces in this theory already appeared in their current form in work by S. G. Krein where they are called Hilbert scales, see the reviews~\cite{Mitjagin1961, KreinPetunin1966} and references therein. 
Later, Guillemin~\cite{Guillemin1985}, Connes--Moscovici~\cite{ConnesMoscovici1995} and Higson~\cite{Higson2003} amongst others have used general pseudodifferential calculi. Of these, Connes and Moscovici introduced the formalism in NCG.

The basic principle is as follows. We take one operator $\Theta$ as reference, a self-adjoint boundedly invertible operator on a separable Hilbert space $\Hc$. With this operator, define the Sobolev spaces 
\[
\Hc^s(\Theta):= \overline{\dom \Theta^s}^{\|\cdot\|_s}, \quad \|\xi\|_s:= \|\Theta^s\xi\|_{\Hc}, \quad s \in \R,
\]
clearly inspired by the Bessel potential Sobolev spaces from before. And as before, here $\dom \Theta^s$ is defined via the functional calculus for unbounded self-adjoint operators~\cite[Section~5.3]{Schmudgen2012}, which is complete in the norm $\|\cdot\|_s$ for $s\geq 0$ but not for $s<0$. Our `pseudodifferential operators' are simply operators on $\Hc$ which extend or restrict to bounded maps
\[
T: \Hc^{s+r}(\Theta) \to \Hc^{s}(\Theta), \quad s \in \R,
\]
in a consistent and well-defined way. These operators are denoted by $\op^r(\Theta)$, sometimes called operators of analytic order (less than) $r$~\cite{Higson2003}.

In this calculus, there is no notion of symbols, which is at once a strength and a weakness: we will not be able to use any techniques based on symbol manipulations, but as a consequence results are valid in any pseudodifferential calculus. And, importantly, this makes the pseudodifferential calculus perfectly adapted for noncommutative geometry where we do not want to interact directly with a geometric base space in the first place.

Some properties of this calculus have been compiled in~\cite{Uuye2011}. The following concrete examples of this abstract pseudodifferential calculus are borrowed from the exposition in~\cite{MOOIs} (joint work of the author with Edward McDonald and Teun van Nuland).
\begin{itemize}
    \item As remarked before, taking $\Theta = (1-\Delta)^{\frac{1}{2}}$ on $L_2(\mathbb{R}^d)$ gives the classical (Bessel potential) Sobolev spaces $\Hc^s = W_2^s(\mathbb{R}^d)$. (Pseudo)differential operators of order $k$, and (unbounded) Fourier multipliers $T_\phi$ with symbols $|\phi(\xi)|\lesssim (1+|\xi|)^k,\, \xi\in \mathbb{R}^d$ are contained in $\op^k(1-\Delta)^{\frac12}$. Note though that $\op(1-\Delta)^{\frac12}$ is a much larger class than this, and for example also contains translation operators.
    \item  For spectral triples $(\Ac, \Hc, D)$, taking $\Theta= (1+D^2)^{\frac{1}{2}}$ recovers the calculus of Connes and Moscovici as used in noncommutative geometry~\cite{ConnesMoscovici1995}.
    \item Taking $\Theta = (1-\Delta)^{\frac12}$, where $\Delta$ is the sub-Laplacian on a stratified Lie group, gives the Sobolev spaces defined by Folland and Stein~\cite{FollandStein1974,RothschildStein1976}. 
    \item Related to the previous example is the anharmonic oscillator $\Theta^2 = 1+ \Delta^{2l} + |x|^{2k} $ on $\R^d$ for integers $l,k\geq 1$ and generalisations thereof, which define Sobolev spaces and a pseudodifferential calculus that appear in the study of sub-Laplacian operators on stratified Lie groups too~\cite{ChatzakouDelgado2021}. The special case where $\Theta^2$ is the harmonic oscillator gives Shubin's Sobolev spaces $Q^s(\R^d)$~\cite[Section~IV.25]{Shubin2001} and   the $\Gamma$-pseudo-differential operators in~\cite[Chapter~2]{NicolaRodino2010}, see also~\cite{BongioanniTorrea2006}.
    \item In the aforementioned pseudodifferential calculus in Van Erp--Yuncken~\cite{vanErpYuncken2019}, and Androulidakis--Mohsen--Yuncken~\cite{AndroulidakisMohsen2022} on filtered manifolds, Sobolev spaces are constructed with a similar procedure.
    \item A similar calculus has been constructed for quantum Euclidean spaces~\cite{GaoJunge2022}.
    \item Finally, we note the case where $\Theta = 1_{\Hc}$, which gives that $\Hc^s = \Hc$ for all $s \in \R$, and $\op^r(1_{\Hc}) = B(\Hc)$, $r \in \R$.
\end{itemize}

It is a basic fact that the square of an operator in $\op^m(\Theta)$ is an operator in $\op^{2m}(\Theta)$, but what can be said of its square root? More generally, one can wonder for which type of abstract pseudodifferential operators and class of functions there exists a well-behaved functional calculus within the framework of this abstract pseudodifferential calculus.
In Chapter~\ref{Ch:FunctCalc}, we resolve this question by defining abstract \textit{elliptic} pseudodifferential operators, and provide a functional calculus for these operators based on the preprint~\cite{MOOIs}.

\section{Multiple operator integrals}\label{S:MOIs}
To an extent, we have physicists to thank for the field of multiple operator integrals. To see their ingenuity in manipulating operator integrals, simply open any book on Quantum Field Theory. For example, chances are that you will encounter the time-dependent operator
\[
V(t) = e^{itH_0} e^{-it(H_0 + H_I)}, \quad t \in \R,
\]
where $H_0$ and $H_I$ are unbounded self-adjoint operators on a Hilbert space, which is expanded by the Dyson series~\cite[Chapter~6]{Folland2008}
\[
V(t) \sim I + \sum_{n=1}^\infty V_n(t),
\]
where
\[
V_n(t) = \frac{1}{i^n}\int_0^t \int_0^{\tau_n}\cdots \int_0^{\tau_2} e^{i\tau_n H_0}H_I e^{i(\tau_{n-1}-\tau_{n})H_0} H_I  \cdots e^{i(\tau_1-\tau_2)H_0} H_I e^{-i\tau_1 H_0}\,d\tau_1 \cdots d\tau_n.
\]
Of course, the symbols `$=$' and `$\sim$' in these formulas mean something different to physicists than they do to mathematicians, and making these identities mathematically rigorous requires some effort. 
The subject of Multiple Operator Integration provides a rigorous foundation for identities like the Dyson series above and generalisations thereof. We will mostly follow~\cite{Peller2016, SkripkaTomskova2019} in this overview, but see also~\cite{PagterWitvlietSukochev2002, BirmanSolomyak2003,ACDS} for important developments of the subject.

Suppose that we are interested in the G\^ateaux derivative
\[
\frac{d}{dt}\bigg|_{t=0}f(A+tB) = \lim_{h\to 0} \frac{f(A+hB) - f(A)}{h},
\]
for $A, B\in B(\Hc)$ and some holomorphic function $f:\C \to \C$ (using the holomorphic functional calculus), taking the limit with respect to the operator norm. Picking a contour $\gamma$ in the complex plane which encircles counterclockwise the union of the spectra of $A+tB$ for $t \in (-\varepsilon, \varepsilon)$, we have
\[
f(A+tB)  =\frac{1}{2\pi i}\int_{\gamma}f(z) (z-A-tB)^{-1}\,dz, \quad t \in (-\varepsilon, \varepsilon),
\]
as a $B(\Hc)$-valued Bochner integral.
Using the resolvent identity, we have that
\begin{align*}
    \frac{d}{dt}\bigg|_{t=0} (z-A-tB)^{-1} & = \lim_{t\to0} \frac{(z-A-tB)^{-1} - (z-A)^{-1}}{t} \\
    &= \lim_{t\to 0} (z-A-tB)^{-1}B(z-A)^{-1} \\
    &=  (z-A)^{-1}B(z-A)^{-1}.
\end{align*}
Similarly, for the holomorphic function $f$ above we can see that
\begin{equation}\label{eq:HoloDeri}
\frac{1}{n!}\frac{d^n}{dt^n}\bigg|_{t=0}f(A+tB)  =\frac{1}{2\pi i}\int_{\gamma}f(z) (z-A)^{-1} B (z-A)^{-1} B \cdots (z-A)^{-1}\,dz,
\end{equation}
where we have an alternating product of $n+1$ factors of $(z-A)^{-1}$ and $n$ factors of $B$. In case $A$ and $B$ commute, this formula simply reduces to
\[
\frac{d^n}{dt^n}\bigg|_{t=0}f(A+tB) = f^{(n)}(A)B^n. 
\]
When they do not, we can devise a way to describe the right-hand side of~\eqref{eq:HoloDeri}. First define the divided differences $f^{[n]}:\C^{n+1}\to \C$ by
\begin{align*}
f^{[0]}(\lambda)&:= f(\lambda),\\
f^{[n]}(\lambda_0, \ldots, \lambda_n)&:= \frac{f^{[n-1]}(\lambda_0, \ldots, \lambda_{n-1})- f^{[n-1]}(\lambda_1, \ldots, \lambda_n)}{\lambda_0- \lambda_n}, \quad \lambda_0 , \ldots, \lambda_{n-1},\lambda_n\neq \lambda_0 \in \C,
\end{align*}
defining $f^{[n]}$ with an appropriate limit when $\lambda_{0}=\lambda_n$, so that
\[
f^{[n]}(\lambda, \ldots, \lambda) = \frac{1}{n!}f^{(n)}(\lambda), \quad \lambda \in \C.
\]
By similar resolvent manipulations as before, it can be shown that for our holomorphic function $f$, for a contour $\gamma$ encircling $\lambda_0, \ldots, \lambda_n \in \C$ counterclockwise,
\[
f^{[n]}(\lambda_0,\ldots,\lambda_n) = \frac{1}{2\pi i}\int_\gamma f(z) (z-\lambda_0)^{-1}\cdots (z-\lambda_n)^{-1}\, dz.
\]
Hence, intuitively, the expression on the right-hand side of~\eqref{eq:HoloDeri} is something like $f^{[n]}(A,\ldots, A)$ with factors of $B$ spliced `in-between' the $n+1$ copies of $A$.

Richard Feynman dealt with the issue~\cite{Feynman1951} by defining a clever system of notation
\[
\overset{2}{A} \overset{1}{B}\overset{3}{C} := BAC. 
\]
In this notation, we formally have
\[
\frac{d^n}{dt^n}f(A+tB) = f^{[n]}(\overset{1}{A}, \overset{3}{A},\ldots, \overset{2n+1}{A})\overset{2}{B}\overset{4}{B} \cdots \overset{2n}{B}.
\]
Obviously, this needs serious work to make well-defined and rigorous, see~\cite{Feynman1951,Maslov1976}\cite[Appendix~I]{KarasevMaslov1993}\cite[Chapter~7]{Jefferies2004}. Daletskii has made some historical comments on this matter~\cite{Daletskii1998}. 

A different approach, originally developed for studying evolution equations in Hilbert spaces, was devised by Daletskii and S. G. Krein~\cite{DaletskiiKrein1951,DaletskiiKrein1956, Daletskii1998}, elements of which had already appeared independently in work by L\"owner~\cite{Lowner1934}. Their work was continued by Birman--Solomyak~\cite{BirmanSolomyak1,BirmanSolomyak2,BirmanSolomyak3,BirmanSolomyak2003}, Pavlov~\cite{Pavlov1971]}, Sten'kin~\cite{SolomyakStenkin1971,Stenkin1977}, and Peller~\cite{Peller2006, Peller2016}. See also~\cite{ACDS, PagterWitvlietSukochev2002,SkripkaTomskova2019}. Given a function $\phi:\R^{n+1} \to \C$, they define operators corresponding to the formal expression
\[
T^{H_0, \ldots, H_n}_\phi(V_1, \ldots, V_n) :=  \int_{\sigma(H_0)\times \cdots \times \sigma(H_n)} \!\!\!\!\!\!\!\!\!\!\!\!\!\!\!\!\!\!\!\!\!\!\!\!\!\!\! \phi(\lambda_0,\ldots, \lambda_n) \, dE_0(\lambda_0) V_1 dE_1(\lambda_1) \cdots V_n dE_n(\lambda_n),
\]
where $dE_i$ is the spectral measure of the self-adjoint operator $H_i$.
While, like Feynman's approach, this too is only formal notation, in various settings it can be made precise. An important class of examples is the following~\cite{Peller2006,Peller2016}. Assume that $\phi:\R^{n+1} \to \C$ can be written as
\begin{equation}\label{eq:BSrep}
\phi(\lambda_0, \ldots, \lambda_n) = \int_\Omega a_0(\lambda_0, \omega) \cdots a_n(\lambda_n, \omega)\, d\nu(\omega),
\end{equation}
where $(\Omega,\nu)$ is a finite measure space, and the functions $a_i:\R \times \Omega \to \C$ are bounded and measurable. Then we can define 
\begin{equation}\label{eq:MOIs1}
T^{H_0, \ldots, H_n}_\phi(V_1, \ldots, V_n)\psi := \int_\Omega a_0(H_0, \omega)V_1 a_1(H_1,\omega) \cdots V_n a_n(H_n, \omega)\psi\, d\nu(\omega), \quad \psi \in \Hc,
\end{equation}
as an $\Hc$-valued Bochner integral, where $V_1, \ldots, V_n \in B(\Hc)$ and $H_0, \ldots, H_n$ are self-adjoint. Peller proved that this defines a well-defined bounded operator, independent of the representation~\eqref{eq:BSrep} chosen~\cite{Peller2006,Peller2016}. Furthermore, he proved that for functions $f:\R \to \C$ in the Besov space $B^n_{\infty, 1}(\R)$, we have that the divided difference $f^{[n]}$ is of the appropriate form and hence $T^{H_0,\ldots, H_n}_{f^{[n]}}(V_1, \ldots, V_n)$ can be defined. If furthermore $f \in B^n_{\infty,1}(\R)\cap B^1_{\infty,1}(\R)$ and $f^{[n]}$ is bounded, the formula
\[
\frac{1}{n!}\frac{d^n}{dt^n}\bigg|_{t=0}f(A+tB) = T^{A,\ldots, A}_{f^{[n]}}(B, \ldots, B) 
\]
can be justified for unbounded self-adjoint $A$ and bounded self-adjoint $B$~\cite[Theorem~5.6]{Peller2006}. The operators $T^{H_0,\ldots, H_n}_{\phi}(V_1, \ldots, V_n)$ are called multiple operator integrals, and when $n=1$ they are often called double operator integrals.

Multiple and double operator integrals (MOIs and DOIs) are not only useful for studying derivatives, but they also come with the identities
\begin{equation}
    \begin{split}\label{eq:MOIIdentities}
        [f(A),B] &= T^{A,A}_{f^{[1]}}([A,B]),\\
    f(A)-f(B) &= T^{A, B}_{f^{[1]}}(A-B),
    \end{split}
\end{equation}
as well as higher-order analogues. 
This indicates that DOIs and MOIs can appear in many settings, and this framework has indeed been a powerful formalism in functional analysis. The technique has been crucial in Potapov and Sukochev's proof of the long-standing conjecture that Lipschitz functions are operator Lipschitz on the Schatten classes $\mathcal{L}_p$, $1<p<\infty$~\cite{PotapovSukochev2011}, that is, the estimate
\[
\|f(A) - f(B)\|_p \leq C_{p,f} \|A-B\|_p, \quad 1<p<\infty.
\]
The case $0<p<1$ has been studied in~\cite{McDonaldSukochev2022}. DOIs have similarly been used to provide different kinds of sharp operator estimates~\cite{AleksandrovPeller2010,AleksandrovPeller2016,CaspersMontgomery-Smith2014,CaspersPotapov2019}.
Furthermore, MOIs are used in the study of spectral flow and the spectral shift function~\cite{AzamovCarey2007, PotapovSkripka2013}, and in many other areas of functional analysis like Biane--Speicher's stochastic analysis and their free It\^o formula~\cite{BianeSpeicher1998,Nikitopoulos2022}. 

A useful consequence of the two identities~\eqref{eq:MOIIdentities} (and their higher-order analogues) is a noncommutative Taylor expansion:
\begin{equation}\label{eq:NonComTaylor}
    f(D+V) \sim  \sum_{n,m=0}^\infty \sum_{m_1 + \dots + m_n = m} \frac{C_{m_1, \dots, m_n}}{(n+m)!} \delta_D^{m_1}(V) \cdots \delta_D^{m_n}(V) f^{(n+m)}(D),
\end{equation}
where $\delta_D(V) := [D,V]$, and $C_{m_1, \dots, m_n}$ are some combinatorial constants specified below in Section~\ref{S:AsympExp}. If the operators $D$ and $V$ happen to commute, this reduces to a standard Taylor expansion. 
These Taylor expansions should for now be interpreted as purely formal, but they have been made rigorous for classical pseudodifferential operators on manifolds and in Banach algebras~\cite{Paycha2007, Paycha2011, HartmannLesch2024}. 
We too will prove a precise version of the noncommutative Taylor expansion for pseudodifferential operators in Chapter~\ref{Ch:MOIs}, using the abstract pseudodifferential calculus described in Section~\ref{S:PSDOs}.

For us, MOIs connect to the other themes of this dissertation due to their appearance in noncommutative geometry, which has at times gone unrecognised. Take for example the JLO cocycle~\cite{JLO1988}, which functions in NCG as the noncommutative analogue of the Chern character in differential geometry. Given an even spectral triple $(\Ac, \Hc, D)$ with grading $\gamma$, the JLO cocycle is a collection of functionals $(\phi^0, \phi^2, \phi^4, \ldots )$ where $\phi^{n}: \Ac^{\otimes n+1} \to \C$ is defined as
\[
\phi^{n}(a_0 \otimes a_1 \cdots \otimes  a_{n}) := \int_{\Sigma^{n}} \Tr(\gamma a^0 e^{-t_0 D^2} [D,a_1] e^{-t_1D^2} \cdots  [D,a_{n}]e^{-t_{n}D^2})\,dt.
\]
Here, $\Sigma_{n}$ is the $n$-simplex with normalised Lebesgue measure $dt$. Comparing this expression with~\eqref{eq:MOIs1}, it is not difficult to see that we have
\[
\phi^{n}(a_0, a_1, \ldots, a_{n}) = \Tr(\gamma a_0 T_{f^{[n]}}^{D^2, \ldots, D^2}([D,a_1], \ldots, [D,a_{n}])),
\]
where $f(x) = \exp(-x)$.

This cocycle is a cornerstone of the local index formula, a landmark result in NCG which generalises the Atiyah--Singer index theorem to noncommutative spaces. For this result by Connes and Moscovici~\cite{ConnesMoscovici1995}, there exist by now several different proofs~\cite{Higson2003,CPRS1,CPRS2, CPRS3, CPRS4}. Stemming from manipulations of the JLO cocycle, these papers contain many elaborate and technical arguments involving operator integrals. The theory of MOIs can potentially be used to simplify these arguments. In fact, the original statement of the local index formula already suggests this possibility.
\begin{thm}[Local index formula~\cite{ConnesMoscovici1995}]\label{T:LocalIndex} Let $(\Ac, \Hc, D)$ be an odd unital spectral triple with simple dimension spectrum, such that $D^{-1} \in \mathcal{L}_{p,\infty}$ for some $p>0$. 
    The following functionals $\phi^n : \Ac^{\otimes n+1}\to \C$ define a cocycle $(\phi^1, \phi^3, \phi^5, \ldots)$ in the $(b,B)$-bicomplex of $\Ac$:
    \[
    \phi^n(a_0, \ldots, a_n) = \sqrt{2i} \sum_{q\geq 0, k_j \geq 0} c_{n,k,q} \tau_q\big(a_0 \delta_{D^2}^{k_1}([D,a_1]) \cdots \delta_{D^2}^{k_n}([D,a_n]) |D|^{-(n+2\sum k_j)} \big),
    \]
    where $ c_{n,k,q}$ are some constants, and $\tau_q(T) := \mathrm{res}_{z=0} z^q \Tr(T|D|^{-2z})$. The cohomology class of the cocycle $(\phi^1, \phi^3, \phi^5, \ldots)$ in $HC^{odd}(\Ac)$ coincides with the cyclic cohomology Chern character $ch_*(\Ac, \Hc, D)$.
\end{thm}
Even without knowing what a simple dimension spectrum, $(B,b)$-cocycles, cyclic cohomology, or the Chern character are, one can see that the cocycle defined in this theorem involves something resembling a noncommutative Taylor expansion~\eqref{eq:NonComTaylor}. In fact, it more closely resembles an expansion of a single MOI, which we will see in Chapter~3 of this thesis (Proposition~\ref{P:Expansion}).

The occurrence of MOI arguments in NCG is not limited to proofs of the local index formula. They also appear in expansions of the spectral action~\cite{Skripka2014,Skripka2018, Suijlekom2011, vNvS21a} and the heat trace (often studied for a noncommutative analogue of curvature)~\cite{ConnesTretkoff2011, FathizadehKhalkhali2012, FathizadehKhalkhali2013, ConnesMoscovici2014, DabrowskiSitarz2015, LeschMoscovici2016, KhalkhaliSitarz2018, Liu2018, ConnesFathizadeh2019, Liu2022, NulandSukochev2025}. See~\cite{FathizadehKhalkhali2019} for a review of the topic of curvature in NCG. An important tool in the area of these curvature computations is the rearrangement lemma~\cite{ConnesTretkoff2011,ConnesMoscovici2014,lesch2017divided,HartmannLesch2024}, which is closely related to MOIs. See~\cite{NulandSukochev2025} for an approach to curvature via multiple operator integrals.

The theory of MOIs can potentially be a useful tool in all these contexts in NCG, but there is a hurdle to be taken. This is best illustrated by seeing what happens when one takes a spectral triple $(\Ac, \Hc, D)$ and applies an identity like~\eqref{eq:MOIIdentities} when the relevant operators are $D^2$ and $a\in \Ac$ like for the JLO cocycle:
\begin{align*}
    [f(D^2),a] = T^{D^2, D^2}_{f^{[1]}}([D^2,a]).
\end{align*}
Such manipulations of MOIs are (implicitly) common in the literature, and indeed necessary to interpret the local index formula, Theorem~\ref{T:LocalIndex}, as a MOI expansion. The problem is that the operator $[D^2,a]$ appearing as an argument in the MOI on the right-hand side is typically \textit{unbounded} in NCG, and the theory of MOIs as in~\cite{Peller2016, SkripkaTomskova2019} is only equipped to deal with bounded arguments. In the NCG literature, many ad-hoc arguments can be found to justify such integral manipulations~\cite{CPRS1, ConnesMoscovici1995, Higson2003}, but an overarching theory is missing.

In Chapters~\ref{Ch:FunctCalc} and~\ref{Ch:MOIs} of this thesis, this is resolved by developing a theory of multiple operator integrals for abstract pseudodifferential operators (introduced in Section~\ref{S:PSDOs}), based on joint work~\cite{MOOIs} with Edward McDonald and Teun van Nuland. Our results extend those by Paycha~\cite{Paycha2011} who proved a noncommutative Taylor expansion for classical pseudodifferential operators on manifolds, and are comparable to those by Hartmann and Lesch~\cite{HartmannLesch2024} on Banach spaces. Our results are more general than those of Paycha, both in terms of the functions and the operators considered, and more closely adapted to operators appearing in NCG than those in~\cite{HartmannLesch2024}. Furthermore, while the Paycha, Hartmann and Lesch papers prove Taylor expansions, our fundamental construction of MOIs also systematises the underlying techniques involving operator integrals that are widespread in NCG, allowing to justify the arguments illustrated above involving unbounded operators. As a demonstration of the power and adaptability of this formalism, Chapter~\ref{Ch:MOIs} resolves an open question regarding the existence of certain asymptotic expansions in spectral triples, posed by Eckstein and Iochum~\cite{EcksteinIochum2018}.

\section{Connes' integral formula}\label{S:ConnesIntegral}
The fact that Riemannian manifolds can be completely captured by spectral triples, indicates that there must be a way to integrate functions on manifolds spectrally. Alain Connes recognised that Dixmier traces, or more general singular traces, are the appropriate tool.

The Dixmier trace was originally constructed as an example by Dixmier~\cite{Dixmier1966}, which settled in the negative the question whether the operator trace is (up to constants) the only trace that can be defined on $B(\Hc)$. With trace we mean a linear functional $\phi: \mathcal{J} \to \C$ on a two-sided ideal $\mathcal{J}\subseteq B(\Hc)$ which vanishes on commutators,
\[
\phi([A,B]) = 0, \quad A \in \mathcal{J}, B\in B(\Hc).
\]
See~\cite{SukochevUsachev2016} for details on Dixmier's original construction. We will take a different approach, following~\cite{LSZVol1,LMSZVol2}.

\begin{defn}
    The weak Schatten class $\mathcal{L}_{p,\infty} \subseteq B(\Hc)$, $p>0$, consists of those compact operators $A$ for which
    \[
    \lambda(k,|A|) = O(k^{-\frac1p}), \quad k \to \infty,
    \]
    where for a compact operator $T$ we write $\{\lambda(k,T)\}_{k=0}^\infty$ for an eigenvalue sequence of $T$ counted with algebraic multiplicity, ordered in decreasing modulus. These sets are two-sided ideals in $B(\Hc)$. The ideal $\mathcal{L}_{1,\infty}$ is called the weak trace-class ideal.
\end{defn}

\begin{defn}
    An extended limit $\omega \in (\ell_\infty(\N))^*$ is a continuous positive linear functional on $\ell_\infty(\N)$ such that $\omega(1) = 1$, and for each sequence $x \in \ell_\infty(\N)$ that converges to zero, we have $\omega(x) = 0$.

    Associated with each extended limit $\omega$, we can define the Dixmier trace $\Tr_\omega: \mathcal{L}_{1,\infty} \to \C$ by
    \begin{align*}
    \Tr_\omega(A):= \omega \bigg(\bigg\{\frac{1}{\log(n+2)} \sum_{k=0}^n \lambda(k,A)\bigg\}_{n=0}^\infty \bigg).
    \end{align*}
    In general, the value of $\Tr_\omega(A)$ may depend on the choice of $\omega$. If it does not, we call $A$ Dixmier measurable. Note that there is some inconsistency in the literature on the definition of Dixmier measurable operators, we follow~\cite{LSZVol1}. See also~\cite{SukochevUsachev2013, UsachevThesis}.
\end{defn}

At first sight, it might come as a surprise that the Dixmier traces are linear maps or even traces. 
It is furthermore a crucial fact that they vanish on the finite-rank operators (a forteriori, on $\mathcal{L}_1$), a property which characterises so-called singular traces.

There is a rich mathematical theory surrounding Dixmier traces and their more general cousins the singular traces. Originally Dixmier~\cite{Dixmier1966}, and later Connes~\cite{Connes1988}\cite[Section~4.2]{Connes1994}, required additional properties of the extended limits $\omega$ like dilation invariance in their definitions of the Dixmier trace, and defined these on a larger ideal $\mathcal{M}_{1,\infty}$ which is now sometimes called the Macaev--Dixmier ideal. There are many subtle points to be made regarding these differing definitions, or the question when different extended limits define the same Dixmier trace, or what the sufficient and necessary conditions are for an operator to be Dixmier measurable and more. For these subtleties we refer to the books~\cite{LSZVol1,LMSZVol2} and the many references therein.

The most important result illustrating that Dixmier traces can be used for integration, comes from Connes' trace theorem~\cite{Connes1988}. This provides that the noncommutative residue (sometimes called Wodzicki residue or Wodzicki--Guillemin residue) for order $-d$ classical pseudodifferential operators on $d$-dimensional manifolds can be given by a Dixmier trace.

If we equip $M$ with a Riemannian metric $g$, and consider the operator $ M_f(1-\Delta_g)^{-\frac{d}{2}}$ on $L_2(M)$, where  $f\in C_c^\infty(M)$ and $\Delta_g$ is the Laplace--Beltrami operator, Connes' trace theorem reduces to Connes' integral formula.
\begin{thm}[Connes' integral formula~\cite{Connes1988}]\label{T:ConnesIntegral} Let $M$ be a $d$-dimensional Riemannian manifold. For all $f\in C^\infty_c(M)$, we have that $M_f(1-\Delta_g)^{-\frac{d}{2}} \in \mathcal{L}_{1,\infty}$, and for all extended limits $\omega \in \ell_\infty^*$
    \begin{equation}\label{eq:ConnesintegralDix}
    \Tr_\omega(M_f(1-\Delta_g)^{-\frac{d}{2}}) = \frac{\vol(\mathbb{S}^{d-1})}{d(2\pi)^d}\int_Mf\,d\nu_g, \quad f\in C^\infty_c(M).
    \end{equation}
\end{thm}
For compact manifolds and $\R^d$, this theorem admits a generalisation to \textit{all} singular traces on $\mathcal{L}_{1,\infty}$ (not just Dixmier traces), and
all $f\in L_2(M)$, see~\cite{LordPotapov2010,KaltonLord2013},~\cite[Chapters~7,8]{LSZVol1} and~\cite[Chapters~2,3]{LMSZVol2}. Interestingly, $f \in L_2(M)$ is both necessary and sufficient for~\eqref{eq:ConnesintegralDix} to hold~\cite{LordPotapov2010,KaltonLord2013}, illustrating that an extension to $L_1(M)$ is impossible as mistakenly claimed in~\cite[Corollary~7.22]{GVF2001}.

The state of the art for Connes' integral formula can be found in~\cite{LordSukochev2020}, supplemented by~\cite{ZaninSukochev2023}. By modifying the left-hand side of~\eqref{eq:ConnesintegralDix} to the symmetric expression 
\[
\phi ((1-\Delta_g)^{-\frac{d}{4}}M_f(1-\Delta_g)^{-\frac{d}{4}}),
\]
where $\phi$ is a singular trace on the aforementioned Dixmier--Macaev ideal $\mathcal{M}_{1,\infty}$, these papers extend Connes' integral formula to functions $f$ in certain function spaces which include the space $L_p(M)$ for $p>1$ (on a compact manifold $M$). Even in this approach, an extension to $L_1(M)$ is demonstrably impossible~\cite[Lemma~5.7]{LordPotapov2010}.

In the field of noncommutative geometry, Connes' integral formula is of significant philosophical importance. It justifies that for $d$-summable spectral triples $(\Ac, \Hc, D)$ (meaning that $(1+D^2)^{-\frac{d}{2}} \in \mathcal{L}_{1,\infty}$), the functional
\[
a \mapsto \Tr_\omega(a(1+D^2)^{-\frac{d}{2}}), \quad a \in \Ac,
\]
can be interpreted as a noncommutative integral, which allows us to talk about integration on a noncommutative space. After all, for a spectral triple the operator $D^2$ functions as the analogue of the Laplace operator (see Section~\ref{S:NCG}). We refer to~\cite{LordSukochev2010, LSZVol1, LMSZVol2} for thorough studies of the noncommutative integral in this context.

The philosophy goes deeper, however. The integral formula shows that the Lebesgue measure can be completely `recovered' from operator based computations with Dixmier traces. It has led Connes to craft a paradigm where he considers compact operators to be infinitesimals, with the weak trace-class $\mathcal{L}_{1,\infty}$ being infinitesimals of order $1$, and the Dixmier trace being the integral~\cite{Connes1994}. This `quantised calculus' is a completely operator-theoretical approach to integration theory. In the words of Alain Connes: ``I~believe that \textit{most} integrals that we know of, are special cases of this''~\cite{ConnesYT}, i.e. are computable via a Dixmier trace.

These Dixmier trace formulas have some advantages over the usual approach to integration. Due to the nature of the Dixmier trace, properties like convergence or positivity of an integral are easy to see on the operator-theoretical side. Furthermore, its robustness towards trace-class perturbations (recall that $\Tr_\omega(T) =0$ if $T$ is trace-class) can lead to deductions about similar robustness properties for the integral in question~\cite[p.558]{Connes1994}.

It should be remarked that in the noncommutative geometry literature, a popular alternative definition for the noncommutative integral is the functional
\[
a \mapsto \mint{-}a(1+D^2)^{\frac{s}{2}} := \mathrm{res}_{s=0} \Tr(a(1+D^2)^{\frac{1+s}{2}}), \quad a \in \Ac.
\]
Whenever $s \mapsto \zeta_{a,D}(s):=\Tr(a(1+D^2)^{\frac{1+s}{2}})$ can be extended meromorphically to a neighbourhood of the origin, we have for all extended limits $\omega \in \ell_\infty^*$~\cite[Theorem~9.1.5(a)]{LSZVol1}
\[
 \mathrm{res}_{s=0} \Tr(a(1+D^2)^{\frac{1+s}{2}}) = \Tr_\omega(a(1+D^2)^{\frac12}).
\]
In general, however, it need not be the case that $\zeta_{a,D}$ allows such an extension, in which case $\mint{-}a(1+D^2)^{\frac12}$ is not well-defined. Counterexamples are given in~\cite[Lemma~17]{CareySukochev2012}\cite[Corollary~6.34, Corollary~7.23]{KaltonLord2013}\cite[Example~9.3.3]{LSZVol1}, of which
\cite[Corollary~6.34, Corollary~7.23]{KaltonLord2013} is notable for providing an example on a $d$-dimensional manifold of a compactly supported pseudodifferential operator of order $-d$ with such behaviour. See also~\cite[Chapters~8,9]{LSZVol1} for extensive analysis on the relation between the Dixmier trace and residue formulas. Throughout this thesis, to maintain maximal generality, we will stick with the Dixmier trace formulation of the noncommutative integral.

In Chapter~\ref{Ch:QE}, the noncommutative integral is studied in relation to the paradigm of spectral truncations in NCG due to Connes and Van Suijlekom~\cite{ConnesvSuijlekom2021,ConnesvSuijlekom2022}. The result provides a vehicle to relate NCG to a field of mathematics called Quantum Ergodicity. In Chapters~\ref{Ch:DOSDiscrete} and~\ref{Ch:DOSManifolds}, we provide a Dixmier trace formula for integrals with respect to a measure called the density of states, on respectively discrete metric spaces and open manifolds based on the papers~\cite{AHMSZ,HekkelmanMcDonald2024}, validating Connes' mantra in this context.

\section{Density of states}\label{S:DOS}
Parts of this section are borrowed from~\cite{AHMSZ, HekkelmanMcDonald2024}, respectively a joint work with Nurulla Azamov, Edward McDonald, Fedor Sukochev and Dmitriy Zanin, and a joint work with Edward McDonald.

Many electrical and thermal properties of a material can be deduced from its associated density of states (DOS), a construction in solid-state physics. Loosely speaking, the DOS describes how many quantum states are admitted for the electrons in the material at each energy level per unit volume. 
Over the years the topic has kept of plenty of mathematicians and mathematical physicists occupied, see for example the books~\cite{PasturFigotin1992,Veselic2008,AizenmanWarzel2015}. 

The rigorous definition of the DOS already takes some effort, and there exist differing approaches. In this thesis we follow Simon~\cite[Section C]{Simon1982}. Given a (possibly unbounded) lower-bounded self-adjoint operator $H$ on the Hilbert space $L_2(X)$ where $X$ is some metric space with a Borel measure written as $|\cdot|$, we consider the limits
\[
\lim_{R\to \infty} \frac{1}{|B(x_0,R)|}\mathrm{Tr}(f(H)M_{\chi_{B(x_0,R)}}), \quad f\in C_c(\R),
\]
where $B(x_0,R)$ denotes the closed ball with center $x_0 \in X$ and radius $R$, and $\chi_{B(x_0,R)}$ is its indicator function. When these limits exist (this includes assuming that $f(H)M_{\chi_{B(x_0,R)}}$ is trace-class in the first place), we have a positive linear functional on $C_c(\R)$ and hence, via the Riesz--Markov--Kakutani theorem, we obtain a Borel measure $\nu_H$ on $\R$~\cite[Proposition C.7.2]{Simon1982} such that
\[
\lim_{R\to \infty} \frac{1}{|B(x_0,R)|}\mathrm{Tr}(f(H)M_{\chi_{B(x_0,R)}}) = \int_{\mathbb{R}}f \,d\nu_H, \quad f\in C_c(\mathbb{R}).
\]
This measure, if it exists, is what we call the density of states of the operator $H.$ This essentially coincides with similar definitions elsewhere in the literature, at least for Schr\"odinger operators on Euclidean space~\cite{Simon1982}. Important to observe is that while we define the DOS via increasing balls, in principle we could also consider taking another increasing sequence of sets $\{\Omega_n\}_{n\in \N}$. However, the DOS in general depends on the choice of this increasing sequence of sets, even for Schr\"odinger operators with radially homogenous potentials on~$\R^d$~\cite{AzamovMcDonald2022}.

Apart from its origin in physics, the density of states can be considered as a substitute for the spectral counting function~\cite{Strichartz2012} which featured in Section~\ref{S:NCG},
\[
N(\lambda) := \# \{ k:\lambda_k\leq \lambda\} =\Tr(\chi_{(-\infty, \lambda]}(H)),
\]
in case $H$ does not have discrete spectrum. Indeed, when $X$ has finite diameter and $|X|<\infty$, every self-adjoint operator $H$ with compact resolvent admits a DOS and we have
\[
\nu_H(-\infty, \lambda] = \frac{N(\lambda)}{|X|}.
\]

Common research areas for studying the DOS are its existence~\cite{AizenmanWarzel2015, BerezinShubin1991, CarmonaLacroix1990, DoiIwatsuka2001, PasturFigotin1992, Shubin1979, Simon1982}, the analytical properties of the function $\lambda \mapsto \nu_H(-\infty, \lambda]$~\cite{AizenmanWarzel2015, BourgainKlein2013, CarmonaLacroix1990, PasturFigotin1992} and its asymptotic behaviour as $\lambda$ approaches boundaries of the support of $\nu_H$~\cite{AizenmanWarzel2015, BourgainKlein2013, CarmonaLacroix1990, Lang1991, PasturFigotin1992}. The study of the Anderson localisation phenomenon, of interest to mathematical physicists, is inseparably related to the study of the density of states~\cite{AizenmanWarzel2015, CarmonaLacroix1990, Lang1991}.

The DOS can be connected to NCG, which will be the focus of the studies of the DOS in this dissertation. Connes' mantra that most integrals can be computed with Dixmier traces was mentioned in Section~\ref{S:ConnesIntegral}, and since the DOS is a measure associated with operators, it is a good case study for Connes' philosophy. This was first observed and investigated by Nurulla Azamov, Edward McDonald, Fedor Sukochev and Dmitriy Zanin on $\R^d$ for Schr\"odinger operators~\cite{AzamovMcDonald2022}.

Let us consider the Laplacian on $\R^d$, $\Delta = \sum_{j=1}^d \partial_j^2$. We know that the classical Weyl law covered in Section~\ref{S:NCG} (equation~\eqref{eq:WeylLaw}) gives for the Laplacian $\Delta_\Omega$ on a bounded domain $\Omega \subseteq \R^d$ with Dirichlet boundary conditions, that
\begin{align*}
    \Tr(\chi_{(-\infty,\lambda]}(-\Delta_\Omega)) &\sim \frac{\vol(\mathbb{S}^{d-1})}{d(2\pi)^{d}}\vol(\Omega) \lambda^{\frac{d}{2}}, \quad \lambda \to \infty.
\end{align*}
It might therefore not come as a surprise that we have the limit
\[
\lim_{R\to \infty} \frac{1}{|B(x_0,R)|} \Tr(\chi_{[-\infty,\lambda)}(-\Delta) M_{\chi_{B(x_0,R)}} ) =  \frac{\vol(\mathbb{S}^{d-1})}{d(2\pi)^{d}} \lambda^{\frac{d}{2}}, \quad \lambda>0,
\]
where we now have an exact equality instead of an asymptotic identity. See~\cite{Strichartz2012} for a precise proof. In fact, the Laplacian admits a DOS $\nu_{-\Delta}$ which satisfies
\begin{equation}\label{eq:IntegratedDOS}
    \nu_{-\Delta}(-\infty, \lambda]  = \frac{\vol(\mathbb{S}^{d-1})}{d(2\pi)^{d}} \lambda^{\frac{d}{2}}, \quad \lambda >0,
\end{equation}
and more generally,
\begin{equation}\label{eq:DOSMeasure}
d\nu_{-\Delta}(\lambda) = \frac{\vol(\mathbb{S}^{d-1})}{2(2\pi)^{d}} \lambda^{\frac{d}{2}-1} d\lambda, \quad \lambda > 0,
\end{equation}
where $d\lambda$ is the Lebesgue measure on $\R$, see e.g.~\cite{Simon1982,Strichartz2012, AMSZ}. The function~\eqref{eq:IntegratedDOS} is sometimes called the integrated DOS.

The following argument was presented in~\cite{AzamovMcDonald2022}. Recall Connes' integral formula (Theorem~\ref{T:ConnesIntegral}), which gives for any extended limit $\omega \in \ell_\infty^*$,
\[
\Tr_\omega(M_f (1-\Delta)^{-\frac{d}{2}}) = \frac{\vol(\mathbb{S}^{d-1})}{d(2\pi)^d}
\int_{\R^d}f(x)\, dx, \quad f\in C_c(\R^d).
\]
Since the Dixmier trace is invariant under unitary transformations, we can perform a Fourier transform on the left-hand side to deduce that
\[
\Tr_\omega(f(-i\nabla) M_{\langle x \rangle}^{-d}) =\frac{\vol(\mathbb{S}^{d-1})}{d(2\pi)^d} \int_{\R^d}f(x) \,dx, \quad f \in C_c(\R^d).
\]
Suppose now that $f$ is radially symmetric and $f(x) = g(|x|^2)$ for a continuous compactly supported function $g \in C_c[0,\infty)$. Then, $f(-i\nabla) = g(-\Delta)$, and switching to spherical coordinates we get
\begin{align*}
\Tr_\omega(g(-\Delta) M_{\langle x \rangle}^{-d}) &=\frac{\vol(\mathbb{S}^{d-1})}{d(2\pi)^d} \int_{\R^d}g(|x|^2) \,dx\\
&=\frac{\big(\vol(\mathbb{S}^{d-1})\big)^2}{d(2\pi)^d} \int_{0}^\infty g(r^2) r^{d-1} \,dr\\
&=\frac{\big(\vol(\mathbb{S}^{d-1})\big)^2}{2d(2\pi)^d} \int_{0}^\infty g(\lambda) \lambda^{\frac{d}{2}-1} \,d\lambda.
\end{align*}
Hence, we find that
\begin{equation}\label{eq:DixDOSLaplace}
\Tr_\omega(g(-\Delta) M_{\langle x \rangle}^{-d}) = \frac{\vol(\mathbb{S}^{d-1})}{d} \int_{\R} g(\lambda) \,d\nu_{-\Delta}(\lambda), \quad g \in C_c(\R),
\end{equation}
which is a Dixmier trace formula for the DOS of the Laplacian.
The results of~\cite{AzamovMcDonald2022} extend the identity~\eqref{eq:DixDOSLaplace} to Schr\"odinger operators $-\Delta + M_V$ which admit a DOS and where $V \in L_\infty(\R^d)$ is real-valued.

Chapters~\ref{Ch:DOSDiscrete} and~\ref{Ch:DOSManifolds} in this thesis focus on extending formula~\eqref{eq:DixDOSLaplace} to discrete metric spaces and manifolds respectively, based on the papers~\cite{AHMSZ,HekkelmanMcDonald2024}.
The DOS in these settings has previously been studied, as discrete spaces can serve as a model or a discrete approximation~\cite{AizenmanWarzel2015, AlexanderOrbach1982, BourgainKlein2013, CarmonaLacroix1990, ChayesChayes1986, Hof1993, KirschMuller2006, PasturFigotin1992, Veselic2005, Wegner1981}, and manifolds allow studying the DOS in even greater generality~\cite{AdachiSunada1993, LenzPeyerimhoff2008, LenzPeyerimhoff2004, PeyerimhoffVeselic2002, Veselic2008}.

Furthermore, a link between the DOS and Dixmier traces has previously appeared in the influential work by Bellissard, van Elst and Schulz-Baldes~\cite{BellissardVanElst1994}, which provides an NCG explanation of the quantum Hall effect. Noting the relation between Dixmier traces and $\zeta$-function residues~\cite{LSZVol1}, a result by Bourne and Prodan~\cite[Lemma~6.1]{BourneProdan2018} also bears some resemblance to formula~\eqref{eq:DixDOSLaplace}.

A notorious aspect of the DOS is that the existence of the DOS cannot be guaranteed in general situations. One advantage of the Dixmier trace formula~\eqref{eq:DixDOSLaplace} is that the Dixmier trace on the left-hand side is guaranteed to be well-defined even if the DOS itself has not been shown to exist. In fact, it can be interpreted as a generalisation of the definition of the DOS. This is unlikely to be any good if this `DOS' depends on the choice of extended limit, so one might at least require $g(H)M_{\langle x \rangle}^{-d}$ to be Dixmier measurable. In Chapter~\ref{Ch:DOSDiscrete} we will confirm on discrete spaces that this is a strictly weaker requirement than existence of the DOS, but in other settings and for Schr\"odinger type operators it is not yet precisely known what the relation is between Dixmier measurability and existence of the DOS.

Finally, the Dixmier trace formula allows questions regarding the DOS to be answered with operator theoretical machinery. We will see examples of this in Chapters~\ref{Ch:DOSDiscrete} and~\ref{Ch:DOSManifolds}, where questions like translation invariance of the DOS are answered in this manner. We will also see a connection between the Dixmier trace formula for the DOS and index theory, namely, Roe's index theorem on open manifolds.

\section{Preliminary material}\label{S:IntroPreliminaries}
This section serves as a brief overview of common notation, convention, and definitions used throughout the rest of this dissertation. 
The covered material is standard and can be found in~\cite{LSZVol1, LMSZVol2, Simon2005}. Some of the content has been copied from~\cite{HekkelmanMcDonald2024}.

The natural numbers are by convention $\N = \{0,1,2,\ldots\}$. For functions $f, g : \R \to \R_{>0}$ (or with other appropriate domains and co-domains) we write
\[
f(x) \sim g(x), \quad x\to \infty,
\]
to mean that
\[
\lim_{x\to \infty}\frac{f(x)}{g(x)} = 1.
\]

All Hilbert spaces are assumed to be complex separable Hilbert spaces, denoted $\mathcal{H}$, on which the inner product $\langle \cdot, \cdot\rangle$ is linear in its second component. The set of bounded operators on $\Hc$ is denoted by $B(\mathcal{H})$ and the ideal of compact operators by $K(\mathcal{H}).$ The operator norm on $B(\mathcal{H})$ is denoted $\|\cdot\|.$ For any compact operator $T \in K(\mathcal{H})$, an eigenvalue sequence $\lambda(T) = \{\lambda(k,T)\}_{k=0}^\infty$ is a sequence of the eigenvalues of $T$ listed with algebraic multiplicity, ordered such that $ \{|\lambda(k,T)|\}_{k=0}^\infty$ is non-increasing. The singular value sequence $\mu(T) = \{\mu(k,T)\}_{k=0}^\infty$ of $T$ is defined by
\[
\mu(k,T) :=\lambda(k,|T|), \quad k \geq 0.
\]
Equivalently,
\[
    \mu(k,T) = \inf\{\|T-R\|\;:\;\mathrm{rank}(R)\leq k\}.
\]
Let $\ell_\infty$ denote the space of complex-valued bounded sequences indexed by $\mathbb{N}.$ For $x\in \ell_\infty$, we will denote by $\mu(x) = \{\mu(k,x)\}_{k=0}^\infty \in \ell_\infty$ the decreasing rearrangement of $ \{|x_k|\}_{k=0}^\infty.$ This is consistent with the notation for the singular value sequence of an operator in the sense that if $\mathrm{diag}(x)$ is the operator given by a diagonal matrix with entries $\{x_k\}_{k=0}^\infty,$ then $\mu(\mathrm{diag}(x)) = \mu(x).$

For $p,q\in (0,\infty)$, we define the Lorentz sequence spaces $\ell_{p,q}$, $\ell_{p,\infty}$ and $\ell_{\infty, q}$ as the spaces of sequences $x\in \ell_\infty$ such that
\[
\| x \|_{p,q} := \left( \sum_{k=0}^\infty (k+1)^{\frac{q}{p}-1} \mu(k,x)^q \right)^{\frac{1}{q}}<\infty,
\]
\[
\| x \|_{p, \infty} := \sup_{k\geq 0}\, (k+1)^{\frac{1}{p}} \mu(k,x) <\infty,
\]
and
\[
\| x \|_{\infty, q} := \left( \sum_{k=0}^\infty (k+1)^{-1} \mu(k,x)^q \right)^{\frac{1}{q}}<\infty,
\]
respectively. The space $\ell_{\infty, \infty}$ is defined as $\ell_{\infty, \infty} := \ell_\infty$, and $\ell_{p,p}$ is denoted as $\ell_p$. Using the previously defined singular value sequences, we give the definition of the Lorentz ideals $\mathcal{L}_{p,q}$ for $p,q \in (0,\infty]$ as the quasi-normed spaces of compact operators $T$ such that
\[
\|T\|_{p,q} := \| \mu(T)\|_{\ell_{p,q}} <\infty.
\]
Like for the sequence spaces, $\mathcal{L}_{p,p}$ is denoted as $\mathcal{L}_p$, and $\mathcal{L}_{\infty, \infty} := B(\mathcal{H})$. Indeed, consistent with these definitions, the operator norm $\| \cdot \|$ on $B(\mathcal{H})$ is sometimes denoted by $\| \cdot \|_\infty$ if confusion might arise with other norms. 

The spaces $\mathcal{L}_{p,q}$, $0 < p,q \leq \infty$ form two-sided ideals in $B(\Hc)$. 
The ideals $\mathcal{L}_p$ are called Schatten classes, $\mathcal{L}_1$ is the familiar ideal of trace-class operators (on which we can define the usual operator trace $\mathrm{Tr}$), and $\mathcal{L}_{1,\infty}$ is called the ideal of weak trace-class operators. The ideal $\mathcal{L}_2$ is called the Hilbert--Schmidt class.
All of these quasi-norms have the property that, for $p,q \in (0,\infty],$
\[
\|ABC\|_{p,q} \leq \|A\|_{\infty}  \|B\|_{p,q}  \|C\|_{\infty}, \quad B \in \mathcal{L}_{p,q}, A, C \in B(\mathcal{H}).
\]

In similar spirit, we define the following function spaces.
\begin{defn}\label{D:LpSpaces}
    Let $(X, \Sigma, \nu)$ be a measure space. We write $L_\infty(X)$ for the quotient space
    \[
     L_\infty(X) := \{ f:X\to \C \text{ measurable} \; | \; \|f\|_\infty<\infty\}\big/\sim,
    \]
    where $\|\cdot\|_\infty$ indicates the essential supremum, and $f \sim g$ if $f(x) = g(x)$ $\nu$-almost everywhere.
    We define similarly for $0<p<\infty$,
    \[
    L_p(X) := \Big\{ f:X\to \C \text{ measurable}  \; | \; \|f\|_p:=\bigg(\int_X |f|^p\,d\nu\bigg)^{\frac1p}<\infty \Big\}\Big/\sim.
    \]
\end{defn}

There is a way to define the spaces $L_{p,q}(X)$ as well, and to furthermore see the spaces $L_{p,q}(X)$, $\mathcal{L}_{p,q}$ and $\ell_{p,q}$ as instances of a more general integration theory on von Neumann algebras~\cite{DPS2023}. We will not delve into that matter in this thesis though.

Let $\mathcal{J}$ be a two-sided ideal in $B(\Hc)$. A linear functional $\phi: \mathcal{J} \rightarrow \mathbb{C}$ is called a trace if it satisfies
\[
\phi(BT ) = \phi(TB), \quad T \in \mathcal{J}, B \in B(\Hc).
\]
This condition is equivalent with requiring
\[
\phi(U^*TU) = \phi(T), \quad T \in \mathcal{J}, U\in U(\Hc),
\]
where $U(\Hc)$ are the unitary operators on $\Hc$. Traces $\phi$ on $\mathcal{J}$ that vanish on finite-rank operators called singular traces.

All traces on $\mathcal{L}_{1,\infty}$ are singular traces. In contrast with the situation on $\mathcal{L}_1$, continuous traces on $\mathcal{L}_{1,\infty}$ are far from unique. A particular kind of continuous traces can be defined on $\mathcal{L}_{1,\infty}$ which are called Dixmier traces, as covered in Section~\ref{S:ConnesIntegral}, repeated here for convenience.
\begin{defn}
    An extended limit $\omega \in (\ell_\infty(\N))^*$ is a continuous positive linear functional on $\ell_\infty(\N)$ such that $\omega(1) = 1$, and for each sequence $x \in \ell_\infty(\N)$ that converges to zero, we have $\omega(x) = 0$.

    Associated with each extended limit $\omega$, we can define the Dixmier trace $\Tr_\omega: \mathcal{L}_{1,\infty} \to \C$ by
    \begin{align*}
    \Tr_\omega(A):= \omega \bigg(\bigg\{\frac{1}{\log(n+2)} \sum_{k=0}^n \lambda(k,A)\bigg\}_{n=0}^\infty \bigg).
    \end{align*}
    In general, the value of $\Tr_\omega(A)$ may depend on the choice of $\omega$. If it does not, we call $A$ Dixmier measurable~\cite{LSZVol1}.
\end{defn}
More can be read about singular and Dixmier traces in~\cite{LSZVol1,LMSZVol2}, for the topic of measurability see~\cite{LSZVol1, UsachevThesis, SukochevUsachev2013}.

\begin{defn}\label{D: V-modulated}
    Given an operator $0 \leq V\in B(\Hc)$, an operator $A \in B(\Hc)$ is said to be $V$-modulated if
    \begin{equation}\label{def V-modulated}
        \sup_{t>0} t^{\frac{1}{2}}\|A(1+tV)^{-1}\|_2<\infty.
    \end{equation}
\end{defn}

As can be seen from the definition, 
a $V$-modulated operator is necessarily Hilbert--Schmidt.
The importance of $V$-modulated operators comes from the following theorem, see~\cite[Theorem 7.1.3]{LSZVol1}. We will make use of this theorem in Chapters~\ref{Ch:QE} and~\ref{Ch:DOSDiscrete}.
\begin{thm}\label{T: Modulated}
If $V \in \mathcal{L}_{1,\infty}(\Hc)$ is strictly positive, $T$ is $V$-modulated and $\{e_k\}_{k=0}^\infty$ is an orthonormal basis
such that $Ve_k = \mu(k,V)e_k$, then as $n\to \infty$,
\begin{equation}\label{original_expectation_values}
    \sum_{k=0}^n \lambda(k,T) = \sum_{k=0}^n \langle e_k,Te_k\rangle +O(1).
\end{equation}
\end{thm}

Note that if $0 \leq V \in \mathcal{L}_{1,\infty}(\Hc)$, then $V$ is automatically $V$-modulated~\cite[Lemma 7.3.4]{LSZVol1}, and if $A$ is bounded and $T$ is $V$-modulated, then $AT$ is $V$-modulated,
which directly follows from \eqref{def V-modulated}.

All manifolds in this thesis are smooth, oriented manifolds. 
The unique volume form on an (oriented) Riemannian manifold $(M,g)$ is written as $\nu_g$, which is given in local coordinates as
\[
\nu_g = \sqrt{\abs{\det(g)}} dx^1 \wedge \cdots \wedge dx^d,
\]
where $d$ is the dimension of $M$. We then define $L_2(M)$ as in Definition~\ref{D:LpSpaces}. On $L_2(M)$ we will often consider the Laplace--Beltrami operator $\Delta_g$, which is given in local coordinates~as
\begin{align}\label{eq:Laplacian}
\Delta_g f = \frac{1}{\sqrt{\abs{\det(g)}}} \partial_i \Big( \sqrt{\abs{\det(g)}} g^{ij} \partial_j f \Big), \quad f \in C_c^\infty(M).
\end{align}
When not mentioned otherwise, we will consider this operator on $L_2(M)$ as the Friedrichs extension of 
~\eqref{eq:Laplacian}~\cite[Chapter~10]{Schmudgen2012}.
Explicitly, we put
\[
\dom \Delta_g = H^2(M),
\]
where $H^k(M)$ are the standard Sobolev spaces on Riemannian manifolds, see e.g.~\cite[Chapter~7]{TriebelTFSII}.
With this domain, $\Delta_g$ is a negative self-adjoint operator on $L_2(M)$.

    \part{The Joys of MOIs}
    \chapter{Functional calculus for abstract pseudodifferential operators}
\label{Ch:FunctCalc}
{\setlength{\epigraphwidth}{\widthof{Take care, and work harder.}}
\epigraph{Take care, and work harder.}{Fedor Sukochev}}
This chapter has appeared in slightly modified form as part of the preprint~\cite{MOOIs}, joint work with Edward McDonald and Teun van Nuland. We thank Dmitriy Zanin for help with Proposition~\ref{P:Invertop}. The review of the abstract pseudodifferential calculus in Section~\ref{SS:Review} contains established results, the results and definitions in the other sections are novel. The main results of this chapter are the functional calculus for positive-order pseudodifferential operators in Theorem~\ref{T:MainFunctCalc}, and the zero-order case in  Theorem~\ref{T:FunctCalcOp0}.

Much of the motivation for this chapter has already been covered in Sections~\ref{S:PSDOs} and~\ref{S:MOIs}. We will look at a scale of Hilbert spaces $\{\Hc_s\}_{s\in \R}$ defined by a single operator $\Theta$, which is a very plain generalisation of Sobolev spaces. On this scale, operators can be defined which map in some well-defined and consistent way
\[
T: \Hc^{s+m} \to \Hc^s, \quad s \in \R,
\]
which forms a  bare-bones abstract pseudodifferential calculus. We will develop a functional calculus for these operators, which will therefore be applicable to any pseudodifferential calculus. Though this is interesting in its own right, we are motivated by developing multiple operator integrals for these operators in Chapter~\ref{Ch:MOIs}, which in turn is motivated by applications in noncommutative geometry. 

Usually, defining a functional calculus for a pseudodifferential calculus involves elliptic operators~\cite{Strichartz1972,MeladzeShubin1987,Bony2013}. We too will follow this strategy, although it breaks down for zero-order operators. For the latter, we can instead directly apply work by Davies~\cite{Davies1995a,Davies1995c}.

\section{Abstract pseudodifferential calculus}

\subsection{Review}\label{SS:Review}
First, the details of the abstract pseudodifferential calculus that we use.  See also~\cite{KreinPetunin1966,Daletskii1967,ConnesMoscovici1995,Uuye2011}.

\begin{defn}\label{def:pseudocalc1}
    Let $\Theta$ be a possibly unbounded invertible positive self-adjoint operator on a separable Hilbert space $\Hc$. Define the Hilbert spaces $\Hc^s := \overline{\dom \Theta^s}^{\| \cdot \|_s}$ for $s \in \mathbb{R}$ where $\|\phi\|_{s} : = \| \Theta^s \phi\|$ -- noting that taking this completion is unnecessary for $s\geq 0$. We write $\Hc^\infty := \bigcap_{s \geq 0} \Hc^s$, which is dense in $\Hc$. 
\end{defn}

This definition dates back to work by S. G. Krein, see the reviews~\cite{Mitjagin1961,KreinPetunin1966} and references therein. Further developments were made in~\cite{Daletskii1967,Bonic1967}. Furthermore, the space $\Hc^\infty$ is a `countably Hilbert space' in the terminology of~\cite[Section~I.3]{GelfandVilenkin1964}.

A quick first observation is that the embedding $\Hc^{s} \hookrightarrow \Hc^{t}$ is continuous for all $s \geq t$.
For $s > 0$, there is a pairing between $\Hc^s$ and $\Hc^{-s}$ given by
\[
\langle u, v\rangle_{(\Hc^s, \Hc^{-s})} :=\langle  \Theta^s u, \Theta^{-s} v \rangle_\Hc, \quad u \in \Hc^s, v\in \Hc^{-s}.
\]
This pairing identifies $\Hc^{-s}$ with the (continuous) anti-linear dual space of $\Hc^s$ and vice-versa.

The space $\Hc^\infty = \bigcap_{s\in \R} \Hc^s$ is a Fr\'echet space equipped with the norms $\| \cdot \|_s$, $s\in \mathbb{R}$. By construction, $\Hc^\infty \subseteq \Hc^s$ for any $s \in \R$, and in fact $\Hc^\infty$ is dense in $\Hc^s$. Since a subspace of a separable metric space is itself separable it follows that every $\Hc^s$ admits an orthonormal basis consisting of vectors in $\Hc^\infty$.

We define $\Hc^{-\infty}$ as the continuous anti-linear dual space of $\Hc^\infty$, which can be identified with
\begin{equation}\label{eq:distributions}
    \Hc^{-\infty} = \bigcup_{s\in \mathbb{R}} \Hc^{s}.
\end{equation}
This is an $LF$-space, in the sense of~\cite[Chapter~13]{Treves1967}. From this perspective $\Hc^\infty$ can be interpreted as a Schwartz space and $\Hc^{-\infty}$ as a space of distributions. In case $\Theta^{-1} \in \mathcal{L}_s$ for some $s>0$, we have that the Gelfand triples $\Hc^\infty \subset \Hc \subset \Hc^{-\infty}$ and $\Hc^{s} \subset \Hc \subset \Hc^{-s}$ are `rigged Hilbert spaces' in the terminology of~\cite[Section~I.4]{GelfandVilenkin1964}.

Given $u \in \Hc^\infty$ and $v \in \Hc^{-\infty}$ it follows from~\eqref{eq:distributions} that $v \in \Hc^{-s}$ for some particular $s \in \R$. It is immediate that $u \in \Hc^s$, and we have
\[
\langle u, v \rangle_{(\Hc^{\infty}, \Hc^{-\infty})} = \langle u, v \rangle_{(\Hc^{s}, \Hc^{-s})}.
\]

\begin{prop}\label{P:Interpolation}
    The Sobolev spaces $\Hc^s$ in Definition~\ref{def:pseudocalc1} form an exact interpolation scale. That is, let $s_0 \leq s_1$, $r_0,r_1\in \R,$ and let $0<\theta<1.$ Set
        \[
            s_\theta := (1-\theta)s_0+\theta s_1,\quad r_\theta = (1-\theta)r_0+\theta r_1.
        \]
        If $T$ is a bounded linear map 
        \[
            T:\Hc^{s_0}\to \Hc^{r_0},\quad T|_{\Hc^{s_1}}:\Hc^{s_1}\to \Hc^{r_1},
        \]
        then $T|_{\Hc^{s_\theta}}$ is bounded from $\Hc^{s_\theta}$ to $\Hc^{r_{\theta}}$ for every $\theta.$ Moreover we have
        \[
            \|T\|_{\Hc^{s_\theta}\to \Hc^{r_{\theta}}} \leq \|T\|_{\Hc^{s_0}\to \Hc^{r_0}}^{1-\theta}\|T\|_{\Hc^{s_1}\to \Hc^{r_1}}^{\theta}.
        \]
\end{prop}
\begin{proof}
    See~\cite{KreinPetunin1966}. Or, after identifying $\Hc^s$ with a weighted $L_2$-space through the spectral theorem, this follows from the Stein--Weiss interpolation theorem for $L_p$-spaces~\cite[Theorem~5.4.1]{BerghLofstrom1976}.
\end{proof}

\begin{defn}\label{D:AbstractPSDO}
    We say that a linear operator $A: \Hc^\infty \to \Hc^\infty$ is in the class $\operatorname{op}^r(\Theta)$ (it has analytic order $\leq r$) if $A$ extends to a continuous operator
    \[
    \overline{A}^{s+r,r}: \Hc^{s+r} \to \Hc^{s}
    \]
    for all $s\in \mathbb{R}$. If no confusion can arise, we often write
    \[
    A: \Hc^{s+r} \to \Hc^s.
    \]
    Furthermore, $\op(\Theta) := \bigcup_{r\in \mathbb{R}} \op^r(\Theta)$, and $\op^{-\infty}(\Theta):= \bigcap_{r\in \mathbb{R}}\op^r(\Theta)$. We define $\OP^r(\Theta)\subseteq \op^r(\Theta)$ as those $A \in \op^r(\Theta)$ for which $\delta^n_\Theta(A) \in \op^r(\Theta)$ for each $n \geq 0$, where $\delta_\Theta(A) := [\Theta, A]$. 
\end{defn}

Since  $\Hc^{s} \hookrightarrow \Hc^{t}$ is continuous for all $s \geq t$, it follows that $\op^r(\Theta) \subseteq \op^t(\Theta)$ for $r \leq t$.
Both $\op(\Theta)$ and $\OP(\Theta):= \bigcup_{r \in \R} \OP^r(\Theta)$ form a filtered algebra, as $\op^r(\Theta) \cdot \op^t(\Theta) \subseteq \op^{r+t}(\Theta)$ and $\OP^r(\Theta)\cdot \OP^t(\Theta) \subseteq \OP^{r+t}(\Theta)$.

Commonly, this paradigm is also used when dealing with unbounded operators
\[
T: \dom(T) \to \Hc.
\]
In this case, one writes $T \in \op^r(\Theta)$ if $\Hc^\infty \subseteq \dom(T)$, $T(\Hc^\infty) \subseteq \Hc^\infty$, and
\[
T|_{\Hc^\infty} \in \op^r(\Theta).
\]
Conversely, for $T \in \op^r(\Theta)$ with $r \geq 0$,
\[
\overline{T}^{r,0}: \Hc^r \subseteq \Hc \to \Hc
\]
can be interpreted as an unbounded operator.
Furthermore, note that $\op^{-r}(\Theta) \subseteq B(\Hc)$ for $r \geq 0$ in the sense that for $A \in \op^{-r}(\Theta)$ we have $\overline{A}^{-r,0}|_{\Hc} \in B(\Hc)$. Similarly if $\Theta^{-1}\in \mathcal{L}_s(\Hc)$, a Schatten class, then $\op^{-s}(\Theta) \subseteq \mathcal{L}_{1}(\Hc)$.

\subsection{Adjoints}
\label{S:AppEllipticAdj}
Here we discuss various ways of defining the adjoint in $\op^r(\Theta)$. 
\begin{defn}
    Let $A \in \op^r(\Theta)$ so that $A$ extends to a bounded operator
    \[
    A: \Hc^{s+r} \to \Hc^{s}
    \]
for all $s\in \mathbb{R}$.
\begin{enumerate}
    \item The adjoint of $A$ as an endomorphism of the topological vector space $\Hc^\infty$ we denote 
\[
A^\dag : \Hc^{-\infty} \to \Hc^{-\infty}
\]
defined by the identity
\[
\langle A u, v \rangle_{(\Hc^\infty, \Hc^{-\infty})} = \langle u, A^\dag v \rangle_{(\Hc^{\infty}, \Hc^{-\infty})}, \quad u \in \Hc^\infty, v\in \Hc^{-\infty}.
\]
\item In similar fashion we denote the adjoint
\[
A'^{_s}: \Hc^{-s} \to \Hc^{-s-r}
\]
defined by the relevant identity
\[
\langle Au, v\rangle_{(\Hc^s, \Hc^{-s})} = \langle u, A'^{_s}v \rangle_{(\Hc^{s+r}, \Hc^{-s-r})}, \quad u \in \Hc^{s+r}, v\in \Hc^{-s}.
\]
\item We define the Hermitian adjoint
\[
A^{\flat_s}: \Hc^{s} \to \Hc^{s+r}
\]
via the identity
\[
\langle Au, v \rangle_{\Hc^s} = \langle u, A^{\flat_s} v \rangle_{\Hc^{s+r}}, \quad u \in \Hc^{s+r}, v\in \Hc^s.
\]
\item In case $r \geq 0$, the map
\[
A: \Hc^{s+r} \subseteq \Hc^s \to \Hc^s
\]
  is an unbounded operator on the Hilbert space $\Hc^s$, so we define another Hermitian adjoint
\[
A^{*_s}: \mathcal{D}_s \to \Hc^{s},
\]
with domain
\[
\mathcal{D}_s := \{u \in \Hc^s \ |\  \exists v \in \Hc^s \forall \phi \in \Hc^{s+r} : \langle u, T\phi \rangle_{\Hc^s} = \langle v , \phi \rangle_{\Hc^{s}} \},
\]
such that
\[
\langle Au, v \rangle_{\Hc^s} = \langle u, A^{*_s} v \rangle_{\Hc^{s}}, \quad u \in \Hc^{s+r}, v\in \mathcal{D}_s.
\]
\end{enumerate}
\end{defn}

These adjoints are related in the following way.
\begin{prop}\label{P:opadjoints}
    Let $A \in \op^r(\Theta)$. Then, for all $s \in \R$,
    \begin{enumerate}
        \item $A'^{_s} = A^\dag \big|_{\Hc^{-s}}$;
        \item $A^{\flat_s}= \Theta^{-2s-2r} A^\dag \Theta^{2s}\big|_{\Hc^s}$.
    \end{enumerate}
    If $r\geq 0$, 
    \begin{enumerate}
        \item[3.] $A^{*_s} = \Theta^{-2s} A^\dag \Theta^{2s}\big|_{\mathcal{D}_s}$.
    \end{enumerate}
\end{prop}
\begin{proof}
    \begin{enumerate}
        \item Take $u\in \Hc^{\infty}\subseteq \Hc^{s+r}$ and $v\in \Hc^{-s} \subseteq \Hc^{-\infty}$. Then
        \begin{align*}
           \langle Au, v \rangle_{(\Hc^s, \Hc^{-s})} &= \langle u, A'^{_s} v\rangle_{(\Hc^{s+r}, \Hc^{-s-r})}\\
           &= \langle u, A'^{_s} v \rangle_{(\Hc^{\infty}, \Hc^{-\infty})}.
        \end{align*}
        We also have
        \begin{align*}
            \langle Au, v \rangle_{(\Hc^s, \Hc^{-s})} &  = \langle Au, v \rangle_{(\Hc^\infty, \Hc^{-\infty})}\\
            &= \langle u, A^\dag v \rangle_{(\Hc^{\infty}, \Hc^{-\infty})}.
        \end{align*}
        Hence it follows that
        \[
        A'^{_s} v = A^\dag v \in \Hc^{-\infty}, \quad v \in \Hc^{-s}.
        \]
        
        \item Take $u \in \Hc^{\infty}$, $v\in \Hc^{s}$. Then on the one hand,
        \begin{align*}
            \langle Au, v \rangle_{\Hc^s} &= \langle Au, \Theta^{2s} v \rangle_{(\Hc^\infty, \Hc^{-\infty})}\\
            &= \langle u, A^\dag \Theta^{2s} v\rangle_{(\Hc^{\infty}, \Hc^{-\infty})},
        \end{align*}
        and on the other hand
        \begin{align*}
            \langle Au, v \rangle_{\Hc^s} &= \langle u, A^{\flat_s} v \rangle_{\Hc^{s+r}}\\
            &= \langle u, \Theta^{2s+2r} A^{\flat_s} v \rangle_{(\Hc^{\infty}, \Hc^{-\infty})}.
        \end{align*}
        We therefore find
        \[
        A^{\flat_s} = \Theta^{-2s-2r} A^\dag \Theta^{2s}\big|_{\Hc^s}.
        \]
        \item Take $u \in \Hc^\infty$, $v \in \mathcal{D}_s \subseteq \Hc^s \subseteq \Hc^{-\infty}$. Then
        \begin{align*}
            \langle Au, v\rangle_{\Hc^s} &= \langle Au, \Theta^{2s} v \rangle_{(\Hc^\infty, \Hc^{-\infty})}\\
            &= \langle u, A^\dag \Theta^{2s} v \rangle_{(\Hc^\infty, \Hc^{-\infty})},
        \end{align*}
        and
        \begin{align*}
            \langle Au, v\rangle_{\Hc^s} &= \langle u, A^{*_s} v\rangle_{\Hc^s}\\
            & =\langle u, \Theta^{2s} A^{*_s} v \rangle_{(\Hc^\infty, \Hc^{-\infty})}.
        \end{align*}
        Hence
        \[
        A^{*_s} = \Theta^{-2s} A^\dag \Theta^{2s} \big|_{\mathcal{D}_s}. \qedhere
        \]
    \end{enumerate}
\end{proof}

An important takeaway from this proposition is that if $A : \Hc^\infty \to \Hc^\infty$, we have a priori that
\[
A^\dag: \Hc^{-\infty} \to \Hc^{-\infty},
\]
but if $A \in \op^r(\Theta)$ we have in fact that $A^\dag \in \op^r(\Theta)$ (or, more precisely, $A^\dag \big|_{\Hc^\infty} \in \op^r(\Theta)$). 

It is now also clear that the Hermitian adjoints $A^{\flat_s}$ and $A^{*_s}$ in general cannot be regarded as operators in $\op(\Theta)$, as the operators $A^{\flat_s}$ and $A^{\flat_t}$ do not agree on the intersection $\Hc^s \cap \Hc^t$ for $s\not = t$, and the same holds for $A^{*_s}$.

\begin{prop}\label{P:symmetric}
    If $A \in \op^r(\Theta)$, $r \geq 0$, then 
    \[
 A: \Hc^r\subseteq \Hc^0 \to \Hc^0
 \] is symmetric if and only if $A = A^\dag$. 
\end{prop}
\begin{proof}
Suppose that $A = A^\dag$. Let $u\in \Hc^\infty, v \in \Hc^{r}$. Then
    \begin{align*}
        \langle Au, v \rangle_{\Hc^0} &= \langle Au, v\rangle_{(\Hc^\infty, \Hc^{-\infty})}\\
        &= \langle u, A^\dag v \rangle_{(\Hc^\infty, \Hc^{-\infty})}\\
        &= \langle u, Av \rangle_{\Hc^0}.
    \end{align*}
 By density of $\Hc^\infty \subseteq \Hc^r$, the above equality holds for $u \in \Hc^r$ as well, and hence $A: \Hc^r\subseteq \Hc^0 \to \Hc^0$ is symmetric.

 On the other hand, if $ A: \Hc^r\subseteq \Hc^0 \to \Hc^0$
 is symmetric, then for $u,v \in \Hc^\infty,$ 
 \begin{align*}
     \langle u, A^\dag v \rangle_{(\Hc^\infty, \Hc^{-\infty})} &= \langle Au, v\rangle_{\Hc^0} \\
    &= \langle u, Av \rangle_{\Hc^0}\\
    &= \langle u, Av \rangle_{(\Hc^\infty, \Hc^{-\infty})},
 \end{align*}
 showing that $A^\dag v = Av \in \Hc^\infty$ which implies that $A=A^\dag \in \op^r(\Theta)$.
\end{proof}

\subsection{Elliptic operators}
\label{SS:EllipticOps}
To prepare the way for a functional calculus on this Hilbert scale, we will define a notion of ellipticity called $\Theta$-ellipticity in this subsection. We then show that for $A \in \op^r(\Theta)$, $r \geq 0$, $\Theta$-elliptic and symmetric we have that $A$ as an operator on $\Hc$ is self-adjoint with domain $\Hc^r$. Furthermore, if $A$ is invertible in an appropriate sense, then $A^{-1} \in \op^{-r}(\Theta)$. 

\begin{defn}\label{D:ThetaElliptic}
    We say that an operator $A \in \op^r(\Theta)$ is $\Theta$-elliptic if there exists a parametrix for $A$ of order $-r$, that is, there exists an operator $P \in \op^{-r}(\Theta)$ such that
\begin{align*}
    AP &= 1_{\Hc^{\infty}} + R_1;\\
    PA &= 1_{\Hc^{\infty}} + R_2,
\end{align*}
where $R_1, R_2 \in \op^{-\infty}(\Theta) = \bigcap_{s \in \R} \op^s(\Theta)$.
\end{defn}
We emphasise that the notion of $\Theta$-ellipticity depends on $\Theta$ and on the order $r \in \R$. 

\begin{rem}
In the definition of $\Theta$-ellipticity above, we could have equivalently required the formally weaker condition of $A$ having a right-parametrix $P_1 \in \op^{-r}(\Theta)$ and a left-parametrix $P_2 \in \op^{-r}(\Theta)$,
\begin{align*}
    AP_1 &= 1_{\Hc^{\infty}} + R_1;\\
    P_2A &= 1_{\Hc^{\infty}} + R_2,
\end{align*}
where $R_1, R_2 \in \op^{-\infty}(\Theta)$. Namely, it is not difficult to deduce that this would imply $P_1 - P_2 \in \op^{-\infty}(\Theta)$.
\end{rem}

We will now provide a lemma which lets us formally weaken the condition of $\Theta$-ellipticity in Definition~\ref{D:ThetaElliptic}.

For a pseudodifferential operator $T\in\op(\Theta)$, we say that there exists an asymptotic expansion 
\[
T \sim \sum_{k=0}^\infty T_k,
\]
if
\[
T - \sum_{k=0}^N T_k \in \op^{m_N}(\Theta), \quad m_N\downarrow -\infty.
\]

    \begin{lem}[Borel lemma]\label{L:Borel}
        Let $\{A_k\}_{k=0}^\infty$ be a sequence of linear operators from $\Hc^\infty$ to $\Hc^\infty$ for which $A_k \in \op^{m_k}(\Theta)$ such that $m_k\downarrow -\infty$
        as $k\to\infty$. There exists a linear operator $A\in \op^{m_0}(\Theta)$ such that
        \begin{equation*}
            A \sim \sum_{k=0}^\infty A_k.
        \end{equation*}
    \end{lem}
    \begin{proof}
        Let $\eta \in C^\infty_c(\R)$ be equal to $1$ in a neighbourhood of zero, and let $\{\varepsilon_k\}_{k=0}^\infty$ be a sequence of positive numbers tending
        to zero in a manner to be determined shortly. Formally we define
        \begin{equation*}
            A := \sum_{k=0}^\infty A_k(1-\eta(\varepsilon_k \Theta)).
        \end{equation*}
        We will prove that $\{\varepsilon_k\}_{k=0}^\infty$ can be chosen such that this series makes sense and $A\in \op^{m_0}(\Theta)$ with the desired asymptotic expansion. 
        
        Let $\xi\in \Hc^\infty$. Then for every $k \geq 0$ and $n\in \Z$, we have
        \begin{align*}
            \|A_k(1-\eta(\varepsilon_k \Theta))\xi\|_{\Hc^{n}} &\leq \|A_k\|_{\Hc^{n+m_k}\to \Hc^n}\|(1-\eta(\varepsilon_k\Theta))\xi\|_{\Hc^{n+m_k}}\\
            &\leq \|A_k\|_{\Hc^{n+m_k}\to \Hc^n}\|1-\eta(\varepsilon_k\Theta)\|_{\Hc^{n+m_0}\to \Hc^{n+m_k}}\|\xi\|_{\Hc^{n+m_0}}.
        \end{align*}
        Let $a>0$ be a number such that $a < \Theta$. The norm of $1-\eta(\varepsilon_k \Theta)$ from $\Hc^{n+m_0}$ to $\Hc^{n+m_k}$ is determined by functional calculus as
        \begin{equation*}
            \sup_{t > a} t^{m_k-m_0} (1-\eta(\varepsilon_k t)) \leq \varepsilon_k^{m_0 - m_k} \sup_{s > 0} s^{m_k - m_0}  (1-\eta(s)) \leq C_\eta\varepsilon_k.
        \end{equation*}
        for some constant $C_\eta$, and for $k$ sufficiently large so that $m_0 - m_k \geq 1$. Now we choose $\varepsilon_k$ sufficiently small such that
        \begin{equation*}
            0 \leq \varepsilon_kC_\eta\max_{|n|\leq k}\{\|A_k\|_{\Hc^{n+m_k}\to \Hc^{n}}\} < 2^{-k}.
        \end{equation*}
        With this choice of sequence $\{\varepsilon_k\}_{k=0}^\infty$, we have just proved that the series
        \begin{equation*}
            A\xi = \sum_{k=0}^\infty A_k(1-\eta(\varepsilon_k \Theta))\xi
        \end{equation*}
        converges in every $\Hc^{n}$, and defines a bounded linear operator
        \begin{equation*}
            A:\Hc^{n+m_0}\to \Hc^{n},\quad n\in \Z.
        \end{equation*}
        Since this holds for every $n\in \Z$, it follows that $A:\Hc^\infty\to \Hc^\infty$ and by interpolation (Proposition~\ref{P:Interpolation}) $A \in \op^{m_0}(\Theta)$.
        Note that with this fixed choice of $\{\varepsilon_k\}_{k=0}^\infty$, we have proved the stronger result that the `tail' of $A$
        \begin{equation*}
            \sum_{k=N+1}^\infty A_k(1-\eta(\varepsilon_k \Theta))
        \end{equation*}
        converges in every $\Hc^s$ and defines a linear operator in $\op^{m_{N+1}}(\Theta)$.
        
        Now we prove that $A$ has the desired asymptotic expansion.
        For every $N>0$ we have
        \begin{equation*}
            A-\sum_{k=0}^N A_k = -\sum_{k=0}^N A_k\eta(\varepsilon_k \Theta) + \sum_{k=N+1}^\infty A_k(1-\eta(\varepsilon_k \Theta))
        \end{equation*} 
        Since $\eta$ is compactly supported, it is easy to see that the first summand has order $-\infty$ for every $N\geq 0$, and the second summand
        has order at most $m_{N+1}$ due to the result just proved. 
    \end{proof}

\begin{cor}\label{C:EllipticParametrix}
    Suppose that $A \in \op^r(\Theta)$ has an inverse $B\in \op^{-r}(\Theta)$ up to order $-1$. That is, 
        \begin{align*}
            AB = 1_{\Hc^\infty}+R_1;\\
            BA = 1_{\Hc^\infty}+R_2
        \end{align*} 
        where $R_1, R_2$ have order $-1$. Then $A$ is $\Theta$-elliptic.
\end{cor}
\begin{proof}
    Since $R_1^j$ has order at most $-j$, we can use the Borel lemma to construct an operator $B'$ such that
    \begin{equation*}
            B' \sim \sum_{k=0}^\infty (-1)^jBR_1^j.
        \end{equation*}
        Then $AB' -1$ has order $-\infty$. Similarly we can construct a left inverse.
\end{proof}

    The condition for $\Theta$-ellipticity in Corollary~\ref{C:EllipticParametrix} corresponds to a definition of ellipticity by Guillemin in an abstract pseudodifferential calculus~\cite{Guillemin1985}.

\begin{prop}\label{P:InvertElliptic}
    Let $A \in \op^r(\Theta)$ be $\Theta$-elliptic. If the bounded extension
    \[
    A: \Hc^{s_0+r} \to \Hc^{s_0}
    \]
    admits a bounded inverse
    \[
    A^{-1}: \Hc^{s_0} \to \Hc^{s_0+r},
    \]
    for a specific $s_0 \in \R$, then $A^{-1}\big|_{\Hc^\infty} \in \op^{-r}(\Theta)$. Simply writing $A^{-1} = A^{-1}\big|_{\Hc^\infty}$, we have that
    \[
    A^{-1} A = A A^{-1} = 1_{\Hc^\infty}.
    \]
\end{prop}
\begin{proof}
    Let $P$ be a parametrix of $A$ and take $x \in \Hc^\infty$, so that
    \begin{align*}
        A^{-1}x &= (PA - R_2) A^{-1}(AP - R_1) x\\
        &= PAPx - R_2 P x- PR_1 x + R_2 A^{-1} R_1 x.
    \end{align*}
    Observe that for $y \in \Hc^t$, $t\in \R$, we have $A^{-1}R_1 y \in \Hc^{s_0+r}$, so that $R_2 A^{-1}R_1 y  \in \Hc^\infty$.
    Hence,
    \[
    R_2 A^{-1} R_1 \in \op^{-\infty}(\Theta),
    \]
    and therefore
    \begin{align*}
        A^{-1} &= PAP - R_2 P - PR_1 + R_2 A^{-1} R_1 \\
        &\in \op^{-r}(\Theta) + \op^{-\infty}(\Theta) = \op^{-r}(\Theta). \qedhere
    \end{align*}
\end{proof}

\begin{prop}\label{P:EllipticClosed}
    Let $A \in \op^r(\Theta)$ be $\Theta$-elliptic of order $r \geq 0$. Then the unbounded operator
    \[
    A: \Hc^{s+r}\subseteq \Hc^s \to \Hc^s
    \]
    is closed for each $s \in \R$.
\end{prop}
\begin{proof}
    Define the graph norm of $A$ on $\dom(A) = \Hc^{s+r}$ as 
    \[
    \| x \|_{G(A)} := \|Ax \|_{s} + \| x\|_s, \quad x \in \Hc^{s+r}.
    \]
    By definition, $A$ is a closed operator if and only if $\dom(A)$ is complete with respect to this graph norm. We will show that for $\Theta$-elliptic operators, the graph norm is equivalent to $\| \cdot \|_{s+r}$, which immediately implies that $\Hc^{s+r}$ is complete with respect to the graph norm. First, we have that
    \begin{align*}
        \|Ax \|_s + \|x\|_s & \leq \|A\|_{\Hc^{s+r}\to \Hc^s} \|x\|_{s+r} + \|\Theta^{-r}\|_{\Hc^s \to \Hc^s} \|\Theta^r x \|_s\\
        & \lesssim \| x\|_{s+r}.
    \end{align*}
    Next, let $P$ be a parametrix for $A$, and take $x \in \Hc^\infty$ so that
    \begin{align*}
        \| x \|_{s+r} & \leq \| PA x \|_{s+r} + \| R_2 x\|_{s+r}\\
        & \leq \| P \|_{\Hc^s \to \Hc^{s+r}} \| Ax \|_{s} + \| R_2 \|_{\Hc^s \to \Hc^{s+r}} \| x \|_s\\
        &\lesssim \|Ax \|_s + \|x\|_s.
    \end{align*}
    The assertion of the proposition is now immediate.
\end{proof}

Elliptic operators have a property which is often called elliptic regularity, or maximal subellipticity.
\begin{prop}\label{P:ellipticreg}
    Let $A \in \op^r(\Theta)$ be a $\Theta$-elliptic operator. If $x \in \Hc^{-\infty}$ is such that $Ax \in \Hc^s$ for an $s \in \R$, then $x \in \Hc^{s+r}$.
\end{prop}
\begin{proof}
    Take $x \in \Hc^{-\infty}$, and suppose that $Ax \in \Hc^{s}$. Then $ PAx\in \Hc^{s+r}$,
    which implies that
    \[
    x = PAx - R_2x \in \Hc^{s+r}.\qedhere
    \]
\end{proof}

Finally we will now show that if $A \in \op^r(\Theta)$, $r \geq 0$ is $\Theta$-elliptic and symmetric, then $A$ is self-adjoint with domain $\Hc^r$. 
Recall from Section~\ref{S:AppEllipticAdj} that $A^\dag|_{\Hc^\infty} \in \op^{r}(\Theta)$, and for any $s\in \R$, we have that
\begin{equation}\label{eq:DaggerAdj}
    \langle A u, v \rangle_{(\Hc^s, \Hc^{-s})} = \langle u, A^\dag v \rangle_{(\Hc^{s+r}, \Hc^{-s-r})}, \quad u \in \Hc^{s+r}, v\in \Hc^{-s}.
\end{equation}
And, by Proposition~\ref{P:symmetric}, for $r \geq 0$ we have that $A = A^\dag$ if and only if $\overline{A}^{r,0}: \Hc^r \subseteq \Hc \to \Hc$ is symmetric.

\begin{prop}\label{P:EllipticSelfAdj}
    Let $A \in \op^r(\Theta)$, $r \geq 0$, be a $\Theta$-elliptic and symmetric operator. Then $A$
    is self-adjoint with domain $\Hc^r$.
\end{prop}
\begin{proof}
    To prove that $A$ is self-adjoint, we need to show that the Hermitian adjoint of the closed operator $A: \Hc^{r} \subseteq \Hc^0 \to \Hc^0$, writing $A^{*_0}:= \big(\overline{A}^{r,0} \big)^*$,
    \[
    A^{*_0}: \dom(A^{*_0}) \subseteq \Hc^0 \to \Hc^0,
    \]
    has domain $\dom(A^{*_0}) = \Hc^r$. 
    Recall that, by definition,
    \begin{align*}
        \dom(A^{*_0}) := \{u \in \Hc^0 : \exists v \in \Hc^0 \text{ such that } \forall \phi \in \Hc^r \ \langle u, A\phi \rangle_{\Hc^0}=\langle v, \phi \rangle_{\Hc^0} \}.
    \end{align*}
    If $u,v \in \Hc^0$ and $\phi \in \Hc^r$, then by~\eqref{eq:DaggerAdj} and Proposition~\ref{P:symmetric},
    \begin{align*}
        \langle u, A\phi \rangle_{\Hc^0} &=\langle A^\dag u, \phi \rangle_{(\Hc^{-r}, \Hc^{r})}\\
        &= \langle A u, \phi \rangle_{(\Hc^{-r}, \Hc^{r})};\\
        \langle v, \phi \rangle_{\Hc^0} &= \langle v, \phi \rangle_{(\Hc^{-r}, \Hc^{r})}.
    \end{align*}
    Since $\Hc^{r}$ separates the points of $\Hc^{-r}$, we have that for $u, v \in \Hc^0$,
    \[
    \langle A u, \phi \rangle_{(\Hc^{-r}, \Hc^{r})} = \langle v, \phi \rangle_{(\Hc^{-r}, \Hc^{r})}, \quad \forall \phi \in \Hc^{r},
    \]
    if and only if
    \[
    Au = v \in \Hc^{-r}.
    \]
    Hence, 
    \begin{align*}
        \dom(A^{*_0}) &= \{u \in \Hc^0 : \exists v \in \Hc^0 \text{ such that }  A u = v \} \\
        &= \{u \in \Hc^0 :  A u \in \Hc^0\}.
    \end{align*}
    By elliptic regularity (Proposition~\ref{P:ellipticreg}) it follows that $\dom(A^{*_0}) = \Hc^r,$
    completing the proof.
\end{proof}

\section{Functional calculus for elliptic operators}
\label{SS:FunctCalc}
Due to the results in the previous section, we know that for a $\Theta$-elliptic symmetric operator $A \in \op^r(\Theta)$, $r \geq 0$, the operator $\overline{A}^{r,0}: \Hc^r \subseteq \Hc \to \Hc$ is self-adjoint. Hence, we can ask when the operator $f(\overline{A}^{r,0})$ defined via the Borel functional calculus is an operator in~$\op(\Theta)$. An obstacle for a naive approach is that $f(\overline{A}^{s+r,s})$ is not easily defined for $s \not =0$ as an operator on $\Hc^s$, as $\overline{A}^{s+r,s}$ is generally \textit{not} normal or symmetric as an operator on $\Hc^s$ when $\overline{A}^{r,0}$ is on $\Hc$. Indeed, Proposition~\ref{P:opadjoints} tells us that this is only the case when
\[
A=\Theta^{-2s}A\Theta^{2s}.
\]
The main idea in this section is the following. A functional calculus for the self-adjoint $\Theta$ itself as a pseudodifferential operator is not difficult to construct. Namely, for measurable $f: \R \to \C$, one can define $f(\Theta)$ as an unbounded operator on $\Hc$~\cite[Theorem~5.9]{Schmudgen2012}, with operator norms
\[
\|f(\Theta)\|_{\Hc^{s+m}\to\Hc^s} = \|\Theta^s f(\Theta)\Theta^{-s-m}\|_{\Hc \to \Hc} = \|f(x)x^{-m}\|_{L_\infty(E)},
\]
which gives a quick and appropriate condition for $f(\Theta)\in \op^m(\Theta)$ (here, $E$ is the spectral measure of $\Theta$, see Definition~\ref{def:LinftySpectral} below for a precise definition of $L_\infty(E)$). We can further exploit this by the insight that for a $\Theta$-elliptic symmetric operator $A\in \op^r(\Theta)$, $r>0$, in fact the norms $\|\Theta^s \xi\|$ and $\|\langle A\rangle^{\frac{s}{r}}\xi\|$ are equivalent, where $\langle x\rangle:=(1+|x|^2)^{\frac12}$. If we can provide that
\[
\Hc^s(\Theta) = \Hc^s(\langle A \rangle^{\frac{1}{r}}), \quad \op^s(\Theta) =\op^s(\langle A \rangle^{\frac{1}{r}}),
\]
we can then obtain a functional calculus for $A$ on the Sobolev spaces $\Hc^s(\langle A \rangle^{\frac{1}{r}})$ as easily as for $\Theta$ above, which lets us conclude that $f(A) \in \op(\Theta)$.

\begin{prop}\label{P:ReplaceTheta}
    Let $A\in \op^r(\Theta)$, $r>0$ be $\Theta$-elliptic and symmetric. Then $\langle A \rangle^{\frac{1}{r}}:=(1+A^2)^{\frac{1}{2r}}$ extends to an invertible positive self-adjoint operator on $\Hc$, and 
    \[
    \dom \Theta^s = \dom \big(\langle A \rangle^{\frac{1}{r}}\big)^s.
    \]
    The norms $\| \Theta^s \xi \|_{\Hc}$ and $\|\langle A \rangle^{\frac{s}{r}} \xi \|_{\Hc}$ are equivalent on this subspace of $\Hc$. Therefore, $\Theta$ and $\langle A \rangle^{\frac{1}{r}}$ define the same Sobolev scale
    \[
    \Hc^s(\Theta) = \Hc^s\big(\langle A \rangle^{\frac{1}{r}}\big),
    \]
    and we have
    \[
    \op^t(\Theta)  = \op^t \big(\langle A \rangle^{\frac{1}{r}}\big).
    \]
\end{prop}
\begin{proof}
The first statement follows from Proposition~\ref{P:EllipticSelfAdj}. For the remaining statements, it suffices to prove that 
    \[
    (1+A^2)^{\alpha} \Theta^{-2\alpha r}
    \]
    extends to a bounded operator on $\Hc$ for all $\alpha \in \R$, as this would imply for $\xi \in \Hc^\infty$,
    \[
    \|\langle A \rangle^{\frac{s}{r}} \xi \|_{\Hc} \leq \| \langle A \rangle ^{\frac{s}{r}} \Theta^{-s}\|_\infty \| \Theta^s \xi\|_{\Hc} \lesssim \| \Theta^s \xi\|_{\Hc},
    \]
    and an analogous estimate in the other direction.

Let $P$ be a parametrix for $A$, so that $AP = 1 + R$ with $R \in \op^{-\infty}(\Theta)$. Since
    \[
    (1+A^2)P^2 = P^2 + A(1+R)P = 1 + P^2 + R + ARP
    \]
    and $P^2 + R + ARP \in \op^{-2r}(\Theta)$ and similarly for $(1+A^2)P^2$, it follows that the operator $1+A^2$ is also $\Theta$-elliptic due to Corollary~\ref{C:EllipticParametrix}. Since $A$ is self-adjoint with domain $\Hc^r$, applying Proposition~\ref{P:InvertElliptic} gives that  $(1+A^2)^{-1} \in \op^{-2r}(\Theta)$. We therefore have $(1+A^2)^k \in \op^{2kr}(\Theta)$, $k \in \Z$. This in turn gives that $\Hc^\infty \subseteq \dom (1+A^2)^z$ for any $z\in \C$.
    
    We use the Hadamard three-line theorem, so define the function
    \[
    F(z) := \langle x, (1+A^2)^{mz} \Theta^{-2 mz r} y \rangle_{\Hc}, \quad z \in \C,
    \]
    where $m \in \Z$ and $x, y \in \Hc^\infty$ are fixed.     
    Let $\{e_n\}_{n\in \N}\subseteq \Hc^\infty$ be an orthonormal basis of $\Hc$, then
    \begin{equation}\label{E:basis_expansion_of_F}
    F(z) = \sum_{n=0}^\infty \langle x, (1+A^2)^{mz} e_n \rangle_\Hc \langle e_n , \Theta^{-2mzr} y\rangle_{\Hc}.
    \end{equation}
    Using the dominated convergence theorem, it can be seen that $z \mapsto  \langle x, (1+A^2)^{mz} e_n \rangle_\Hc$ and $z \mapsto \langle e_n , \Theta^{-2mzr} y\rangle_{\Hc}$ are continuous maps. Applying the Cauchy--Schwarz inequality to the series \eqref{E:basis_expansion_of_F} yields
    \[
        \sum_{n=0}^\infty |\langle x,(1+A^2)^{mz}e_n\rangle_{\Hc}||\langle e_n,\Theta^{-2mzr}y\rangle_{\Hc}| \leq \|(1+A^2)^{m\overline{z}}x\|_{\Hc}\|\Theta^{-2mzr}y\|_{\Hc},
    \]
    which is uniformly bounded on compact subsets of $\C$ due to the continuity of the right-hand side. 
    We can therefore apply the dominated convergence theorem again to deduce that $F(z)$ is a continuous function itself. 
    Furthermore, this uniform boundedness yields through Fubini's theorem that if $\gamma$ is a closed loop
    in $\mathbb{C}$ then
    \[
        \int_{\gamma} F(z)\,dz = \sum_{n=0}^\infty \int_{\gamma}\langle x,(1+A^2)^{mz}e_n\rangle_{\Hc}\langle e_n,\Theta^{-2mzr}y\rangle_{\Hc}\,dz.
    \]
    Using Fubini's theorem once more,
    we have that 
    \[
    \int_\gamma F(z)\, dz = \sum_{n=0}^\infty  \int_{\sigma(\Theta^{-2r})}\int_{\sigma(1+A^2)} \int_{\gamma} (\lambda \mu)^{mz} \, dz \langle x,dE^{1+A^2}e_n\rangle_{\Hc} \langle e_n,dE^{\Theta^{-2r}}y\rangle_{\Hc} = 0,
    \] 
    so that we can conclude by Morera's theorem that $F(z)$ is holomorphic.    
    
    Since $1+A^2$ and $\Theta$ are positive operators and
    \[
    \sup_{x > 0} |x^{it}| = 1,
    \]
    it follows from the Borel functional calculus that for $s \in \R,$
    \begin{align*}
        |F(is)| &= |\langle  (1+A^2)^{-ims} x, \Theta^{-2 imsr} y \rangle_{\Hc}|\\
        &\leq \| x \|_{\Hc} \|y\|_{\Hc}.
    \end{align*} 
    Likewise,
    \begin{align*}
        |F(1+is)| &= |\langle  (1+A^2)^{-ims} x, (1+A^2)^m \Theta^{-2mr} \Theta^{-2imsr} y \rangle_{\Hc}|\\
        &\leq \| (1+A^2)^m \Theta^{-2mr} \|_{\Hc \to \Hc} \|x \|_{\Hc} \|y\|_{\Hc},
    \end{align*}
    which we know to be finite since $m$ is an integer.

    The Hadamard three-line theorem (see e.g. \cite[Lemma 1.1.2]{BerghLofstrom1976}) now gives that for $\alpha \in (0,1)$,
    \begin{align*}
        |F(\alpha)| &\leq \max_{s\in \R} |F(\alpha + is)|\\
        &\leq \big( \max_{s\in \R} |F(is)|\big)^{(1-\alpha)} \big( \max_{s\in \R} |F(1 + is)|\big)^\alpha\\
        &\leq  \| (1+A^2)^m \Theta^{-2mr} \|_{\Hc \to \Hc}^\alpha \|x\|_{\Hc} \|y\|_{\Hc}.
    \end{align*}
    Hence, with $\alpha \in (0,1)$ and $m \in \Z$,
    \begin{align*}
        \|(1+A^2)^{m\alpha} \Theta^{-2m \alpha r}\|_{\Hc\to \Hc} \leq  \| (1+A^2)^m \Theta^{-2mr} \|_{\Hc \to \Hc}^\alpha,
    \end{align*}
    which proves that $(1+A^2)^{m\alpha} \Theta^{-2m\alpha r}$ extends to a bounded operator on $\Hc$ for all $m\in \Z$, $\alpha \in [0,1]$.
\end{proof}

\begin{defn}\label{def:LinftySpectral}
Let $E$ be a spectral measure on $\R$ with the Borel sigma algebra. For a Borel measurable function $f:\R\to\C$ we define the essential supremum seminorm
$$\|f\|_{L_\infty(E)}:=\sup\{y\in\R~:~E(|f|^{-1}((y,\infty)))=0\},$$
which defines $L_\infty(E)$ in the usual way, namely as the quotient of the set of measurable functions with finite seminorm, by the set of those of zero seminorm. In the same way we define $L_\infty^\beta(E)$ by the seminorm
$$\|f\|_{L_\infty^\beta(E)}:=\|x\mapsto f(x)\langle x\rangle^{-\beta}\|_{L_\infty(E)},$$
for any $\beta\in\R$, where $\langle x \rangle = (1+|x|^2)^{1/2}$.
\end{defn}

\begin{thm}\label{T:MainFunctCalc}
   Let $A\in \op^r(\Theta)$, $r >0$, be $\Theta$-elliptic and symmetric, and let $E$ denote its spectral measure. If $f \in L^\beta_\infty(E), \beta \in \R,$ then 
        \[
        f(\overline{A}^{r,0}) \in \op^{\beta r}(\Theta),
        \]
        and we simply write $f(A) := f(\overline{A}^{r,0})$.
        More precisely,
        \[
        \|f(A)\|_{\Hc^{s+\beta r} \to \Hc^s} \leq C_{s,A} \|f \|_{L^\beta_\infty(E)}.
        \]
\end{thm}
\begin{proof}
    Let $A \in \op^r(\Theta)$, $r>0$ be $\Theta$-elliptic and symmetric. Using Proposition~\ref{P:ReplaceTheta}, we replace $\Theta$ by $\langle A \rangle^{\frac{1}{r}}$ so that $A \in \op^r(\langle A \rangle^{\frac{1}{r}})$ is $\Theta$-elliptic and symmetric. By Proposition~\ref{P:EllipticSelfAdj}, the operator
    \[
    \overline{A}^{r,0}: \Hc^r(\langle A \rangle^{\frac{1}{r}}) \subseteq \Hc \to \Hc
    \]
    is self-adjoint; we denote its spectral measure by $E$. Then for $f \in L^\beta_\infty(E)$, using Borel functional calculus to define $f(\overline{A}^{r,0})$, we have
    \begin{align*}
        \|\langle A \rangle^{\frac{s}{r}} f(\overline{A}^{r,0})\langle A \rangle^{-\frac{s}{r} - \beta}\|_{\Hc \to \Hc} = \| f(\overline{A}^{r,0})\langle A \rangle^{- \beta}\|_{\Hc \to \Hc} = \|f\|_{L^\beta_\infty(E)}<\infty,
    \end{align*}
    which shows that $f(\overline{A}^{r,0})|_{\Hc^\infty} \in \op^{\beta r}(\langle A \rangle^{\frac{1}{r}})$. Converting this estimate back into an estimate on the spaces $\Hc^s(\Theta)$ introduces the constant $C_{s,A}$.
\end{proof}

Theorem~\ref{T:MainFunctCalc} has a converse in the following sense. If $A \in \op^r(\Theta)$, $r>0$ is an arbitrary $\Theta$-elliptic symmetric operator and if $f: \R \to \C$ is such that $f(A) \in \op^{\beta r}(\Theta)$, then the proof of Proposition~\ref{P:ReplaceTheta} gives that $f(A) (1+A^2)^{-\beta/2}$ is a bounded operator on $\Hc$. This happens if and only if $f \langle x \rangle^{-\beta} \in L^0_\infty(E)$~\cite[Theorem~5.9]{Schmudgen2012}, i.e. $f\in L^\beta_\infty(E)$.

\begin{cor}\label{C:IndepNorm}
    If $A\in \op^r(\Theta), r>0$ is symmetric and $\Theta$-elliptic and if $f :\R \to \C$ is a bounded Borel measurable function, then for any $t\in \R$ we have
    \[
    \|f(tA)\|_{\Hc^s \to \Hc^s} \leq C_{s,A} \sup_{x\in \R}|f(x)|,
    \]
    independent of $t\in \R$.
\end{cor}
\begin{proof}
    Like in the proof Theorem~\ref{T:MainFunctCalc}, we have
    \begin{align*}
        \|f(tA)\|_{\Hc^s(\langle A \rangle^{\frac{1}{r}}) \to \Hc^s(\langle A \rangle^{\frac{1}{r}})} & = \|f_t\|_{L^0_\infty(E)} \leq \sup_{x\in \R}|f(x)|,
    \end{align*}
    where we wrote $f_t(x) := f(tx)$.
\end{proof}

We say that two normal (potentially unbounded) operators $A, B$ strongly commute if all their respective spectral projections commute.

The functional calculus constructed in Theorem~\ref{T:MainFunctCalc} can easily be extended to a larger class of operators. For example, on $\R^d$  we have that $i\frac{d}{dx}$ is not $(1-\Delta)^{\frac12}$-elliptic, but it does commute strongly with a $(1-\Delta)^{\frac12}$-elliptic symmetric operator, namely $(1-\Delta)^{\frac12}$ itself. The following proposition shows that a functional calculus for $i\frac{d}{dx}$ does exist in $\op(1-\Delta)^{\frac12}$ for this reason.
\begin{prop}\label{P:BigFunctCalc}
    Let $A$ be a self-adjoint operator on $\Hc$ with spectral measure $E$. If there exists a $\Theta$-elliptic symmetric operator $H \in \op^h(\Theta)$, $h>0$ such that $A$ strongly commutes with $\overline{H}^{h,0}: \Hc^h \subseteq \Hc \to \Hc$, then for $f \in L^0_\infty(E)$ we have that $f(A) \in \op^0(\Theta)$. If $A \in \op^r(\Theta)$ itself for some $r \in \R$, we have that $f(A) \in \op^{\beta r}(\Theta)$ for $f \in L^\beta_\infty(E)$, $\beta \geq 0$.
\end{prop}
\begin{proof}
    In light of Proposition~\ref{P:ReplaceTheta}, we can assume without loss of generality that $H = \Theta$. If $f \in L^0_\infty(E)$, then $f(A):\Hc \to \Hc$ is a bounded operator, and for $\xi \in \Hc^\infty$ we have
    \[
    \| f(A) \xi\|_{\Hc^k} = \|\Theta^k f(A) \xi\|_{\Hc} \leq \|f(A)\|_\infty \|\xi\|_{\Hc^k}, \quad k \in \Z, 
    \]
    which shows that $f(A) \in \op^0(\Theta)$ through interpolation (Proposition~\ref{P:Interpolation}). The second part of the proposition is proved similarly, after the observation that the Hadamard three-line argument in the proof of Proposition~\ref{P:ReplaceTheta} goes through for $A$ if $m \in \N$, i.e.\\ $(1+A^2)^\alpha \Theta^{-2\alpha r}$ is bounded for $\alpha \geq 0$. 
\end{proof}

Let us now compare the functional calculus of Theorem~\ref{T:MainFunctCalc} with other examples. For classical pseudodifferential operators, pseudodifferential operators in the Beals--Fefferman calculus, and a generalisation thereof on manifolds, general functions in $f\in L_\infty^\beta(E)$ are much too rough to guarantee that $f(T)$ is again a pseudodifferential operator of the same class~\cite{Strichartz1972,Bony2013}. To emphasise, our results provide that for such rough functions $f$, the operator $f(T)$ is well-defined and maps boundedly between appropriate Sobolev spaces, but nothing more. The typical function class that allows concluding that $f(T)$ is again a pseudodifferential operator of the right type, is the following~\cite{Strichartz1972,Bony2013}. 

\begin{defn}\label{def:Sbeta}
    For $I\subseteq \R$ an interval and $\beta \in \R$, we define $S^\beta(I)$ as the class of smooth functions $f: I \to \C$ such that
\[
\|f\|_{S^\beta(I), k} := \sup_{x\in I} |f^{(k)}(x)| \langle x \rangle^{k-\beta} < \infty, \quad k \in \N.
\]
The quantities above are seminorms.
\end{defn}

From this perspective, the operator class $\OP(\Theta)$ (recall Definition~\ref{D:AbstractPSDO}) behaves more like typical pseudodifferential operators.

\begin{thm}\label{T:FunctCalcOPEll}
    Let $A\in \OP^r(\Theta)$, $r >0$, be $\Theta$-elliptic and symmetric, with spectrum $\sigma(A)$. If $f \in S^\beta(\sigma(A)), \beta \in \R,$ then the operator $f(A) \in \op^{\beta r}(\Theta)$ defined in Theorem~\ref{T:MainFunctCalc}, satisfies 
        \[
        f(A) \in \bigcap_{\varepsilon>0}\OP^{\beta r+\varepsilon}(\Theta).
        \]
\end{thm}

To prove this theorem, we need to estimate the appropriate operator norms of the repeated commutators $\delta^n_\Theta(f(A))$. As mentioned in Section~\ref{S:MOIs}, this is exactly what multiple operator integrals are for. The proof of this theorem will hence be provided in Chapter~\ref{Ch:MOIs}, as a consequence of Theorem~\ref{T:OPMOOIS}.

\begin{rem}
    It is unknown to the author whether the conclusion of Theorem~\ref{T:FunctCalcOPEll} can be strengthened to
    \[
    f(A) \in \OP^{\beta r}(\Theta).
    \]
\end{rem}

\section{Functional calculus for zero-order operators }\label{S:AppFunctCalc0}
In Section~\ref{SS:FunctCalc} we proved that $\Theta$-elliptic symmetric operators in $\op^r(\Theta)$ for $r>0$ admit a functional calculus. The approach of that section does not apply for the case $r=0$. To illustrate how different the zero-order case is, consider the situation where $\Theta = (1-\Delta)^{\frac12}$ on $L_2(M)$ where $M$ is a compact subset of $\R^d$. We have that for $\phi: M \to \mathbb{R}$, the multiplication operator
    \begin{align*}
        M_\phi: L_2(M) &\to L_2(M)\\
        \xi &\mapsto \phi \cdot \xi,
    \end{align*}
    where $\phi \cdot \xi (x) = \phi(x) \xi(x)$, is in $\op^0(1-\Delta)^{\frac12}$ if and only if $\phi$ is smooth. If $f(M_\phi) \in \op^0 (1-\Delta)^{\frac12}$ for all $M_\phi \in \op^0 (1-\Delta)^{\frac12}$, then the identity 
    \[
    f(M_\phi) = M_{f \circ \phi}
    \]
    shows that the function $f$ has to be smooth itself and no functional calculus with general functions in $L^\beta_\infty(E)$ is possible; smoothness is a necessary condition.

For $\op^0(\Theta)$ we therefore use a different strategy altogether. 
An approach by Davies~\cite{Davies1995c, Davies1995a} on the construction of a functional calculus using almost analytic extensions directly applies. The technique was introduced by H\"ormander~\cite{Hormander1969, Hormander1970}, and subsequently used in various contexts by many authors, though the formula is often called the Helffer--Sj\"ostrand formula due to their independent rediscovery in~\cite{HelfferSjostrand1988}. For detailed notes on the historical origins, see~\cite{Hormander1969}. Using almost analytic extensions to obtain a functional calculus for pseudodifferential operators has precedent in the works of, amongst others,  H\"ormander~\cite{Hormander1969}, Dynkin~\cite{Dynkin1970, Dynkin1972}, Helffer--Sj\"ostrand~\cite{HelfferSjostrand1988}, Dimassi--Sj\"ostrand~\cite[Chapter 8]{DimassiSjostrand1999}, and Bony~\cite{Bony2013}. In an interesting twist, some of the earliest applications by Helffer and Sj\"ostrand of this formula occured in the study of the Schr\"odinger operator and its density of states~\cite{HelfferSjostrand1988, HelfferSjostrand1989, Sjostrand1991}, the main subject of Chapters~\ref{Ch:DOSDiscrete} and~\ref{Ch:DOSManifolds} in this thesis (but in which these extensions make no appearance). For the details of this construction we follow Davies~\cite{Davies1995a}\cite[Section~2.2]{Davies1995c}.

\begin{defn}[\cite{Davies1995a}]\label{def:AlmostAnalyticExtension}
    Let $f \in C^\infty_c(\mathbb{R})$. Recall that $\langle x \rangle:=(1+|x|^2)^{\frac12}$. We define an extension $\tilde{f}: \mathbb{C} \to \mathbb{C}$ by
    \[
    \tilde{f}(x+iy) := \tau(y/\langle x \rangle ) \sum_{k=0}^N f^{(k)}(x) \frac{(iy)^k}{k!},
    \]
    where $N \geq 1$ and $\tau:\mathbb{R} \to \mathbb{R}$ is a smooth bump function with $\tau(s) = 0$ for $|s|>2$, $\tau(s) =1$ for $|s|<1$.  Then we have
    \[
    f(x) = -\frac{1}{\pi} \int_\mathbb{C} \left( \frac{\partial \tilde{f}}{\partial \overline{z}}(z)\right) (z-x)^{-1} \,dz, \quad x \in \mathbb{R},
    \]
    independent of the choice of $\tau$ and $N$. We refer to $\tilde{f}$ as an \textit{almost analytic extension} of $f$.
\end{defn}

\begin{defn}\label{def:Tbeta}
We define $T^{\beta}(I)$ as the space of smooth functions $f: I \to \C$ such that 
    \[
    \|f\|_{T^\beta(I), k}:= \int_I |f^{(k)}(x)|\langle x\rangle^{k-\beta-1}\, dx < \infty, \quad k \in \N.
    \]
\end{defn}
    Recall that we previously put
    \[
\|f\|_{S^\beta(I), k} = \sup_{x\in I} |f^{(k)}(x)| \langle x \rangle^{k-\beta} < \infty, \quad k \in \N.
\]
    Hence, in case $I = \R$, we note the inclusions
    \[
    \bigcup_{\alpha<\beta} S^{\alpha}(\R) \subsetneq T^{\beta}(\R) \subsetneq S^\beta(\R) \subsetneq L^\beta_\infty(E), \quad \beta \in \R,
    \]
    for any spectral measure $E$.
    
\begin{thm}[\cite{Davies1995a}]\label{T:Helffer-Sjostrand}
    Let $f \in C^\infty_c(\mathbb{R})$ with almost analytic extension $\tilde{f}$ as in Definition~\ref{def:AlmostAnalyticExtension}, so that
    \[
    f(x) = -\frac{1}{\pi} \int_\mathbb{C} \left( \frac{\partial \tilde{f}}{\partial \overline{z}}(z)\right) (z-x)^{-1} \, dz, \quad x \in \mathbb{R}.
    \]
    For any closed, densely defined operator $H$ with $\sigma(H) \subseteq \mathbb{R}$, if for some $\alpha \in \mathbb{R}_{\geq 0}$ we have the estimate
    \[
    \| (z-H)^{-1} \|_\infty \leq C \frac{1}{|\Im(z)|}\left(\frac{\langle z \rangle}{|\Im(z)|}\right)^\alpha, \quad z\in \mathbb{C}\subseteq \mathbb{R},
    \]
    then we have that
    \[
    f(H) := -\frac{1}{\pi} \int_\mathbb{C} \frac{\partial \tilde{f}}{\partial \overline{z}} (z-H)^{-1}\, dz
    \]
    defines a bounded operator on $\Hc$ independent of the choice of $N > \alpha$ and $\tau$ in the construction of the extension $\tilde{f}$, with
    \[
    \|f(H)\|_\infty \leq \sum_{k=0}^{N+1}\|f\|_{T^0(\R), k}.
    \]
    The integral should be interpreted as a $B(\Hc)$-valued Bochner integral. In case $H$ is self-adjoint, this agrees with the continuous functional calculus.
\end{thm}
We thank Dmitriy Zanin for providing a key step in the following proof, which is an adaptation of an argument by Beals~\cite[Lemma~3.1]{Beals1977}.

\begin{prop}\label{P:Invertop}
    Let $X \in \op^r(\Theta)$ be such that $[\Theta, X] \in \op^r(\Theta)$. If the extension 
    \[
    X: \Hc^{s_0+r} \to \Hc^{s_0}
    \]
    has a bounded inverse 
    \[
    X^{-1}:\Hc^{s_0} \to \Hc^{s_0+r}
    \]
    for one particular $s_0 \in \mathbb{R}$, then $X^{-1}\big|_{\Hc^\infty} \in \op^{-r}(\Theta)$. We have $X X^{-1}|_{\Hc^\infty} = X^{-1}X|_{\Hc^{\infty}} = 1_{\Hc^\infty}$. In particular, if $X \in \op^r(\Theta)$ and $[\Theta, X] \in \op^r(\Theta)$ with $r \geq 0$, then we have as (unbounded) operators 
    \[
    \sigma(X: \Hc^{s_0 +r} \subseteq \Hc^{s_0} \to \Hc^{s_0}) = \sigma(X: \Hc^{s+r}\subseteq \Hc^s \to \Hc^s)
    \]
    for all $s\in \mathbb{R}$, where $\sigma$ denotes the spectrum of the operator.
\end{prop}
\begin{proof}
Since $\op^r(\Theta) = \op^0(\Theta) \cdot \Theta^r$, it suffices to prove the proposition for $r=0$. 

Suppose that $X \in \op^0(\Theta)$ is a bijection on $\Hc^{s_0} \to \Hc^{s_0}$ and write $X^{-1}:\Hc^{s_0} \to \Hc^{s_0}$. Then $X$ restricts to a necessarily injective map $\Hc^{{s_0}+1} \to \Hc^{{s_0}+1}$. We now prove that $X: \Hc^{{s_0}+1} \to \Hc^{{s_0}+1}$ is also surjective.

We follow~\cite[Lemma~3.1]{Beals1977}, filling in some omitted details. Take $v \in \Hc^{s_0 + 1}$, then there exists $u \in \Hc^{s_0}$ with $Xu = v$. Let $\varepsilon >0$, then $\frac{\Theta}{1+\varepsilon \Theta} \in \op^0(\Theta)$, and
\begin{align*}
    \frac{\Theta}{1+ \varepsilon \Theta} u &= \frac{\Theta}{1+ \varepsilon \Theta} X^{-1} v \\
    &= X^{-1} \frac{\Theta}{1+\varepsilon \Theta}v + X^{-1}\left[\frac{\Theta}{1+ \varepsilon \Theta}, X\right] X^{-1}v.
\end{align*}
Now, the $\Hc^{s_0}$ norm of the right-hand side is bounded independent of $\varepsilon$: 
\begin{align*}
    \left \| \frac{\Theta}{1+\varepsilon \Theta} \right \|_{\Hc^{s_0+1} \to \Hc^{s_0}} &\leq \| \Theta \|_{\Hc^{s_0+1} \to \Hc^{s_0}};\\
\left\|\left[\frac{\Theta}{1+ \varepsilon \Theta}, X\right] \right\|_{\Hc^{s_0} \to \Hc^{s_0}} &= \|(1+\varepsilon\Theta)^{-1}[\Theta,X](1+\varepsilon \Theta)^{-1}\|_{\Hc^{s_0}\to H^{s_0}} \leq \|[\Theta, X]\|_{\Hc^{s_0}\to \Hc^{s_0}}.
\end{align*}
This implies that $u \in \Hc^{s_0+1}$, a fact that can be quickly verified with the spectral theorem and Fatou's lemma. Therefore, 
\[
    X: \Hc^{s_0 + 1} \to \Hc^{s_0+1}
\]
is a bijection. By induction and interpolation (Proposition~\ref{P:Interpolation}), the same assertion holds for each $\Hc^s$, $s \geq s_0$.

Finally, it is a basic fact that the adjoint of a bijective operator is bijective, i.e.
\[
X'^{_{s_0}} = X^\dag|_{\Hc^{-s_0}}:\Hc^{-s_0} \to \Hc^{-s_0}
\]
is a bijection (recall the notation from Section~\ref{S:AppEllipticAdj}). Since
\[
[\Theta, X^\dag] = - [\Theta, X]^{\dag} \in \op^0(\Theta),
\]
we can apply the same arguments as above to deduce that
\[
X^{\dag}:\Hc^{-s} \to \Hc^{-s}
\]
is a bijection for all $-s \geq -s_0$. This implies that 
\[
X = X^{\dag \dag}: \Hc^{s} \to \Hc^{s}
\]
is a bijection for all $s \leq s_0$.
\end{proof}

As an aside, one may wonder how for $r>0$ the condition $A, [\Theta,A] \in \op^r(\Theta)$ compares to the condition of $A$ being $\Theta$-elliptic used in the previous section. The following consequence of Proposition~\ref{P:Invertop} shows that, under the assumption that $A$ is self-adjoint with domain $\Hc^r$, $\Theta$-ellipticity is a weaker condition.
\begin{cor}\label{C:EllipticEquiv}
    Let $A\in \op^r(\Theta), r>0,$ be such that
    \[
    A:\Hc^r\subseteq \Hc^0 \to \Hc^0
    \]
    has a non-empty resolvent set (for example if $A$ is self-adjoint with domain $\Hc^r$), and suppose that $[\Theta, A] \in \op^r(\Theta)$ as well. Then $A$ is $\Theta$-elliptic.
\end{cor}
\begin{proof}
    By assumption, there exists $z \in \C$ such that 
    \[
    z-A: \Hc^r \subseteq \Hc^0 \to \Hc^0
    \]
    is invertible, and because $[\Theta, A] \in \op^r(\Theta)$ it follows that $(z-A)^{-1} \in \op^{-r}(\Theta)$ by Proposition~\ref{P:Invertop}. Now,
    \[
    A (z-A)^{-1} = - 1 + z (z-A)^{-1}. 
    \]
    In other words, $-(z-A)^{-1}$ is an inverse of $A$ modulo $\op^{-r}(\Theta)$. Corollary~\ref{C:EllipticParametrix} gives that $A$ is $\Theta$-elliptic.
\end{proof}

This following type of estimate on the resolvent also appears in $L_p$-boundedness problems, see \cite{Davies1995b, Davies1995a, JensenNakamura1994}.
\begin{lem}\label{L:ResolventEstimate}
    Let $A \in \op^0(\Theta)$ be such that $[\Theta, A] \in \op^0(\Theta)$ and $\overline{A}^{0,0}:\Hc \to \Hc$ is self-adjoint. 
    Then for all $s \in \mathbb{R},$ there is a constant $C_s > 0$ such that
    \begin{align*}
        \|(z-A)^{-1}\|_{\Hc^s\to \Hc^s} &\leq C_s \frac{1}{|\Im(z)|} \left(\frac{\langle z \rangle}{|\Im(z)|}\right)^{2^{|s|}-1}, \quad z \in \mathbb{C}\setminus \mathbb{R}.
    \end{align*}
\end{lem}
\begin{proof}
    The proof is by induction and interpolation (Proposition~\ref{P:Interpolation}). For $s=0$, the estimate holds by self-adjointness of $\overline{A}^{0,0}$. Note that $(z-A)^{-1} \in \op^0(\Theta)$ due to Proposition~\ref{P:Invertop}.
    
    Suppose the inequality is proved for a fixed $s\in\R_{\geq 0}$. Then for $z \in \mathbb{C}\setminus \mathbb{R}$,
    \begin{align*}
        & \|(z-A)^{-1}\|_{\Hc^{s+1}\to \Hc^{s+1}} \\
        &= \|\Theta(z-A)^{-1}\Theta^{-1}\|_{\Hc^s\to \Hc^s}\\
        &\leq \|(z-A)^{-1}\|_{\Hc^s\to \Hc^s}+\|(z-A)^{-1}[\Theta,A](z-A)^{-1}\Theta^{-1}\|_{\Hc^s\to \Hc^s}\\
        &\leq  \|(z-A)^{-1}\|_{\Hc^s\to \Hc^s} \Big(1+ \|[\Theta,A](i+A)^{-1}\|_{\Hc^s\to \Hc^s}\|(i+A)(z-A)^{-1}\|_{\Hc^s\to \Hc^s} \|\Theta^{-1}\|_{\Hc^s \to \Hc^s} \Big).
    \end{align*}
    Note that $(i+A)^{-1} \in \op^{0}(\Theta)$ by Proposition~\ref{P:Invertop}, so that for some constant $B_s >0$,
    \[
    \|[\Theta,A](i+A)^{-1}\|_{\Hc^s\to \Hc^s} \|\Theta^{-1}\|_{\Hc^s\to \Hc^s}\leq B_s.
    \]
    Using the resolvent identity, we have
    \[
        (i+A)(z-A)^{-1} = (i+z)(z-A)^{-1}-1
    \]
    and therefore
    \begin{align*}
        \|(i+A)(z-A)^{-1}\|_{\Hc^s\to \Hc^s} & \leq 1+ |z+i|\|(z-A)^{-1}\|_{\Hc^s\to \Hc^s}.
    \end{align*}
    This yields
    \begin{align*}
        \|(z-A)^{-1}\|_{\Hc^{s+1}\to \Hc^{s+1}} &\leq  \|(z-A)^{-1}\|_{\Hc^s\to \Hc^s} (1+B_s )\\
        &\quad + |z+i| B_s  \|(z-A)^{-1}\|_{\Hc^s\to \Hc^s}^2 \\
        &\leq (1+B_s) \|(z-A)^{-1}\|_{\Hc^s\to \Hc^s} \cdot (1+ |z+i|\|(z-A)^{-1}\|_{\Hc^s\to \Hc^s}).
    \end{align*}
    This estimate also holds with $|z-i|$ on the right-hand side, and $\min(|z+i|,|z-i|)\leq\langle z\rangle$, so that
    \begin{align*}
        \|(z-A)^{-1}\|_{\Hc^{s+1}\to \Hc^{s+1}} &\leq  (1+B_s)  \|(z-A)^{-1}\|_{\Hc^s\to \Hc^s} \cdot (1+ \langle z \rangle \|(z-A)^{-1}\|_{\Hc^s\to \Hc^s}),
    \end{align*}
    from which the claimed estimate follows. Induction and interpolation now provide the estimate for all $s \geq 0$. 

    The case $s \leq 0$ is proved in the same manner, using induction in the negative direction. Namely, the norm 
    \[
    \| (z-A)^{-1}\|_{\Hc^{s-1} \to \Hc^{s-1}} = \|\Theta^{-1}(z-A)^{-1}\Theta \|_{\Hc^s\to \Hc^s}
    \]
    can be estimated as before.
\end{proof}

We now arrive at the main result of this section, a functional calculus for $\op^0(\Theta)$. 

\begin{thm}\label{T:FunctCalcOp0}
    Let $A \in \op^0(\Theta)$ be such that $[\Theta, A] \in \op^0(\Theta)$ and $\overline{A}^{0,0}: \Hc \to \Hc$ is self-adjoint. For $f \in C^\infty(\R)$, 
    \[
    f(A) \in \op^{0}(\Theta).
    \]
    Specifically, we have the estimate
    \[
    \|f(A)\|_{\Hc^s \to \Hc^s} \leq \sum_{k=0}^{\lceil 2^{|s|}\rceil+1} \|f\|_{T^0(\R), k}.
    \]
\end{thm}
\begin{proof}
Without loss of generality, we assume that $f\in C_c^\infty(\R)$. As $A\in \op^0(\Theta)$, it extends to a bounded operator
\[
\overline{A}^{s,s}: \Hc^s \to \Hc^s, \quad s\in \R.
\]
Furthermore, by Proposition~\ref{P:InvertElliptic}, we have
\[
(z-A)^{-1} \in \op^0(\Theta), \quad z \in \C\setminus\R.
\]
Theorem~\ref{T:Helffer-Sjostrand} and Lemma~\ref{L:ResolventEstimate} combined give that
\[
f(\overline{A}^{s,s}): \Hc^s\to \Hc^s, \quad s \in \R,
\]
is a bounded operator with the norm bound as claimed. By construction, for $\xi \in \Hc^\infty$ we have that
\begin{align*}
    f(\overline{A}^{s,s})\xi = \int_{\C} \frac{\partial \tilde{f}}{\partial \overline{z}} (z-\overline{A}^{s,s})^{-1} \xi \,dz \in \Hc^s,
\end{align*}
as an $\Hc^s$-valued Bochner integral. It is clear that
\[
(z-\overline{A}^{s,s})^{-1} \big|_{\Hc^\infty} = (z-A)^{-1},
\]
and therefore these arguments show that for $\xi \in \Hc^\infty$ the integral
\[
\int_{\C} \frac{\partial \tilde{f}}{\partial \overline{z}} (z-A)^{-1} \xi\, dz
\]
can be evaluated as a Bochner integral in each Sobolev space $\Hc^s$. Hence
\[
\xi \mapsto \int_{\C} \frac{\partial \tilde{f}}{\partial \overline{z}} (z-A)^{-1} \xi \,dz
\]
forms a bounded linear map on $\Hc^\infty$; denote this operator by $f(A): \Hc^\infty \to \Hc^\infty$. Then since $f(A)$ agrees with $f(\overline{A}^{s,s})$ on $\Hc^\infty$, we must have
\[
\overline{f(A)}^{s,s} = f(\overline{A}^{s,s}),
\]
and thus we have $f(A) \in \op^0(\Theta)$.
\end{proof}

Similarly to the $\Theta$-elliptic case, we also provide a theorem for $\OP^0(\Theta)$ (recall Definition~\ref{D:AbstractPSDO}) for which the proof will follow in Chapter~\ref{Ch:MOIs} using a multiple operator integral technique. Specifically, it is a consequence of Theorem~\ref{T:OPMOOIS}.

\begin{thm}\label{T:FunctCalcOP0}
     Let $A \in \OP^0(\Theta)$ be such that $\overline{A}^{0,0}: \Hc \to \Hc$ is self-adjoint. For $f \in C^\infty(\R)$,
     \[
     f(A) \in \OP^{0}(\Theta).
     \]
\end{thm}

Theorem~\ref{T:FunctCalcOp0} and Theorem~\ref{T:FunctCalcOP0} are consistent with~\cite[Theorem~3]{Bony2013}, which provides a functional calculus for zero-order pseudodifferential operators on manifolds for $f \in C^\infty(\R)$. Since our pseudodifferential calculus is vastly more general, our results are both more general and less specific than that in~\cite{Bony2013}.
    \chapter{Multiple operator integrals as abstract pseudodifferential operators}
\label{Ch:MOIs}
{\setlength{\epigraphwidth}{\widthof{Thanks again, and work even harder!!!}}
\epigraph{Thanks again, and work even harder!!!}{Fedor Sukochev}}
Like Chapter~\ref{Ch:FunctCalc}, the content and exposition in this chapter has appeared previously as part of the preprint~\cite{MOOIs}, joint work with Edward McDonald and Teun van Nuland. For this chapter, we are indebted to discussions with Fedor Sukochev and Dmitriy Zanin. The main result of this chapter, Theorem~\ref{T:MainMOIConstruction}, is a construction of multiple operator integrals adapted to the abstract pseudodifferential calculus of Connes and Moscovici. The other main results of this chapter are a more specific version of this MOI construction in Theorem~\ref{T:MOOIforNCG1}, a noncommutative Taylor expansion in Theorem~\ref{T:CombiExpansion}, and a result on the existence of asymptotic trace expansions in spectral triples in Corollary~\ref{C:AsympExpNCG}.

Due to the results in Chapter~\ref{Ch:FunctCalc}, we have a functional calculus for abstract pseudodifferential operators $A \in \op(\Theta)$. As explained in Section~\ref{S:MOIs}, it would be incredibly useful in the field of noncommutative geometry if we can define multiple operator integrals (MOIs) for this abstract pseudodifferential calculus. There are two closely related constructions we could employ here: Peller's approach~\cite{Peller2006,Peller2016} and Azamov--Carey--Dodds--Sukochev's~\cite{ACDS}. The former works with a very general class of symbol functions, constructing MOIs for operators on Hilbert spaces. The latter considers a more restricted class of functions and achieves stronger results as a consequence (Fr\'echet versus Gateaux differentiability), and notably achieves this in the more general setting of semi-finite von Neumann algebras. Since we wish to keep our symbol functions as general as possible and the results in Peller's work suffice for our purposes,  we will use Peller's approach as a blueprint to construct MOIs of pseudodifferential operators.

In Section~\ref{S:AsympExp} we will provide asymptotic expansions for these MOIs including a noncommutative Taylor expansion, which we will apply in Section~\ref{S:AsympExpNCG} to noncommutative geometry, answering a previously open question on the existence of asymptotic expansions of heat traces in spectral triples.

Throughout this chapter, we fix a positive invertible operator $\Theta$ on the separable Hilbert space $\Hc$, which yields a scale of separable Hilbert spaces $\{\Hc^s\}_{s\in \R}$ by Definition~\ref{def:pseudocalc1}.

\section{Operator integrals}
We start with standard definitions and results on measurability and integrability of operator valued functions.
\begin{defn}
    Let $\Hc_0, \Hc_1$ be separable Hilbert spaces and let $(\Omega, \Sigma, \nu)$ be a measure space with complex measure. A function $f: \Omega \to B(\Hc_1, \Hc_0)$ is called weak operator topology measurable (weakly measurable for short) if for all $\eta \in \Hc_0$, $\xi \in \Hc_1$ the complex-valued function
    \[
    \omega \mapsto \langle \eta, f(\omega) \xi \rangle_{\Hc_0}, \quad \omega \in \Omega,
    \]
    is measurable. Similarly, $f$ is said to be weak operator topology integrable if for all $\xi$ and $\eta$ the above map is integrable.
\end{defn}

\begin{lem}{\cite[Lemma~1.4.2]{LMSZVol2}}\label{L:PettisIntegral}
Let $\Hc_0, \Hc_1, (\Omega, \Sigma, \nu)$ be as above, and let $f: \Omega \to B(\Hc_1, \Hc_0)$ be weakly measurable. Then the norm function
    \[
    \omega \mapsto \|f(\omega)\|_{\Hc_1 \to \Hc_0}, \quad \omega \in \Omega,
    \]
    is measurable. If moreover
    \[
    \int_\Omega \| f(\omega)\|_{\Hc_1 \to \Hc_0} \, d|\nu|(\omega) < \infty,
    \]
    then there exists a unique $I_f \in  B(\Hc_1, \Hc_0)$ such that
    \[
    \langle \eta, I_f \xi \rangle_{\Hc_0} = \int_\Omega \langle \eta, f(\omega)\xi \rangle_{\Hc_0} \, d\nu(\omega), \quad \eta \in \Hc_0, \xi \in \Hc_1,
    \]
    and
    \[
    \| I_f \|_{\Hc_1 \to \Hc_0} \leq \int_\Omega \| f(\omega)\|_{\Hc_1 \to \Hc_0} \,d|\nu|(\omega).
    \]
    We then write $I_f = \int_\Omega f(\omega) \,d\nu(\omega)$.
\end{lem}

\begin{prop}{\cite[Lemma~3.11]{DoddsDodds2020II}}\label{P:WeakMeasFunctCalc}
    Let $(\Omega, \Sigma, \nu)$ be a $\sigma$-finite measure space, let $a:\R \times \Omega \to \C$ be measurable and bounded, and let $H$ be an (unbounded) self-adjoint operator on $\Hc$. Then
    \[
    \omega \mapsto a(H,\omega)
    \]
    is weakly measurable.
\end{prop}
\begin{proof}
Though~\cite[Lemma~3.11]{DoddsDodds2020II} is only formulated for bounded $H$, the unbounded case follows with the same proof. 
\end{proof}

\begin{lem}\label{L:WeaklyMeasComp}
    Let $\Hc_0, \ldots, \Hc_{2n+1}$ be separable Hilbert spaces, $X_i \in B(\Hc_{2i}, \Hc_{2i-1})$, and let $f_i:\Omega \to B(\Hc_{2i+1}, \Hc_{2i} )$ be weakly measurable functions. Then
    \begin{align*}
        \Omega & \to B(\Hc_{2n+1}, \Hc_0)\\
        \omega &\mapsto f_0(\omega) X_1 f_1(\omega) \cdots X_n f_n(\omega)
    \end{align*}
    is weakly measurable. Furthermore, if 
    \[
    \int_\Omega \| f_0(\omega)\|_{\Hc_1 \to \Hc_0} \cdots \|f_n(\omega)\|_{\Hc_{2n+1}\to \Hc_{2n}}\, d|\nu|(\omega) < \infty,
    \]
    the map
    \begin{align*}
    B(\Hc_{1}, \Hc_{0}) \times \cdots \times B(\Hc_{2n-1}, \Hc_{2n-2}) &\to B(\Hc_{2n}, \Hc_0)\\
    (X_1, \ldots, X_n) \quad &\mapsto \int_\Omega f_0(\omega) X_1f_1(\omega) \cdots X_n f_n(\omega)\, d\nu(\omega),
\end{align*}
whose existence follows from Lemma~\ref{L:PettisIntegral}, is $so$-continuous when restricted to the unit ball in each argument $B(\Hc_{2i}, \Hc_{2i-1})$.
\end{lem}
\begin{proof}
    The first part of the lemma is a consequence of the fact that the pointwise product of weakly measurable functions is weakly measurable, see~\cite[Lemma~3.7]{DoddsDodds2020II}.
    
    The $so$-continuity follows from the joint continuity of the multiplication
    \[
     (X_1, \ldots, X_n) \mapsto  a_0(H_0, \omega) X_1 a_1(H_1, \omega) \cdots X_n a_n(H_n, \omega)
    \]
    in the strong operator topology when restricting to the unit balls~\cite[Section~I.3.2]{Blackadar2006}, in combination with the Dominated Convergence Theorem for the Bochner integral of Hilbert space-valued functions~\cite[Corollary~III.6.16]{DunfordSchwartzI}.
\end{proof}

Recall the definition of $L^\beta_\infty(E)$ in Definition~\ref{def:LinftySpectral}, and that we use the notation $\| \cdot \|_\infty$ for the essential supremum.

\begin{lem}\label{L:EssSupMeasurable}
    Let $(\Omega, \Sigma, \nu)$ be a $\sigma$-finite measure space, let $a:\R \times \Omega \to \C$ be measurable and bounded, and let $E$ be a spectral measure on $\Hc$. Then the functions
    \begin{align*}
        \omega &\mapsto \|a(\cdot, \omega)\|_\infty, \qquad        \omega \mapsto \|a(\cdot, \omega)\|_{L^0_\infty(E)}
    \end{align*}
    are measurable.
\end{lem}
\begin{proof}
    Both claims can be proved with the Fubini--Tonelli Theorem or by combining Lemma~\ref{L:PettisIntegral} and Proposition~\ref{P:WeakMeasFunctCalc}. It is vital that $\| \cdot \|_\infty$ is the essential supremum, and that $E$ is a spectral measure on a \textit{separable} Hilbert space as pointed out in~\cite[Remark~4.1.3]{Nikitopoulos2023}.
\end{proof}

\section{Multiple operator integrals as abstract pseudodifferential operators}\label{S:Basics}
In this section we construct multiple operator integrals following Peller~\cite{Peller2006, Peller2016}. His MOIs are constructed with symbols $\phi$ in integral projective tensor products \\$L_\infty(E_0)\hat{\otimes}_i \cdots \hat{\otimes}_i L_\infty(E_n)$ for spectral measures $E_0, \ldots, E_n$. A precise study of this integral tensor product can be found in~\cite{Nikitopoulos2023}.
We first generalise this integral tensor product.

\begin{defn}\label{def:ProjIntBox}
    Let $\Gamma_0, \ldots, \Gamma_n$ be function spaces of bounded measurable functions $\R \to \R$ equipped with (semi)norms $\|\cdot \|_{\Gamma_i, k}$, $k \in \N$.
    We define $\Gamma_0 \boxtimes_i \cdots \boxtimes_i \Gamma_n$ as the set of functions $\phi: \mathbb{R}^{n+1} \to \mathbb{C}$ for which there exists a decomposition
    \begin{equation}\label{eq:boxrep}
        \phi(\lambda_0, \ldots, \lambda_n) = \int_\Omega a_0(\lambda_0, \omega) \cdots a_n(\lambda_n, \omega) \,d\nu(\omega)
    \end{equation}
    where $(\Omega, \nu)$ is a $\sigma$-finite measure space, $a_i : \mathbb{R} \times \Omega \to \mathbb{C}$ is measurable, $a_i(\cdot, \omega) \in \Gamma_i$, the functions $\omega \mapsto \|a_i(\cdot, \omega)\|_{\Gamma_i, k}$ are measurable for each $i$ and $k$, and
    \begin{equation}\label{eq:boxnorm}
        \int_\Omega \|a_0(\cdot, \omega)\|_{\Gamma_0, k_0} \cdots \|a_n(\cdot, \omega)\|_{\Gamma_n, k_n}\, d|\nu|(\omega) < \infty.
    \end{equation}
    We define the seminorm
    \[
    \| \phi \|_{\Gamma_0 \boxtimes_i \cdots \boxtimes_i \Gamma_n, k_0, \ldots, k_n}, \quad k_0, \ldots, k_n \in \N
    \]
    to be the infimum of the quantity~\eqref{eq:boxnorm} over all representations~\eqref{eq:boxrep}.
\end{defn}

\begin{rem}
We have that 
\[
\Gamma_0 \otimes \cdots\otimes \Gamma_n \subseteq \Gamma_0 \boxtimes_i \cdots \boxtimes_i \Gamma_n,
\]
where $\otimes$ denotes the algebraic tensor product for topological vector spaces. We refrain from answering the question whether the space in Definition~\ref{def:ProjIntBox} is in general the completion of the algebraic tensor product under the given seminorms. 
\end{rem}

Note that due to Lemma~\ref{L:EssSupMeasurable}, if $a: \R \times \Omega \to \C$ is measurable and $a(\cdot, \omega) \in S^\beta(\R)$, we have that $\omega \mapsto \|a(\cdot, \omega)\|_{S^\beta(\R), k}$ is measurable, and the same claim holds for $T^\beta(\R)$ and $L_\infty^\beta(E)$ (recall the definitions of the spaces $S^\beta(\R)$, $T^\beta(\R)$ and $L_\infty^\beta(E)$ respectively in Definition~\ref{def:Sbeta}, Definition~\ref{def:Tbeta} and Definition~\ref{def:LinftySpectral}). Hence the construction in Definition~\ref{def:ProjIntBox} can be applied for these spaces without this extra assumption.

For $L_\infty^\beta(E)$ with $E$ a spectral measure, this gives the integral projective tensor product
\[
L_\infty^{\beta_0}(E_0) \boxtimes_i \cdots \boxtimes_i L_\infty^{\beta_n}(E_n) = L_\infty^{\beta_0}(E_0) \hat{\otimes}_i \cdots \hat{\otimes}_i L_\infty^{\beta_n}(E_n),
\]
which appears in particular for $\beta_0 = \cdots = \beta_n = 0$ in the works by Peller~\cite[pp.6, 7]{Peller2006}.

\begin{rem}\label{rem:SipvL}
    Observe that
    \[
    S^{\beta_0}(\R) \boxtimes_i \cdots \boxtimes_i S^{\beta_n}(\R) \subseteq  L_\infty^{\beta_0}(E_0) \hat{\otimes}_i \cdots \hat{\otimes}_i L_\infty^{\beta_n}(E_n)
    \]
    no matter what spectral measures $E_i$ are taken. 
\end{rem}

\begin{prop}\label{P:BoxProdProd}
If $\phi \in S^{\alpha_0}(\R) \boxtimes_i \cdots \boxtimes_i S^{\alpha_n}(\R)$ and $\psi \in S^{\beta_0}(\R) \boxtimes_i \cdots \boxtimes_i S^{\beta_n}(\R)$, then
\[
\Phi(\lambda_0, \ldots, \lambda_n):= \phi(\lambda_0, \ldots, \lambda_n) \psi(\lambda_0, \ldots, \lambda_n)\in S^{\alpha_0 + \beta_0}(\R) \boxtimes_i \cdots \boxtimes_i S^{\alpha_n + \beta_n}(\R).
\]
An analogous statement holds for the spaces $L^\beta_\infty(E)$.
\end{prop}
\begin{proof}
    According to Definition~\ref{def:ProjIntBox}, we can find $\sigma$-finite measure spaces $(\Omega, \nu)$, $(\Sigma, \mu)$ and measurable functions $a_i: \R \times \Omega \to \C$, $b_i: \R \times \Sigma \to \C$ such that
    \begin{align*}
        \phi(\lambda_0, \ldots, \lambda_n) &= \int_{\Omega} a_0(\lambda_0, \omega) \cdots a_n(\lambda_n, \omega)\, d\nu(\omega);\\
        \psi(\lambda_0, \ldots, \lambda_n) &= \int_{\Sigma} b_0(\lambda_0 ,\sigma) \cdots b_n(\lambda_n, \sigma) \, d\mu(\sigma).
    \end{align*}
As observed above, the maps
    \[
    \omega \mapsto \| a_i(\cdot, \omega)\|_{S^{\alpha_i}(\R), k}
    \]
    are measurable, similarly for the functions $b_i$.
    
    Using Tonelli's theorem,
    \begin{align*}
        \int_{\Omega\times \Sigma} \big|a_0(\lambda_0, \omega)b_0(\lambda_0, \sigma)&\cdots a_n(\lambda_n, \omega)b_n(\lambda_n, \sigma)\big| \,d(\nu\times \mu)(\omega, \sigma)\\
        &\leq \langle \lambda_0 \rangle^{\alpha_0+\beta_0} \cdots \langle \lambda_n \rangle^{\alpha_n + \beta_n} \int_{\Omega}  \| a_0(\cdot, \omega)\|_{S^{\alpha_0},0} \cdots \|a_n(\cdot, \omega)\|_{S^{\alpha_n}, 0} \, d\nu(\omega) \\
        &\quad \times \int_{\Sigma}\|b_0(\cdot, \sigma)\|_{S^{\beta_0}, 0} \cdots \|b_n(\cdot, \sigma)\|_{S^{\beta_n},0} \, d\mu( \sigma)<\infty.
    \end{align*}
    Hence, by Fubini's theorem
    \[
    \Phi(\lambda_0, \ldots, \lambda_n) = \int_{\Omega \times \Sigma} a_0(\lambda_0, \omega)b_0(\lambda_0, \sigma) \cdots a_n(\lambda_n, \omega) b_n(\lambda_n, \sigma)\, d(\nu \times \mu)(\omega, \sigma).
    \]
    The fact that $\Phi \in S^{\alpha_0+\beta_0}(\R) \boxtimes_i \cdots \boxtimes_i S^{\alpha_n + \beta_n}(\R)$ now follows from the computation
    \[
    \| a_k(\cdot, \omega) b_k(\cdot, \sigma)\|_{S^{\alpha_k + \beta_k}, m} \leq \sum_{j=0}^m \binom{m}{j} \|a_k(\cdot, \omega)\|_{S^{\alpha_k}, j} \| b_k(\cdot, \sigma)\|_{S^{\beta_k}, m-j}.
    \]
\end{proof}

The following theorem is the first main result of this chapter. It is a construction of multiple operator integrals adapted to the abstract pseudodifferential calculus by Connes and Moscovici which we have seen in Chapter~\ref{Ch:FunctCalc}.

\begin{thm}\label{T:MainMOIConstruction}
Let $H_i \in \op^{h_i}(\Theta)$, $h_i > 0$, $i=0, \ldots, n$, be $\Theta$-elliptic symmetric operators with spectral measures $E_i$, and let $\phi\in L_\infty^{\beta_0}(E_0)  \boxtimes_i \cdots  \boxtimes_i  L_\infty^{\beta_n}(E_n)$, given with explicit representation
\[
\phi(\lambda_0,\ldots,\lambda_n)=\int_\Omega a_0(\lambda_0,\omega)\cdots a_n(\lambda_n,\omega)\, d\nu(\omega), \quad \lambda_i \in \sigma(H_i), i=0,\ldots, n.
\] 
Then the integral
\[
T_\phi^{H_0,\ldots,H_n}(X_1,\ldots,X_n)\psi:=\int_\Omega a_0(H_0,\omega)X_1 a_1(H_1,\omega)\cdots X_n a_n(H_n,\omega)\psi\,d\nu(\omega), \quad \psi\in\Hc^\infty,
\]
for $X_1,\ldots,X_n\in\op(\Theta)$, converges as a Bochner integral in $\Hc^s$ for every $s \in \R$, and defines an $n$-multilinear map 
\[
T_\phi^{H_0,\ldots,H_n}:\op^{r_1}(\Theta)\times\cdots\times\op^{r_n}(\Theta)\to\op^{\sum_j r_j + \sum_j \beta_j h_j}(\Theta)
\]
depending on $\Omega$ and the functions $a_0,\ldots,a_n$ only through the symbol~$\phi$. For $s\in \R$ we have the estimate
   \begin{align*}
        \big \|T_\phi^{H_0,\ldots, H_n}(X_1, \ldots,X_n)\big\|_{\Hc^{s+\sum_{j}r_j + \sum_j \beta_j h_j} \to \Hc^{s}} \lesssim \prod_{j=1}^n \|X_j\|_{\Hc^{s_{j}+r_j} \to \Hc^{s_{j}}} \|\phi\|_{L_\infty^{\beta_0}(E_0) \boxtimes_i \cdots \boxtimes_i L_\infty^{\beta_n}(E_n)},
    \end{align*}
    for some specific $s_1, \ldots, s_n \in \R$ depending on $h_j, r_j, \beta_j$.
\end{thm}

For $\Theta = 1_{\Hc}$ and $\beta_0 = \cdots = \beta_n = 0$, it is immediate that $a_j(H_j, \omega) \in \op(1_{\Hc})= B(\Hc)$ and
\[
\|a_j(H_j, \omega)\|_{\Hc^{s} \to \Hc^s} = \|a_j(\cdot, \omega)\|_{L^0_\infty(E_j)},
\]
and the above theorem reduces to Peller's construction of MOIs~\cite{Peller2006, Peller2016}. Hence, this theorem is a strict generalisation of his results.
The proof is a subtle modification of (aspects of) the proofs presented in~\cite{ACDS, Peller2006}.

\begin{proof}[Proof of Theorem~\ref{T:MainMOIConstruction}]
Fix the $\Theta$-elliptic symmetric operators $H_i \in \op^{h_i}(\Theta)$, $h_i > 0$, $i=0, \ldots, n$, with spectral measures $E_i$, and the function $\phi:\R^{n+1}\to\C$
\[
\phi(\lambda_0,\ldots,\lambda_n)=\int_\Omega a_0(\lambda_0,\omega)\cdots a_n(\lambda_n,\omega)\, d\nu(\omega),
\]
where $(\Omega,\nu)$ is a finite measure space and the functions $a_j:\R\times\Omega\to\C$ are measurable and such that $(x,\omega)\mapsto a_j(x,\omega)\langle x \rangle^{-\beta_j}$ is $E_j \times \nu$-a.e. bounded for $\beta_j \in \R$. 

By Theorem~\ref{T:MainFunctCalc}, we have that $\Hc^\infty \subseteq \dom a_j(H_j, \omega)$. Now fix $\omega \in \Omega$ and take $\eta, \xi \in \Hc^\infty$. Then~\cite[Theorem~4.13]{Schmudgen2012} gives that
\[
a_j(H_j, \omega) \xi = \lim_{n\to \infty} a_j(H_j, \omega) \chi_{[-n, n]}(H_j) \xi,
\]
where $\chi_{[-n,n]}$ is the indicator function of the interval $[-n,n]$, because $\operatorname{ess\,sup}_{|\lambda|\leq n} |a_j(\lambda, \omega)| < \infty$. Now Proposition~\ref{P:WeakMeasFunctCalc} gives that
\[
\omega \mapsto \langle \eta, a_j(H_j,\omega) \xi \rangle_{\Hc^s} =  \lim_{n\to\infty}\langle \Theta^{2s} \eta, a_j(H_j, \omega) \chi_{[-n,n]}(H_j) \xi \rangle_{\Hc}
\]
is measurable for all $s\in \R$.

Let now $X_i \in \op^{r_i}(\Theta)$, $i = 1, \ldots, n$. 
 Fix $s \in \R$ and define $s_0, \ldots, s_{2n+1} \in \R$ with
 \[
 s_0:=s, \quad s_{2n+1} := s + \sum_{i=0}^{n} \beta_i h_i + \sum_{i=1}^n r_i,
 \]
 so that
    the operators $a_j(H_j, \omega)$ and $X_j$ extend to bounded operators
    \begin{align*}
        a_j(H_j, \omega) &\in B(\Hc^{s_{2j+1}}, \Hc^{s_{2j}}),\\
        X_j &\in B(\Hc^{s_{2j}}, \Hc^{s_{2j-1}}).
    \end{align*}
    By the previous argument, 
    \[
    \omega \mapsto a_j(H_j, \omega) \in B(\Hc^{s_{2j+1}}, \Hc^{s_{2j}})
    \]
    is weakly measurable since $\Hc^{\infty}$ is dense in both $\Hc^{s_{2j}}$ and $\Hc^{s_{2j+1}}$.
    
    Using the functional calculus in Theorem~\ref{T:MainFunctCalc},
    \[
    \|a_j(H_j,\omega)\|_{\Hc \indices{^{s_{2j+1}}}\to\Hc\indices{^{s_{2j}}}}\leq C_{s,H_j}\|a_j(\cdot,\omega)\|_{L^{\beta_j}_{\infty}(E_j)},
    \]
    and so we have that
    \begin{align*}
        \int_\Omega \|a_0(H_0, \omega )X_1 a_1(H_1,\omega) &\cdots X_n a_n(H_n, \omega) \|_{\Hc^{s_{2n+1}} \to \Hc^{s_0}}\,d\nu(\omega)\\
        &\lesssim \prod_{j=1}^n \|X_j\|_{\Hc^{s_{2j}} \to \Hc^{s_{2j-1}}}  \int_\Omega \prod_{j=0}^n \| a_j(\cdot, \omega)\|_{L^{\beta_j}_\infty(E_j)} \, d|\nu|(\omega) < \infty,
    \end{align*}
    where Lemma~\ref{L:EssSupMeasurable} ensures the right-hand side is defined. This is a finite quantity since $a_j(x, \omega) \langle x \rangle^{-\beta_j} $ is $E_j \times \nu-a.e.$ bounded and $\nu$ is a finite measure space. Therefore, Lemma~\ref{L:PettisIntegral}
    provides that
    \[
    \int_\Omega a_0(H_0, \omega )X_1 a_1(H_1,\omega) \cdots X_n a_n(H_n, \omega)\, d\nu(\omega)
    \]
    defines an operator in the weak sense in $B(\Hc^{s_{2n+1}}, \Hc^{s_0})$ with
    \begin{equation}\label{eq:MOIOperatorNorm}
    \begin{split}
        \bigg \|\int_\Omega a_0(H_0, \omega )X_1 a_1(H_1,\omega) &\cdots X_n a_n(H_n, \omega) \,d\nu(\omega)\bigg\|_{\Hc^{s_{2n+1}} \to \Hc^{s_0}}\\
        &\lesssim \prod_{j=1}^n \|X_j\|_{\Hc^{s_{2j}} \to \Hc^{s_{2j-1}}} \int_\Omega \prod_{j=0}^n \| a_j(\cdot, \omega)\|_{L^{\beta_j}_\infty(E_j)} \,d|\nu|(\omega).
    \end{split}
    \end{equation}
    With Pettis' theorem~\cite[Propositions~1.9 and~1.10]{VakhaniaTarieladze1987}, it now follows that for $\psi \in \Hc^{s_{2n+1}}$, 
    \[
    \omega \mapsto a_0(H_0, \omega )X_1 a_1(H_1,\omega) \cdots X_n a_n(H_n, \omega) \psi \in \Hc^{s}
    \]
    is Bochner integrable in $\Hc^s$. This holds in particular for $\psi \in \Hc^\infty$, and as $s \in \R$ was taken arbitrarily it follows that for $\psi \in \Hc^\infty$,
    \[
    \int_\Omega  a_0(H_0, \omega )X_1 a_1(H_1,\omega) \cdots X_n a_n(H_n, \omega) \psi \, d\nu(\omega) \in \Hc^\infty.
    \]
    It is therefore clear that we have a well-defined operator
    \[
    \int_\Omega  a_0(H_0, \omega )X_1 a_1(H_1,\omega) \cdots X_n a_n(H_n, \omega)  \, d\nu(\omega) \in \op^{ \sum_j r_j + \sum_j \beta_j h_j}(\Theta).
    \]
    
The claim that this operator is independent of the chosen representation of $\phi$
\[
\phi(\lambda_0,\ldots,\lambda_n)=\int_\Omega a_0(\lambda_0,\omega)\cdots a_n(\lambda_n,\omega)\, d\nu(\omega)
\]
permits taking the infimum over all representations on the right-hand side of the estimate~\eqref{eq:MOIOperatorNorm} giving the norm estimate in the assertion of this theorem. This independence follows from the proof of~\cite[Lemma~4.3]{ACDS}. Namely, given $\eta, \xi \in \Hc^\infty$, it is easy to check that $\theta_{\eta, \xi}:\Hc^\infty \to \Hc^\infty$ defined by
\[
\theta_{\eta, \xi} (\psi) := \langle \eta, \psi \rangle_{\Hc} \xi, \quad \psi \in \Hc^\infty,
\]
is an element of $\op^{-\infty}(\Theta)$.
The computations in~\cite[Lemma~4.3]{ACDS} give that, for $\eta_k, \xi_k \in \Hc^\infty$, the integral
\[
\int_\Omega a_0(H_0, \omega) \theta_{\eta_1, \xi_1} a_1(H_1, \omega) \cdots \theta_{\eta_n, \xi_n} a_n(H_n, \omega)\, d\nu(\omega) \in B(\Hc)
\]
does not depend on the chosen representation of $\phi$, and so neither does
\[
\int_\Omega a_0(H_0, \omega) \theta_{\eta_1, \xi_1} a_1(H_1, \omega) \cdots \theta_{\eta_n, \xi_n} a_n(H_n, \omega)\, d\nu(\omega) \bigg|_{\Hc^\infty} \in \op^{-\infty}(\Theta).
\]
The $so$-density of the span of $\{\theta_{\eta, \xi} : \eta, \xi\in \Hc^\infty\}$ in $B(\Hc^{s_{2i}}, \Hc^{s_{2i-1}})$ combined with Lemma~\ref{L:WeaklyMeasComp} concludes the proof.
\end{proof}

\begin{prop}
    The MOI constructed in Theorem~\ref{T:MainMOIConstruction} is linear in its symbol:
    \[
    T_{\alpha\phi + \beta \psi}^{H_0,\ldots, H_n}(X_1, \ldots, X_n) = \alpha T_{\phi}^{H_0,\ldots, H_n}(X_1, \ldots, X_n) + \beta T_{\psi}^{H_0,\ldots, H_n}(X_1, \ldots, X_n), \quad \alpha, \beta \in \C.
    \]
\end{prop}
\begin{proof}
    If both $\phi, \psi: \R^{n+1}\to \C$ have an integral representation of the required form over measure spaces $\Omega$ and $\Sigma$ respectively, then $\alpha \phi + \beta \psi$ can be decomposed appropriately as an integral over the disjoint union $\Omega \sqcup \Sigma$. The assertion then follows by elementary arguments.
\end{proof}

\begin{rem}
    The MOI constructed in Theorem~\ref{T:MainMOIConstruction} is independent of the operator $\Theta$ defining the abstract pseudodifferential calculus in the following sense. If $H_i$ and $X_i$ are operators on $\Hc$ such that $X_i |_{\Hc^\infty} \in \op^{r_i}(\Theta)$ and $a_i(H_i, \omega)|_{\Hc^{\infty}(\Theta)} \in \op^{\beta_i h_i}(\Theta)$ satisfying the conditions of Theorem~\ref{T:MainMOIConstruction}, then the proof of Theorem~\ref{T:MainMOIConstruction} shows that
    we can define $T_\phi^{H_0,\ldots, H_n}(X_1,\ldots, X_n): \Hc^{\sum_i r_i + \sum_i \beta_i h_i} \to \Hc$ by
    \begin{align*}
    T_\phi^{H_0, \ldots, H_n}(X_1,\ldots, X_n) \psi = \int_\Omega a_0(H_0, \omega) 
 V_1 a_1(H_1, \omega) \cdots V_n a_n(H_n, \omega) &\psi\, d\nu(\omega) \in \Hc, \\&\psi \in \Hc^{\sum_i r_i + \sum_i \beta_i h_i},
    \end{align*}
    which is a map that, apart from the definition of its domain, does not depend on $\Theta$.
\end{rem}

Furthermore, using the functional calculus for $\op^0(\Theta)$ in Theorem~\ref{T:FunctCalcOp0}, we can similarly define multiple operator integrals for $\op^0(\Theta)$.

\begin{thm}\label{T:MOOIsop0}
    Let $H_i, [\Theta, H_i]\in \op^0(\Theta)$, $i=0, \ldots, n$, be such that each $\overline{H_i}^{0,0}$ is self-adjoint. For $\phi \in T^0(\R) \boxtimes_i \cdots \boxtimes_i T^0(\R)$ with corresponding representation
    \[
\phi(\lambda_0,\ldots,\lambda_n)=\int_\Omega a_0(\lambda_0,\omega)\cdots a_n(\lambda_n,\omega) \, d\nu(\omega),
\]
the integral
\[
T_\phi^{H_0,\ldots,H_n}(X_1,\ldots,X_n)\psi:=\int_\Omega a_0(H_0,\omega)X_1 a_1(H_1,\omega)\cdots X_n a_n(H_n,\omega)\psi\,d\nu(\omega), \quad \psi\in\Hc^\infty
\]
for $X_1,\ldots,X_n\in\op(\Theta)$, converges as a Bochner integral in $\Hc^s$ for every $s \in \R$, and defines an $n$-multilinear map $T_\phi^{H_0,\ldots,H_n}:\op^{r_1}(\Theta)\times\cdots\times\op^{r_n}(\Theta)\to\op^{\sum_j r_j}(\Theta)$.
\end{thm}
\begin{proof}
    By definition of $T^0(\R) \boxtimes_i \cdots \boxtimes_i T^0(\R)$ we have for all $k_0, \ldots, k_n \in \N$,
    \[
    \int_\Omega \|a_0(\cdot, \omega)\|_{T^0(\R), k_0} \cdots \|a_n(\cdot, \omega)\|_{T^0(\R), k_n} \, d|\nu|(\omega) < \infty.
    \]
    The proof of the theorem is then identical to the proof of Theorem~\ref{T:MainMOIConstruction}, using that
    \[
    \|a_j(H_j, \omega)\|_{\Hc^s \to \Hc^s} \leq \sum_{k=0}^{\lceil 2^{|s|}\rceil+1} \|a_j(\cdot, \omega)\|_{T^0(\R), k},
    \]
    see Theorem~\ref{T:FunctCalcOp0}.
\end{proof}

\section{Divided differences}\label{S:FunctSpac}
The conditions that appear in Theorem~\ref{T:MainMOIConstruction} and Theorem~\ref{T:MOOIsop0} on the symbol $\phi: \R^{n+1}\to \C$ need to be analysed more closely in order to prepare these multiple operator integral constructions for practical applications. The main result of this section is Lemma~\ref{L:DivDifT}, which gives that for $f \in T^\beta(\R)$, $\beta \in \R$, the divided difference $f^{[n]}$ has an integral representation satisfying the conditions of the MOI construction for $\Theta$-elliptic operators in Theorem~\ref{T:MainMOIConstruction}, and that for $f\in C^\infty(\R)$ the divided difference $f^{[n]}$ can be used in the MOIs for zero-order operators in Theorem~\ref{T:MOOIsop0}.

First of all, for functions $\phi:\R^{n+1} \to \C$ it is an equivalent condition to admit a representation
\[
\phi(\lambda_0,\ldots,\lambda_n)=\int_\Omega a_0(\lambda_0,\omega)\cdots a_n(\lambda_n,\omega)\, d\nu(\omega)
\]
for a finite complex measure space (technically: of finite variation) $(\Omega,\nu)$ and measurable functions $a_j:\R\times\Omega\to\C$ such $(x,\omega)\mapsto a_j(x,\omega)\langle x \rangle^{-\beta_j}$ is $E_j \times \nu$-a.e. bounded for $\beta_j \in \R$,
or a representation
\[
\phi(\lambda_0,\ldots,\lambda_n)=\int_\Sigma b_0(\lambda_0,\sigma)\cdots b_n(\lambda_n,\sigma) \, d\mu(\sigma),
\]
where $\Sigma$ is a $\sigma$-finite measure space, $b_j: \R \times \Omega \to \C$ measurable, and
\[
\int_\Sigma \|b_0(\cdot, \sigma)\|_{L^{\beta_0}_\infty(E_0)} \cdots \|b_n(\cdot, \sigma)\|_{L^{\beta_n}_\infty(E_n)}\, d|\mu|(\sigma)<\infty.
\]
Namely, given the second representation, a representation of the first type can be obtained by putting~\cite[p.48]{SkripkaTomskova2019}
\[
a_j(\lambda_j, \sigma) := \frac{b_j(\lambda_j, \sigma)}{\|b_j(\cdot, \sigma)\|_{L_\infty^{\beta_j}(E_j)}}, \quad \nu(\sigma):= \|b_0(\cdot, \sigma)\|_{L^{\beta_0}_\infty(E_0)} \cdots \|b_n(\cdot, \sigma)\|_{L^{\beta_n}_\infty(E_n)}\mu(\sigma).
\]

We now again use the construction of an almost analytic extension (see Definition~\ref{def:AlmostAnalyticExtension}) to provide an explicit integral representation for $f\in T^\beta(\R)$. Recall the definition of the product $\boxtimes_i$, Definition~\ref{def:ProjIntBox}.

\begin{lem}\label{L:DivDifT}
\begin{enumerate}
    \item Take $n \in \N$, let $\alpha$ be some real number with $-1 \leq \alpha \leq n$, and consider any collection of real numbers $-1 \leq \beta_0, \ldots, \beta_n \leq 0$ such that $\sum \beta_j = \alpha - n$. Then
    \[
f \in T^\alpha(\R) \Rightarrow f^{[n]} \in S^{\beta_0}(\R) \boxtimes_i \cdots \boxtimes_i S^{\beta_n}(\R),
\]
where for each $k_0, \ldots, k_n \in \N$ we have
\begin{align*}
\|f^{[n]}(\lambda_0, \ldots, \lambda_n)\|_{S^{\beta_0}(\R) \boxtimes_i \cdots \boxtimes_i S^{\beta_n}(\R), k_0, \ldots, k_n} \lesssim \sum_{r=0}^{n+\sum_{j=0}^n k_j+2} \| f \|_{T^\alpha(\R), r}.
\end{align*}

\item Let $\alpha \leq n$. Then
\[
f \in T^\alpha(\R) \Rightarrow f^{[n]} \in \sum_{\substack{\beta_0, \ldots, \beta_n \leq 0 \\ \sum \beta_j= \alpha -n}}S^{\beta_0}(\R) \boxtimes_i \cdots \boxtimes_i S^{\beta_n}(\R).
\]
For each component $\phi \in S^{\beta_0}(\R) \boxtimes_i \cdots \boxtimes_i S^{\beta_n}(\R)$ in the (finite) decomposition, we have
\begin{align*}
\|\phi\|_{S^{\beta_0}(\R) \boxtimes_i \cdots \boxtimes_i S^{\beta_n}(\R), k_0, \ldots, k_n} \lesssim \sum_{r=0}^{n+\sum_{j=0}^n k_j+2} \| f \|_{T^\alpha(\R), r}.
\end{align*}

\item Let $\alpha \geq n$. Then
\[
f \in T^\alpha(\R) \Rightarrow f^{[n]} \in \sum_{ \sum \beta_j= \alpha -n}S^{\beta_0}(\R) \boxtimes_i \cdots \boxtimes_i S^{\beta_n}(\R).
\]
For each component $\phi \in S^{\beta_0}(\R) \boxtimes_i \cdots \boxtimes_i S^{\beta_n}(\R)$ in the (finite) decomposition, we have
\begin{align*}
\|\phi\|_{S^{\beta_0}(\R) \boxtimes_i \cdots \boxtimes_i S^{\beta_n}(\R), k_0, \ldots, k_n} \lesssim \sum_{r=0}^{n+\sum_{j=0}^n k_j+2} \| f \|_{T^\alpha(\R), r}.
\end{align*}
\end{enumerate}
\end{lem}
\begin{proof}
\begin{enumerate}
    \item For $g \in C^\infty_c(\mathbb{R})$ with almost analytic extension $\tilde{g}$, we have
\[
g(x) = -\frac{1}{\pi} \int_{\mathbb{C}} \frac{\partial \tilde{g}}{\partial \overline{z}} (z-x)^{-1}\,dz,
\]
and hence
\begin{align}\label{eq:divdifcompact}
g^{[n]}(\lambda_0, \ldots, \lambda_n) = \frac{(-1)^{n}}{\pi} \int_{\mathbb{C}} \frac{\partial \tilde{g}}{\partial \overline{z}} (z-\lambda_0)^{-1} \cdots (z-\lambda_n)^{-1}\, dz.
\end{align}

Now take $f \in T^\alpha(\R)$ with $\alpha \leq n$ and with almost analytic extension $\tilde{f}$. Directly from Definition~\ref{def:AlmostAnalyticExtension}, writing $\sigma(z) := \tau(\frac{\Im(z)}{\langle\Re(z)\rangle})$, it follows that (cf. \cite[Section~2.2]{Davies1995b})
\begin{align*}
    \frac{\partial \tilde{f}}{\partial \overline{z}} = \frac{1}{2} \bigg(\sum_{r=0}^N f^{(r)}(\Re(z)) \frac{(i\Im(z))^r}{r!} \bigg)(\sigma_x(z) + i\sigma_y(z) ) + \frac{1}{2} f^{(N+1)}(\Re(z)) \frac{(i\Im(z))^N}{n!}\sigma(z).
\end{align*}
We define
\[
U:=\{z \in \C : \langle \Re(z) \rangle < |\Im(z)| < 2 \langle \Re(z) \rangle \}, \qquad V:=\{z \in \C : 0 \leq |\Im(z)| < 2 \langle \Re(z) \rangle \},
\]
and note that the support of $\sigma$ is contained in $V$, while the support of $\sigma_x$ and $\sigma_y$ are contained in $U$. More precisely,
\[
|\sigma_x(z) + i \sigma_y(z)| \lesssim \frac{1}{\langle \Re(z) \rangle} \chi_{U}(z).
\]
Therefore we have the estimate~\cite[Lemma~1]{Davies1995a}\begin{align*}
    \int_{\mathbb{C}} \bigg| &\frac{\partial \tilde{f}}{\partial \overline{z}}\bigg| |z-\lambda_0|^{-1} \cdots |z-\lambda_n|^{-1}\,dz\\
    & \lesssim \sum_{r=0}^N \int_U |f^{(r)}(\Re(z))| |\Im(z)|^{r-n-1} \langle \Re(z) \rangle^{-1} \,dz + \int_V |f^{(N+1)}(\Re(z))| |\Im(z)|^{N-n-1} \,dz\\
    &\lesssim \sum_{r=0}^{N+1} \int_{\R} |f^{(r)}(x)|\langle x \rangle^{r-n-1}\,dx = \sum_{r=0}^{N+1} \|f\|_{T^{n}(\R), r},
\end{align*}
where the last estimate (integration over the imaginary direction) is justified when $N \geq n+1$.

Hence the integral
\begin{align*}
    \frac{(-1)^{n}}{\pi} \int_{\mathbb{C}} \frac{\partial \tilde{f}}{\partial \overline{z}} (z-\lambda_0)^{-1} \cdots (z-\lambda_n)^{-1}\, dz
\end{align*}
converges.

Since~\cite[Lemma~6]{Davies1995a} gives that $C^\infty_c(\R)$ is dense in $T^\alpha(\R)$, the Lebesgue dominated convergence theorem then gives that the identity~\eqref{eq:divdifcompact} extends to all $f \in T^\alpha(\R)$, $\alpha \leq n$, i.e.
\[
f^{[n]}(\lambda_0, \ldots, \lambda_n) = \frac{(-1)^{n}}{\pi} \int_{\mathbb{C}} \frac{\partial \tilde{f}}{\partial \overline{z}} (z-\lambda_0)^{-1} \cdots (z-\lambda_n)^{-1}\,dz.
\]
In order to show that this is a decomposition as described in Definition~\ref{def:ProjIntBox}, with $a_j(\lambda_j, z) = (z-\lambda_j)^{-1}$, we will now estimate the expressions
\begin{equation}\label{eq:Sbetanorm}\begin{split}
    \|(z-\cdot)^{-1}\|_{S^{\beta}(\R), k}&=\sup_{\lambda \in \mathbb{R}} \langle \lambda \rangle^{k-\beta} \bigg|\frac{\partial^{k}}{\partial \lambda^{k}} (z-\lambda)^{-1}\bigg| \\
    &\lesssim \bigg(\sup_{\lambda \in \mathbb{R}} \langle \lambda \rangle^{-\beta} |z-\lambda|^{-1}\bigg) \cdot \bigg(\sup_{\lambda \in \mathbb{R}} \langle \lambda \rangle^{k} |z-\lambda|^{-k}\bigg).
\end{split}
\end{equation}

Note that, for $\lambda\in \R$ and $z \in \C \setminus \R$,
\begin{align}\begin{split}\label{eq:Stap1}
    \frac{\langle \lambda \rangle}{|z-\lambda|} = \frac{|\lambda \pm i|}{|z-\lambda|} &\leq 1 + \frac{|z\pm i|}{|z-\lambda|}\\
    &\leq 1+ \frac{\langle z \rangle }{|\Im(z)|},
\end{split}
\end{align}
as  $\min(|z+i|,|z-i|)\leq\langle z\rangle$. Next, we estimate
\[
\sup_{\lambda \in \R} \langle \lambda \rangle^{-\beta} |z-\lambda|^{-1}
\]
for $-1 \leq \beta \leq 0$. We estimate the supremum over $\lambda > 1$, $|\lambda| \leq 1$ and $\lambda < -1$ separately. First, for $|\lambda| \leq 1$ we have $1 \leq \langle \lambda \rangle \leq \sqrt{2}$, and so
\[
\sup_{|\lambda|\leq 1} \langle \lambda \rangle^{-\beta} |z-\lambda|^{-1} 
\lesssim \frac{1}{|\Im(z)|}.
\]
For $\lambda > 1$, we have $\langle \lambda \rangle \leq \sqrt{2} \lambda$, so that
\begin{align*}
    \sup_{\lambda > 1} \langle \lambda \rangle^{-\beta} |z-\lambda|^{-1} &\lesssim 
    \sup_{\lambda > 1-\Re(z)}  \frac{(\lambda+\Re(z))^{-\beta}}{\big(\lambda^2+\Im(z)^2\big)^{\frac12}}.
\end{align*}
Writing $v = (\lambda, \Im(z)) \in \R^2$, then by using Cauchy--Schwarz for the inner product on $\R^2$, we have
\begin{align*}
    \sup_{\lambda > 1-\Re(z)}  \frac{(\lambda+\Re(z))^{-\beta}}{\big(\lambda^2+\Im(z)^2\big)^{\frac12}} & = \sup_{\lambda > 1-\Re(x)}  \frac{(v \cdot (1, \frac{\Re(z)}{\Im(z)}))^{-\beta}}{\|v\|}\\
     &\leq \sup_{\lambda > 1-\Re(z)}  \frac{\| (1, \frac{\Re(z)}{\Im(z)})\|^{-\beta}}{\|v\|^{1+\beta}}\\
    &\leq  \frac{|z|^{-\beta}}{ |\Im(z)|}.
\end{align*}

For $\lambda < -1$ we have a similar estimate, and hence combined we have
\begin{equation}\label{eq:Stap2}
    \sup_{\lambda \in \R}\langle \lambda \rangle^{-\beta} |z-\lambda|^{-1} \lesssim  \frac{1}{|\Im(z)|} \max(1, |z|^{-\beta}) \leq \frac{\langle z \rangle^{-\beta}}{|\Im(z)|}.
\end{equation}

Combining~\eqref{eq:Sbetanorm},~\eqref{eq:Stap1} and~\eqref{eq:Stap2} we get an estimate 
\[
\sup_{\lambda \in \mathbb{R}} \langle \lambda \rangle^{k-\beta} \bigg|\frac{\partial^{k}}{\partial \lambda^{k}} (z-\lambda)^{-1}\bigg| \lesssim \frac{\langle z \rangle^{-\beta}}{|\Im(z)|} \bigg(1+\frac{\langle z \rangle}{|\Im(z)|}\bigg)^k.
\]

Let $-1 \leq \beta_0, \ldots, \beta_n \leq 0$. Taking the inequality above and proceeding as before with $N \geq n+1+\sum_{j=0}^n k_j$, we have
\begin{align*}
    \int_{\mathbb{C}} \bigg| \frac{\partial \tilde{f}}{\partial \overline{z}}\bigg|& \left(\sup_{\lambda_0 \in \mathbb{R}} \langle \lambda_0 \rangle^{k_0-\beta_0} \bigg|\frac{\partial^{k_0}}{\partial \lambda_0^{k_0}} (z-\lambda_0)^{-1}\bigg| \right) \cdots \left(\sup_{\lambda_n \in \mathbb{R}} \langle \lambda_n \rangle^{k_n-\beta_n} \bigg|\frac{\partial^{k_n}}{\partial \lambda_n^{k_n}} (z-\lambda_0)^{-1}\bigg| \right)\,dz \\
    &\lesssim \sum_{r=0}^{N+1}\int_{\mathbb{R}} |f^{(r)}(x)|  \langle x \rangle^{r-n-1-\sum_{j=0}^n \beta_j} \,dx < \infty.
\end{align*}
This converges in particular for $\sum_{j=0}^n \beta_j = \alpha - n$. Since $-n-1 \leq \sum_{j=0}^n \beta_j \leq 0$, this choice is possible if $-1 \leq \alpha \leq n$. We have therefore proved for $-1 \leq \alpha \leq n$, and $-1 \leq \beta_0, \ldots, \beta_n \leq 0$ such that $\sum_{j=0}^n \beta_j = \alpha - n$, that 
\begin{align*}
\|f^{[n]}(\lambda_0, \ldots, \lambda_n)&\|_{S^{\beta_0}(\R) \boxtimes_i \cdots \boxtimes_i S^{\beta_n}(\R), k_0, \ldots, k_n} \\
&= \int_{\mathbb{C}} \|  (z-\cdot)^{-1} \|_{S^{\beta_0}(\R), k_0} \cdots \|(z-\cdot)^{-1} \|_{S^{\beta_n}(\R), k_n} \bigg| \frac{\partial \tilde{f}}{\partial \overline{z}}\bigg|\,dz \\
&\lesssim \sum_{r=0}^{n+\sum_{j=0}^n k_j+2} \| f \|_{T^n(\R), r} \leq \sum_{r=0}^{n+\sum_{j=0}^n k_j+2} \|f\|_{T^\alpha(\R),r}.
\end{align*}

\item For $f \in T^{\alpha}(\R)$, $-1 \leq \alpha \leq n$, we have by the first part of the lemma that for each $n \in \N$,
\[
f^{[n]} \in S^{\beta_0}(\R) \boxtimes_i \cdots \boxtimes_i S^{\beta_n}(\R)
\]
where each $\beta_j$ can be chosen to lie in the interval $[-1, 0]$, and $\sum \beta_j = \alpha - n$. 

For $f \in T^\alpha(\R)$ with $\alpha \leq -1$, we can write $f  = g \cdot (x+i)^{-k}$ where $g \in T^\beta(\R)$, $-1 \leq \beta \leq 0$ and $k \in \N$.
The Leibniz rule for divided differences dictates
\[
f^{[n]}(\lambda_0, \ldots, \lambda_n) = \sum_{l=0}^{n} g^{[l]}(\lambda_0, \ldots, \lambda_l) \big((x+i)^{-k}\big)^{[n-l]}(\lambda_l, \ldots, \lambda_n).
\]
From part 1 and the explicit form of the divided differences of $\big((x+i)^{-k}\big)^{[n]}$ we therefore conclude that each term is an element of
\[
\sum_{\substack{\beta_0, \ldots, \beta_n \leq 0 \\ \sum \beta_j= \alpha -k-n}}S^{\beta_0}(\R) \boxtimes_i \cdots \boxtimes_i S^{\beta_n}(\R),
\]
with the required estimate of norms.

\item This follows analogously to assertion 2, by analysing $\big(g(x+i)^k\big)^{[n]}$ for $g \in T^{\alpha}(\R)$ with $-1 \leq \alpha \leq 0$.
\qedhere\end{enumerate}
\end{proof}

\begin{rem}\label{rem:fToLinfty}
    The proof of Lemma~\ref{L:DivDifT} in fact shows that if $f \in C^{n+2}(\R)$ such that $\| f\|_{T^\beta(\R), k} < \infty$ for $k=0, \ldots, n+2$, then given any spectral measures $E_0, \ldots, E_n$, we have
    \[
    \| f^{[n]} \|_{L_\infty^{\beta_0}(E_0) \hat{\otimes}_i \cdots \hat{\otimes}_i L_\infty^{\beta_n}(E_n)} \leq \|f^{[n]}\|_{S^{\beta_0}(\R) \boxtimes_i \cdots \boxtimes_i S^{\beta_n}(\R), 0, \ldots, 0} \leq \sum_{k=0}^{n+2} \| f\|_{T^\beta(\R), k} < \infty.
    \]
    For $n = 0$, the space of functions that satisfy this condition closely resembles the space $\mathfrak{F}_m(\R)$ used in~\cite{CareyGesztesy2016} in the context of double operator integrals.
\end{rem}

\begin{thm}\label{T:MOOIforNCG1}
    Let $H_0, \ldots, H_n$ be such that each $H_i \in \op^{h}(\Theta), h > 0,$ is $\Theta$-elliptic and symmetric, and let $f \in C^{n+2}(\R)$ such that $\|f\|_{T^\beta(\R), k} <\infty$, $k=0, \ldots, n+2$ for some $\beta \in \R$. Then for any $X_i\in \op^{r_i}(\Theta)$, $i=1,\ldots, n$,
     \[
     T^{H_0, \ldots, H_n}_{f^{[n]}}(X_1, \ldots, X_n) \in \op^{r+(\beta-n)h}(\Theta),
     \]
     with the estimate
     \[
      \| T^{H_0, \ldots, H_n}_{f^{[n]}}(X_1, \ldots, X_n)\|_{\Hc^{s+r+ (\beta-n)h} \to \Hc^s} \leq C_{s,H_0, \ldots, H_n} \left(\sum_{k=0}^{n+2} \|f\|_{T^\beta(\R), k} \right) \prod_{i=1}^n \|X_i\|_{\Hc^{s_i+r_i} \to \Hc^{s_i}} 
     \]
     for some $s_1, \ldots, s_n \in \R$.    
\end{thm}
\begin{proof}
    Lemma~\ref{L:DivDifT} and Remark~\ref{rem:fToLinfty}, combined with Theorem~\ref{T:MainMOIConstruction}.
\end{proof}

\begin{thm}\label{T:MOOIforNCG2}
    Let $H_i, [\Theta, H_i] \in \op^0(\Theta)$, $i=0, \ldots, n$, be such that each $\overline
    {H_i}^{0,0}:\Hc \to \Hc$ is self-adjoint, and let $f\in C^\infty(\R)$. Then
     \[
     T^{H_0, \ldots, H_n}_{f^{[n]}}(X_1, \ldots, X_n) \in \op^{\sum_j r_j}(\Theta).
     \]  
\end{thm}
\begin{proof}
    Lemma~\ref{L:DivDifT} and Remark~\ref{rem:fToLinfty}, combined with Theorem~\ref{T:MainMOIConstruction}.
\end{proof}

Due to Lemma~\ref{L:DivDifT}, the divided difference $f^{[n]}$ for $f \in C^\infty(\R)$ is in particular a permitted symbol for Theorem~\ref{T:MOOIsop0}.

\section{MOI identities}
\label{S:MOIIdentities}
The most important identities for our applications of our multiple operator integrals are the following. These are generalisations of the identities
\begin{align*}
    [f(H),X] &= T^{H,H}_{f^{[1]}}([H,X]);\\
    f(H+X)-f(H)&= T^{H+X, H}_{f^{[1]}}(X),
\end{align*}
seeing that $f(H) = T_{f^{[0]}}^H()$.

\begin{prop}\label{P:UMOIcom}
Let $a, X_1,\ldots,X_n \in\op(\Theta)$, let $H_i\in \op^{h_i}(\Theta)$, $h_i > 0$ be symmetric and $\Theta$-elliptic and let $f \in T^\beta(\R)$, $\beta \in \R$. Then
\begin{align}
T^{H_0, \ldots, H_n}_{f^{[n]}}(X_1,\ldots,X_j,aX_{j+1},\ldots,X_n)&-T_{f^{[n]}}^{H_0, \ldots, H_n}(X_1,\ldots,X_ja,X_{j+1},\ldots,X_n)\label{eq:comeq1}\\
&=T_{f^{[n+1]}}^{H_0, \ldots, H_j, H_j, \ldots, H_n}(X_1,\ldots,X_j,[H_j,a],X_{j+1},\ldots,X_n);\notag\\
    T^{H_0, \ldots, H_n}_{f^{[n]}}(aX_1,\ldots,X_n) - a T^{H_0, \ldots, H_n}_{f^{[n]}}&(X_1,\ldots,X_n) \\
    &= T^{H_0, H_0, H_1, \ldots, H_n}_{f^{[n+1]}}([H_0,a],X_1,\ldots,X_n) ;\notag\\
    T^{H_0, \ldots, H_n}_{f^{[n]}}(X_1,\ldots,X_n)a -  T^{H_0, \ldots, H_n}_{f^{[n]}}&(X_1,\ldots,X_na)\\ &= T^{H_0, \ldots, H_n, H_n}_{f^{[n+1]}}(X_1,\ldots,X_n,[H_n,a]).\notag
\end{align}
Moreover, for $A \in \op^a(\Theta), a>0$, $B \in \op^b(\Theta), b>0$ symmetric and $\Theta$-elliptic,
\begin{align}\label{eq:SuperscriptDiff}
    \begin{split}
        &T^{H_0, \ldots, H_{j-1}, A, H_{j+1}, \ldots, H_n}_{f^{[n]}}(X_1, \ldots, X_n) - T^{H_0, \ldots, H_{j-1}, B, H_{j+1}, \ldots, H_n}_{f^{[n]}}(X_1, \ldots, X_n)\\
        &=T^{H_0, \ldots, H_{j-1}, A, B, H_{j+1}, \ldots, H_n}_{f^{[n+1]}}(X_1, \ldots, X_j, A-B, X_{j+1}, \ldots, X_n).
    \end{split}
\end{align}
The same assertions hold for self-adjoint $H_i, A, B\in \op^0(\Theta)$ such that $[\Theta,H_i], [\Theta,A], [\Theta,B] \in \op^0(\Theta)$ and with $f\in C^\infty(\R)$.
\end{prop}
\begin{proof}
We prove equation~\ref{eq:comeq1} in the $\Theta$-elliptic case, the other identities and the zero-order case follow analogously. For $f\in T^\beta(\R)$, the multiple operator integrals appearing in the equation are then well-defined, see Theorem~\ref{T:MOOIforNCG1}.

Write
\begin{align*}
    F_j(\lambda_0, \ldots, \lambda_{n+1}) &:= f^{[n]}(\lambda_0, \ldots, \lambda_{j-1}, \lambda_{j+1}, \ldots \lambda_{n+1}),
\end{align*}
and observe that
\[
F_{j+1}(\lambda_0, \ldots, \lambda_{n+1}) - F_j(\lambda_0, \ldots, \lambda_{n+1}) = (\lambda_j - \lambda_{j+1}) f^{[n+1]}(\lambda_0, \ldots, \lambda_{n+1}).
\]
Hence, 
\begin{align*}
    T_{f^{[n+1]}}^{H_0, \ldots, H_j, H_j, \ldots, H_n}&(X_1,\ldots,X_j,[H_j,a],X_{j+1},\ldots,X_n)\\
    &= T_{(\lambda_j - \lambda_{j+1}) f^{[n+1]}}^{H_0, \ldots, H_j, H_j, \ldots, H_n}(X_1,\ldots,X_j,a,X_{j+1},\ldots,X_n)\\
    &= T_{F_{j+1}}^{H_0, \ldots, H_j, H_j, \ldots, H_n}(X_1,\ldots,X_j,a,X_{j+1},\ldots,X_n) \\
    &\quad - T_{F_j}^{H_0, \ldots, H_j, H_j, \ldots, H_n}(X_1,\ldots,X_j,a,X_{j+1},\ldots,X_n) \\
    &= T^{H_0, \ldots, H_n}_{f^{[n]}}(X_1,\ldots,X_j,aX_{j+1},\ldots,X_n)\\
    &\quad -T_{f^{[n]}}^{H_0, \ldots, H_n}(X_1,\ldots,X_ja,X_{j+1},\ldots,X_n). \qedhere
\end{align*}
\end{proof}

With these identities in hand, we can show that the MOI constructed in the previous section is an element of $\OP(\Theta)$ if all its components are and the symbol is a divided difference.

\begin{thm}\label{T:OPMOOIS}
     Let $H_0, \ldots, H_n$ be such that each $H_i \in \OP^{h}(\Theta), h > 0,$ is symmetric and $\Theta$-elliptic, and let $f \in T^\beta(\R)$ for some $\beta \in \R$. For operators $X_i \in \OP^{r_i}(\Theta)$, $r:= \sum_{i=1}^n r_i$, we have that
     \[
     T_{f^{[n]}}^{H_0, \ldots, H_n}(X_1, \ldots, X_n) \in \OP^{h(\beta-n)+ r}(\Theta).
     \]
     Similarly, if each $H_i$ instead is such that $H_i, [\Theta, H_i] \in \OP^0(\Theta)$ and $H_i$ is self-adjoint, then for any $f \in C^\infty(\R)$ we have that
     \[
     T_{f^{[n]}}^{H_0, \ldots, H_n}(X_1, \ldots, X_n) \in \OP^{r}(\Theta).
     \]
\end{thm}
\begin{proof}
     We focus on the $\Theta$-elliptic case, the zero-order case follows similarly. Taking $n=1$ to ease notation, using Proposition~\ref{P:UMOIcom} gives that
     \[
     [\Theta, T^{H_0, H_1}_{f^{[1]}}(X_1)] = T^{H_0, H_0, H_1}_{f^{[2]}}([\Theta, H_0],X_1) + T^{H_0, H_1}_{f^{[1]}}([\Theta,X_1]) + T^{H_0, H_1, H_1}_{f^{[2]}}(X_1, [\Theta, H_1]).
     \]
     As $[\Theta, H_i]\in \op^h(\Theta)$ and $[\Theta, X_1] \in \op^{r_1}(\Theta)$, Lemma~\ref{L:DivDifT} combined with Theorem~\ref{T:MainMOIConstruction} gives that
     \[
     T^{H_0, H_0, H_1}_{f^{[2]}}([\Theta, H_0],X_1), T^{H_0,H_1}_{f^{[1]}}([\Theta,X_1]), T^{H_0, H_1, H_1}_{f^{[2]}}(X_1, [\Theta, H_1]) \in \op^{h(\beta-n)+r}(\Theta).
     \]
     Higher commutators and $n > 1$ follow analogously.
\end{proof}

We have now proven Theorem~\ref{T:FunctCalcOPEll} and Theorem~\ref{T:FunctCalcOP0} which were claimed in the previous chapter: they are special cases of Theorem~\ref{T:OPMOOIS}. 

In the setting that $\Theta^{-1} \in \mathcal{L}_s$, $s>0$, it is immediate from Theorem~\ref{T:MainMOIConstruction} and Theorem~\ref{T:MainFunctCalc} that for $H_0, \ldots, H_n \in \op^{h}(\Theta), h >0$ symmetric and $\Theta$-elliptic, and $f\in T^\beta(\R)$, the multiple operator integral
\begin{equation}\label{eq:TypicalMOI}
    T_{f^{[n]}}^{H_0, \ldots, H_n}(X_1, \ldots, X_n) \in \op^{(\beta-n)h + r}(\Theta)
\end{equation}
can be considered to be a trace-class operator on $\Hc$ if $\beta$ is small enough. Namely, we have
\[
\|A\|_{1} \leq \|\Theta^{-s}\|_1 \|A\|_{\Hc^0 \to \Hc^{s}}.
\]

\section{Asymptotic expansions}
\label{S:AsympExp}
Through the identities proved in the previous section, the theory of multiple operator integrals lends itself well for establishing asymptotic expansions of operators. As an immediate example, we prove a noncommutative Taylor expansion for pseudodifferential operators.

\begin{thm}\label{T:Taylor}
     Let $f \in T^\beta(\R)$, $H \in \op^h(\Theta)$, $h>0$, $\Theta$-elliptic and symmetric, and let $V \in \op^r(\Theta)$ be symmetric. If the order of the perturbation $V$ is strictly smaller than that of $H$, i.e. $r < h$, we have
    \begin{align}\label{eq:Taylor asymptotic formula}
    f(H + V) \sim &\sum_{n=0}^\infty T^{H}_{f^{[n]}}(V, \ldots, V),
    \end{align}
    in the sense that
    \begin{align*}
        &f(H+V) - \sum_{n=0}^N T^{H}_{f^{[n]}}(V, \ldots, V)\in \op^{m_N}(\Theta)
    \end{align*}
    with $m_N \downarrow -\infty$.
\end{thm}
\begin{proof}
Using the last part of Proposition~\ref{P:UMOIcom} with $A = H+V$, $B = H$, we have
\begin{align*}
    f(H+V) - f(H) &= T^{H+V}_{f^{[0]}}() - T^H_{f^{[0]}}()\\
    &= T_{f^{[1]}}^{H+V, H}(V).
\end{align*}
Repeating the argument, we get
    \begin{align*}
            f(H+V) - &\sum_{n=0}^N  T^{H}_{f^{[n]}}(V, \ldots, V) = T_{f^{[N+1]}}^{H+V, H, \ldots, H}(V, \ldots, V).
    \end{align*}
Now, if $r < h$, Theorem~\ref{T:MOOIforNCG1} gives that
\[
T_{f^{[N+1]}}^{H+V, H, \ldots, H}(V, \ldots, V) \in \op^{(\beta-N-1)h+Nr}(\Theta),
\]
with 
\[
(\beta-N-1)h+Nr = N(r-h) + (\beta-1)h \downarrow -\infty.\qedhere
\]
\end{proof}

Note that, if $H$ and $V$ are commuting operators, \eqref{eq:Taylor asymptotic formula} recovers the classic Taylor expansion formula $f(H+V)\sim\sum_{n=0}^\infty \frac{f^{(n)}(H)}{n!}V^n$. Hence, the expansion~\eqref{eq:Taylor asymptotic formula} can be interpreted as a type of noncommutative Taylor expansion. 
Each term in the expansion can itself be expanded as follows. Recall that we write $\delta_H(X) := [H,X]$, $\delta_H^n(X) := \delta_H(\cdots \delta_H(\delta_H(X)) \cdots)$.

\begin{prop}\label{P:Expansion} Let $X_i \in \op^{r_i}(\Theta)$, $H\in \op^h(\Theta)$, $h >0$ symmetric and $\Theta$-elliptic, and $f\in T^\beta(\R)$. Then
    \begin{align*}
          T^{H}_{f^{[n]}}(X_1, \dots, X_n) &= \sum_{m=0}^{N}  \sum_{m_1 + \dots + m_n = m} \frac{C_{m_1, \dots, m_n}}{(n+m)!} \delta_H^{m_1}(X_1) \cdots \delta_H^{m_n}(X_n) f^{(n+m)}(H) \\
        &\quad +   S^n_{N}(X_1, \dots, X_n),
    \end{align*}
    where 
    \[
    C_{m_1, \dots, m_n} :=  \prod_{j=1}^{n} \binom{j+m_1 + \dots + m_{j}-1}{m_j}
    \]
    and the remainder $   S^n_{N}(X_1, \dots, X_n)$ is a sum of terms of the form
    \[
 \delta_H^{m_1}(X_1)\cdots \delta_H^{m_k}(X_{k})   T^{H}_{f^{[n+N+1]}}(1, \ldots, 1, \delta_H^{N+1-m_1 - \cdots - m_k}(X_{k+1}), 1, \ldots, 1, X_{k+2}, \ldots, X_n).
\]
If the commutators $\delta_H^k(X_j)$ have a lower order than the expected $r_j+kh$, explicitly if
    \[
    \delta_H^k(X_j) \in \op^{r_j+k(h-\varepsilon)}(\Theta)
    \]
    for some $\varepsilon > 0$, then the above gives an asymptotic expansion
    \[
        T^{H}_{f^{[n]}}(X_1, \dots, X_n) \sim \sum_{m=0}^{\infty}  \sum_{m_1 + \dots + m_n = m} \frac{C_{m_1, \dots, m_n}}{(n+m)!} \delta_H^{m_1}(X_1) \cdots \delta_H^{m_n}(X_n) f^{(n+m)}(H),
    \]
    in the sense that the remainder term 
    \[
    S^n_{N}(X_1, \ldots, X_n) \in \op^{m_N}(\Theta)
    \]
    with $m_N = \sum_j r_j + (\beta-n)h-\varepsilon (N+1) \downarrow -\infty$.
\end{prop}

The proof of this proposition is a lengthy combinatorial exercise, the computations of which are standard and appear in many proofs of the local index formula~\cite[Lemma~6.11]{CPRS1}\cite[Equation~(71)]{ConnesMoscovici1995}\cite[Lemma~2.12]{Higson2003}, see also~\cite{Daletskii1998,Paycha2011}. The novelty is that they can be performed in the very general context of (unbounded) MOIs. In effect, Proposition~\ref{P:Expansion} is a generalisation of the cited results.

\begin{defn}
    The multiset coefficient $\multinom{n}{k}$ for $n, k \in \N$ is defined as 
    \[
    \multinom{n}{k} := \binom{n+k-1}{k}.
    \]
\end{defn}

\begin{lem}\label{L:Commute1} For $f\in T^\beta(\R)$, $H\in \op^h(\Theta)$, $h> 0$ symmetric and $\Theta$-elliptic, $X_i \in \op^{r_i}(\Theta)$, we have
\begin{align*}
    &T^{H} _{f^{[n+j]}}(\underbrace{1, \dots, 1}_{j}, X_1, \dots, X_n ) \\
    &= \sum_{m=0}^{N}  \multinom{m+1}{j} \delta_H^{m}(X_1) T^{H}_{f^{[n+j+m]}}(\underbrace{1, \dots, 1}_{j+1+m}, X_2, \dots, X_n) + R^n_{j,N}(X_1, \dots, X_n),
\end{align*}
where 
\[
R^n_{j,N}(X_1, \dots, X_n) :=  \sum_{l=0}^{j} \multinom{N+1}{l} T^{H}_{f^{[n+j+N+1]}}(\underbrace{1, \dots, 1}_{j-l}, \delta_H^{N+1}(X_1), \underbrace{1, \dots, 1}_{N+1+l}, X_2, \dots, X_n).
\]
\end{lem}
\begin{proof}
    Multiset coefficients have the property that
    \[
    \sum_{l=0}^j \multinom{m}{l} = \multinom{m+1}{j}.
    \]
    The assertion of the lemma follows by induction on $N$. For $N=0$,
    \begin{align*}
        &T^{H}_{f^{[n+j]}}(\underbrace{1, \dots, 1}_{j}, X_1, \dots, X_n )\\
        &= X_1 T^{H}_{f^{[n+j]}}(\underbrace{1, \dots, 1}_{j+1}, X_2, \dots, X_n ) + \sum_{l=0}^j T^{H}_{f^{[n+j+1]}}(\underbrace{1, \dots, 1}_{j-l}, \delta_H(X_1), \underbrace{1, \dots, 1}_{1+l}, X_2, \dots, X_n)
    \end{align*}
    by applying Proposition~\ref{P:UMOIcom} $j+1$ times on $X_1$.

    Suppose that the assertion holds for $N-1$. Then 
    \begin{align*}
         &T^{H}_{f^{[n+j]}}(\underbrace{1, \dots, 1}_{j}, X_1, \dots, X_n ) \\
         &= \sum_{m=0}^{N-1}  \multinom{m+1}{j} \delta_H^m(X_1) T^{H}_{f^{[n+j+m]}}(\underbrace{1, \dots, 1}_{j+1+m}, X_2, \dots, X_n) \\
    & \quad +  \sum_{l=0}^{j} \multinom{N}{l} T^{H}_{f^{[n+j+N]}}(\underbrace{1, \dots, 1}_{j-l}, \delta_H^{N}(X_1), \underbrace{1, \dots, 1}_{N+l}, X_2, \dots, X_n)\\
    &=^*  \sum_{m=0}^{N-1} \multinom{m+1}{j} \delta_H^m(X_1) T^{H}_{f^{[n+j+m]}}(\underbrace{1, \dots, 1}_{j+1+m}, X_2, \dots, X_n) \\
    & \quad +  \sum_{l=0}^{j} \multinom{N}{l} \delta_H^{N}(X_1) T^{H}_{f^{[n+j+N]}}(\underbrace{1, \dots, 1}_{N+1+j}, X_2, \dots, X_n)\\
    & \quad +  \sum_{l=0}^{j} \multinom{N}{l} \sum_{k=0}^{j-l} T^{H}_{f^{[n+j+N+1]}}(\underbrace{1, \dots, 1}_{j-l-k}, \delta_H^{N+1}(X_1), \underbrace{1, \dots, 1}_{N+1+l+k}, X_2, \dots, X_n),
    \end{align*}
    where in the step marked with $*$ we applied Proposition~\ref{P:UMOIcom} $j-l$ times on $\delta_H^{N}(X_1)$. Continuing on, 
    \begin{align*}
 &T^{H}_{f^{[n+j]}}(\underbrace{1, \dots, 1}_{j}, X_1, \dots, X_n )\\
    &=  \sum_{m=0}^{N-1} \multinom{m+1}{j} \delta_H^m(X_1) T^{H}_{f^{[n+j+m]}}(\underbrace{1, \dots, 1}_{j+1+m}, X_2, \dots, X_n) \\
    & \quad +   \multinom{N+1}{j} \delta_H^{N}(X_1) T^{H}_{f^{[n+j+N]}}(\underbrace{1, \dots, 1}_{N+1+j}, X_2, \dots, X_n)\\
    & \quad +  \sum_{l=0}^{j} \sum_{k=0}^{j-l}  \multinom{N}{l} T^{H}_{f^{[n+j+N+1]}}(\underbrace{1, \dots, 1}_{j-l-k}, \delta_H^{N+1}(X_1), \underbrace{1, \dots, 1}_{N+1+l+k}, X_2, \dots, X_n).
    \end{align*}
 In the last sum, relabel $r:= k+l$, so that
\begin{align*}
    &T^{H}_{f^{[n+j]}}(\underbrace{1, \dots, 1}_{j}, X_1, \dots, X_n ) \\
    &= \sum_{m=0}^{N}  \multinom{m+1}{j} \delta_H^m(X_1) T^{H}_{f^{[n+j+m]}}(\underbrace{1, \dots, 1}_{j+1+m}, X_2, \dots, X_n) \\
    & \quad +  \sum_{r=0}^{j} \sum_{l=0}^{r}  \multinom{N}{l} T^{H}_{f^{[n+j+N+1]}}(\underbrace{1, \dots, 1}_{j-r}, \delta_H^{N+1}(X_1), \underbrace{1, \dots, 1}_{N+1+r}, X_2, \dots, X_n)\\
    &= \sum_{m=0}^{N} \multinom{m+1}{j} \delta_H^m(X_1) T^{H}_{f^{[n+j+m]}}(\underbrace{1, \dots, 1}_{j+1+m}, X_2, \dots, X_n) \\
    & \quad +  \sum_{r=0}^{j}   \multinom{N+1}{r} T^{H}_{f^{[n+j+N+1]}}(\underbrace{1, \dots, 1}_{j-r}, \delta_H^{N+1}(X_1), \underbrace{1, \dots, 1}_{N+1+r}, X_2, \dots, X_n).
\end{align*}
This concludes the induction step.
\end{proof}

\begin{proof}[Proof of Proposition~\ref{P:Expansion}]
Apply Lemma~\ref{L:Commute1} first to the first entry of $T^{H}_{f^{[n]}}(X_1, \dots, X_n)$,
\begin{align*}
    T^{H}_{f^{[n]}}(X_1, \dots, X_n ) &= \sum_{m_1=0}^{N}  \multinom{m_1+1}{0} \delta_H^{m_1}(X_1) T^{H}_{f^{[n+m_1]}}(\underbrace{1, \dots, 1}_{m_1+1}, X_2, \dots, X_n) \\
    &\quad + R^{n}_{0,N}(X_1, \dots, X_n).
\end{align*}
Apply Lemma~\ref{L:Commute1} once more, expanding up to order $N-m_1$ instead of $N$.
\begin{align*}
    &T^{H}_{f^{[n]}}(X_1, \dots, X_n ) \\
    &= \sum_{m_1=0}^{N} \sum_{m_2=0}^{N-m_1} \multinom{m_1+1}{0} \multinom{m_2+1}{m_1+1} \delta_H^{m_1}(X_1) \delta_H^{m_2}(X_2) T^{H}_{f^{[n+m_1+m_2]}}(\underbrace{1, \dots, 1}_{2+m_1+m_2}, X_3, \dots, X_n)\\
    & \quad + \sum_{m_1=0}^{N}\multinom{m_1+1}{0}  \delta_H^{m_1}(X_1) R^{n-1}_{1+m_1,N-m_1}(X_2, \dots, X_n) + R^{n}_{0,N}(X_1, \dots, X_n).
\end{align*}
Repeating gives the formula
\begin{align*}
    &T^{H}_{f^{[n]}}(X_1, \dots, X_n )\\
    &= \sum_{m=0}^{N}  \sum_{m_1 + \dots + m_n = m} \prod_{j=1}^n  \multinom{m_j+1}{j-1+m_1+\dots + m_{j-1}}\delta_H^{m_1}(X_1) \cdots \delta_H^{m_n}(X_n) T^{H}_{f^{[n+m]}}(\underbrace{1, \dots, 1}_{n+m})\\
    & \quad + S^n_{N}(X_1, \dots X_n),
\end{align*}
where
\begin{align*}
     S^n_{N}(X_1, \dots X_n)&:=   \sum_{k=0}^{n-1} \sum_{m_1 + \dots + m_k \leq N} \prod_{j=1}^{k}  \multinom{m_j+1}{j-1+m_2+\dots + m_{j-1}} \\
     &\quad \times \delta_H^{m_1}(X_1)\cdots \delta_H^{m_k}(X_{k})  R^{n-k}_{k+m_1+\dots+m_{k},N-m_1-\dots-m_{k}}(X_{k+1}, \dots, X_n).
\end{align*}

The observation that 
\[
\multinom{n}{k} = \multinom{k+1}{n-1},
\]
and the definition
\[
\multinom{n}{k} = \binom{n+k-1}{k},
\]
finishes the proof of the proposition.
\end{proof}

Putting these results together, we now have a noncommutative Taylor expansion.

\begin{thm}\label{T:CombiExpansion}
    Let $f \in T^\beta(\R)$, $H \in \op^h(\Theta)$, $h>0$, $\Theta$-elliptic and symmetric, and let $V \in \op^r(\Theta)$, $r < h$, be symmetric. Assume that 
    \[
    \delta_H^k(V) \in \op^{r+k(h-\varepsilon)}(\Theta)
    \]
    for some $\varepsilon > 0$. Then    
    we have the noncommutative Taylor expansion
    \[
    f(H+V) \sim  \sum_{n,m=0}^\infty \sum_{m_1 + \dots + m_n = m} \frac{C_{m_1, \dots, m_n}}{(n+m)!} \delta_H^{m_1}(V) \cdots \delta_H^{m_n}(V) f^{(n+m)}(H).
    \]
\end{thm}
\begin{proof}
    For $N, M > 0$, Theorem~\ref{T:Taylor} and Proposition~\ref{P:Expansion} give
    \begin{align*}
        f(H+V) &= \sum_{n=0}^{N} \sum_{m=0}^{M} \sum_{m_1 + \dots + m_n = m} \frac{C_{m_1, \dots, m_n}}{(n+m)!} \delta_H^{m_1}(V) \cdots \delta_H^{m_n}(V) f^{(n+m)}(H)\\
        &\quad + \op^{N(r-h)+(\beta-1) h}(\Theta) + \sum_{n=0}^N \op^{nr + (\beta-n) h-\varepsilon(M+1)}(\Theta).
    \end{align*}
    As $(N,M) \to (\infty, \infty)$, we see that the order of the remainder decreases to $-\infty$.
\end{proof}

The noncommutative Taylor expansion \eqref{eq:Taylor asymptotic formula} features prominently in~\cite{Paycha2007} for classical pseudodifferential operators on manifolds, and in~\cite{Daletskii1998,Paycha2011} in a more abstract sense. For bounded operators on Banach spaces, it has been studied in~\cite{HartmannLesch2024}. For related expansions for bounded operators $V$, see also
~\cite{Hansen,Suijlekom2011,Skripka2018,vNvS21a}.

\section{Trace expansions and NCG}
\label{S:AsympExpNCG}
The observation in the previous section that the local index formula is closely connected to the noncommutative Taylor expansion in Theorem~\ref{T:CombiExpansion} leads us naturally to the topic of asymptotic expansions of trace formulas. In various contexts of noncommutative geometry and beyond, expansions are studied of the kind
\begin{equation}\label{eq:AsympExp1}
    \Tr(f(tH+tV)) \underset{t\downarrow 0}{\sim} \sum_{k=0}^\infty c_k t^{r_k},
\end{equation}
for an increasing sequence $r_k \uparrow \infty$ and constants $c_k \in \mathbb{C}$, which means that as $t \downarrow 0$
\[
\Tr (f(tH+tV)) = \sum_{k=0}^N c_k t^{r_k} + O(t^{r_{N+1}})
\]
for every $N \in \R$. Or, more generally (c.f.~\cite{EcksteinIochum2018}),
\begin{equation}\label{eq:AsympExp2}
     \Tr(f(tH+tV)) \underset{t\downarrow 0}{\sim} \sum_{k=0}^\infty \rho_k(t),
\end{equation}
where $\rho_k(t) = O(t^{r_k})$ and \[
\Tr(f(tH+tV)) = \sum_{k=0}^N \rho_k(t) + O(t^{r_{N+1}}).
\]

To study the asymptotic expansion of expressions like
\[
\Tr(a e^{-t(D+V)^2}),
\]
we will use a modified version of Theorem~\ref{T:Taylor} and Proposition~\ref{P:Expansion}. 
For this purpose, we need to make a more detailed analysis of norm bounds of MOIs and of the remainder in the Taylor expansion in Theorem~\ref{T:Taylor}.

\begin{prop}\label{P:TraceEstimatewitht}
Let $H_i \in \op^h(\Theta)$ for a fixed $h > 0$ be symmetric and $\Theta$-elliptic operators. If $X_i \in \op^{r_i}(\Theta)$, $r:= \sum_{i=1}^n r_i$, $f \in T^\alpha(\mathbb{R})$ with $\alpha \leq n$, $t\leq 1$, then
    \[
    \| T^{tH_0, \ldots, tH_n}_{f^{[n]}}(X_1, \ldots, X_n) \|_{\Hc^{q+r+(\alpha-n)h} \to \Hc^q} \lesssim t^{\alpha-n}, \quad q\in \R.
    \]
Rephrased, if $f \in T^{\frac{u-r}{h}+n}$, $u\leq r$, then
\[
    \| T^{tH_0, \ldots, tH_n}_{f^{[n]}}(X_1, \ldots, X_n) \|_{\Hc^{q+u} \to \Hc^q} \lesssim t^{\frac{u-r}{h}}, \quad q\in \R.
    \]
\end{prop}
\begin{proof}
Lemma~\ref{L:DivDifT} gives that
\[
      f^{[n]}\in \sum_{\beta_0 + \cdots + \beta_n = \alpha-n} S^{\beta_0}(\R) \boxtimes_i \cdots \boxtimes_i S^{\beta_n}(\R),
\]
where each $\beta_i \leq 0$ since $\alpha \leq n$ (recall Definition~\ref{def:ProjIntBox} for this notation).
   Consider one of the summands, $\phi \in S^{\beta_0}(\R) \boxtimes_i \cdots \boxtimes_i S^{\beta_n}(\R)$. Then 
   \[
   \phi(\lambda_0, \ldots, \lambda_n) = (\lambda_0 + i)^{\beta_0} \cdots (\lambda_n + i)^{\beta_n} \cdot \psi(\lambda_0, \ldots, \lambda_n),
   \]
   for a function $\psi \in S^0(\R) \boxtimes_i \cdots \boxtimes_i S^0(\R)$, and thus, by Remark~\ref{rem:SipvL} and Theorem~\ref{T:MainMOIConstruction},
\begin{align*}
    T^{tH_0, \ldots, tH_n}_{\phi}&(X_1, \ldots, X_n)\\
    &=  T^{tH_0, \ldots, tH_n}_{\psi}((tH_0+i)^{\beta_0} X_1 (tH_1+i)^{\beta_1}, X_2(tH_2+i)^{\beta_2},\ldots, X_n (tH_n+i)^{\beta_n}).
\end{align*}
Corollary~\ref{C:IndepNorm} and Theorem~\ref{T:MainMOIConstruction} give that
\begin{align*}
    &\|  T^{tH_0, \ldots, tH_n}_{\psi}((tH_0+i)^{\beta_0}  X_1 (tH_1+i)^{\beta_1}, X_2(tH_2+i)^{\beta_2},\ldots, X_n (tH_n+i)^{\beta_n}) \|_{\Hc^{q+r+(\alpha-n)h}\to\Hc^q} \\
    &\lesssim\| (tH_0+i)^{\beta_0}\|_{\Hc^{q_0 + \beta_0h} \to \Hc^{q_0}} \cdots \| (tH_n+i)^{\beta_n}\|_{\Hc^{q_n+\beta_nh}\to \Hc^{q_n}},
\end{align*}
for $q_i$ some real numbers. Theorem~\ref{T:MainFunctCalc} gives that 
\[
\|(tH_j+i)^{\beta_ij}\|_{\Hc^{q_j+\beta_jh}\to \Hc^{q_j}} \lesssim \sup_{x\in \R} |(tx+i)^{\beta_j} | \langle x\rangle^{-\beta_j} \lesssim t^{\beta_j}.
\]
Therefore, 
\begin{align*}
    \|  T^{tH_0, \ldots, tH_n}_{\phi}(X_1, X_2,\ldots, X_n )& \|_{\Hc^{q+r+(\alpha-n)h} \to \Hc^q} \lesssim t^{\beta_0 + \cdots + \beta_k}= t^{\alpha-n}.\qedhere
\end{align*}
\end{proof}

\begin{prop}\label{P:TaylorTrace}Let $\Theta^{-1} \in \mathcal{L}_{s}$, $s > 0$, $f \in T^\beta(\R)$, $H \in \op^h(\Theta), h> 0$ $\Theta$-elliptic and symmetric and $V \in \op^r(\Theta)$ symmetric. Let $h > r \geq 0$ and $\beta < -\frac{s}{h}$.
    For every $N \in \N$, we have as $t \downarrow 0$,
    \[
    \Tr(f(tH+tV)) = \sum_{n=0}^{N} t^{n} \Tr( T^{tH}_{f^{[n]}}(V, \ldots, V) ) + O(t^{(N+1)(1-\frac{r}{h}) - \frac{s}{h}}).
    \]
\end{prop}
\begin{proof}
    The proof of Theorem~\ref{T:Taylor} gives that
    \[
    f(tH + tV) = \sum_{n=0}^{N} T_{f^{[n]}}^{tH}(tV, \ldots, tV) + T_{f^{[N+1]}}^{tH + tV, tH, \ldots, tH}(tV, \ldots, tV).
    \]
    The condition $\beta < -\frac{s}{h}$ assures that all terms on the left and right-hand side are trace-class (cf. \eqref{eq:TypicalMOI}). Furthermore, we have $f \in T^\beta(\R)\subseteq T^{(N+1)(1-\frac{r}{h})-\frac{s}{h}}(\R)$ so that Proposition~\ref{P:TraceEstimatewitht} provides that
    \begin{align*}
        \|T_{f^{[N+1]}}^{tH + tV, tH, \ldots, tH}&(tV, \ldots, tV)\|_1 \\
        &\lesssim t^{N+1} \|T_{f^{[N+1]}}^{tH + tV, tH, \ldots, tH}(V, \ldots, V)\|_{\Hc^{-s}\to\Hc^0}\\
        &\lesssim t^{(N+1)(1-\frac{r}{h}) - \frac{s}{h}}.\qedhere
    \end{align*}
\end{proof}

This proposition makes it clear that to determine the coefficients of asymptotic expansions of the type~\eqref{eq:AsympExp1} or~\eqref{eq:AsympExp2}, it suffices to study the asymptotic expansions of the multiple operator integral
\[
\Tr( T^{tH}_{f^{[n]}}(V, \ldots, V) ),
\]
which we do with Proposition~\ref{P:Expansion}.

\begin{thm}\label{T:GeneralAsymp}
    Let $\Theta^{-1} \in \mathcal{L}_{s}$, $s > 0$, $f \in T^\beta(\R)$, $H\in \op^h(\Theta)$, $h>0$ symmetric and $\Theta$-elliptic, $V \in \op^r(\Theta)$ symmetric. If $h > r \geq0 $, $\beta \leq -\frac{s}{h}$, and $\delta_H^n(V) \in \op^{r+n(h-\varepsilon)}(\Theta)$, then $\Tr(f(tH+tV)) $ admits an asymptotic expansion as $t\downarrow 0$ of type~\eqref{eq:AsympExp2} given by
    \begin{align*}
        \Tr(f(tH+tV)) = \sum_{n=0}^{N} \sum_{m=0}^{N}   \sum_{m_1 + \dots + m_n = m} t^{n+m}\frac{C_{m_1, \dots, m_n}}{(n+m)!} \Tr\big( \delta_H^{m_1}(V) \cdots \delta_H^{m_n}(V) & f^{(n+m)}(tH)\big) \\
        &+ O(t^{m_N}),
    \end{align*}
    where $m_N := (N+1)\min\big(\frac{\varepsilon}{h}, (1-\frac{r}{h}) \big) - \frac{s}{h},$ so that $m_N \uparrow \infty$ as $N \to \infty$.
\end{thm}
\begin{proof}
    Combining Propositions~\ref{P:TaylorTrace} and~\ref{P:Expansion}, we have that
    \begin{align*}
        \Tr(f(tH+tV)) &= \sum_{n=0}^{N} \sum_{m=0}^{N}  t^{n+m} \sum_{m_1 + \dots + m_n = m} \frac{C_{m_1, \dots, m_n}}{(n+m)!} \Tr\big( \delta_H^{m_1}(V) \cdots \delta_H^{m_n}(V) f^{(n+m)}(tH)\big)\\
        &\quad + \sum_{n=0}^{N} t^n \Tr(S^n_{N,t}(V, \ldots, V))+ O(t^{(N+1)(1-\frac{r}{h}) - \frac{s}{h}}),
    \end{align*}
    where $S^n_{N,t}(V, \ldots, V)$ is a sum of terms of the form
    \[
    t^{N+1}   \delta_H^{m_1}(V)\cdots \delta_H^{m_k}(V)   T^{tH}_{f^{[n+N+1]}}(1, \ldots, 1, \delta_H^{N+1-m_1 - \cdots - m_k}(V), 1, \ldots, 1, V, \ldots, V).
    \]
    We then estimate
    \begin{align*}
        &\big \|t^{N+1}   \delta_H^{m_1}(V)\cdots \delta_H^{m_k}(V)   T^{tH}_{f^{[n+N+1]}}(1, \ldots, 1, \delta_H^{N+1-m_1 - \cdots - m_k}(V), 1, \ldots, 1, V, \ldots, V)\big\|_1\\
        &\lesssim t^{N+1} \big \|  \delta_H^{m_1}(V)\cdots \delta_H^{m_k}(V)   T^{tH}_{f^{[n+N+1]}}(1, \ldots, 1, \delta_H^{N+1-m_1 - \cdots - m_k}(V), 1, \ldots, 1, V, \ldots, V)\big\|_{\Hc^{-s}\to\Hc^0}\\
        &\leq t^{N+1} \big\|\delta_H^{m_1}(V)\cdots \delta_H^{m_k}(V) \big\|_{\Hc^{kr+(m_1 + \cdots + m_k)(h-\varepsilon)} \to \Hc^{0}}\\
        &\quad \times \big\| T^{tH}_{f^{[n+N+1]}}(1, \ldots, 1, \delta_H^{N+1-m_1 - \cdots - m_k}(V), 1, \ldots, 1, V, \ldots, V)\big\|_{\Hc^{-s}\to \Hc^{kr+(m_1 + \cdots + m_k)(h-\varepsilon)}}.
    \end{align*}
    Applying Proposition~\ref{P:TraceEstimatewitht} then provides that
    \begin{align*}
        \| S^n_{N, t}(V, \ldots, V)\|_1 &\lesssim t^{N+1-\frac{s+kr+(m_1 + \cdots + m_k)(h-\varepsilon)}{h}-\frac{(n-k)r+(N+1-m_1 - \cdots - m_k)(h-\varepsilon)}{h}}\\
        &\quad = t^{-\frac{s+nr}{h}+(N+1)\frac{\varepsilon}{h}},
    \end{align*}
    and hence
    \begin{align*}
        \sum_{n=0}^{N} t^n \Tr(S^n_{N,t}(V, \ldots, V)) \lesssim \sum_{n=0}^{N} t^{n(1-\frac{r}{h})+(N+1)\frac{\varepsilon}{h}-\frac{s}{h}}.
    \end{align*}
    Defining
    \[
    m_N := (N+1)\min\bigg(\frac{\varepsilon}{h}, \big(1-\frac{r}{h}\big) \bigg) - \frac{s}{h}
    \]
    concludes the proof.
\end{proof}

This expansion in fact partially answers a specific open problem posed by Eckstein and Iochum in~\cite{EcksteinIochum2018}. 
Given a spectral triple $(\Ac, \Hc, D)$ it is a common assumption to require the existence of an asymptotic expansion as $t \downarrow 0$ of
\[
\Tr(ae^{-tD^2}),
\]
where $a \in \Ac$. Their question is whether the existence of asymptotic expansions of
\[
\Tr(ae^{-t(D+V)^2})
\]
can be deduced for suitable perturbations $V$ from this, and whether it could be enough for $\Tr(e^{-tD^2})$ to admit an asymptotic expansion.

First, we address the second part of the question by giving an explicit example where the asymptotic expansion of $\Tr(e^{-tD^2})$ provides no control over the expansions of $\Tr(a e^{-tD^2})$.

\begin{thm}{\cite[Theorem~3.2]{EcksteinIochum2018}}\label{T:ExpansionToZeta}
    For a bounded operator $a$ and invertible positive operator $D$ such that $D^{-1} \in \mathcal{L}_s$, $s >0$, the existence of an asymptotic expansion 
    \[
    \Tr(ae^{-tD^2}) \underset{t \downarrow 0}{\sim} \sum_{k=0}^\infty \rho_k(t),
    \]
    where 
    \[
    \rho_k(t) := \sum_{z \in X_k} \bigg( \sum_{n=0}^d c_{n,k} \log^n t \bigg) t^{-z},
    \]
    with $c_{n,k} \in \C$ and for suitable sets $X_k \subset \C$ (for details, see~\cite[Theorem~3.2]{EcksteinIochum2018}),
    implies the existence of a meromorphic continuation of
    \[
    \zeta_{D^2,a}(s) := \Tr(a |D|^{-2s})
    \]
    to the complex plane, with poles of order at most $d+1$ located at points in $\bigcup_{k=0}^\infty X_k \subset \C$.
\end{thm}

\begin{ex}\label{E:Counterexample}
    Let $\Ac = \ell_{\infty}(\Z_{\geq 1}),$ $\Hc = \ell_2(\Z_{\geq 1})$ where $\Ac$ is represented
    on $\Hc$ by pointwise multiplication, and let $D$ be the diagonal operator on $\Hc$ given by
    \[
        De_n = ne_n,\quad n\geq 1.
    \]
This is a spectral triple for trivial reasons: $\Ac$ acts on $\Hc$ by bounded operators,
and $[D,a]=0$ for all $a\in\Ac.$
Despite being atypical, $(\Ac,\Hc,D)$ satisfies most of the assumptions commonly made in the literature in terms of smoothness or summability.
The algebra $\Ac$ is not separable, but all of the following arguments can be performed in a separable (even finite dimensional) subalgebra of $\Ac.$

It is a classical result that we have the asymptotic expansion
\[
\Tr(e^{-tD^2})= \sum_{n=1}^\infty e^{-tn^2} \underset{t \downarrow 0}{\sim} \frac{\sqrt{\pi}}{2} t^{-\frac{1}{2}} - \frac{1}{2},
\]
see for example~\cite[Lemma~3.1.3]{Gilkey2004}. 
Nonetheless, the functions $\Tr(a|D|^{-2s})$ for $a\in \Ac$ are very badly behaved. For example, let
\[
    a:= \sum_{n=2}^\infty \frac{1}{\log n} e_n \in \Ac,
\]
so that
\[
\zeta_{a,D^2}(s)= \Tr(aD^{-2s}) = \sum_{n=2}^\infty  \frac{1}{\log n}n^{-2s},\quad \Re(s)>\frac{1}{2},
\]
which is holomorphic on $\Re(s) > \frac{1}{2}$. Now,
\begin{align*}
    \frac{d}{ds}\zeta_{a,D^2}(s) &= -2\sum_{n=2}^\infty n^{-2s}  \\
    &=2-2\zeta(2s),\quad \Re(s)>\frac{1}{2},
\end{align*}
where $\zeta$ is the Riemann zeta function which has a simple pole at $1$. Therefore, $\zeta_{a,D^2}(s) + \log(2s-1)$ is the antiderivative of an entire function, which implies that $\zeta_{a,D^2}(s) = -\log(2s-1) + f(s)$ where $f(s)$ is entire~\cite[Theorem~10.14]{PapaRudin}.
We conclude that $\zeta_{a,D^2}$ does not admit a meromorphic extension to the complex plane, and thus $\Tr(ae^{-tD^2})$ does not have an asymptotic expansion as $t\downarrow 0$ of the type in Theorem~\ref{T:ExpansionToZeta}.

An even more pathological example is 
\[
    b := \sum_{n=2}^\infty \frac{\Lambda(n)}{\log(n)}e_n
\]
where $\Lambda$ is the von Mangoldt function which satisfies
\[
\Lambda(n) :=
\begin{cases}
    \log(p) &\text{ if } n=p^k \text{ for } p \text{ prime};\\
    0 &\text{ otherwise}.
\end{cases}
\] 
A classical formula asserts that~\cite[p.4]{Titchmarsh1986}
\[
    \zeta_{b,D}(s)=\Tr(b|D|^{-s}) = \sum_{n=2}^\infty \frac{\Lambda(n)}{\log(n)}n^{-s} = \log \zeta(s), \quad \Re(s) > 1
\]
which is badly behaved at every zero of $\zeta$ and at $s= 1$.
\end{ex}

To answer the first part of the question, we deduce the following.

\begin{cor}\label{C:AsympExpNCG}
    Let $(\Ac,\Hc,D)$ be a regular $s$-summable spectral triple, $s> 0$. Let $V, P \in \mathcal{B}$, $V$ self-adjoint and bounded, where $\mathcal{B}$ is the algebra generated by $\Ac$ and $D$. 
   Then for all $f \in T^{\beta}(\R)$ with $\beta < -s$, the expressions
\begin{align*}
\Tr(f(tD + tV)), \qquad \Tr(Pe^{-t(D+V)^2}) \quad \text{and} \quad \Tr(Pe^{-t|D+V|})
\end{align*}
admit an asymptotic expansion as $t \downarrow 0$ given respectively by
\begin{align*}
    \Tr(f(tD+tV)) = \sum_{n=0}^{N} \sum_{m=0}^{N}   \sum_{m_1 + \dots + m_n = m} t^{n+m} \frac{C_{m_1, \dots, m_n}}{(n+m)!} \Tr\big( \delta_D^{m_1}(V) \cdots \delta_D^{m_n}&(V) f^{(n+m)}(tD)\big)\\
    &+ O(t^{N+1-s}),
\end{align*}
where $C_{m_1, \dots, m_n}$ is the same as in Proposition~\ref{P:Expansion}, 
\begin{align*}
\Tr(Pe^{-t(D+V)^2}) = \sum_{n = 0}^{N} \sum_{m=0}^{N} \sum_{m_1 + \dots + m_n = m}  (-t)^{n+m} \frac{C_{m_1, \dots, m_n}}{(n+m)!} \Tr(PA^{(m_1)}\cdots &A^{(m_n)}\exp(-tD^2)) \\&+ O(t^{\frac{N+1-s}{2}}),
\end{align*}
where $A := DV + VD + V^2$, and $A^{(m)} := \delta^m_{D^2}(A)$, 
and
\begin{align*}
    \Tr(Pe^{-t|D+V|}) = \sum_{n = 0}^{N} \sum_{m=0}^{N} \sum_{m_1 + \dots + m_n = m}  (-t)^{n+m} \frac{C_{m_1, \dots, m_n}}{(n+m)!} \Tr(P\delta_{|D|}^{m_1}(B)\cdots &\delta_{|D|}^{m_n}(B)\exp(-t|D|)) \\
    &+ O(t^{(N+1)(1-\varepsilon) - s}),
\end{align*}
where $B := |D+V| - |D|$ and $\varepsilon > 0$ can be picked arbitrarily small.
\end{cor}
\begin{proof}
Given a regular $s$-summable spectral triple  $(\Ac,\Hc,D)$, write $\Theta := (1+D^2)^{\frac12}$. Let $V, P \in \mB$, $V$ self-adjoint and bounded, where $\mB$ is the algebra generated by $\Ac$ and $D$. If $f \in T^{\beta}(\R)$ with $\beta < -s$, Theorem~\ref{T:GeneralAsymp} immediately gives that
  \begin{align*}
    \Tr(f(tD+tV)) = \sum_{n=0}^{N} \sum_{m=0}^{N}   \sum_{m_1 + \dots + m_n = m} t^{n+m}\frac{C_{m_1, \dots, m_n}}{(n+m)!} \Tr\big( \delta_D^{m_1}(V) \cdots \delta_D^{m_n}&+(V) f^{(n+m)}(tD)\big) \\&+ O(t^{N+1-s}).
    \end{align*}

Regarding the expansion of $\Tr(Pe^{-t(D+V)^2})$ we have that $D^2 \in \OP^2(\Theta)$, $A \in \OP^1(\Theta)$ since $(\Ac, \Hc, D)$ is regular. As $[D^2, A] = [\Theta^2, A] = \Theta[\Theta, A] + [\Theta,A] \Theta$, we have that $A^{(m)} \in \OP^{1+m}(\Theta)$. Furthermore, $\mB \subseteq \op(\Theta)$. The proof of this corollary is then the same as the proof of Theorem~\ref{T:GeneralAsymp}. Filling in
\[
m_N = (N+1)\min\bigg(\frac{\varepsilon}{h}, \big(1-\frac{r}{h}\big) \bigg) - \frac{s}{h} = \frac{N+1-s}{2}
\]
gives the order of the error term.

For the expansion of $\Tr(Pe^{-t|D+V|})$, while $|D| \in \OP^1(\Theta)$ due to Theorem~\ref{T:MainFunctCalc}, to conclude something similar for $|D+V|$ we have to do more work. Note that $D$ has discrete spectrum since $(1+D^2)^{-\frac12} \in \mathcal{L}_s$. If $V \in \OP^0(\Theta)$ is self-adjoint, then $D+V$ has real discrete spectrum too. Hence we can modify the function $x \mapsto |x|$ slightly on a small neighbourhood around $x=0$ to get a smooth function $f$ which has the property that $f(x) = |x|$ on $\sigma(D+V) \cup \sigma(D)$, and $f \in S^1(\R)$ since the second and higher derivatives of $f$ are all compactly supported. Using Theorem~\ref{T:MOOIforNCG1} and the observation that $S^1(\R) \subseteq T^{1+\varepsilon}(\R)$ for all $\varepsilon > 0$, we have
\[
|D+V| \in \OP^{1+\varepsilon}(\Theta), \ \ |D+V|-|D| = T^{D+V, D}_{f^{[1]}}(V) \in \OP^{\varepsilon}(\Theta).
\]
Therefore, we get
\begin{align*}
    \Tr(Pe^{-t|D+V|}) = \sum_{n = 0}^{N} \sum_{m=0}^{N} \sum_{m_1 + \dots + m_n = m}  (-t)^{n+m} \frac{C_{m_1, \dots, m_n}}{(n+m)!} \Tr(P\delta_{|D|}^{m_1}(B)\cdots& \delta_{|D|}^{m_n}(B)\exp(-t|D|)) \\
    &+ O(t^{(N+1)(1-\varepsilon) - s}),
\end{align*}
where $B := |D+V|-|D|$ and $\varepsilon > 0$ can be chosen arbitrarily small.
\end{proof}

Apart from providing a perturbative expansion of the spectral action, Corollary~\ref{C:AsympExpNCG} shows that if for all $P \in \mB$ we have an expansion
\begin{equation}\label{eq:ExistenceAsExp}
\Tr (P e^{-tD^2}) \underset{t\downarrow 0}{\sim} \sum_{k=0}^\infty c_k(P) t^{r_k}
\end{equation}
for constants $c_k(P) \in \C$, then for all $P \in \mB$ self-adjoint and bounded there exist constants $c_k(P,V) \in \C$ such that
\[
\Tr(P e^{-t(D+V)^2}) \underset{t\downarrow 0}{\sim} \sum_{k=0}^\infty c_k(P, V) t^{r_k}.
\]
Similarly, if each $\Tr (P e^{-tD^2})$ admits an asymptotic expansions of the type in Theorem~\ref{T:ExpansionToZeta}, then $\Tr (P e^{-t(D+V)^2})$ does too.

\begin{rem}
    Corollary~\ref{C:AsympExpNCG} can be modified to work for non-unital spectral triples. Given a spectral triple $(\Ac, \Hc, D)$ with non-unital algebra $\Ac$ writing $\Theta = (1+D^2)^{\frac12}$, if one assumes instead of $\Theta^{-1} \in \mathcal{L}_s$ that there exists $p \geq 1$ such that $a \Theta^{-s} \in \mathcal{L}_1$ for all $a \in \Ac \cup [D,\Ac]$ and $s > p$ as is proposed in~\cite{CGRS2}, then we also have $a \cdot \op^{-s}(\Theta) \in \mathcal{L}_1$ for $s > p$. It follows that as $t\downarrow 0$
    \begin{align*}
    \Tr(af(tD+tV)) = \sum_{n=0}^{N} \sum_{m=0}^{N}  t^{n+m} \sum_{m_1 + \dots + m_n = m} \frac{C_{m_1, \dots, m_n}}{(n+m)!} \Tr\big(a \delta_D^{m_1}(V) \cdots \delta_D^{m_n}&(V) f^{(n+m)}(tD)\big)\\& + O(t^{N+1-p})
    \end{align*}
    for $a \in \Ac \cup [D,\Ac]$, $V \in \mB$ bounded and self-adjoint, and $f \in T^{\beta}(\R)$, $\beta< -p$.
\end{rem}

We can now answer the first part of the question by Eckstein and Iochum. More precisely, in~\cite[Chapter~5]{EcksteinIochum2018} the question is asked when, for a spectral triple $(\Ac, \Hc, D)$, the existence of an asymptotic expansion of $\Tr(e^{-t|D|})$ implies the existence of an expansion of $\Tr(e^{-t|D+V|})$ for a suitable perturbation $V$. Corollary~\ref{C:AsympExpNCG} compared with Example~\ref{E:Counterexample} suggests that this is not generally possible. We illustrate this with the following example.

\begin{ex} 
    Let us revisit Example~\ref{E:Counterexample} where $(\Ac, \Hc, D) = (\ell_\infty(\Z_{\geq 1}), \ell_2(\Z_{\geq 1}), D)$, and $D$ is defined by
    \[
        De_n = ne_n,\quad n\geq 1.
    \]
    We take as before
    \[
    a:= \sum_{n=2}^\infty \frac{1}{\log n} e_n.
    \]
    Noting that $|D+a| - |D| = a$, we can apply Corollay~\ref{C:AsympExpNCG} to get (in this situation we can choose $\varepsilon = 0$)
    \[
    \Tr(e^{-t|D+a|}) = \Tr(e^{-tD}) -t \Tr(ae^{-tD}) + O(t).
    \]
    Taking the Mellin transform we get for $\Re(s)>1$ (see~\cite[Proposition~2.10]{EcksteinIochum2018})
    \begin{align*}
    \Tr(|D+a|^{-s}) &= \frac{1}{\Gamma(s)}\int_0^\infty t^{s-1}\Tr(e^{-t|D+a|})\,dt\\
        &= \frac{1}{\Gamma(s)}\int_0^1 t^{s-1}\Tr(e^{-t|D+a|})\,dt + \frac{1}{\Gamma(s)}\int_1^\infty t^{s-1}\Tr(e^{-t|D+a|})\,dt.
\end{align*}
Since
\[
    \Tr(e^{-t|D+a|)}) \leq \Tr(e^{-tD}) = (e^t-1)^{-1} \leq 2 e^{-t},\quad t\geq 1,
\]
we have that
\[
    s\mapsto \frac{1}{\Gamma(s)}\int_1^\infty t^{s-1}\Tr(e^{-t|D+a|})\,ds
\]
is holomorphic.
It follows that
\begin{align*}
    \Tr(|D+a|^{-s}) &= \frac{1}{\Gamma(s)}\int_0^1 t^{s-1}\Tr(e^{-t|D+a|})\,dt + \mathrm{holo}_{\C}(s)\\
    &=   \frac{1}{\Gamma(s)}\int_0^1 t^{s-1} \Tr(e^{-tD}) -t \Tr(ae^{-tD})\, dt + \mathrm{holo}_{\Re(s)>-1}(s)\\
    &= \zeta_{0,D}(s) - s \zeta_{a,D}(s+1) + \mathrm{holo}_{\Re(s)>-1}(s).
\end{align*}
Since $s \mapsto \zeta_{a,D}(s+1)$ does not extend holomorphically to any punctured neighbourhood of $s=0$ as the computations in Example~\ref{E:Counterexample} show, we conclude that $\zeta_{a,D}(s) = \Tr(|D+a|^{-s})$ does not admit a meromorphic extension to the entire complex plane. By Theorem~\ref{T:ExpansionToZeta}, 
\[
\Tr (e^{-t|D+a|})
\]
does not admit an asymptotic expansion of the type listed in the theorem.
\end{ex}

We conclude by remarking that the existence of an asymptotic expansion of  $\Tr( P e^{-t D^2})$ for $P \in \mB$ of the type in Theorem~\ref{T:ExpansionToZeta} is guaranteed for commutative spectral triples~\cite[Theorem~2.7]{GrubbSeeley1995}, and by the same theorem also for almost commutative spectral triples~\cite[Chapter~10]{Suijlekom2025}.
    \part{Dixmier Trace Formulas}
    \chapter{Connes' integral formula and quantum ergodicity}
\label{Ch:QE}
{\setlength{\epigraphwidth}{0.55\textwidth}
\epigraph{I am extremely impressed by your taste displayed in the opening epigraphs to each chapter.}{Fedor Sukochev}}
This chapter is based on~\cite{QE}, joint work with Edward McDonald. Thanks are extended towards Nigel Higson for helpful comments and suggestions, and Magnus Goffeng for pointing out the condition $[D,A]$ being bounded is sufficient for Lemma~\ref{L:Widom}. We also acknowledge Eric Leichtnam, Qiaochu Ma, and Rapha\"el Ponge for their assistance. The main results of this chapter are a formula for the noncommutative integral in Theorem~\ref{T:Log-CesaroMean}, a noncommutative Szeg\H{o} formula in Theorem~\ref{T:Szego}, and a link between NCG and Quantum Ergodicity in Theorem~\ref{T:MainQE}.

For applications of NCG in physics and numerical computations in NCG, it is important to know how well spectral triples can be approximated by a finite truncation, since this is all we can measure physically or compute numerically. A physical detector in a measurement for example is typically only effective in a certain energy range. Hence  for a spectral triple $(\Ac, \Hc, D)$ it makes sense to consider a spectral projection $P = \chi_I(D)$, and truncate the triple into the form
\[
(P\Ac P, P\Hc, PD).
\]
Now, $P\Ac P$ is not generally an algebra anymore. Connes and Van Suijlekom introduced the concept of operator system spectral triples for this purpose~\cite{ConnesvSuijlekom2021}, a generalisation of usual spectral triples.
\begin{defn}
    A unital operator system spectral triple $(\Ac, \Hc, D)$ consists of a unital space $\Ac$ of bounded operators on a Hilbert space $\Hc$ such that its norm closure is $*$-invariant, $D$ is a self-adjoint operator on $\Hc$ with compact resolvent, and for all $T \in \Ac$ we have that $T \dom(D) \subseteq \dom(D)$ and $[D,T]$ extends to a bounded operator. 
\end{defn}

Developments in this direction were made in~\cite{DAndreaLizzi2014, GlaserStern2020, GlaserStern2021, ConnesvSuijlekom2022, DAndreaLandi2022, Hekkelman2022, GielenvSuijlekom2023, Rieffel2023, LeimbachvSuijlekom2024, Suijlekom2024a, Suijlekom2024b} amongst others.

We will connect this paradigm with Connes' (normalised) noncommutative integral
\begin{equation}\label{eq:NoncomInt}
a \mapsto \frac{\Tr_\omega(a\langle D\rangle ^{-d})}{\Tr_\omega(\langle D\rangle^{-d})}, \quad a \in B(\Hc),
\end{equation}
recall that  $\langle x \rangle := (1+|x|^2)^{\frac{1}{2}}$.
Namely, we will show that given a finite-rank spectral projection $P_\lambda : = \chi_{[-\lambda,\lambda]}(D)$ where $\chi_{[-\lambda,\lambda]}$ is the indicator function of the interval $[-\lambda, \lambda]\subseteq \R$, the functional
\begin{equation}\label{eq:TruncInt}
    P_\lambda a P_\lambda \mapsto \frac{\Tr(P_\lambda a P_\lambda)}{\Tr(P_\lambda)}, \quad a \in B(\Hc),
\end{equation}
approximates the noncommutative integral on spectrally truncated compact spectral triples (Proposition~\ref{P:HeattoInt}, Theorem~\ref{T:Log-CesaroMean}). This is a result in the spirit of~\cite{Stern2019}, where finite-rank approximations of zeta residues are given. We however do not assume the existence of a full asymptotic expansion of the heat trace. Instead, we focus our efforts on the computation of the first term of this expansion, which is the noncommutative integral.

The language involved is closely tied to the field of quantum ergodicity, the inception of which can largely be credited to Shnirelman, Zelditch and Colin de Verdi\`ere~\cite{Shnirelman1974,ColindeVerdiere1985,Zelditch1987}. For reviews of this field, see~\cite{Zelditch2010, Zelditch2017}. Quantum ergodicity is a property of an operator which can mean various things. A common definition is that, given a compact Riemannian manifold $M$ and a positive self-adjoint operator $\Delta$ on $L_2(M)$ with compact resolvent, the operator $\Delta$ is said to be quantum ergodic if for every orthonormal basis $\{e_n\}_{n=0}^\infty$ of $L_2(M)$ consisting of eigenfunctions of $\Delta$ there exists a density one subsequence $J\subseteq \N$ such that for all zero-order pseudodifferential operators $\mathrm{Op}(\sigma)$ with principal symbol $\sigma \in C^\infty(S^*M)$,
\[
\lim_{J \ni j \to \infty}\langle e_{j}, \mathrm{Op}(\sigma) e_{j}\rangle_{L_2(M)} = \int_{S^*M} \sigma \, d\nu,
\]
and where $\nu$ is the measure on the cotangent sphere $S^*M$ induced by the Riemannian metric. In this context, a density one subsequence means that 
\[
\frac{\# J \cap \{0,\ldots, n\}}{n+1} \to 1, \quad n \to \infty.
\]
Quantum ergodicity implies in particular that the functions $|e_j|^2$ become uniformly distributed over $M$ as $J\ni j \to \infty$, in the sense that the measures $|e_j|^2 d\nu_g$ converge to $d\nu_g$ in the weak$*$-topology, see Figure~\ref{F:eigenfunctions}.

\begin{figure}[h!]
\centering
    \includegraphics[width=\linewidth]{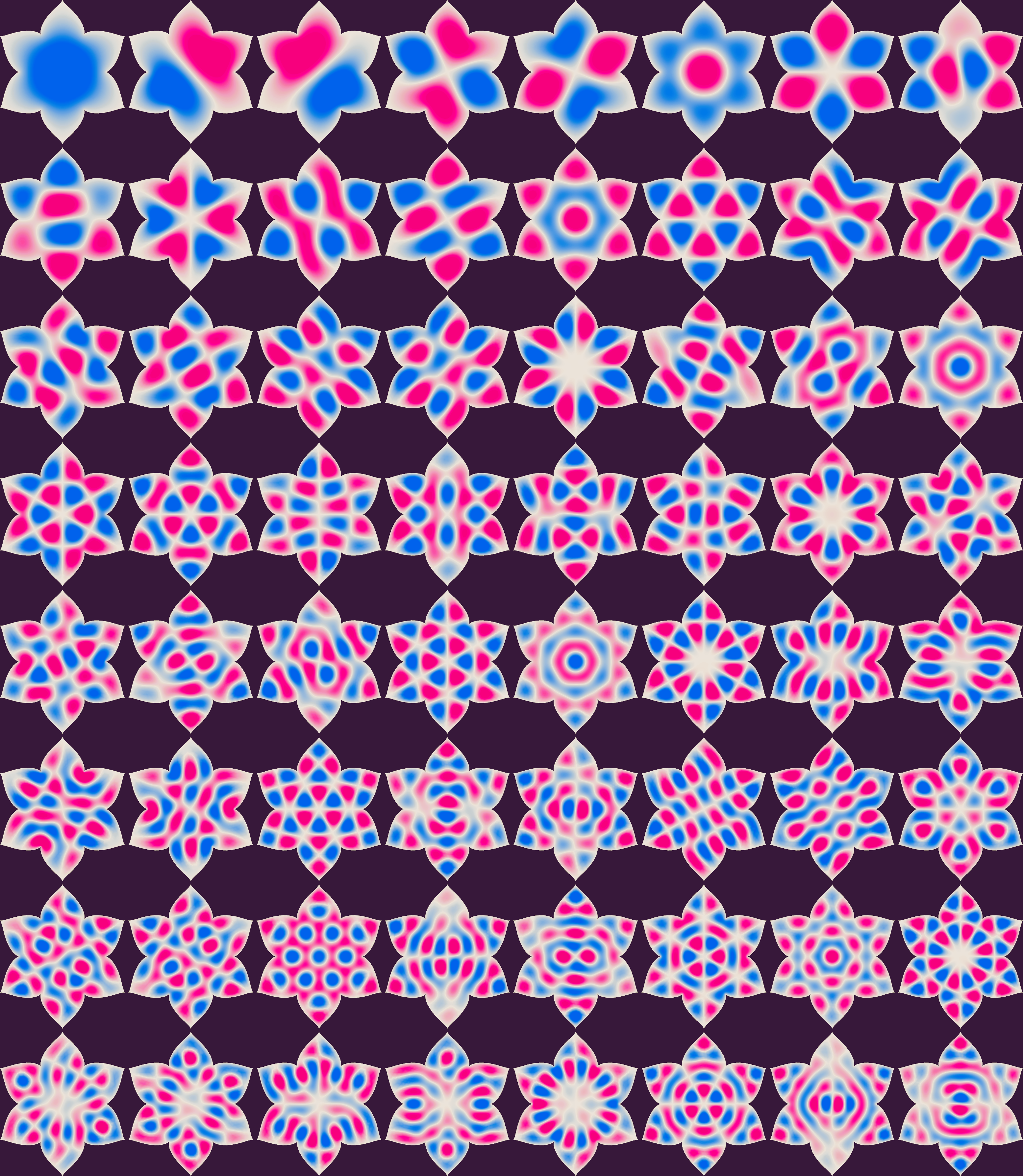}
    \caption{An illustration of 64 eigenfunctions of the Laplacian on a rose-shaped domain with Dirichlet boundary conditions, corresponding to the 64 smallest eigenvalues (counting multiplicities). Generated in Python with the finite-element method, using FEniCSx~\cite{DOLFINx}. Quantum ergodicity would imply that there exists a density one subsequence of eigenfunctions $\{e_j\}_{j\in J}$ along which the measures $|e_j|^2 d\lambda$ converge to the uniform (Lebesgue) measure $d\lambda$ in the weak$^*$-topology.}
    \label{F:eigenfunctions}
\end{figure}

Although quantum ergodicity shares a philosophical link with NCG --- emerging from a functional-analytic approach to ergodic geodesic flow on compact Riemannian manifolds --- there has yet to be made an explicit connection between the two fields, despite their contemporary development. We will show in Section~\ref{S:Ergodicity} that our results on the noncommutative integral on truncated spectral triples provide the means with which the gap can be bridged.

We furthermore propose a straightforward noncommutative generalisation of the property of ergodic geodesic flow on compact Riemannian manifolds for spectral triples, and explore what some results from the field of quantum ergodicity provide in this context. Our definition of ergodicity is known in the study of $C^*$-dynamical systems as uniqueness of the vacuum state, and hence a result by Zelditch~\cite{Zelditch1996} can now be recognised as an NCG version of the classical result by Colin de Verdi\`ere that ergodicity of the geodesic flow implies quantum ergodicity of the Laplace--Beltrami operator~\cite{ColindeVerdiere1985}, see Theorem~\ref{T:MainQE} below.
\newpage
Additionally, we will draw from a result of Widom~\cite{Widom1979} on the asymptotic behaviour of the functional~\eqref{eq:TruncInt}, which directly implies a Szeg\H{o} limit formula for spectral triples that satisfy the Weyl law (Theorem~\ref{T:Szego}). This provides that for all self-adjoint $A \in B(\Hc)$ which map $\dom|D|$ into itself and such that $[D,A]$ is bounded,
\[
\Tr_\omega (\langle D \rangle^{-d}) \cdot \omega \circ M \bigg(  \frac{\Tr(f(P_{\lambda_n} A P_{\lambda_n}))}{\Tr(P_{\lambda_n})}\bigg) = \Tr_\omega (f(A)\langle D\rangle^{-d}), \quad f \in C(\R),\,f(0)=0.
\]
Here, $M: \ell_\infty \to \ell_\infty$ is a logarithmic averaging operator, and $\omega \in \ell_\infty^*$ is an extended limit.
Details are provided in Section~\ref{S:Szego}. Note that we use the short-hand notation $\omega \circ M (a_n)$ for $ \omega \circ M (\{a_n\}_{n=1}^\infty)$. We remark that this result provides the insight that Szeg\H{o} limit theorems can be interpreted as versions of Connes' integral formula.

An outline of this chapter is as follows. We first explore and make precise the relation between the functionals~\eqref{eq:NoncomInt} and~\eqref{eq:TruncInt} in Section~\ref{S:Integration}.
Section~\ref{S:Szego} provides the mentioned Szeg\H{o} limit theorem for NCG. Next, we discuss a way of interpreting the functional~\eqref{eq:TruncInt} when the noncommutative integral~\eqref{eq:NoncomInt} is not defined, for example in $\theta$-summable or $\mathrm{Li}_1$-summable spectral triples. Namely, we relate the functional~\eqref{eq:TruncInt} to a functional that is sometimes called the Fr\"ohlich functional, which has been studied extensively in~\cite{GoffengRennie2019} as a KMS state. Finally, in Section~\ref{S:Ergodicity} we exhibit our study in quantum ergodicity and its relation to NCG through our results on the noncommutative integral.

\section{Integration on truncated spectral triples}
\label{S:Integration}
Let us fix a closed self-adjoint operator $D$ on a separable Hilbert space $\Hc$ such that $\langle D\rangle^{-d}\in \mathcal{L}_{1,\infty}$, where $d>0$ and $\langle x \rangle := (1+|x|^2)^{\frac{1}{2}}$. We fix an extended limit $\omega \in \ell_{\infty}^*$ and assume that $\Tr_\omega(\langle D\rangle^{-d}) > 0$.
We write $P_\lambda := \chi_{[-\lambda, \lambda]}(D)$. This situation is modeled after (compact) $d$-dimensional spectral triples $(\Ac, \Hc, D)$.

We first provide the most straight-forward approach to the noncommutative integral on truncated triples, using standard techniques that are employed in quantum ergodicity~\cite{ColindeVerdiere1985}. We write 
\[
f(t) \sim C t^{-\alpha}
\]
to mean
\[
\lim_{t\to 0} t^{\alpha}f(t) = C.
\]

\begin{prop}\label{P:HeattoInt}
    Let $a \in B(\Hc)$. If there exist constants $C, C(a) \in \R$ with 
    \[
    \Tr(e^{-tD^2}) \sim C t^{-\frac{d}{2}} , \quad \Tr(a e^{-tD^2}) \sim C(a)t^{-\frac{d}{2}},
    \]
    then
    \[
    \Tr_\omega (a \langle D\rangle^{-d}) = \Tr_\omega(\langle D \rangle^{-d}) \lim_{\lambda \to \infty}\frac{\Tr(P_\lambda a P_\lambda)}{\Tr (P_\lambda)}.
    \]
\end{prop}
\begin{proof}
    By~\cite[Corollary~8.1.3]{LSZVol1} we have that 
    \[
    C = \Gamma(\frac{d}{2}+1) \Tr_\omega(\langle D\rangle^{-d}), \quad C(a) = \Gamma(\frac{d}{2}+1) \Tr_\omega(a \langle D\rangle^{-d}).
    \]
    Recall that we assume $\Tr_\omega(\langle D\rangle^{-d})>0$. An application of the Hardy--Littlewood Tauberian theorem~\cite[Theorem~XII.5.2]{Feller1971} to the function $\Tr(e^{-tD^2})$ shows that 
    \[
    \Tr(P_\lambda) \sim \Tr_\omega(\langle D\rangle^{-d}) \lambda^{d}, \quad \lambda \to \infty.
    \] Applying the theorem again to the function $\Tr(a e^{-tD^2})$ then gives that $\lim_{\lambda \to \infty}\frac{\Tr(P_\lambda a P_\lambda)}{\Tr (P_\lambda)}$ exists and is equal to $\frac{\Tr_\omega(a\langle D\rangle^{-d})}{\Tr_\omega(\langle D\rangle^{-d})}$. 
\end{proof}

\begin{rem}
    The Hardy-Littlewood Tauberian theorem implies that the condition \\ $\Tr(e^{-tD^2})\sim Ct^{-\frac{d}{2}}$ as $t\to 0$ is equivalent to $\lambda(k,D^2)\sim \widetilde{C} k^{\frac{2}{d}}$ as $k\to\infty$~\cite[Theorem~XII.5.2]{Feller1971}.
    \end{rem}

\begin{defn}\label{D:WeylLaw}
    We say that $D^2$ (as fixed at the start of this section) satisfies a Weyl law if $\Tr(e^{-tD^2})\sim Ct^{-\frac{d}{2}}$, and it satisfies a local Weyl law for an operator $a \in B(\Hc)$ if $ \Tr(a e^{-tD^2}) \sim C(a)t^{-\frac{d}{2}}$. 
\end{defn}

    See~\cite{McDonaldSukochev2022WeylLaw} for an investigation of the validity of the (local) Weyl law for spectral triples, and~\cite{Ponge2023} for an extensive study of Weyl's law in relation to Connes' integral formula. The latter, work by Ponge, answers some questions regarding Weyl laws and the noncommutative integral related to measurability of operators.

Although the local Weyl laws hold for Riemannian manifolds and a wide class of spectral triples~\cite{GrubbSeeley1995, EcksteinIochum2018, McDonaldSukochev2022WeylLaw}, we have seen an example of a spectral triple in which such behaviour does not hold in Chapter~\ref{Ch:MOIs}, see Example~\ref{E:Counterexample}. In the remainder of this section we show what can be deduced without this condition.
We now fix an orthonormal basis $\{e_n\}_{n=0}^\infty$ of eigenvectors of $|D|$, ordered such that the corresponding eigenvalues $\{\lambda_n\}_{n=0}^\infty$ are non-decreasing.

\begin{lem}\label{L:WeakestInt}
    Let $A \in B(\Hc)$. Then
    \[
    \Tr_\omega(A\langle D\rangle^{-d}) = \omega\bigg(\frac{1}{\log(n+2)} \sum_{k=0}^n \langle\lambda_k\rangle^{-d} \langle e_k, A e_k \rangle\bigg).
    \]
    If $D^2$ satisfies Weyl's law, i.e. $\lambda_k \sim C k^{\frac{1}{d}}$, this simplifies to
    \[
    \Tr_\omega (A\langle D\rangle^{-d}) = \Tr_\omega(\langle D\rangle^{-d}) \cdot  \omega\bigg(\frac{1}{\log(n+2)} \sum_{k=0}^n  \frac{\langle e_k, A e_k \rangle}{k+1} \bigg).
    \]
\end{lem}
\begin{proof}
    The first part is~\cite[Corollary~7.1.4(c)]{LSZVol1}, i.e. a corollary of Theorem~\ref{T: Modulated}, the second claim is~\cite[Theorem~7.1.5(a)]{LSZVol1} or~\cite{LordSukochev2011}.
\end{proof}

What appears in the lemma above is the logarithmic mean $M: \ell_\infty \to \ell_\infty$, defined by
\begin{align*}
    M: x &\mapsto \bigg\{\frac{1}{\log(n+2)} \sum_{k=0}^n \frac{x_k}{k+1} \bigg\}_{n=0}^\infty.
\end{align*}
This can be compared with the Ces\`aro mean
\begin{align*}
    C: x &\mapsto \bigg\{\frac{1}{n+1} \sum_{k=0}^n x_k \bigg\}_{n=0}^\infty.
\end{align*}

\begin{lem}\label{L:LogToCes}
    For any sequence $x \in \ell_\infty$, we have
    \[
    (M(x))_n = (M \circ C(x))_n + o(1), \quad n \to \infty.
    \]
\end{lem}
\begin{proof}
    For $x \in \ell_\infty$ and $k \geq 0$ we have
    \begin{align*}
         \frac{x_k}{k+1}  &= \bigg(\frac{1}{k+1}\sum_{l=0}^k x_l\bigg) - \frac{k}{k+1}\bigg(\frac{1}{k}\sum_{l=0}^{k-1} x_l\bigg)\\
         &= (C(x))_k - (C(x))_{k-1} + \frac{1}{k+1} (C(x))_{k-1}.
    \end{align*}
    Hence, as $n\to \infty$
    \begin{align*}
        (M(x))_n &= \frac{1}{\log(n+2)} \sum_{k=0}^n \frac{x_k}{k+1}\\
        &= \frac{1}{\log(n+2)}\bigg( (C(x))_n + \sum_{k=1}^n \frac{1}{k+1} (C(x))_{k-1} \bigg)\\
        &= (M\circ T \circ C(x))_n + o(1),
    \end{align*}
    where $T : (x_0, x_1, x_2, \ldots) \mapsto (0, x_0, x_1, \ldots)$ is the right-shift operator on $\ell_\infty$. Finally, for any bounded sequence $a \in \ell_\infty$, we have that
    \[
    (M\circ T(a))_n - (M(a))_n = o(1), \quad n \to \infty,
    \]
    which can be found in~\cite[Lemma~6.2.12]{LSZVol1}.
\end{proof}

Since both $M$ and $C$ are regular transformations in Hardy's terminology~\cite[Chapter~III]{Hardy1949}, meaning that $M(x)_n \to c $ whenever $x_n \to c$, it is a consequence of Lemma~\ref{L:LogToCes} that for $x \in \ell_\infty$, if $C(x)_n \to c$ then $M(x)_n \to c$ as $n \to \infty$. We introduce one more crucial lemma. Namely, writing $Q_n$ for the projection onto $\{ e_0, \ldots, e_n\}$, we want to switch freely between 
\[
\frac{\Tr(P_\lambda a P_\lambda)}{\Tr(P_\lambda)}, \quad \frac{\Tr(Q_n a Q_n )}{\Tr(Q_n)}.
\]
The first can be written as $\frac{\Tr(Q_{N(\lambda)} a Q_{N(\lambda)} )}{\Tr(Q_{N(\lambda)})},$ where $N(\lambda)$ is the greatest $k\geq 0$ such that $\lambda_k\leq \lambda$, and thus can be interpreted as a subsequence of the second. The following lemma can therefore be applied, which appeared as~\cite[Lemma~4.8]{AHMSZ}, i.e. the published version of Chapter~\ref{Ch:DOSDiscrete}, in a slightly weaker form and in a different context.

\begin{lem}\label{L:AHMSZ}
    Let $\phi: \mathbb{N} \to \R_{> 0}$ be an increasing function such that $\phi(n)\to \infty$ as $n\to \infty$, let $\{a_k\}_{k \in \mathbb{N}} \subseteq \mathbb{R}$ be a sequence such that $\big\{\frac{1}{\phi(n)}\sum_{k=0}^n |a_k| \big\}_{n=0}^\infty$ is bounded, and let $\{k_0, k_1, \dots \}$ be an infinite, increasing sequence of positive integers such that 
    \[
    \lim_{n\to \infty} \frac{\phi(k_{n+1})}{\phi(k_n)} = 1,
    \] 
    and
    \[
    \frac{1}{\phi(k_{n})} \sum_{k=k_{n-1}+1}^{k_{n}} |a_k| = o(1), \quad n\to\infty.
    \]
    Labeling $k_{i_n} := \min\{k_i : k_{i} \geq n \}$,     
    we have that
    \[
    \frac{1}{\phi(n)}\sum_{k=0}^n a_k = \frac{1}{\phi(k_{i_n})}\sum_{k=0}^{k_{i_n}} a_k + o(1), \quad n \to \infty.
    \]
\end{lem}
\begin{proof}
Without loss of generality, we can assume that $\{a_k\}_{k\in \N}$ is a positive sequence. We have \begin{align*}
    \frac{1}{\phi(n)} \sum_{k=1}^n a_k -  \frac{1}{\phi(k_{i_n})} \sum_{k=1}^{k_{i_n}} a_k 
    &\leq  \bigg( \frac{\phi(k_{i_n})}{\phi(k_{i_n-1})}-1\bigg)\frac{1}{\phi(k_{i_n})}\sum_{k=1}^{k_{i_n}} a_k =o(1);\\
    \frac{1}{\phi(k_{i_n})} \sum_{k=1}^{k_{i_n}} a_k-\frac{1}{\phi(n)} \sum_{k=1}^n a_k & \leq \frac{1}{\phi(k_{i_n})} \sum_{k=k_{i_n-1}+1}^{k_{i_n}} a_k =o(1). \qedhere
\end{align*}
\end{proof}

We can now prove the main result of this section.
\begin{thm}\label{T:Log-CesaroMean}
    Let $A \in B(\Hc)$. If $D^2$ satisfies Weyl's law (Definition~\ref{D:WeylLaw}), then
    \[
    \frac{\Tr_\omega(A\langle D\rangle^{-d})}{\Tr_\omega(\langle D\rangle^{-d})} =(\omega \circ M)\big(\langle e_n, A e_n \rangle\big) =(\omega \circ M) \bigg( \frac{\Tr(Q_n A Q_n)}{\Tr(Q_n)} \bigg)=
     (\omega \circ M) \bigg( \frac{\Tr(P_{\lambda_n} A P_{\lambda_n})}{\Tr(P_{\lambda_n})} \bigg).
    \]
    If furthermore $Q$ is an operator with $\bigcap_{n\geq 0}\dom(D^n) \subseteq \dom Q$ such that for some $s\geq -d,$ $Q\langle D\rangle^{-s}$ extends to a bounded operator, e.g. if $Q \in \op^s(\langle D \rangle)$, we have
    \begin{align*}
        \Tr_\omega (Q) &= \omega \bigg(\frac{\Tr(P_{\lambda_n}QP_{\lambda_n})}{\log(\Tr(P_{\lambda_n}))} \bigg), \quad s = -d;\\
        \frac{\Tr_\omega (Q\langle D\rangle^{-s-d})}{\big(\Tr_\omega(\langle D\rangle^{-d})\big)^{\frac{s}{d}+1}} &= \Big(\frac{s}{d}+1\Big) \omega \circ M\bigg(\frac{\Tr(P_{\lambda_n} QP_{\lambda_n})}{\Tr(P_{\lambda_n})^{\frac{s}{d}+1}}\bigg), \quad s > -d.
    \end{align*}
\end{thm}
\begin{proof}
    The first equality in the first equation appeared in Lemma~\ref{L:WeakestInt}, the second equality is a consequence of Lemma~\ref{L:LogToCes} and the trivial identity
    \[
     \frac{\Tr(Q_n A Q_n)}{\Tr(Q_n)} = \frac{1}{n+1} \sum_{k=0}^n \langle e_k, A e_k\rangle.
    \] 
    The last equality follows from Lemma~\ref{L:AHMSZ} when taking $\phi(n) = n+1$, since the Weyl law gives that $\frac{N(\lambda_n)}{N(\lambda_{n+1})} \to 1.$ The assumption
    \[
    \frac{1}{N(\lambda_n)} \sum_{k=N(\lambda_{n-1})+1}^{N(\lambda_n)} \langle e_k, A e_k\rangle = o(1), \quad n\to\infty
    \]
    in Lemma~\ref{L:AHMSZ} is satisfied, since
    \[
    \frac{1}{N(\lambda_n)} \sum_{k=N(\lambda_{n-1})+1}^{N(\lambda_n)}| \langle e_k, A e_k\rangle | \leq \|A\|_\infty\frac{N(\lambda_n) - N(\lambda_{n-1})}{N(\lambda_{n})} = o(1), \quad n \to \infty.
    \]

    Now take an operator $Q$ with $\bigcap_{n\geq 0}\dom(D^n) \subseteq \dom Q$ such that $Q\langle D\rangle^{-s}$ extends to a bounded operator. For $s=-d$, the given formula for $\Tr_\omega(Q)$ is a combination of Lemma~\ref{L:WeakestInt} and Lemma~\ref{L:AHMSZ}. For $s \neq -d$, first, due to Weyl's law
    \begin{align*}
       \langle\lambda_k\rangle^{-s-d} = \big(\Tr_\omega(\langle D\rangle^{-d})\big)^{\frac{s}{d}+1} (k+1)^{-\frac{s}{d}-1} + o\big((k+1)^{-\frac{s}{d}-1}\big), \quad k \to \infty
    \end{align*}
    and hence, since $(k+1)^{-\frac{s}{d}}\langle e_k, Qe_k \rangle$ is bounded, we have that
    \begin{align*}
       \langle\lambda_k\rangle^{-s-d} \langle e_k, Q e_k \rangle = \big(\Tr_\omega(\langle D\rangle^{-d})\big)^{\frac{s}{d}+1} (k+1)^{-\frac{s}{d}-1} \langle e_k, Q e_k\rangle + o\big((k+1)^{-1}\big), \quad k \to \infty.
    \end{align*}
    Now applying Lemma~\ref{L:WeakestInt} and then Lemma~\ref{L:LogToCes},
    \begin{align*}
        \Tr_\omega (Q\langle D\rangle^{-s-d}) &=
        \omega \bigg(\frac{1}{\log(n+2)} \sum_{k \leq n} \langle\lambda_k\rangle^{-s-d}\langle e_k, Q e_k \rangle \bigg)\\
        &=\big(\Tr_\omega(\langle D\rangle^{-d})\big)^{\frac{s}{d}+1} \omega \circ M \big( (n+1)^{-\frac{s}{d}}\langle e_n, Q e_n \rangle \big)\\
        &=\big(\Tr_\omega(\langle D\rangle^{-d})\big)^{\frac{s}{d}+1} \omega \circ M \bigg( \frac{1}{n+1} \sum_{k \leq n}(k+1)^{-\frac{s}{d}}\langle e_k, Q e_k \rangle \bigg).
    \end{align*}
    Using Abel's summation formula, as $n \to \infty$
    \begin{align*}
         \frac{1}{n+1} \sum_{k \leq n}(k+1)^{-\frac{s}{d}}\langle e_k, Q e_k \rangle  &= (n+1)^{-\frac{s}{d}-1}\sum_{k \leq n} \langle e_k, Q e_k \rangle\\
         &\quad- \frac{1}{n+1}\sum_{k \leq n-1}\big((k+2)^{-\frac{s}{d}}-(k+1)^{-\frac{s}{d}}\big) \sum_{j \leq k}\langle e_j, Q e_j \rangle.
    \end{align*}
    By Taylor's formula, we have
    \[
        (k+2)^{-\frac{s}{d}}-(k+1)^{-\frac{s}{d}}+\frac{s}{d}(k+1)^{-\frac{s}{d}-1} = \frac{s}{d}\big(\frac{s}{d}+1\big)\int_0^1 (1-\theta)(k+1+\theta)^{-\frac{s}{d}-2}\,d\theta.
    \]
    Therefore    
    \begin{align*}
         &\frac{1}{n+1} \sum_{k \leq n}(k+1)^{-\frac{s}{d}}\langle e_k, Q e_k \rangle \\&=  (n+1)^{-\frac{s}{d}-1}\sum_{k \leq n} \langle e_k, Q e_k \rangle + \frac{s}{d} C\big( \big\{(k+1)^{-\frac{s}{d}-1} \sum_{j \leq k}\langle e_j, Q e_j \rangle\big\}_{k=0}^\infty\big)_n\\
         &\quad -\frac{s}{d}\big(\frac{s}{d}+1\big)\int_0^1 (1-\theta) C\big(\big\{(k+1+\theta)^{-\frac{s}{d}-2}\sum_{j\leq k}\langle e_j,Qe_j\rangle\big\}_{k=0}^\infty\big)_n\,d\theta,
    \end{align*}
    where $C: \ell_\infty \to \ell_\infty$ is the Ces\`aro operator. 
    Since $(j+1)^{-\frac{s}{d}}\langle e_j,Qe_j\rangle$ is bounded and $s>-d$ we have
    \[
        \big|(k+1+\theta)^{-\frac{s}{d}-2}\sum_{j\leq k}\langle e_j,Qe_j\rangle\big|=
            O((k+1)^{-1}),\quad k\to\infty.
    \]
    Thus
    \begin{align*}
        \frac{1}{n+1} \sum_{k \leq n}(k+1)^{-\frac{s}{d}}\langle e_k, Q e_k \rangle & =  (n+1)^{-\frac{s}{d}-1}\sum_{k \leq n} \langle e_k, Q e_k \rangle \\
        &\quad + \frac{s}{d} C\big( \big\{(k+1)^{-\frac{s}{d}-1} \sum_{j \leq k}\langle e_j, Q e_j \rangle\big\}_{k=0}^\infty\big)_n + O\Big(\frac{\log(n+2)}{n+1}\Big).
    \end{align*}   
    Using Lemma~\ref{L:LogToCes} again, we have
    \begin{align*}
        \Tr_\omega (Q\langle D\rangle^{-s-d}) &= \big(\Tr_\omega(\langle D\rangle^{-d})\big)^{\frac{s}{d}+1} \big(1+\frac{s}{d} \big) \omega \circ M\bigg(\frac{1}{(n+1)^{\frac{s}{d}+1}}\sum_{k \leq n} \langle e_k, Q e_k \rangle \bigg).
        \end{align*}
    To apply Lemma~\ref{L:AHMSZ}, taking $\phi(n) = (n+1)^{\frac{s}{d}+1}$ and $k_n = N(\lambda_n)$, we need to check that
    \begin{align*}
        \frac{1}{N(\lambda_n)^{\frac{s}{d}+1}} \sum_{k=N(\lambda_{n-1})+1}^{N(\lambda_n)} |\langle e_k, Q e_k\rangle | &\lesssim \frac{1}{N(\lambda_n)^{\frac{s}{d}+1}}\sum_{k=N(\lambda_{n-1})+1}^{N(\lambda_n)} k^{\frac
    {s}{d}}\\
    &\lesssim \frac{N(\lambda_n)^{\frac{s}{d}+1}-N(\lambda_{n-1})^{\frac{s}{d}+1}}{N(\lambda_n)^{\frac{s}{d}+1}}\\
    &= o(1), \quad n \to \infty.
    \end{align*}
    Hence Lemma~\ref{L:AHMSZ} applies, and we conclude that
    \begin{align*}
        \Tr_\omega (Q\langle D\rangle^{-s-d})    &= \big(\frac{s}{d}+1\big)\big(\Tr_\omega(\langle D\rangle^{-d})\big)^{\frac{s}{d}+1} \omega \circ M\bigg(\frac{\Tr(P_{\lambda_n} QP_{\lambda_n})}{\Tr(P_{\lambda_n})^{\frac{s}{d}+1}}\bigg). \qedhere
    \end{align*}
\end{proof}

As an obvious consequence of Theorem~\ref{T:Log-CesaroMean}, if for $A \in B(\Hc)$
\[
\frac{\Tr(P_\lambda AP_\lambda)}{\Tr(P_\lambda)}
\]
converges, it follows that, provided $D^2$ satisfies Weyl's law, the limit must necessarily be the noncommutative integral of $A$. Furthermore, if the noncommutative integral is independent of $\omega$, meaning that $A\langle D\rangle^{-d}$ is Dixmier measurable (see e.g.~\cite{LSZVol1,LMSZVol2,Ponge2023}) one can replace $\omega \circ M$ by $\lim \circ M$ on the right hand sides of Theorem~\ref{T:Log-CesaroMean}. Finally, with a Weyl law, for self-adjoint $A \in B(\Hc)$ we have
\begin{align*}
\liminf_{k \to \infty} \langle e_k, A e_k \rangle \leq \liminf_{\lambda \to \infty} \frac{\Tr(P_{\lambda} A P_{\lambda})}{\Tr(P_{\lambda})} \leq \frac{\Tr_\omega(A\langle D\rangle^{-d})}{\Tr_\omega(\langle D \rangle^{-d})} \leq \limsup_{\lambda \to \infty} \frac{\Tr(P_{\lambda} A P_{\lambda})}{\Tr(P_{\lambda})} \leq \limsup_{k \to \infty} \langle e_k, A e_k \rangle.  
\end{align*}

All results achieved in this section are different flavours of the observation that the noncommutative integral is the limit point --- in a weak, averaging notion --- of the sequence $\{\langle e_k,  A e_k \rangle \}_{k=0}^\infty$. For the circle $\mathbb{S}^1$ this is not surprising; given $f =\sum_{k=-\infty}^\infty a_k e_k \in L_1(\mathbb{S}^1)$ in Fourier basis, we have for \textit{every} $k \in \Z$
\[
\langle  e_k, M_f e_k \rangle = a_0 = \int_{\mathbb{S}^1} f(t) \, dt.
\]
More generally, Proposition~\ref{P:HeattoInt} combined with Connes' integral formula (Theorem~\ref{T:ConnesIntegral}) and Lemma~\ref{L:AHMSZ} shows that for any $d$-dimensional closed Riemannian manifold $M$ with volume form $\nu_g$ we have that the Ces\`aro mean of the sequence 
\[
\langle e_k, M_f  e_k \rangle, \quad f \in C(M)
\]
converges to $ \int_M f\, d\nu_g$. This fact is precisely what started investigations into quantum ergodicity. Recall that this covers the study of to what extent the matrix elements $\langle e_k,  M_f e_k \rangle$ themselves converge to an integral of $f$. More details will be provided in Section~\ref{S:Ergodicity}.

Previously, in~\cite{LordPotapov2010, LordSukochev2011, LSZVol1} it had already been observed that for spectral triples $(\Ac, \Hc, D)$ where $D^2$ satisfies Weyl's law that if the noncommutative integral
    \[
     \frac{\Tr_\omega(a \langle D\rangle^{-d})}{\Tr_\omega( \langle D\rangle^{-d})}, \quad a \in \Ac,
    \]
    is independent of $\omega$, then
    \[
    \frac{1}{\log(n+2)} \sum_{k=0}^n \frac{\langle  e_k, a e_k\rangle}{k+1}, \quad a \in \Ac,
    \]
    converges as $n\to \infty$, which was interpreted as being related to quantum ergodicity.

In quantum ergodicity and related fields, there is a vast literature on the properties and asymptotics of the operators $P_\lambda a P_\lambda$. Through the results established in this section, the link with Connes' integral formula unlocks this literature for study from the perspective of noncommutative geometry. One result from this cross-pollination is a Szeg\H{o} limit theorem for truncated spectral triples.

\section{Szeg\H{o} limit theorem}\label{S:Szego}
Szeg\H{o} proved various limit theorems concerning determinants of Toeplitz matrices, inspired by a conjecture by P\'olya and after work on these determinants by Toeplitz, Caratheodory and Fej\'er, see~\cite{Szego1915} and references therein. Much later, Widom provided a generalisation of these results with a simplified proof~\cite{Widom1979}, see also~\cite{LaptevSafarov1996} for a version for elliptic selfadjoint (pseudo)differential operators on manifolds without boundary. We now provide a translation of the results of Widom into noncommutative geometry. We thank Magnus Goffeng for pointing out that  instead of requiring that $[|D|,A]$ is bounded, it suffices to assume in the following lemma that $[D,A]$ is bounded.

\begin{lem}[\cite{Widom1979}]\label{L:Widom}
    Let $D^2$ satisfy Weyl's law (Definition~\ref{D:WeylLaw}), and let $A\in B(\Hc)$ and $B \in B(\Hc)$ map $\dom|D|$ into itself, and be such that $[D,A]$ and $[D,B]$ are bounded. 
    Then
    \[
    \lim_{\lambda \to \infty} \frac{\Tr(P_\lambda A (1-P_\lambda) B P_\lambda)}{\Tr(P_\lambda)} = 0.
    \]
\end{lem}
\begin{proof}
    First, $[D,A]$ being bounded implies that $[\langle D \rangle^{\frac{1}{2}}, A]$ is bounded due to our results in Chapter~\ref{Ch:MOIs}, namely Theorem~\ref{T:MOOIforNCG1} and Proposition~\ref{P:UMOIcom} (alternatively, see~\cite[Lemma~10.13]{GVF2001}). Hence, replacing $D$ by $\langle D \rangle^{\frac{1}{2}}$, we can assume that $D$ is positive and that $[|D|,A]$ and $[|D|, B]$ are bounded. Then, by the Cauchy-Schwarz inequality, an equivalent formulation of the statement is that for every $B$ such that $[|D|,B]$ is bounded, we have
    \[
    \lim_{\lambda \to\infty}\frac{\|P_\lambda B (1-P_\lambda)\|_{HS}^2}{\Tr(P_\lambda)} = 0,
    \]
    where $\| \cdot \|_{HS}$ is the Hilbert--Schmidt norm. The following argument is essentially due to Widom~\cite{Widom1979},  see also \cite[Lemma 3.4]{Guillemin1979}.

    Let $N>0.$ A quick computation~\cite[p.~145]{Widom1979} shows that
    \[
        \|P_{\lambda}B(1-P_{\lambda+N})\|_{HS}^2 \leq N^{-2}\|P_{\lambda}[|D|,B](1-P_{\lambda+N})\|_{HS}^2.
    \]
    By the triangle inequality, we have
    \begin{align*}
        \|P_{\lambda}B(1-P_{\lambda})\|_{HS}^2&\leq 2\|P_{\lambda}B(P_{\lambda+N}-P_{\lambda})\|_{HS}^2+2\|P_{\lambda}B(1-P_{\lambda+N})\|_{HS}^2\\
        &\leq 2\|B\|_{\infty}^2\Tr(P_{\lambda+N}-P_{\lambda})+2N^{-2}\Tr(P_{\lambda})\|[|D|,B]\|_{\infty}^2.
    \end{align*}
    Weyl's law implies that
    \[
        \Tr(P_{\lambda+N}-P_{\lambda}) = o(\Tr(P_{\lambda})),\quad \lambda\to\infty,
    \]
    and hence
    \[
        \limsup_{\lambda\to\infty} \frac{\|P_{\lambda}B(1-P_{\lambda})\|_{HS}^2}{\Tr(P_{\lambda})} \leq 2N^{-2}\|[|D|,B]\|_{\infty}^2.
    \]
    Since $N$ is arbitrary, this completes the proof.
\end{proof}

Following Widom~\cite{Widom1979} further, Lemma~\ref{L:Widom} can be combined with the characterisation of Connes' integral theorem in Theorem~\ref{T:Log-CesaroMean} into a Szeg\H{o} limit theorem.

\begin{thm}\label{T:Szego}
    Let $D^2$ satisfy Weyl's law (Definition~\ref{D:WeylLaw}), and let $A \in B(\Hc)$ be self-adjoint and such that it maps $\dom|D|$ into itself and $[D,A]$ is bounded. 
    Then 
    \[
    \Tr_\omega(\langle D \rangle^{-d})\cdot(\omega \circ M) \bigg( \frac{\Tr(f(P_{\lambda_n} A P_{\lambda_n}))}{\Tr(P_{\lambda_n})}\bigg)= \Tr_\omega(f(A)\langle D\rangle^{-d}), \quad f\in C(\R),\, f(0)=0.
    \]
    If for every positive integer $k$ there is some constant $C_k\in \R$ with
    \[
    \Tr(A^ke^{-tD^2}) \sim C_k t^{-\frac{d}{2}},
    \]
    then for every $f \in C(\R)$ with $f(0)=0$ we have
    \[
    \Tr_\omega(\langle D\rangle^{-d})\lim_{\lambda \to \infty}\frac{\Tr(f(P_{\lambda} A P_{\lambda}))}{\Tr(P_{\lambda})}=\Tr_\omega( f(A)\langle D\rangle^{-d}).
    \]
\end{thm}
\begin{proof}
    We sketch the proof of the stronger identity 
    \begin{equation}\label{eq:StrongerSzego}
    \Tr_\omega(\langle D \rangle^{-d})\cdot(\omega \circ M) \bigg( \frac{\Tr( P_{\lambda_n}f(P_{\lambda_n} A P_{\lambda_n}) P_{\lambda_n})}{\Tr(P_{\lambda_n})}\bigg)= \Tr_\omega(f(A)\langle D\rangle^{-d}), \quad f \in C(\R).
    \end{equation}
    Lemma~\ref{L:Widom} gives that
    \begin{equation}\label{e:Widom_identity}
    \lim_{\lambda \to \infty}\frac{\Tr\big(P_\lambda A^k P_\lambda - (P_\lambda A P_\lambda)^k\big)}{\Tr(P_\lambda)} = 0, \quad k\geq 1,
    \end{equation}
    which implies equation~\eqref{eq:StrongerSzego} for polynomial $f$ through Theorem~\ref{T:Log-CesaroMean}. An application of the Stone--Weierstrass theorem provides an extension to continuous functions.  Details can be found in~\cite{Widom1979}.  

    The second assertion for polynomial functions $f$ is a combination of \eqref{e:Widom_identity} and Proposition~\ref{P:HeattoInt}. If $f$ is a continuous function on $\R$ with $f(0)=0$, let $\varepsilon>0$ and choose a polynomial function $p$ with $p(0)=0$ such that 
    \[
        \|f-p\|_{L_{\infty}([-\|A\|_{\infty},\|A\|_{\infty}])}<\varepsilon.
    \]
    Then 
    \[
        \Big|\frac{\Tr((f-p)(P_{\lambda}AP_{\lambda}))}{\Tr(P_{\lambda})}\Big|< \varepsilon
    \]
    and
    \[
        |\Tr_{\omega}((f-p)(A)\langle D\rangle^{-d})| \leq \varepsilon\|\langle D\rangle^{-d}\|_{1,\infty}.
    \]
    Hence
    \[
        \limsup_{\lambda\to\infty} \Big|\Tr_\omega(\langle D\rangle^{-d})\frac{\Tr( f(P_{\lambda}AP_{\lambda}))}{\Tr(P_{\lambda})}-\Tr_{\omega}(f(A)\langle D\rangle^{-d})\Big| \leq 2\varepsilon \|\langle D\rangle^{-d}\|_{1,\infty}.
    \]
    Since $\varepsilon$ is arbitrary, this implies the result.
\end{proof}

We emphasise that Theorem~\ref{T:Szego} shows that the classical Szeg\H{o} theorems for determinants of Toeplitz matrices and Widom's generalisations thereof can be interpreted as properties of the noncommutative integral on spectral triples and their spectral truncations. 

\section{Fr\"{o}hlich functional}\label{S:Frohlich}
So far, we have considered situations modeled after $d$-dimensional spectral triples, where $\langle D\rangle^{-d} \in \mathcal{L}_{1,\infty}$. There are many examples of spectral triples that do not satisfy this condition, however. Instead, one could consider the property of $\theta$-summability, which says that $\Tr(e^{-tD^2}) < \infty$ for all $t>0$, or $\mathrm{Li}_1$-summability which requires $\Tr(e^{-t|D|}) < \infty$ for $t$ large enough.

For this section, we therefore assume that $D$ is a self-adjoint operator with compact resolvent, but we do not assume Weyl laws. The functional
\[
a \mapsto \lim_{t\to \beta}\frac{\Tr(ae^{-t|D|})}{\Tr(e^{-t|D|})}, \quad a \in B(\Hc),
\]
which is sometimes called the Fr\"{o}hlich functional after \cite{FrohlichGrandjeanRecknagel1995,FrohlichChamseddineFelder1993}, has been studied extensively in the literature~\cite{GoffengMesland2018,GoffengRennie2019}. We highlight the relation between this functional and the one that has been the object of study in this chapter. 

\begin{prop}
    Assume there exists $\beta \geq 0$ such that $\Tr(e^{-t|D|}) < \infty$ for $t > \beta$ and $\lim_{t \searrow \beta}\Tr(e^{-t|D|}) = \infty$. Then, for any extended limit $\omega \in \ell_\infty^*$ there exists an extended limit $\hat{\omega}_{D,\beta}\in \ell_\infty^*$ depending on $D$ and $\beta$ such that
    \[
    \omega \bigg(\frac{\Tr(ae^{- (\beta + \frac{1}{n}) |D|})}{\Tr(e^{-(\beta + \frac{1}{n})|D|})}  \bigg) = \hat{\omega}_{D,\beta} \bigg(\frac{\Tr(P_{\lambda_n} a P_{\lambda_n})}{\Tr(P_{\lambda_n})} \bigg), \quad a \in B(\Hc).
    \]
    Furthermore,
    \[
    \lim_{\lambda \to \infty} \frac{\Tr(P_\lambda a P_\lambda)}{\Tr(P_\lambda)} = \lim_{t\to \beta}\frac{\Tr(ae^{-t|D|})}{\Tr(e^{-t|D|})}, \quad a \in B(\Hc),
    \]
    in the sense that if the LHS limit exists, then the RHS limit exists and the equality holds.
\end{prop}
\begin{proof}
Write $\{r_k\}_{k=0}^\infty$ for the eigenvalues of $|D|$ counted \textit{without} multiplicity so that $r_0 < r_1 < \cdots$. Observe the identity $r_k = \lambda_{N(r_k)}$ where $N(\lambda):= \# \{k : \lambda_k \leq \lambda \}$ is the spectral counting function of $|D|$ and $\{\lambda_n\}_{n=0}^\infty$ are the eigenvalues of $|D|$ counted \textit{with} multiplicity. 
Then,
\begin{align*}
    \frac{\Tr(ae^{- (\beta + \frac{1}{n}) |D|})}{\Tr(e^{-(\beta + \frac{1}{n})|D|})} &= \frac{1}{\Tr(e^{-(\beta + \frac{1}{n})|D|})} \sum_{k=0}^\infty \big(\sum_{\lambda_j = r_k} \langle e_j, a e_j \rangle\big)e^{-(\beta + \frac{1}{n})r_k}\\
    &= \frac{1}{\Tr(e^{-(\beta + \frac{1}{n})|D|})} \sum_{k=0}^\infty \bigg(\Tr(P_{r_k})\frac{\Tr(P_{r_k}a P_{r_k})}{\Tr(P_{r_k})}\\
    &\quad \quad \quad \quad \quad \quad \quad\quad\quad\quad\quad\quad- \Tr(P_{r_{k-1}}) \frac{\Tr(P_{r_{k-1}} a P_{r_{k-1}})}{\Tr(P_{r_{k-1}})} \bigg)   e^{-(\beta + \frac{1}{n})r_k}.
\end{align*}
Hence if we define $\hat{\omega}_{D,\beta} \in (\ell_\infty)^*$ by
\[
\hat{\omega}_{D,\beta}(b_n) := \omega \bigg( \frac{1}{\Tr(e^{-(\beta + \frac{1}{n})|D|})} \sum_{k=0}^\infty \Big(\Tr(P_{r_k}) b_{N(r_k)}  - \Tr(P_{r_{k-1}}) b_{N(r_{k-1})} \Big)   e^{-(\beta + \frac{1}{n})r_k}  \bigg), \quad b \in \ell_\infty,
\]
we have by construction that 
\[
    \omega \bigg(\frac{\Tr(ae^{-(\beta +  \frac{1}{n}) |D|})}{\Tr(e^{-(\beta + \frac{1}{n})|D|})}  \bigg) = \hat{\omega}_{D,\beta} \bigg(\frac{\Tr(P_{\lambda_n} a P_{\lambda_n})}{\Tr(P_{\lambda_n})} \bigg), \quad a \in B(\Hc).
\]
Crucially, $\hat{\omega}_{D, \beta}$ is an extended limit if and only if $\lim_{t \searrow \beta}\Tr(e^{-t|D|}) = \infty$, see~\cite[Theorem~III.2]{Hardy1949}. The second assertion of the proposition is proved through the continuous version of the cited theorem, namely~\cite[Theorem~III.5]{Hardy1949}.

Finally, if the limit 
\[
    \lim_{\lambda \to \infty} \frac{\Tr(P_\lambda a P_\lambda)}{\Tr(P_\lambda)}
\]
exists, then all extended limits on the sequence $\frac{\Tr(P_{\lambda_n} a P_{\lambda_n})}{\Tr(P_{\lambda_n})}$ coincide, and so we conclude
\[
    \lim_{\lambda \to \infty} \frac{\Tr(P_\lambda a P_\lambda)}{\Tr(P_\lambda)} = \lim_{t\to \beta}\frac{\Tr(ae^{-t|D|})}{\Tr(e^{-t|D|})}, \quad a \in B(\Hc).\qedhere
\]
\end{proof}

Writing $P_D := \chi_{[0,\infty)}$ and applying the above results to $P_D D$ instead of $D$, we have that
\[
\omega \bigg(\frac{\Tr(P_D ae^{- (\beta + \frac{1}{n})D})}{\Tr(P_D e^{-(\beta + \frac{1}{n})D})}  \bigg)= \hat{\omega}_{D,\beta} \bigg(\frac{\Tr(\chi_{[0,\lambda_n]}(D) a \chi_{[0,\lambda_n]}(D))}{\Tr(\chi_{[0,\lambda_n]}(D))} \bigg), \quad a \in B(\Hc),
\]
which is a functional that is extensively studied in~\cite{GoffengRennie2019}. In particular, it defines a KMS state of inverse temperature $\beta$ on the Toeplitz algebra generated by a $\mathrm{Li}_1$-summable spectral triple $(\Ac, \Hc, D)$ satisfying some extra conditions.

\section{Noncommutative ergodicity}\label{S:Ergodicity}
Quantum ergodicity began as a study of geodesic flow on manifolds through abstract operator theoretical language. On a closed Riemannian manifold $(M,g)$ we can define the geodesic flow as a map $G_t: SM \to SM$, where $SM$ is the unit sphere in the tangent bundle of the manifold $M$. For a point $(x,v) \in SM$, one simply takes the unique geodesic $\gamma: \R \to M$ with $\gamma(0) = x$ and $\gamma'(0) = v$, and defines $G_t(x,v):= (\gamma(t), \gamma'(t))$. This flow is said to be ergodic if every measurable function $f \in L_\infty(SM)$ which is fixed by the flow (i.e. $f \circ G_t = f$ almost everywhere) is constant almost everywhere. Equivalently, the geodesic flow can be defined on $S^*M$, the unit sphere in the cotangent bundle.

Let $\{e_k\}_{k=0}^\infty$ be any orthonormal basis of eigenvectors of the Laplace--Beltrami operator $\Delta_g$ and let $P_\lambda := \chi_{[0,\lambda]}(-\Delta_g)$. 
Related to the result derived in Section~\ref{S:Integration}, it is known that for compact Riemannian manifolds we have that
\[
\frac{\Tr(P_\lambda \mathrm{Op}(a) P_\lambda)}{\Tr(P_\lambda)} \to \int_{S^*M}a \, d\nu,
\]
where $a \in C^\infty(S^*M)$ and $\mathrm{Op}(a)$ is a classical pseudodifferential operator with principal symbol $a.$
Colin de Verdi\`ere showed~\cite{ColindeVerdiere1985} that this fact can be used, if $M$ has ergodic geodesic flow, to show that there exists a density one subsequence $\{e_{j}\}_{j \in J}$ of $\{e_k\}_{k=0}^\infty$, meaning that $\frac{\#(J \cap \{0,\ldots, n\})}{n+1} \to 1$, such that
\[
\lim_{J\ni j\to\infty} \langle e_{j}, \mathrm{Op}(a) e_{j}\rangle =\int_{S^*M}a \, d\nu.
\]
This and related properties are called quantum ergodicity of the operator $\Delta_g$. See also Figure~\ref{F:eigenfunctions}.

Before we start to put quantum ergodicity results into a noncommutative geometrical context, let us observe first that our labours in Section~\ref{S:Integration} provide a result in the other direction. The Weyl measure of an operator, which is the relevant measure for quantum ergodicity~\cite[Section~4]{ColindeVerdiereHillairet2018}, clearly admits a Dixmier trace formula.

\begin{defn}\label{D:WeylMeasure}
    Let $M$ be a manifold equipped with a nonvanishing density $\rho$, and let $\Delta$ be a self-adjoint positive operator on $L_2(M,\rho)$ with compact resolvent. Let $\{e_k\}_{k=0}^\infty$ be an orthonormal basis of $L_2(M,\rho)$ consisting of eigenvectors of $\Delta$ with corresponding eigenvalues $\{\lambda_k\}_{k=0}^\infty$. If
\[
\lim_{\lambda \to \infty} \frac{1}{N(\lambda)} \sum_{\lambda_k \leq \lambda} \langle  e_k, M_f e_k \rangle
\]
exists for all $f \in C_c(M)$, then there exists a measure $\mu_\Delta$ such that
\[
\lim_{\lambda \to \infty} \frac{1}{N(\lambda)} \sum_{\lambda_k \leq \lambda} \langle  e_k, M_f e_k \rangle = \int_M f \, d\mu_\Delta.
\]
This measure is called the \textit{local Weyl measure} of $\Delta$. 
\end{defn}

\begin{prop}\label{P:DixTraceWeyl}
    If $\Delta$ as in Definition~\ref{D:WeylMeasure} satisfies Weyl's law
    \[
    \lambda(k, \Delta) \sim C k^\frac{2}{d}
    \]
    for some $0 < d \in \R$ and admits a local Weyl measure $\mu_\Delta$, then
    \begin{equation}\label{eq:DixTraceWeyl}
        \Tr_\omega(M_f (1+\Delta)^{-\frac{d}{2}}) =  \Tr_\omega((1+\Delta)^{-\frac{d}{2}}) \int_M f \, d\mu_\Delta.    
    \end{equation}
\end{prop}
\begin{proof}
    Consequence of Theorem~\ref{T:Log-CesaroMean}.
\end{proof}

This is relevant for sub-Riemannian manifolds, in which case one can take $\Delta$ to be the sub-Laplacian and $\mu_\Delta$ is not necessarily the usual volume form on the manifold $M$. Notably, a rescaling of this measure was found very recently in~\cite{KordyukovSukochev2024} to be a spectrally correct sub-Riemannian volume of $M$, additionally providing in that context a generalisation of the above Dixmier trace formula to any normalised continuous trace $\phi$. This measure is studied extensively in this context in~\cite{ColindeVerdiereHillairet2018} as well.

We will now shift our attention to results in quantum ergodicity which are interesting when viewed from the perspective of noncommutative geometry. To start, we provide an analogue of ergodicity of the geodesic flow -- a property a compact Riemannian manifold can have, which we should therefore be able to see as a property of a spectral triple. For this purpose we recall the following construction and theorem by Connes~\cite{Connes1995}. For a spectral triple $(\Ac, \Hc, D)$ we write $\Psi^0$ for the set of operators admitting an asymptotic expansion
    \[
    P = b_0 + b_{-1}\langle D\rangle^{-1} + b_{-2}  \langle D\rangle^{-2} + \cdots, \quad b_j \in \mathcal{B},
    \]
    with $\mathcal{B}$ generated by $\Ac$ and $\delta^n(\Ac)$, where $\delta(a):= [|D|,a]$. The asymptotic expansion is taken in the sense of Chapters~\ref{Ch:FunctCalc} and~\ref{Ch:MOIs}, which means here that the difference between $P$ and the $n$th partial summand is an element of $\op^{-n}(\langle D \rangle)$. 
    Note that we do not include the operators $[D,\Ac]$ in~$\mathcal{B}$. For $A \in B(\Hc)$, we write $\sigma_t(A):= e^{it|D|}Ae^{-it|D|}$, $t \in \R$.
\begin{thm}\label{T:ConnesS*A}
    For a unital regular spectral triple $(\Ac, \Hc, D)$, where regular means that $\delta^n(a) \in B(\Hc)$ for all $a \in \Ac$, $n \in \mathbb{N}$, define
    \[
    S^*\Ac := C^*\bigg(\bigcup_{t\in \R}\sigma_t\big(\Psi^0\big) + K(\Hc)\bigg) \big/ K(\Hc).
    \]
    This $C^*$-algebra comes equipped with automorphisms 
    \[
    \sigma_t(A + K(\Hc)):= e^{it|D|} A e^{-it|D|} + K(\Hc).
    \]
    For $(\Ac, \Hc, D) \simeq (C^\infty(M), L_2(S), D_M)$, the Dirac spectral triple associated with a compact Riemannian spin manifold, we have $S^*\Ac \simeq C(S^*M)$. Furthermore, if $A = Op^F(\sigma) \in \Psi^0$ for $\sigma \in C^\infty(S^*M)$ we have $\sigma_t(A+ K(\Hc)) = \mathrm{Op}^F(\sigma \circ G_t) + K(\Hc)$ where $G_t$ is the geodesic flow on~$S^*M$. 
\end{thm}

The identity $\sigma_t(\mathrm{Op}^F(\sigma)) = \mathrm{Op}^F(\sigma \circ G_t) + K(\Hc)$ is known as Egorov's theorem, see e.g.~\cite[Section~9.2]{Zelditch2017}\cite[Section~11.1]{Zworski2012}. This provides the basis for interpreting $\sigma_t$ as an analogue of geodesic flow even in the noncommutative case.

In Theorem~\ref{T:ConnesS*A}, it is important that the operators $[D,\Ac]$ are not included in $\Psi^0$. For illustration, in the commutative case the principal symbol of $|D|$ acts as a scalar on the vector bundle $S$, meaning that $\mathcal{B}$ and hence $\Psi^0$ can be regarded as acting on $L_2(M)$ instead of $L_2(S)$. The isomorphism $S^*C^\infty(M) \cong C(S^*M)$ in Theorem~\ref{T:ConnesS*A} is then an extension of the symbol map $\sigma: \Psi^0_{cl} \to C^\infty(S^*M)$ on classical pseudodifferential operators on $M$.

\begin{rem}
    For a (unital) spectral triple, $\langle D\rangle^{-1}$ is compact and hence
    \[
    S^*\Ac = C^*\bigg(\bigcup_{t\in \R}\sigma_t(\mB) + K(\Hc)\bigg) \big/ K(\Hc).
    \]
    Furthermore, since for $b\in \mB$ we have the convergence in norm
    \[
    \lim_{t\to 0} \frac{\sigma_t(b)-b}{t} = [|D|,b],
    \]
    we in fact have
    \[
    S^*\Ac = C^*\bigg(\bigcup_{t\in \R}\sigma_t(\Ac) + K(\Hc)\bigg) \big/ K(\Hc).
    \]
\end{rem}

A few examples of this construction are given in~\cite{GolseLeichtnam1998}. In the context of foliations of manifolds, it has been covered in~\cite{Kordyukov2005}.

The automorphisms $\sigma_t$ provide an action of $\R$ on the $C^*$-algebra $S^*\Ac$, and this noncommutative cotangent sphere is thus an example of a $C^*$-dynamical system.

\begin{defn}
    A $C^*$-dynamical system $(\Ac, G, \alpha)$ consists of a $C^*$-algebra $\Ac$, a locally compact group $G$, and a strongly continuous representation $\alpha: G \to \operatorname{Aut}(\Ac)$. 
\end{defn}

There is a vast literature on $C^*$-dynamical systems, see~\cite[Section~2.7]{BratteliRobinson1987} for a start. In particular it has been a popular object of study in the field of quantum ergodicity, see e.g.~\cite{Zelditch1996}.

In Connes's point of view, the automorphisms $\sigma_t$ on $S^*\Ac$ are the noncommutative analogue of geodesic flow. Recall that for compact manifolds this flow is said to be ergodic if the only measurable functions that are invariant almost everywhere under the geodesic flow are the functions that are constant almost everywhere. This definition is measure-theoretic in nature, and to translate it into a statement on spectral triples we therefore define the noncommutative $L_2$-space on $S^*\Ac$, which corresponds with $L_2(S^*M)$ in the commutative case.

\begin{prop}
    Let $(\Ac, \Hc, D)$ be a unital regular spectral triple where $D^2$ satisfies Weyl's law (Definition~\ref{D:WeylLaw}). The functional
    \[
    \tau(A + K(\Hc)):= \frac{\Tr_\omega(A\langle D\rangle^{-d})}{\Tr_\omega(\langle D \rangle^{-d})}, 
    \]
    defines a finite positive trace on $S^*\Ac$.
\end{prop}
\begin{proof}
This is well-known. We remark that the traciality of $\tau$ in fact follows from Theorem~\ref{T:Log-CesaroMean}, Widom's Lemma~\ref{L:Widom}, and the trivial identity
\[
\Tr(P_\lambda A P_\lambda B P_\lambda) = \Tr(P_\lambda B P_\lambda A P_\lambda). \qedhere
\]
\end{proof}

\begin{defn}\label{Def:L2SA}
    We define $L_2(S^*\Ac)$ as the Hilbert space $\Hc_\tau$ in the GNS construction $(\pi_\tau, \Hc_\tau)$. Explicitly, writing $I= \{A + K(\Hc) \in S^*\Ac : \tau(A^*A) = 0 \}$, we define
    \[
    L_2(S^*\Ac) := \overline{S^*\Ac / I}^{\|\cdot\|_{L_2}},
    \]
    where the completion is taken in the semi-norm $\| A + I \|_{L_2} = \big(\tau(A^*A)\big)^{\frac{1}{2}}$. The space $L_2(S^*\Ac)$ is a Hilbert space with inner product defined via
    \[
    \langle A + I, B+I \rangle_{L_2} := \tau(B^*A), \quad A,B \in S^*\Ac.
    \]
    \end{defn}
\begin{nota}
    In accordance with the paradigm called the $C^*$-algebraic approach to the principal symbol~\cite{Cordes1979, Dao1, Dao2, Dao3}, we write $\sym$ for the quotient map
    \begin{align*}
     \sym :C^*\bigg(\bigcup_{t\in \R}\sigma_t(\Ac) + K(\Hc)\bigg) &\to S^*\Ac\\
     A \qquad \quad \quad \quad &\mapsto A + K(\Hc),
    \end{align*}
    which is understood as a symbol map. Writing furthermore $\pi$ for the other quotient mapping
    \begin{align*}
        \pi: S^*\Ac &\to L_2(S^*\Ac)\\
        A \quad &\mapsto A+ I,
    \end{align*}
    where $I$ is as in Definition~\ref{Def:L2SA}, we will furthermore use the notation 
    \[
    \syml := \pi \circ \sym :C^*\bigg(\bigcup_{t\in \R}\sigma_t\big(\Ac\big) + K(\Hc)\bigg) \to L_2(S^*\Ac).
    \]
\end{nota}

\begin{ex}\label{ex:L2SAspaces}
    \begin{enumerate}
        \item For the Dirac spectral triple coming from a compact Riemannian spin manifold,  $(C^\infty(M), L_2(S), D_M)$, we have that  $S^*C^\infty(M) \simeq C(S^*M)$ with $\tau_{S^*\Ac} = \int_{S^*M}$. Hence $L_2(S^*C^\infty(M)) \simeq L_2(S^*M)$. The action $\sigma_t$ agrees with the usual geodesic flow.
        \item Given an even dimensional compact Riemannian manifold, and a finite dimensional spectral triple $(\Ac_F, \Hc_F, D_F)$, we have for the almost commutative manifold (see Section~\ref{S:IntroPreliminaries}) 
        \[
        (\Ac := C^\infty(M) \otimes \Ac_F, L_2(S)\otimes \Hc_F, D:=D_M \otimes 1 + \gamma_M \otimes D_F),
        \]
        that 
        $S^*\Ac \simeq C(S^*M)\otimes \Ac_F$ with $\tau_{S^*\Ac} = \int_{S^*M} \otimes \Tr$. Hence $L_2(S^*\Ac) \simeq L_2(S^*M)\otimes HS_F$, where $HS_F$ is simply $\Ac_F$ equipped with the Hilbert--Schmidt norm. The automorphisms $\sigma_t$ act as $\sigma_t^M \otimes 1$, where $\sigma_t^M$ is the usual geodesic flow on $S^*M$. This corrects~\cite[Lemma~2.2]{GolseLeichtnam1998}.
        \item For the noncommutative torus $(C^\infty(\mathbb{T}^d_\theta), L_2(\mathbb{T}^d_\theta) \otimes \mathbb{C}^{N_d}, D)$ (see Section~\ref{S:IntroPreliminaries}), we have that $S^*C(\mathbb{T}_\theta^d) \simeq C(\mathbb{T}^d_\theta) \otimes C(\mathbb{S}^{d-1})$ with $\tau_{S^*\mathbb{T}^d_\theta} = \tau_{\mathbb{T}^d_\theta} \otimes \int_{\mathbb{S}^{d-1}}$ and $\otimes$ is the minimal $C^*$-tensor product. Hence $L_2(S^* \mathbb{T}^d_\theta) \simeq  L_2(\mathbb{T}^d_\theta) \otimes L_2(\mathbb{S}^{d-1})$. The automorphisms $\sigma_t$ act as
        \[
        \sigma_t(u_{k} \otimes g) = u_k \otimes \exp(it\, k \cdot x)g, \quad t\in\R, k \in \mathbb{Z}^d, x\in \mathbb{S}^{d-1}\subseteq \mathbb{R}^d, g \in C(\mathbb{S}^{d-1}).
        \]
        \item Let $\Ac$ be the Toeplitz algebra, i.e. the $C^*$-algebra generated by the shift operator on $\ell_2(\mathbb{N})$, and let $D$ be the operator on $\ell_2(\mathbb{N})$ defined on the standard basis $\{ e_j \}_{j \in \mathbb{N}}$
        \[
        D: e_j \mapsto j e_j, \quad j \in \mathbb{N}.
        \]
        For the spectral triple $(\Ac, \ell_2(\mathbb{N}), D)$, we have $S^*\Ac \simeq C(\mathbb{S}^1)$ with $\tau_{S^*\Ac} = \int_{\mathbb{S}^1}$. Hence $L_2(S^*\Ac) \simeq L_2(\mathbb{S}^1)$. The automorphism $\sigma_t$ is given by rotation. 
        \end{enumerate}
\end{ex}
\begin{proof}
    $(1)$ can be found in~\cite{Connes1995}.\\
    $(2)$: Since $|D| = \sqrt{D_M^2\otimes 1 + 1 \otimes D_F^2}$, it follows that $|D| - |D_M|\otimes 1$ is a compact operator on $L_2(S) \otimes \Hc_F$. We will show this with the double operator integrals from Chapter~\ref{Ch:MOIs}. First, one can omit the kernels of $|D|$ and $|D_M|\otimes 1$ from the Hilbert space as the projection onto the kernel of either operator is finite-rank and thus compact. Both operators have compact resolvent. Hence, after this modification, the function $f(x) = \sqrt{x}$ is smooth on a neighbourhood of the spectra of the operators  $|D|$ and $|D_M|\otimes 1$. Define the Sobolev spaces $\Hc^s := \dom |D_M|^s \otimes 1$, $s \in \R$, and apply Proposition~\ref{P:UMOIcom} and Theorem~\ref{T:MOOIforNCG1} to find that, for $0<\varepsilon < 1$,
\[
|D| - |D_M|\otimes 1 = T^{D^2, D_M^2 \otimes 1}_{f^{[1]}}(1 \otimes D_F^2) \in \op^{-1+\varepsilon}(|D_M|\otimes 1) \subseteq  K(\Hc),
\]
It now follows from Duhamel's formula (see e.g.~\cite[Lemma~5.2]{ACDS}) that
\[
e^{it|D|} - e^{it(|D_M|\otimes 1)} = it \int_0^1 e^{ist|D|} (|D| - |D_M| \otimes 1) e^{i(1-s)t(|D_M|\otimes 1)} \, ds  \in K(\Hc).
\]
Therefore,
\[
\bigcup_{t\in \R}\sigma_t(\mathcal{B}) + K(\Hc) = \bigcup_{t\in \R}\sigma^M_t(\mathcal{B}_M) \otimes \Ac_F + K(\Hc),
\]
where $\sigma_t^M$ is the geodesic flow on $M$ and $\mB_M$ the algebra generated by $\delta_{|D_M|}^n(C^\infty(M)) \subseteq B(L_2(M))$. We conclude that $S^*\Ac \simeq C(S^*M)\otimes \Ac_F$. This proof shows that the action $\sigma_t$ on $ C(S^*M)\otimes \Ac_F  $ is given by $\sigma_t = \sigma_t^M \otimes 1$.
      \\   
    $(3)$: Although the Hilbert space of the spectral triple is $L_2(\mathbb{T}^d_\theta) \otimes \mathbb{C}^{N_d}$, since $|D| = \sqrt{-\Delta} \otimes 1$ (see Section~\ref{S:IntroPreliminaries}) similarly to the manifold case $\mB$ acts trivially on the $\mathbb{C}^{N_d}$-component. We can therefore make the identification $\mB \subseteq B(L_2(\mathbb{T}^d_\theta))$. In fact, we claim that $\overline{\mB}^{\|\cdot\|}$ is a $C^*$-algebra stable under the action $\sigma_t(\cdot) = e^{it\sqrt{-\Delta}} (\cdot) e^{-it\sqrt{-\Delta}}$, and therefore
    \[
    S^*C(\mathbb{T}^d_\theta) \simeq (\overline{\mB}^{\|\cdot\|} + K(L_2(\mathbb{T}^d_\theta)))/K(L_2(\mathbb{T}^d_\theta)).
    \]
    The claim holds since formally $\sigma_t(a) = \sum_{k=0}^\infty \frac{(it)^k}{k!} \delta^k(a)$, and this sum is actually norm convergent for     
    $a \in \mathrm{Poly}(\mathbb{T}^d_\theta) := \mathrm{span}\{u_n\}_{n\in \mathbb{Z}^d}$. Denoting the generated $*$-algebra $\mB_{poly}:= \langle a, \delta^n(a)\rangle_{a \in \mathrm{Poly}(\mathbb{T}^d_\theta)}$ we therefore have
    \[
    \sigma_t : \mB_{poly} \to \overline{\mB}^{\|\cdot\|}.
    \]
    Since $ \mathrm{Poly}(\mathbb{T}^d_\theta)$ is dense in $C^\infty(\mathbb{T}^d_\theta)$ and $\sigma_t$ is an isometry on $B(L_2(\mathbb{T}^d_\theta))$, it is easily seen that this implies that $\sigma_t$ maps $\overline{\mB}^{\|\cdot\|}$ into itself.
    
    By construction $C(\mathbb{T}^d_\theta)$ is represented on $L_2(\mathbb{T}^d_\theta)$ as bounded left-multiplication operators (denote the representation $\pi_1$), and $C(\mathbb{S}^{d-1})$ is as well via the representation
    \[
    \pi_2(g) = g\big(\frac{D_1}{\sqrt{-\Delta}}, \ldots, \frac{D_d}{\sqrt{-\Delta}} \big),\quad g \in C(\mathbb{S}^{d-1}),
    \]
    where $D_i: u_k \mapsto k_i u_k$. It is shown in~\cite{Dao2} that, writing $\Pi(C(\mathbb{T}^{d}_\theta), C(\mathbb{S}^{d-1}) )$ for the $C^*$ -algebra generated by $\pi_1( C(\mathbb{T}^{d}_\theta))$ and $\pi_2( C(\mathbb{S}^{d-1}))$ inside $B(L_2(\mathbb{T}^d_\theta))$, we have
    \[
    \Pi(C(\mathbb{T}^{d}_\theta), C(\mathbb{S}^{d-1}) ) / K(L_2(\mathbb{T}^d_\theta)) \simeq C(\mathbb{T}^{d}_\theta) \otimes C(\mathbb{S}^{d-1}).
    \]    
    To determine that $S^*\mathbb{T}^d_\theta \simeq C(\mathbb{T}^d_\theta) \otimes C(\mathbb{S}^{d-1})$, it therefore suffices to show that 
    \[
    \overline{\mB}^{\|\cdot\|} + K(L_2(\mathbb{T}^d_\theta))= \Pi(C(\mathbb{T}^{d}_\theta), C(\mathbb{S}^{d-1}) ) + K(L_2(\mathbb{T}^d_\theta)) \subseteq B(L_2(\mathbb{T}^d_\theta)).
    \]
    To start, it is immediately obvious that $\pi_1(C(\mathbb{T}^{d}_\theta)) \subseteq \overline{\mB}^{\|\cdot\|}$. Next, the operators $\frac{D_i}{\sqrt{-\Delta}}$ generate $\pi_2(C(\mathbb{S}^{d-1}))$ as a $C^*$-algebra, and we claim that
    \[
    u_{e_j}^*[\sqrt{-\Delta}, u_{e_j}] - \frac{D_j}{\sqrt{-\Delta}} \in K(L_2(\mathbb{T}^d_\theta)),
    \]
    where $e_j \in \mathbb{Z}^d$ is the standard unit vector.
    This would imply
    \[
     \Pi(C(\mathbb{T}^{d}_\theta), C(\mathbb{S}^{d-1}) ) + K(L_2(\mathbb{T}^d_\theta)) \subseteq \overline{\mB}^{\|\cdot\|} + K(L_2(\mathbb{T}^d_\theta)).
    \]
    The claim is proven by writing 
    \begin{align*}
        \bigg(u_{e_j}^*[\sqrt{-\Delta}, u_{e_j}] - \frac{D_j}{\sqrt{-\Delta}}\bigg)u_k &= \bigg(|k+e_j| - |k| - \frac{k_j}{|k|} \bigg) u_{k}.        
    \end{align*}
    Now define 
    \[
    f(t, k):= |k+ t e_j|, \quad k\in \mathbb{Z}^d, t \in \R,
    \]
    and note that its derivatives in the $t$ variable are
    \[
    f'(t,k) = \frac{k_j+t}{|k+te_j|}, \quad f''(t,k) = \frac{|k+te_j|^2 - (k_j+t)^2}{|k+te_j|^3}.
    \]
    Hence 
    \begin{align*}
        \bigg(u_{e_j}^*[\sqrt{-\Delta}, u_{e_j}] - \frac{D_j}{\sqrt{-\Delta}}\bigg)u_k &= (f(1,k) - f(0,k) - f'(0,k))u_k\\
        &= \int_{0}^1 (1-t) f''(t, k) \, dt \cdot u_k.
    \end{align*}
    From the form of $f''(t,k)$ above, we therefore have
    \[
    |f(1,k) - f(0,k) - f'(0,k)| \in c_0(\mathbb{Z}^d),
    \]
    which indeed shows that $u_{e_j}^*[\sqrt{-\Delta}, u_{e_j}] - \frac{D_j}{\sqrt{-\Delta}}$ is a compact operator.

    For the other direction, the above arguments already show that
    \[
    [\sqrt{-\Delta}, u_{e_j}] \in \Pi(C(\mathbb{T}^{d}_\theta), C(\mathbb{S}^{d-1}) ) + K(L_2(\mathbb{T}^d_\theta)).
    \]
    Since by explicit computation
    \[
    u_{e_j}^*\delta_{\sqrt{-\Delta}}^n(u_{e_j}) = (u_{e_j}^*[\sqrt{-\Delta}, u_{e_j}])^n, 
    \]
    and since $u_{e_j}^m = u_{m e_j}$, we have that
    \[
    \delta^n_{\sqrt{-\Delta}}(u_k) \in \Pi(C(\mathbb{T}^{d}_\theta), C(\mathbb{S}^{d-1}) ) + K(L_2(\mathbb{T}^d_\theta)), \quad n \in \mathbb{Z}_{\geq 0}, k \in \mathbb{Z}^d,
    \]
    and hence
    \[
    \overline{\mB}^{\| \cdot \|} \subseteq \Pi(C(\mathbb{T}^{d}_\theta), C(\mathbb{S}^{d-1}) ) + K(L_2(\mathbb{T}^d_\theta)).
    \]   
    In conclusion, $S^*\mathbb{T}^d_\theta \simeq C(\mathbb{T}^d_\theta) \otimes C(\mathbb{S}^{d-1})$. For the automorphism $\sigma_t$, first note that
    \[
    \sigma_t(g(\frac{D_1}{\sqrt{-\Delta}}, \ldots, \frac{D_d}{\sqrt{-\Delta}})) = g(\frac{D_1}{\sqrt{-\Delta}}, \ldots, \frac{D_d}{\sqrt{-\Delta}}), \quad g\in C(\mathbb{S}^{d-1}).
    \]
    Next,
    \begin{align*}
        e^{it\sqrt{-\Delta}} u_{e_j}e^{-it\sqrt{-\Delta}} u_k &= e^{it(|k+e_j| - |k|)} u_{e_j} u_k\\
        &= u_{e_j}\exp(it u_{e_j}^* [\sqrt{-\Delta},u_{e_j}])  u_k\\
        &=u_{e_j}\exp\big(it \frac{D_j}{\sqrt{-\Delta}}\big) u_k + u_{e_j}\bigg(\exp(it u_{e_j}^* [\sqrt{-\Delta},u_{e_j}])  - \exp\big(it \frac{D_j}{\sqrt{-\Delta}}\big) \bigg)  u_k.
    \end{align*}
    We have already seen that $  u_{e_j}^* [\sqrt{-\Delta},u_{e_j}]  -  \frac{D_j}{\sqrt{-\Delta}} \in K(L_2(\mathbb{T}^d_\theta))$, and hence as in the proof of (2) it follows from Duhamel's formula that
    \begin{align*}
    \exp(it u_{e_j}^* [\sqrt{-\Delta},u_{e_j}]) - \exp\big(it \frac{D_j}{\sqrt{-\Delta}}\big) 
    \in K(L_2(\mathbb{T}^d_\theta)).
    \end{align*}
    Thus, we see that
    \[
        \sigma_t(u_{n} \otimes g) = u_n \otimes \exp(it\, n \cdot x)g, \quad t\in\R, n \in \mathbb{Z}^d, x\in \mathbb{S}^{d-1}\subseteq \mathbb{R}^d, g \in C(\mathbb{S}^{d-1}).
    \]
    $(4)$: It is well-known that, after identifying $\ell_2(\mathbb{N})$ with the Hardy space $H^2$, any element in the Toeplitz algebra $\Ac$ can be written as $T_\phi + K$, where $T_\phi$ is the Toeplitz operator with symbol $\phi \in C(\mathbb{S}^1)$ and $K \in K(\ell_2(\mathbb{N}))$, see e.g.~\cite[Section~3.5]{Murphy1990}. By an explicit computation, it can be seen that
    \[
    e^{it|D|} T_\phi e^{-it|D|} = T_{\phi \circ R_t},
    \]
    where $R_t$ is rotation by the angle $t$. Hence $\sigma_t(\Ac) = \Ac$, and 
    \[
    S^*\Ac = \Ac / K(\ell_2(\mathbb{N})) \simeq C(\mathbb{S}^1).
    \]
    For the noncommutative integral, we can use the diagonal formula in Theorem~\ref{T:Log-CesaroMean}, so that for an arbitrary element $T_\phi + K \in \Ac$,
    \[
    \Tr_\omega((T_\phi + K)\langle D\rangle^{-1}) = \omega \circ M (\langle e_k, (T_\phi + K) e_k \rangle ) = \int_{\mathbb{S}^1} \phi(t) \, dt. \qedhere
    \]
\end{proof}

Observe that the automorphism $\sigma_t$ on $S^*\Ac$ extends to a unitary operator $\sigma_t \in B(L_2(S^*\Ac))$. 

\begin{defn}
    We say that $(\Ac, \Hc, D)$ is classically ergodic if for $a \in L_2(S^*\Ac)$, we have $\sigma_t(a) =a$ for all $t\in \R$ if and only if $a = \lambda \cdot 1 \in L_2(S^*\Ac)$ for some $\lambda \in \C$.
\end{defn}

The construction of $L_2(S^*\Ac)$ has now reached its goal; for spectral triples derived from compact Riemannian manifolds, this definition is precisely the usual definition of ergodicity of the geodesic flow.

We now immediately claim the following theorem, the NCG analogue of the classic result in quantum ergodicity by Shnirelman, Zelditch, and Colin de Verdi\`ere~\cite{Shnirelman1974, ColindeVerdiere1985, Zelditch1987}.
\begin{thm}\label{T:MainQE}
    Let  $(\Ac, \Hc, D)$ be a unital regular spectral triple with local Weyl laws for $a \in \Ac$ (Definition~\ref{D:WeylLaw}). Assume that the closure of $\Ac$ in $B(\Hc)$ is separable. If the triple is classically ergodic, then for every basis $\{e_n\}_{n=0}^\infty$ of eigenvectors of $|D|$ there exists a density one subset $J \subseteq \mathbb{N}$ such that
    \[
    \lim_{J\ni j\to \infty}\langle e_j, a e_j \rangle = \frac{\Tr_{\omega}(a \langle D\rangle^{-d})}{\Tr_{\omega}(\langle D\rangle^{-d})}, \quad a \in \Ac.
    \]
\end{thm}
\begin{proof}
    Classical ergodicity of $(\Ac, \Hc,D)$ means precisely that the $C^*$-dynamical system $(S^*\Ac, \mathbb{R}, \sigma_t)$ has a unique vacuum state in the terminology of~\cite{Zelditch1996}. Hence, due to Proposition~\ref{P:HeattoInt}, the theorem is a consequence of~\cite[Lemma~2.1]{Zelditch1996}.
\end{proof}

This theorem, while its mathematical core is already an established result in quantum ergodicity, gives a fresh perspective on the criterion of a $C^*$-dynamical system having a `unique vacuum state'. And while the vast majority of results in the paper~\cite{Zelditch1996} are formulated for `quantised abelian' $C^*$-dynamical systems, which in our case would mean $S^*\Ac$ is represented as a commutative algebra on $L_2(S^*\Ac)$, the philosophy of noncommutative geometry provides solid reason to study not quantised abelian $C^*$-dynamical systems but ones with a unique vacuum state, as proposed by Zelditch~\cite{Zelditch1996}.

\begin{ex}We continue Example~\ref{ex:L2SAspaces}.
    \begin{enumerate}
        \item The canonical spectral triple corresponding to a compact Riemannian spin manifold, $(C^\infty(M), L_2(S), D_M)$ is classically ergodic if and only if $M$ has ergodic geodesic flow.
        \item Any nontrivial almost commutative manifold $( C^\infty(M) \otimes \Ac_F, L_2(S)\otimes \Hc_F, D_M \otimes 1 + \gamma_M \otimes D_F)$ is not classically ergodic. Note that this corrects~\cite[Corollary~(3.1)]{Zelditch1996}, which was already known to experts to be false.
        \item The noncommutative torus, like the commutative torus, is not classically ergodic.
        \item The spectral triple of the Toeplitz algebra is classically ergodic. See~\cite[Example~(D)]{Zelditch1996} for a generalisation.  
    \end{enumerate}
\end{ex}
\begin{proof} $(1)$: See~\cite{ColindeVerdiere1985}.\\
    $(2)$: Since $\sigma_t$ acts on $L_2(S)\otimes HS_F$ by $\sigma_t^M \otimes 1$, any element of the form $1 \otimes a$ is a fixed point of $\sigma_t$.\\
    $(3)$: It follows from Example~\ref{ex:L2SAspaces} that for any $f \in C(\mathbb{S}^{d-1})$, the element $\syml(1\otimes f) \in L_2(S^*\mathbb{T}^d_\theta)$ is a fixed point of $\sigma_t$.\\
    $(4)$: Since the only rotationally invariant functions in $L_2(\mathbb{S}^1)$ are the constant functions, the claim follows.
\end{proof}

We note that the well-studied examples of spectral triples in noncommutative geometry often possess a high degree of symmetry, and in geometric examples a high degree of symmetry can obstruct ergodicity.

We now conclude the paper by giving some equivalent conditions for classical ergodicity.
First, we invoke von Neumann's mean ergodic theorem.
\begin{prop}\label{P:ErgodicThms}
    For any $a \in L_2(S^*\Ac)$, there exists $a_{avg} \in L_2(S^*\Ac)$ such that, putting 
    \[
    a_T := \frac{1}{T}\int_{0}^T \sigma_t(a) \, dt,
    \]
    we have
    \[
    \lim_{T\to \infty} \| a_T- a_{avg}\|_{L_2} \to 0.
    \]
    Furthermore,
    \[
    \langle 1, a_{avg}\rangle_{L_2} = \langle 1, a\rangle_{L_2},
    \]
    and the map $a\mapsto a_{avg}$ is $L_2$-continuous.
\end{prop}
\begin{proof}
    The existence of $a_{avg}$ and the $L_2$-convergence of $a_T$ to $a_{avg}$ follows from von Neumann's mean ergodic theorem. Next, since $\sigma_t$ is a unitary operator with $\sigma_t(1) = 1$, we have 
    \[
    \langle 1, \sigma_t(a)\rangle_{L_2} = \langle 1, a \rangle_{L_2}.
    \]
    Hence,
    \begin{align*}
        \big|\langle 1, a_{avg}\rangle_{L_2} - \langle 1, a\rangle_{L_2} \big| &\leq \big|\langle 1, a_{avg}\rangle_{L_2} - \langle1, a_T\rangle_{L_2} \big| + \smash{\underbrace{ \big|\langle1,a_T\rangle_{L_2} - \langle1,a\rangle_{L_2} \big|}_{=0}}\\
        &=|\langle 1, a_{avg} - a_T\rangle_{L_2}|\\
        &\leq \| a_T- a_{avg}\|_{L_2},
    \end{align*}
    which converges to $0$ as $T \to \infty$ due to the first part.

    Finally, for the continuity of $a \mapsto a_{avg}$, note that
    \[
    \|a_T\|_{L_2} \leq \frac{1}{T} \int_0^T \|\sigma_t( a) \|_{L_2}\,dt =  \| a\|_{L_2},
    \]
    and taking the limit $T \to \infty$,
    \begin{align*}
        \|a_{avg} \|_{L_2}&\leq \|a\|_{L_2}. \qedhere 
    \end{align*}
\end{proof}

\begin{rem}
    Since $S^*\Ac$ is represented as bounded operators on $L_2(S^*\Ac)$, we can consider the von Neumann algebra $\pi_\tau(S^*\Ac)''$ in $B(L_2(S^*\Ac))$, denoted as $L_\infty(S^*\Ac)$, to which $\tau$ extends as a faithful normal tracial state. 
    We can define the noncommutative $L_p$ spaces $L_p(S^*\Ac):= L_p(\tau)$ for $1 \leq p \leq \infty$ via standard constructions (we recover $L_2(S^*\Ac)$ for $p=2$). This is precisely how the spaces $L_p(\mathbb{T}^d_\theta)$ are constructed~\cite[Section~3.5]{LMSZVol2}. It is possible to show that $\sigma_t: L_p(S^*\Ac) \to L_p(S^*\Ac)$ are isometries for all $1 \leq p \leq \infty$, and the averages in Proposition~\ref{P:ErgodicThms} exist and converge in every $L_p(S^*\Ac)$.
\end{rem}

\begin{prop}\label{P:ClassicErgod}
    Given a unital regular spectral triple  $(\Ac, \Hc, D)$ satisfying Weyl's law, the following are equivalent:
    \begin{enumerate}
        \item the spectral triple is classically ergodic;
        \item for all $a \in L_2(S^*\Ac)$,
    \[
     a_{avg} = \langle 1, a\rangle_{L_2} \cdot 1;
    \]
    \item writing 
    \[
    A_T := \frac{1}{T}\int_{0}^T \sigma_t(A) \, dt, \quad A \in \big \langle \bigcup_{t \in \R} \sigma_t(\mathcal{A}) \big \rangle,
    \]
    where $\big \langle \bigcup_{t \in \R} \sigma_t(\mathcal{A}) \big \rangle$ is the $*$-algebra generated by $\bigcup_{t \in \R} \sigma_t(\mathcal{A})$, for all $A \in \big \langle \bigcup_{t \in \R} \sigma_t(\mathcal{A}) \big \rangle$ we have
    \[
    \Tr_\omega(\langle D\rangle^{-d})^2\cdot \lim_{T \to \infty} \omega \circ M \Big( \langle e_k, |A_T|^2 e_k \rangle \Big) = \Big| \Tr_{\omega} (A \langle D\rangle^{-d})  \Big|^2.
    \]
    \end{enumerate} 
\end{prop}
\begin{proof}
    $(1) \Leftrightarrow (2)$ is easily seen from the fact that $a_{avg}$ is a fixed point of $\sigma_t$.\\
    Next, if $A \in \big \langle\bigcup_{t \in \R} \sigma_t(\mathcal{A})\big \rangle$, then by Theorem~\ref{T:Log-CesaroMean} it follows that 
    \begin{align*}
    \Tr_\omega(\langle D\rangle^{-d})^2\cdot \omega \circ M \Big( \langle e_k, |A_T|^2 e_k \rangle \Big) &= \Tr_\omega(\langle D\rangle^{-d}) \Tr_\omega(|A_T|^2\langle D\rangle^{-d})\\
    &= \Tr_\omega(\langle D\rangle^{-d})^2 \langle \syml(A_T), \syml(A_T)\rangle_{L_2}.
    \end{align*}
    Since $\syml(A_T) = \syml(A)_T$, Proposition~\ref{P:ErgodicThms} gives that
    \begin{equation}\label{eq:SunadaStep}
    \Tr_\omega(\langle D\rangle^{-d})^2\cdot \lim_{T \to \infty} \omega \circ M \Big( \langle e_k, |A_T|^2 e_k \rangle \Big) = \Tr_\omega(\langle D\rangle^{-d})^2 \langle \syml(A)_{avg}, \syml(A)_{avg}\rangle_{L_2}.
    \end{equation}
    $(2) \Rightarrow (3)$: This now follows from Equation~\eqref{eq:SunadaStep} and the identity $\langle 1, \syml(A) \rangle_{L_2} = \frac{\Tr_{\omega}(A\langle D\rangle^{-d})}{\Tr_\omega(\langle D\rangle^{-d})}$.\\
    $(3) \Rightarrow (2)$: For $a \in L_2(S^*\Ac)$, 
    Proposition~\ref{P:ErgodicThms} gives that $\langle a_{avg},1\rangle_{L_2} = \langle a, 1 \rangle_{L_2}$, and hence
    \begin{align*}
         \|a_{avg} - \langle 1,a\rangle \cdot 1\|_{L_2}^2 
        &=  \langle a_{avg}, a_{avg}\rangle_{L_2} - \langle a_{avg}, 1\rangle_{L_2} \overline{\langle 1,a\rangle_{L_2}} - \langle 1, a_{avg}\rangle _{L_2} \langle 1,a\rangle _{L_2} + |\langle 1,a\rangle_{L_2}|^2 \\
        &= \langle a_{avg}, a_{avg}\rangle_{L_2} - |\langle 1,a\rangle_{L_2}|^2.
    \end{align*}
    Therefore, assumption (3) combined with Equation~\eqref{eq:SunadaStep} gives for all $A \in \big \langle \bigcup_{t \in \R} \sigma_t(\mathcal{A})\big\rangle$,
    \[
    \syml(A)_{avg} = \langle 1, \syml(A)\rangle_{L_2} \cdot 1.
    \]
    The image of $\big \langle \bigcup_{t \in \R} \sigma_t(\mathcal{A})\big\rangle$ under the map $\syml$ being dense in $L_2(S^*\Ac)$, and the map $a \to a_{avg}$ being $L_2$-continuous, we can conclude that
    \[
    a_{avg} = \langle 1, a\rangle_{L_2} \cdot 1
    \]
    for all $a \in L_2(S^*\Ac)$.
\end{proof}
    \chapter{The density of states on discrete spaces}\label{Ch:DOSDiscrete}
{\setlength{\epigraphwidth}{\widthof{Best wishes and remember: WORK HARDER!!!}}
\epigraph{Best wishes and remember: WORK HARDER!!!}{Fedor Sukochev}}
This chapter is an adaptation of~\cite{AHMSZ}, joint work with Nurulla Azamov, Edward McDonald, Fedor Sukochev, and Dmitriy Zanin. The main result in this chapter is a Dixmier trace formula for the density of states on discrete metric spaces,  Theorem~\ref{T: Main}.

Given a metric space $(X,d)$ with a Borel measure, and a self-adjoint operator $H$ on $L_2(X)$, a Dixmier trace formula for the density of states is an equality of two Borel measures on $\R$ that can be associated with $H$. On the one hand, we can fix a weight $w: X \to \C$ such that $M_w \in \mathcal{L}_{1,\infty}$, which implies that $f(H)M_w \in \mathcal{L}_{1,\infty}$ for all $f \in C_c(\R)$, inducing a measure
\begin{equation}\label{eq:Measure1}
\Tr_\omega(f(H)M_w)  =\int_\R f \, d\nu_1, \quad f \in C_c(\R),
\end{equation}
where $\omega \in \ell_\infty^*$ is an extended limit. This measure $\nu_1$ is guaranteed to exist due to the Riesz--Markov--Kakutani theorem, and a priori it might depend on our choice of~$\omega$. On the other hand, we can look more geometrically towards the measure based on the limits
\begin{equation}\label{eq:Measure2}
\lim_{R \to \infty} \frac{1}{|B(x_0, R)|} \Tr(f(H)M_{\chi_{B(x_0, R)}}) = \int_\R f\, d\nu_2, \quad f \in C_c(\R),
\end{equation}
where $x_0 \in X$ is some chosen basepoint and $|B(x_0,R)|$ is the measure of the closed ball $B(x_0, R)$. For the definition of this measure to make sense we must require that the volume of the balls $B(x_0, R)$ is finite for all $R > 0$ and that $f(H)M_{\chi_{B(x_0,R)}} \in \mathcal{L}_1$, but even then the limits in equation~\eqref{eq:Measure2} are not guaranteed to exist. If the limits \textit{do} converge, we say that $H$ admits a density of states (DOS), and we commonly write $\nu_H$ for $\nu_2$.

In the paper~\cite{AMSZ}, it has been proven for the choice $X = \R^d$, $w(x) =(1+|x|^2)^{-\frac d2}$, and $H = -\Delta + M_V$ with $V \in L_\infty(\R^d)$ real-valued, that the measures~\eqref{eq:Measure1} and~\eqref{eq:Measure2} are equal up to a constant depending on $d$ (and independent of $\omega$). Hence, the measure~\eqref{eq:Measure1} in that setting is in fact \textit{independent} of $\omega$. It was explained in Section~\ref{S:DOS} that the case $V=0$ is a Fourier transform of Connes' integral formula (Theorem~\ref{T:ConnesIntegral}), the general case is substantially more work. 

In this chapter, we prove this Dixmier trace formula for countable discrete metric spaces. This comes with an additional condition on the metric space, a condition that ensures the volumes $|B(x_0, R)|$ grow in a regular and sub-exponential manner as $R \to \infty$. Namely, we will require that
\[
\frac{|B(x_0, r_{k+1})|}{|B(x_0, r_{k})|} \to 1, \quad k \to \infty,
\]
where $\{r_k\}_{k=0}^\infty$ is the set $\{d(y,x_0) \; : \; y \in X\}$ (\textit{without} multiplicities) ordered in increasing manner. We will call this Property $(C)$. The main theorem of this chapter is then the following, the proof of which is specific to the discrete case and is based on recent advances in operator theory and the notion of $V$-modulated operators hatched in the theory of singular traces, see~\cite{KaltonLord2013} or~\cite[Section~7.3]{LSZVol1} and Section~\ref{S:IntroPreliminaries}.
\begin{thm}\label{T: Main}
    Let $(X,d_X)$ be a countably infinite discrete metric space such that every ball contains at most finitely many points, and let $x_0 \in X$. Then the image of the map $d_X(\cdot, x_0): X \rightarrow \mathbb{R}_{\geq 0}$ is a collection of isolated points which can be ordered in an increasing way, denote this by $\{r_k\}_{k \in \mathbb{N}} \subseteq \mathbb{R}$. Suppose that \begin{equation}\tag{C}\label{E: Condition}
        \lim_{k\rightarrow \infty}\frac{|B(x_0, r_{k+1})|}{|B(x_0, r_k)|} = 1.
    \end{equation} Then for any positive, radially strictly decreasing function $w:X\to \C$ with $M_w \in \mathcal{L}_{1,\infty}$ we have that for every extended limit $\omega \in \ell_\infty^*$
    \begin{equation} \label{E: general DOS formula}
        \Tr_\omega(TM_w) = \Tr_\omega(M_w)\lim_{k\to \infty} \frac{\Tr(TM_{\chi_{B(x_0,r_k)}})}{|B(x_0,r_k)|},
    \end{equation}
    for all bounded linear operators $T$ on $\ell_2(X)$ for which the limit on the right-hand side exists.

    In particular, if $H$ is a self-adjoint, possibly unbounded, operator on $\ell_2(X)$ admitting a density of states measure $\nu_H$, then for all extended limits $\omega$,
\begin{equation}
    \mathrm{Tr}_\omega(f(H)M_w) = \int_{\mathbb{R}} f \, d\nu_H, \quad f \in C_c(\R).
\end{equation}
\end{thm}
Because $\{r_k\}_{k=0}^\infty$ denotes all possible distances $d_X(x_0, \cdot)$ in the discrete space $X$, the limit on the right-hand side of equation~\eqref{E: general DOS formula} exists if and only if the continuous limit 
\[
\lim_{R\to \infty} \frac{\Tr(TM_{\chi_{B(x_0,R)}})}{|B(x_0,R)|}
\] exists, and these limits are necessarily equal. This is therefore in line with the definition of the density of states~\eqref{eq:Measure2}.

In comparison with the Euclidean case, we do not have to make any assumptions on the operator $T$ beyond the existence of the existence of the limit~\eqref{E: general DOS formula}, whereas~\cite{AMSZ} is specific to the case $T = f(-\Delta + M_V)$. Furthermore,~\cite{AMSZ} has the requirement $d\geq 2$ on the dimension of the space $\R^d$, an analogue of which is not necessary in the discrete setting. Finally, Theorem~\ref{T: Main} holds for any choice of radially symmetric strictly decreasing function $w : X \to \C$ such that $M_w \in \mathcal{L}_{1,\infty}$ (the latter can be written more briefly as $w \in \ell_{1,\infty}(X)$). 
In particular, we can always pick
\begin{equation*}
    w(x) := \frac{1}{1+|B(x_0,d_X(x_0,x))|},\quad x \in X,
\end{equation*}
for which $\Tr_\omega(M_w) = 1$, see Corollary~\ref{C:MwTrace}. This provides that the constant of proportionality in~\eqref{E: general DOS formula}, $\Tr_\omega(M_w)$, for this choice of weight $w$ is independent of the choice of~$\omega$. Hence, for this choice of $w$, when $T = f(H)$ for an operator $H$ admitting a density of states, this ensures that the measure~\eqref{eq:Measure1} is independent of the choice of $\omega$ as in the Euclidean case.

By definition, independence of the measure~\eqref{eq:Measure1} with respect to $\omega$ is equivalent to the operators $f(H)M_w$ being Dixmier measurable. Let us conclude this introduction with an example illustrating that this is a strictly weaker requirement than $H$ admitting a density of states.

\begin{ex}
Let $X = \mathbb{N}$ with the usual metric $d(x,y) = |x-y|$, and take the base-point $x_0 = 0$. On $\ell_2(\N)$, a diagonal operator $H(e_n) = \lambda_n e_n$ admits a density of states if the limits
\[
\lim_{n\to \infty} \frac{1}{|B(x_0,n)|} \Tr(f(H)M_{\chi_{B(x_0,n)}}) =  \lim_{n\to\infty}\frac{1}{n+1}\sum_{k=0}^n f(\lambda_k), \quad f\in C_c(\R),
\]
exist.

With the choice $w(n):= \frac{1}{n+1}$, we have that $M_w \in \mathcal{L}_{1,\infty}$. Then, for $f\in C(\R)$, we have that $f(H)M_w \in \mathcal{L}_{1,\infty} $ is Dixmier measurable if and only if the limit
\[
\lim_{n\to\infty}\frac{1}{\log(2+n)} \sum_{k=0}^n \lambda(k,f(H)M_w)
\] 
exists~\cite[Theorem~9.1.2(c)]{LSZVol1}, where $\lambda(f(H)M_w)$ is an eigenvalue sequence of $f(H)M_w$ ordered in decreasing modulus. Since $f(H)M_w$ is $M_w$-modulated (see Section~\ref{S:IntroPreliminaries} or~\cite{KaltonLord2013} and~\cite[Section~7.3]{LSZVol1}), we have through Theorem~\ref{T: Modulated} that 
\[
 \sum_{k=0}^n \lambda(k,f(H)M_w) = \sum_{k=0}^n \frac{f(\lambda_k)}{k+1} + O(1), \quad n\to\infty.
\]
Hence, $f(H)M_w$ is Dixmier measurable if and only if the limit
\[
\lim_{n\to\infty} \frac{1}{\log(n+2)}\sum_{k=0}^n \frac{f(\lambda_k)}{k+1}, 
\]
exists.

With the choice
\begin{equation*}
\lambda_n := 
\begin{cases}
\begin{aligned}
0, \quad&n \in [2^{2m}+1, 2^{2m+1}], &&m=0,1,2,\dots,\\
1, \quad&n \in [2^{2m-1}+1, 2^{2m}], &&m=1,2,3,\dots,
\end{aligned}
\end{cases}
\end{equation*}
the diagonal operator $H$ does not admit a density of states, but $f(H)M_w$ is Dixmier measurable for all $f \in C(\R)$.
\end{ex}
\begin{proof}
    Observe that $\frac{1}{n+1}\sum_{k=0}^n \lambda_n$ does not converge as $n\to \infty$, since 
\begin{align*}
\frac{1}{2^{2m}+1} \sum_{k=0}^{2^{2m}} \lambda_n &= \frac{1}{2^{2m}+1}\sum_{k=0}^{2m}(-2)^k\\
&= \frac{1}{2^{2m}+1}\left(\frac{1-(-2)^{2m+1}}{1+2}\right)\\
&\to \frac{2}{3};\\
\frac{1}{2^{2m+1}+1} \sum_{k=0}^{2^{2m+1}} \lambda_n &= \frac{1}{2^{2m+1}+1}\sum_{k=0}^{2m}(-2)^k\\
&= \frac{1}{2^{2m+1}+1}\left(\frac{1-(-2)^{2m+1}}{1+2}\right)\\
&\to \frac{1}{3},
\end{align*}
where we used the closed-form formula for the sum of a geometric series. Hence $H$ cannot admit a density of states. 

Next, we show that $\frac{1}{\log(2+n)}\sum_{k=0}^n \frac{\lambda_k}{k+1}$ converges as $n\to \infty$. For convenience, we will prove the convergence of $\frac{1}{\log(n)}\sum_{k=1}^n \frac{\lambda_k}{k}$, taking $n>2$, and we will use that the harmonic number $H(n)=\sum_{k=1}^n \frac{1}{k}$ has the expansion $H(n) = \log(n)+\gamma+\frac{1}{2n}+O\left(\frac{1}{n^2}\right)$, where $\gamma$ is the Euler-Mascheroni constant~\cite[Section~1.2.7]{Knuth1997}. For $m\geq 1$,
\begin{align*}
\frac{1}{\log(2^{2m})}\sum_{k=1}^{2^{2m}}\frac{\lambda_k}{k} &= \frac{1}{2m\log(2)}\left(1+\sum_{k=1}^{m} H(2^{2k})-H(2^{2k-1})\right)\\
&= \frac{1}{2m\log(2)}\sum_{k=1}^{m} \log(2) (2k-(2k-1))+\\
& \quad + \frac{1}{2m\log(2)}+\frac{1}{2m\log(2)}\sum_{k=1}^{m} \frac{1}{4k} - \frac{1}{4k-2} + O\left(\frac{1}{k^2}\right)\\
&\rightarrow \frac{1}{2},
\end{align*}
where for the last term we used that for any sequence that converges to zero, its Ces\`aro mean also converges to zero. Likewise,\begin{align*}
\frac{1}{\log(2^{2m+1})}\sum_{k=1}^{2^{2m+1}}\frac{\lambda_k}{k} &= \frac{1}{(2m+1)\log(2)}\left( 1+ \sum_{k=1}^{m} H(2^{2k})-H(2^{2k-1})\right)\\
&= \frac{m}{2m+1}+ \frac{1}{(2m+1)\log(2)}+\\
& \quad +\frac{1}{(2m+1)\log(2)}\sum_{k=1}^{m} \frac{1}{4k} - \frac{1}{4k-2} + O\left(\frac{1}{k^2}\right)\\
&\rightarrow \frac{1}{2}.
\end{align*}
Note that $\frac{1}{\log(n)}\sum_{k=1}^n \frac{\lambda_k}{k}$ increases monotonically on $n\in [2^{2m-1}+1, 2^{2m}]$ and decreases monotonically on $n\in[2^{2m}+1,2^{2m+1}]$. Therefore the above shows that $\frac{1}{\log(n)}\sum_{k=1}^n \frac{\lambda_k}{k} \rightarrow \frac{1}{2}$, and hence also $\frac{1}{\log(n)}\sum_{k=1}^n \frac{1-\lambda_k}{k} \rightarrow \frac{1}{2}$. 
Hence, for $f\in C(\R)$,
\[
\frac{1}{\log(2+n)} \sum_{k=0}^n \frac{f(\lambda_k)}{k}\to \frac{f(0)+f(1)}{2}, \quad n \to \infty,
\] 
which thus shows that $f(H)M_w$ is Dixmier measurable.
\end{proof}

Do note that the above example is not a (discrete) Schr\"odinger type operator. It is unknown whether these two conditions still differ when restricting to such operators. This is also an open question for the Euclidean case,
and in~\cite{AMSZ} it was conjectured that these conditions are in fact equivalent for Schr\"odinger operators on Euclidean space.

The structure of the rest of the chapter is as follows. First we will discuss Property~\eqref{E: Condition} in Section~\ref{S: Metric Condition} and give examples of spaces where it is satisfied. In a way, we will return to the origin of the density of states, and show that crystals are included, the original source of the concept of a DOS (see~\cite{HoddesonBaym1987} and a very early use of the DOS in this context in 1929 by F. Bloch~\cite{Bloch1929}). Other concrete examples taken from physics will be considered too, demonstrating that condition~\eqref{E: Condition} appearing in this theorem is very natural.
Section~\ref{S: Proof} is the heart of the chapter where we prove the main result, Theorem~\ref{T: Main}. Finally, in 
Section~\ref{S: Translation} we apply the Dixmier trace formula for the DOS to provide a new proof of the equivariance under translations for the DOS on lattices.

\section{Metric condition}
\label{S: Metric Condition}
The main theorem of this chapter is applicable to countably infinite discrete metric spaces such that every ball contains at most finitely many points and that also satisfy property~\eqref{E: Condition} holds. Namely, we require that \begin{equation}\tag{C}
    \lim_{k\rightarrow \infty}\frac{|B(x_0, r_{k+1})|}{|B(x_0, r_k)|} = 1,
\end{equation} where $\{r_k\}_{k\in \mathbb{N}}$ is the increasing sequence created by ordering the set $\{d_X(x_0, y) : y \in X\}$ in increasing manner (which results in a sequence $r_k \rightarrow \infty$ since every ball in $X$ contains at most finitely many points).

First, observe that property~\eqref{E: Condition} is a condition on the so-called \textit{crystal ball sequence} $\{|B(x_0, r_k)|\}_{k \in \mathbb{N}}$ of the metric space $X$~\cite{ConwaySloane1997}, or alternatively after defining $S(x_{0}, r_k) := B(x_0, r_k) \setminus B(x_0, r_{k-1})$ it is a condition on the \textit{coordination sequence} $\{|S(x_0, r_k)|\}_{k \in \mathbb{N}}$~\cite{Brunner1979, ConwaySloane1997, OKeeffe1995}. 

To build some intuition, consider the following comment by J.E. Littlewood. Upon encountering the condition $\lim_{n\rightarrow \infty}\frac{\lambda_{n+1}}{\lambda_n} = 1$ he remarks~\cite{Littlewood1911}: ``[This condition is] satisfied when $\lambda_n$ is any function of less order than $e^{\varepsilon n}$ for all values of $\varepsilon$, which increases in a regular manner. When, however, $\lambda_n > e^{\varepsilon n}$, the theorem breaks down altogether.'' This observation is apt, indeed our restriction on the metric space $X$ is a strictly stronger assumption than sub-exponential growth of the sequence $\{|B(x_0, r_k)|\}_{k \in \mathbb{N}}$ (with respect to $k$), but exactly what kind of regular growth plus subexponential growth would imply condition~\eqref{E: Condition} is difficult to pin down.

There is an equivalent description of property~\eqref{E: Condition}, the proof of which can be found in a recent paper by Cipriani and Sauvageot~\cite[Proposition~2.9]{CiprianiSauvageot2021}. The proposition they prove is slightly different, but the given proof is immediately applicable to the following.

\begin{prop}\label{P: CiprianiAsymptote}
    Let $(X,d_X)$ be an infinite, discrete metric space such that each ball contains at most finitely many points, choose some point $x_0\in X$ and order $\{d_X(x_0, y) : y \in X\}$ in increasing manner to define the sequence $\{r_k\}_{k\in \mathbb{N}}$. Then \[\lim_{k\rightarrow \infty}\frac{|B(x_0, r_{k+1})|}{|B(x_0, r_k)|} = 1\] if and only if \[|B(x_0, r)| \sim \varphi(r)\] for some continuous function $\varphi: \mathbb{R}_+ \rightarrow \mathbb{R}_+$.
\begin{proof}
The proof is exactly the same as in~\cite[Proposition~2.9]{CiprianiSauvageot2021} after replacing $N_L(x)$ by $|B(x_0, r)|$ and $M_k$ by $|B(x_0, r_{k})|$.
\end{proof}
\end{prop}

\begin{rem}
    For the Cayley graph of the free group $\mathbb{F}_2$ we have $r_k = k$, $|B(x_0, k)| = 2^k$ and hence $|B(x_0, r)| = 2^{\lfloor r \rfloor} $, but \[\frac{|B(x_0, r)|}{2^r} = 2^{\lfloor r \rfloor-r}\] which does not converge as $r\rightarrow \infty$. This illustrates that $|B(x_0, r)| \sim \varphi(r)$ for some continuous function $\varphi$ is a stronger assumption that one might expect.
\end{rem}

In the same paper, another condition is given which is sufficient for property~\eqref{E: Condition} to be satisfied~\cite[Proposition~2.8]{CiprianiSauvageot2021}.

\begin{prop}\label{P: CondC}
    Let $X$ be a metric space as in Proposition~\ref{P: CiprianiAsymptote}. If $|B(x_0, r_k)| \sim f(k)$ (letting now $k\rightarrow \infty$ over the integers) for a function $f\in C^1(0,\infty)$ such that $\frac{f'(x)}{f(x)} \rightarrow 0$ as $x\rightarrow \infty$, then \[\lim_{k\rightarrow \infty}\frac{|B(x_0, r_{k+1})|}{|B(x_0, r_k)|} = 1.\] In particular, if $|B(x_0, r_k)|$ is a polynomial in $k$, then $\frac{|B(x_0, r_{k+1})|}{|B(x_0, r_k)|}\rightarrow 1$ as $k\rightarrow \infty$.
\begin{proof}
See~\cite[Proposition~2.8]{CiprianiSauvageot2021}.
\end{proof}
\end{prop}

To be used later on, we also provide the following lemma.

\begin{lem}\label{L: CondC}
    Let $X$ be a metric space as in Proposition~\ref{P: CiprianiAsymptote}. If there exist constants $C_1, C_2, d$ and $R$ such that for $r_k > R$ we have \begin{equation}\label{E: Coord seq cond1}
    C_1 k^d < |S(x_0, r_k)| < C_2 k^d,
    \end{equation} then $X$ has property~\eqref{E: Condition}.
\begin{proof}
    If $C_1 k^d < |S(x_0, r_k)| < C_2 k^d$, we can deduce that also $|B(x_0, r_k)| \geq \frac{C_1}{d+1} k^{d+1} + O(k^d)$ for $r_k > R$, and therefore \begin{equation}\label{E: Folner}
        \lim_{k\rightarrow \infty}\frac{|S(x_0, r_{k+1})|}{|B(x_0, r_k)|} = 0,
    \end{equation} which is equivalent with \[\lim_{k\rightarrow \infty}\frac{|B(x_0, r_{k+1})|}{|B(x_0, r_k)|} = 1. \qedhere\] 
\end{proof}
\end{lem}

As a final general comment, when writing condition~\eqref{E: Condition} in the manner of Equation \eqref{E: Folner}, it resembles a type of F\o lner condition. In particular, it is reminiscent of work by Adachi and Sunada on the DOS on amenable groups where a closely related property is the subject of interest, namely Property (P) in~\cite[Proposition~1.1]{AdachiSunada1993}, also compare with~\cite[Lemma~3.2]{AdachiSunada1993}.

\subsection{Solid matter}
Consider any kind of rigid matter whose atoms or molecules are arranged in Euclidean space in such a way that it can be described by a tiling of that space. To be precise, we mean a tiling generated by only a finite selection of different tiles, with each type of tile having a fixed arrangement of atoms within (at least 1). Any crystal can be described in this a way using only one tile by considering its underlying Bravais lattice~\cite[Chapter~4]{Ashcroft1976}, but the definition above includes quasicrystals~\cite{ShechtmanBlech1984, LevineSteinhardt1984}. For the approach of quasicrystals by tilings see for example~\cite{Hof1995, Jaric1989, Nelson1986}. Specifically,~\cite{Hof1993} establishes the existence of the integrated DOS, which is the existence of the function $\lambda \mapsto \nu_H(-\infty,\lambda)$, for every self-adjoint vertex-pattern-invariant operator on aperiodic self-similar tilings.

If we define the set $X$ of the metric space $(X,d_X)$ as the atoms or molecules of the material and impose the induced Euclidean metric, then we claim that this space has property~\eqref{E: Condition}.

\begin{prop}
    Let $X$ be a discrete subset of $\mathbb{R}^d$ with the inherited Euclidean metric, such that $X$ can be defined by a tiling as described above. Then $(X,d_X)$ has property~\eqref{E: Condition}.
\begin{proof}
Without loss of generality, assume that the diameters of the tiles are all less than $1$. Hence $r_{k+1} \in (r_k, r_k +2]$,
and therefore it suffices to show that \[\frac{|B(x_0, r_k +2) \setminus B(x_0, r_k)|}{|B(x_0, r_k)|} \xrightarrow{k \rightarrow \infty} 0.\]
Now, the number of vertices contained in $B(x_0, r_k +2) \setminus B(x_0, r_k)$ is bounded from above by some constant times $(r_k)^{d-1}$: if the smallest tile has volume $V$, and each tile contains at most $n$ atoms, and $\tilde{B}(x_0, r_k +2)$ denotes the ball in $\mathbb{R}^d$, then there can be at most $n \frac{|\tilde{B}(x_0, r_k +2) \setminus \tilde{B}(x_0, r_k)|}{V}$ vertices in $B(x_0, r_k +2) \setminus B(x_0, r_k)$, which is bounded by $C_1(r_k)^{d-1}$. 

If the volume of the biggest tile is $W$, the number of tiles that are \textit{fully} contained in $B(x_0, r_k)$ is similarly bounded from below by $\frac{|\tilde{B}(x_0, r_k - 1)|}{W}$ because we assumed that the diameter of the tiles is less than 1. Recall that we assumed that each tile contains at least one atom. Then the number of atoms in $B(x_0, r_k)$ can be bounded from below by $\frac{|\tilde{B}(x_0, r_k - 1)|}{W}$, which is of the form $C_2 (r_k)^d + O((r_k)^{d-1})$. Hence indeed \[0 \leq \frac{|B(x_0, r_k +2) \setminus B(x_0, r_k)|}{|B(x_0, r_k)|} \leq \frac{C_1(r_k)^{d-1}}{C_2 (r_k)^{d} + O((r_k)^{d-1})} \xrightarrow{k \rightarrow \infty} 0,\] and we see that this metric space satisfies condition~\eqref{E: Condition}.
\end{proof}
\end{prop}
\subsection{Crystals}\label{SS: Crystals}
In another approach, we can take the atoms of crystals as the vertices of a graph and define our discrete metric space $(X,d_X)$ as this graph with shortest-path metric (as a first observation, note that for any graph with the shortest-path metric, we have $r_k = k$). Common choices for such a construction are the contact graph~\cite{ConwaySloane1997} and the Voronoi graph~\cite[p.~33]{ConwaySloane1993}. See for example the very recent paper~\cite{PradodeOliveira2021} based on the model~\cite{deOliveiraPrado2005} which shows the existence of the DOS measure on $\mathbb{Z}$ for a suitable family of Dirac operators, the very recent~\cite{PapaefstathiouRobaina2021} or~\cite{Pastur1980} for an older result.

For crystals specifically, such a graph $\Gamma$ comes with a free $\mathbb{Z}^d$ action such that the quotient graph $\Gamma/\mathbb{Z}^d$ is finite. This is a simple observation by considering the underlying Bravais lattice of any crystal~\cite[Chapter~4]{Ashcroft1976}. A recent advancement by Y. Nakamura,  R. Sakamoto, T. Masea and J. Nakagawa~\cite{NakamuraSakamoto2021} concerns exactly such (even possibly directed) graphs $\Gamma$ with a free $\mathbb{Z}^d$ action such that $\Gamma/\mathbb{Z}^d$ is finite. Namely, these authors have proven that the coordination sequence $\{|S(x_0, k)|\}_{k \in \mathbb{N}}$ is then of quasi-polynomial type, by which they mean the following. A quasi-polynomial is defined as a function $p: \N \rightarrow \mathbb{Z}$ with $p(k) = c_m(k) k^m + c_{m-1} k^{m-1}+ \cdots + c_0(k)$ where $c_i(k)$ are all periodic with an integral period. Equivalently, it is a function such that for some integer $N$ \[p(k) = \begin{cases} p_0(k) & k = 0 \text{ mod } N\\
p_1(k) & k = 1 \text{ mod } N\\
&\vdots\\
p_{N-1}(k) & k = N-1 \text{ mod } N,\\
\end{cases}\]
where $p_0, \dots, p_{N-1}$ are polynomials. A function of quasi-polynomial type is then defined as a function $f: \N \rightarrow \mathbb{Z}$ such that there exists some quasi-polynomial $p$ and an integer $M$ such that $f(n) = p(n)$ for all $n \geq M$. See also~\cite[Section~4.4]{Stanley1986}.

\begin{prop}
Let $\Gamma$ be a graph with a free $\mathbb{Z}^d$ action such that the quotient graph $\Gamma/\mathbb{Z}^d$ is finite. Then $(\Gamma, d_\Gamma)$, where we take $d_\Gamma$ as the shortest-path metric, has property~\eqref{E: Condition}.
\end{prop}
\begin{proof} 
By~\cite[Theorem~1.1]{NakamuraSakamoto2021} we have that $|S(x_0, k)|$ is a quasi-polynomial, denote this quasi-polynomial by $p(k)$ and its constituent polynomials by $p_1, \dots, p_{N-1}$. Suppose that the polynomial $p_r$ is one with maximal degree, i.e. $\deg(p_i) \leq \deg(p_r) = t$ for all $i=0, \dots, N-1$. Then there exist some constants $C_1, C_2$ and $L$ such that $p(k) \leq C_1 k^t$ for all $k \geq L$, and also $p_r(k) = C_2 k^t + O(k^{t-1})$. Now take $k \in \mathbb{Z}_{\geq 0}$ arbitrary, and define $a \in \mathbb{Z}_{\geq 0}$ as the smallest positive integer such that $k = mN+r +a$ for some $m \in \mathbb{Z}$. It follows that $a \in \{0, \dots, N-1\}$, and hence note that $mN = k-r-a \geq k-2N$, i.e. $\frac{k}{N} -2\leq m \leq \frac{k}{N}$.
\begin{align*}
    \sum_{n=1}^k p(n) & \geq \sum_{n=0}^m p_r(nN + r)\\
    & = \sum_{n=0}^m C_2(nN + r)^t + O((m+1) (nN + r)^{t-1})\\
    & = C_2\sum_{n=0}^m (nN)^d + O(k^{t})\\
    & = C_3 k^{t+1} + O(k^t).
\end{align*}
Therefore we have that 
\[
0 \leq \frac{p(k+1)}{\sum_{n=1}^{k} p(n)} \leq \frac{C_1 (k+1)^t}{C_3 k^{t+1} + O(k^t)} \xrightarrow{k \rightarrow \infty} 0,
\] 
i.e. $\frac{|S(x_0, k+1)|}{|B(x_0, k)|}$ converges to zero as $k\rightarrow \infty$.
\end{proof}

\subsection{The integer lattice}
From the previous two subsections it follows that respectively $(\mathbb{Z}^d, \norm{\cdot}_2)$ and $(\mathbb{Z}^d, \norm{\cdot}_1)$ have property~\eqref{E: Condition}, where we define \[\norm{v}_p := \left(\sum_{i=1}^d |v_i|^p\right)^{1/p}\]for $1 \leq p < \infty$ and for $p=\infty$ \[\norm{v}_\infty:= \sup_{i=1, \dots, d} \abs{v_i}.\]
Even though the base space of these two metric spaces is the same, the difference between these is the domain undergoing the thermodynamic limit in the definition of the DOS, which can make a difference as demonstrated in~\cite{AzamovMcDonald2022}. As mentioned, the existence of the DOS in the case $\mathbb{Z}$ has been established for a suitable family of Dirac operators~\cite{deOliveiraPrado2005}. Another example is that, using $(\mathbb{Z}^d, \norm{\;\cdot\;}_\infty)$ as a model, the existence of a surface DOS was established for a quantum model with a surface~\cite{EnglischKirsch1988}.

In fact, $(\mathbb{Z}^d, \norm{\cdot}_p)$ has property~\eqref{E: Condition} for all $1 \leq p \leq \infty$. This is simply because in the metric space $(\mathbb{Z}^d, \norm{\cdot}_p)$ we have $\abs{B(0,r)} = V_p(d) r^d + O(r^{d-1})$ where $V_p(d)$ denotes the volume of the $\ell_p$ unit ball in $\mathbb{R}^d$, and Proposition~\ref{P: CiprianiAsymptote} then implies condition~\eqref{E: Condition}.

\subsection{Quasicrystals}
Analogously to Subsection~\ref{SS: Crystals} one can consider graphs constructed from quasicrystals, but these will be aperiodic by definition~\cite{Jaric1989, Nelson1986}. For some investigations of the DOS on quasicrystals, see~\cite{Hof1993, Hof1995}.

On a case-by-case basis, there are some aperiodic tilings for which condition~\eqref{E: Condition} can be expected to hold on their vertex graph. Firstly, a Penrose tiling. This proposition is entirely based on recent results by A. Shutov and A. Maleev~\cite{ShutovMaleev2015, ShutovMaleev2018}.

\begin{prop}
    Consider a 2D Penrose tiling in the construction of~\cite{ShutovMaleev2015}, which is a Penrose tiling with five-fold symmetry with respect to a chosen origin 0. Consider the graph induced by this tiling (i.e. with the same vertices and edges as the tiles) and take this graph with the shortest-path metric as the definition of the metric space $(X,d_X)$. Then this metric space has property~\eqref{E: Condition}.
\begin{proof}
For this particular tiling, Shutov and Maleev showed~\cite{ShutovMaleev2018} that \[|S(0, k)| = C(n)n + o(n)\] where $C(n)$, denoting $\tau = (1+\sqrt{5})/2$, takes a value between $10 \tau^{-2} \approx 3.8$ and $10 \tau^{-2} + (5/2)\tau^{-1} \approx 5.4$ depending on $n$. By Lemma~\ref{L: CondC}, we can then conclude that the vertex graph of this Penrose tiling satisfies condition~\eqref{E: Condition}. 
\end{proof}
\end{prop}

\begin{rem}
Similar asymptotic behaviour can be expected when taking a different base point, or different Penrose tilings, but this is still an open problem. Numerical data supports this conjecture for Penrose tilings~\cite{BaakeGrimm2006}, as well as that this metric condition will be satisfied by aperiodic tilings like the Ammann-Beenker tiling~\cite{BaakeGrimm2006}, a certain class of quasi-periodic self-similar tilings~\cite{ShutovMaleev2010} and two-dimensional quasiperiodic Ito–Ohtsuki tilings~\cite{ShutovMaleev2008}. To this author's knowledge, no quasiperiodic tiling is known that can serve as a counter-example to property~\eqref{E: Condition}.
\end{rem}

\subsection{Percolation}
A successful model for studying conduction properties of a crystal with impurities via the density of states is percolation (in this context also called quantum percolation)~\cite{AlexanderOrbach1982, ChayesChayes1986, KirschMuller2006, Veselic2005, YakuboNakayama1987}, see also~\cite[Section~13.2]{Grimmett1999} for a general approach. In broadest generality, percolation describes the study of a statistical procedure on a graph, which means adding, removing or otherwise manipulating edges or vertices based on some probabilistic method~\cite[Chapter~1]{Grimmett1999}.

Let us focus on so-called bond percolation on the graph $\mathbb{Z}^d$. Choosing some chance $0 \leq p \leq 1$, we declare each edge on the graph $\mathbb{Z}^d$ to be \textit{open} with probability $p$ (in an independent manner), and closed with probability $1-p$. It is a well-known fact that there exists a phase transition at a critical probability $p_c(d)$~\cite[Chapter~1]{Grimmett1999}. Namely, for $p>p_c(d)$ there exists almost surely a unique infinite cluster of vertices connected by open edges, while for $p<p_c(d)$ almost surely all clusters of vertices that are connected by open edges are finite.

In the super-critical region, meaning $p > p_c(d)$, one can wonder if our metric condition holds on the infinite cluster, meaning that we take $(X,d_X)$ as the infinite cluster with induced shortest-path metric. This turns out to be true, which follows quite directly from a result by R. Cerf and M. Th\'eret~\cite{CerfTheret2016}.

\begin{prop}
    Let $X$ be the (almost surely unique) infinite cluster on $\mathbb{Z}^d$ after super-critical bond percolation (with $p > p_c(d)$), and let $d_X$ be the shortest path metric on $X$. Then $(X,d_X)$ has property~\eqref{E: Condition}.
\begin{proof}
We first set up the exact situation as in~\cite{CerfTheret2016}. Denote the set of edges in our graph $\mathbb{Z}^d$ by $\mathbb{E}^d$, and consider a family of i.i.d. random variables $(t(e), e\in \mathbb{E}^d)$ taking values in $[0, \infty]$ (including $\infty$), with common distribution $F$. To be very precise, for each variable $t(e)$ we take $[0,\infty]$ as sample space, we define a $\sigma$-algebra by declaring $A \subseteq [0,\infty]$ measurable if $A \setminus \{\infty\}$ is Lebesgue measurable in $\mathbb{R}$, and $F: [0,\infty] \rightarrow [0,1]$ is a measurable function that defines the distribution of the variable $t(e)$.

These variables can be interpreted as the time it takes to travel along the corresponding edge. Then consider the random extended metric on $\mathbb{Z}^d$ by defining for $x, y\in \mathbb{Z}^d$ \[T(x, y) = \inf\{\sum_{e \in \gamma} t(e) : \gamma \text{ is a path from } x \text{ to } y \}.\] Note that if the distribution $F$ is such that $F(\{1\}) = p$, $F(\{\infty\}) = 1-p$, $T(x,y)$ is always either a positive integer or infinite and defines precisely the usual induced shortest-path metric on $\mathbb{Z}^d$ after bond percolating on $\mathbb{Z}^d$ with chance $p$. Also observe that generally, if $F([0,\infty)) > p_c(d)$, there exists almost surely a unique infinite connected cluster of vertices~\cite{CerfTheret2016}.

Now define $B(x_0, t) := \{y \in \mathbb{Z}^d : T(x_0 ,y)\leq t\}$. Then by~\cite[Theorem~5(ii)]{CerfTheret2016}, if $F([0,\infty)) > p_c(d)$, $F(\{0\}) < p_c(d)$ and $x_0$ is a vertex on the infinite cluster, as $t\rightarrow \infty$ we have almost surely \[\frac{|B(x_0, t)|}{t^d} \rightarrow C\] for some constant $C \in \mathbb{R}$. Returning to the distribution $F(\{1\}) = p$, $F(\{\infty\}) = 1-p$, we can then conclude that \[\frac{|B(x_0, k+1)|}{|B(x_0, k)|} \frac{k^d}{(k+1)^d} \rightarrow 1 \] as $k\rightarrow \infty$, and hence \[\frac{|B(x_0, k+1)|}{|B(x_0, k)|} \rightarrow 1. \qedhere\] \end{proof}
\end{prop}

\begin{rem}
Site percolation is similar to bond percolation, but as the name suggests we would assign each vertex to be open with probability $p$ and closed with probability $1-p$~\cite[Section~1.6]{Grimmett1999}. Arguably, this is a more physically suitable model of an alloy. The argument used above to prove the cited result for bond percolation can be used for site percolation as well, but this author is not aware of an explicit demonstration.
\end{rem}

\section{Dixmier trace formula for the DOS}\label{S: Proof}
In this section we prove Theorem~\ref{T: Main} based on a series of lemmas. 

To start off, we will demonstrate the existence of a function $w$ such that $M_w$ has positive Dixmier trace as required by Theorem~\ref{T: Main}.
For that purpose we first need a lemma generalising the fact that $\sum_{k=1}^n \frac{1}{k} \sim \log(n)$.
\begin{lem}\label{L:general_harmonic_series}
Let $\{a_k\}_{k\in \mathbb{N}}$ be a divergent, increasing sequence of strictly positive real numbers, such that $\lim_{k\to \infty} \frac{a_{k+1}}{a_k}=1$.  We use the convention that $a_{-1} = 0$. Then 
\[
\sum_{j=0}^k \frac{a_j - a_{j-1}}{a_j} \sim \log(a_k), \quad k \to \infty.
\]
\end{lem}
\begin{proof}
Note first that $\sum_{j=0}^k \frac{a_j - a_{j-1}}{a_j}$ equals 1 plus a lower Riemann sum of the integral $\int_{a_0}^{a_k} \frac{dx}{x}$. Hence, \[\sum_{j=0}^k \frac{a_j - a_{j-1}}{a_j} \leq \log(a_k) + 1 - \log(a_0).\]

Next, since $\lim_{k\to \infty} \frac{a_{k+1}}{a_k}=1$, for $\varepsilon >0$ we can choose $K$ such that for $k> K$ we have $\frac{a_k}{a_{k-1}} < 1+\varepsilon$, which can be rearranged to $\frac{1}{a_{k-1}} - \frac{1}{a_k} < \frac{\varepsilon}{a_k}$. Note that $\log(a_k) - \log(a_0) \leq \sum_{j=1}^{k} \frac{a_j - a_{j-1}}{a_{j-1}}$ since it is an upper Riemann sum of the integral $\int_{a_0}^{a_k} \frac{dx}{x}$. Therefore, for $k>K$,
\begin{align*}
    \log(a_k) - \log(a_0) &\leq \sum_{j=1   }^{k} \frac{a_j - a_{j-1}}{a_{j-1}}\\
    &= \sum_{j=1}^{k} \frac{a_j - a_{j-1}}{a_{j}} + \sum_{j=1}^{K} (a_j - a_{j-1}) \left(\frac{1}{a_{j-1}} - \frac{1}{a_j} \right) \\
    &\quad+ \sum_{j=K+1}^k(a_j - a_{j-1}) \left(\frac{1}{a_{j-1}} - \frac{1}{a_j} \right)\\
    &< (1+\varepsilon) \sum_{j=1}^{k} \frac{a_j - a_{j-1}}{a_{j}} + \sum_{j=1}^{K} (a_j - a_{j-1}) \left(\frac{1}{a_{j-1}} - \frac{1}{a_j} \right).
\end{align*}
Dividing by $\log(a_k)$ (assuming without loss of generality that $a_k \neq 1$) and taking the $\liminf$, we get 
\[
1\leq (1+\varepsilon) \liminf_{k\to \infty}  \frac{1}{\log(a_k)} \sum_{j=1}^{k} \frac{a_j - a_{j-1}}{a_{j}}.
\] 
Combined with our earlier estimate, we can conclude that 
\[
\sum_{j=0}^k \frac{a_j - a_{j-1}}{a_j} \sim \log(a_k), \quad k \to \infty. \qedhere
\]
\end{proof}

We will now make use of Lemma~\ref{L:AHMSZ}, which originally appeared in a slightly weaker form in the paper this chapter is based on~\cite[Lemma~4.8]{AHMSZ}. For convenience, let us restate it here. 

\begin{lem}\label{L: subsequence summation}
    Let $\phi: \mathbb{N} \to \R_{> 0}$ be an increasing function such that $\phi(n)\to \infty$ as $n\to \infty$, let $\{a_k\}_{k \in \mathbb{N}} \subseteq \mathbb{R}$ be a sequence such that $\big\{\frac{1}{\phi(n)}\sum_{k=0}^n |a_k| \big\}_{n=0}^\infty$ is bounded, and let $\{k_0, k_1, \dots \}$ be an infinite, increasing sequence of positive integers such that 
    \[
    \lim_{n\to \infty} \frac{\phi(k_{n+1})}{\phi(k_n)} = 1,
    \] 
    and
    \[
    \frac{1}{\phi(k_{n})} \sum_{k=k_{n-1}+1}^{k_{n}} |a_k| = o(1), \quad n\to\infty.
    \]
    Labeling $k_{i_n} := \min\{k_i : k_{i} \geq n \}$,     
    we have that
    \[
    \frac{1}{\phi(n)}\sum_{k=0}^n a_k = \frac{1}{\phi(k_{i_n})}\sum_{k=0}^{k_{i_n}} a_k + o(1), \quad n \to \infty.
    \]
\end{lem}

\begin{cor}\label{C:MwTrace}
Let $(X,d_X)$ be a countably infinite discrete metric space such that every ball contains at most finitely many points, satisfying property (C). Let $x_0 \in X$, and define $w(x) := (1+|B(x_0, d_X(x,x_0))|)^{-1}$. Then $w \in \ell_{1,\infty(X)}$, $M_w$ is Dixmier measurable, and $\mathrm{Tr}_\omega (M_w) = 1$ for all extended limits $\omega \in \ell_\infty^*$.
\end{cor}
\begin{proof}
Applying Lemma~\ref{L:general_harmonic_series} to the sequence $a_k = 1+|B(x_0, r_k)|$, we obtain
\[
\lim_{k\to \infty} \frac{1}{\log(2+|B(x_0, r_k)|)} \sum_{j=1}^k |\partial B(x_0, r_j)| (1+|B(x_0, r_j)|)^{-1} = 1.
\] 
By Lemma~\ref{L: subsequence summation}, it follows that
\[
\lim_{k\to \infty} \frac{1}{\log(2+k)} \sum_{j=1}^k \mu(k, M_w) = 1.\qedhere
\]
\end{proof}

We now present a modified Toeplitz lemma, which follows from much the same proof as in Shiryaev~\cite[Chapter IV, \S 3, Lemma 1]{Shiryaev1996}.

\begin{lem}[Toeplitz lemma]\label{toeplitz_lemma}
    Let $\{c_n\}_{n=0}^\infty$ be a sequence of non-negative numbers, and let $\{z_n\}_{n=0}^\infty$ be a sequence of complex numbers such that $z_n\rightarrow L \in \Cplx$. 
    If $d_n = \sum_{k=0}^n c_k$ diverges, then:
    \begin{equation*}
        \sum_{k=0}^n c_k z_k = L d_n + o(d_n).
    \end{equation*} 
    as $n\to \infty$.
\end{lem}
\begin{proof}
    Let $\varepsilon > 0$ and take $K>0$ sufficiently large such that if $k > K$ then $|z_k - L| < \varepsilon$. For any $n> K,$ 
    rewriting the left hand side of the equality above as
    \begin{align*}
        \sum_{k=0}^n c_k z_k &= \sum_{k=0}^n c_k L 
                           +\sum_{k=0}^K c_k(z_k - L) + \sum_{k=K+1}^n c_k(z_k - L),
    \end{align*}
    we see that 
    \begin{equation*}
        \left|\frac{1}{d_n}\sum_{k=0}^n c_k(z_k - L) \right| \leq \frac{1}{d_n}\sum_{k=0}^K c_k|z_k - L| + \varepsilon.
    \end{equation*}
    Since $d_n\to \infty$ as $n\to \infty,$ it follows that:
    \begin{equation*}
        \limsup_{n\to \infty}\left|\frac{1}{d_n} \sum_{k=0}^n c_k(z_k - L) \right| \leq \varepsilon.
    \end{equation*}
    Since $\varepsilon$ is arbitrary, we have:
    \begin{equation*}
        \sum_{k=0}^n c_k(z_k - L)  = o(d_n)
    \end{equation*}
    and this completes the proof.
\end{proof}

\begin{lem}\label{toeplitz_corollary}
    Let $\{x_k\}_{k=0}^\infty$ and $L \in \Cplx$ be such that as $n\to\infty$,
    \begin{equation*}
        \sum_{k=0}^n x_k = L n+o(n).
    \end{equation*}
    Let $\{a_n\}_{n=0}^\infty$ be a sequence of non-negative numbers such that:
    \begin{enumerate}
        \item $\{a_n\}_{n=0}^\infty$ is non-increasing;
        \item $\sup_{k\geq 0} ka_k < \infty$, that is, the sequence is in $\ell_{1,\infty}$;
        \item $b_n := \sum_{k=0}^n a_k$ diverges.
    \end{enumerate}
    Then   
    \begin{equation*}
        \sum_{k=0}^n a_kx_k = L b_n+o(b_n)
    \end{equation*}
    as $n\to \infty$.
\end{lem}
\begin{proof}
    Let $y_n = \sum_{k=0}^{n-1} x_k$ with $y_{0}=0$. Abel's summation formula gives
    \begin{align*}
        \sum_{k=0}^n a_kx_k &= \sum_{k=0}^n a_k(y_{k+1}-y_{k})\\
                            &= y_{n+1}a_n-\sum_{k=1}^n (a_k-a_{k-1})y_{k}\\
                            &= y_{n+1}a_n+\sum_{k=1}^n \frac{y_k}{k}\cdot k(a_{k-1}-a_{k}).
    \end{align*}
    By assumption (1) we have $a_{k-1} \geq a_k$ so the sequence $c_k := k(a_{k-1}-a_{k})$ is non-negative, and moreover as $n\to\infty$,
    \begin{equation*}
        \sum_{k=1}^n c_k = b_{n-1}-na_n\rightarrow \infty,
    \end{equation*}
    since by assumption (3) $b_n \to \infty$ and by assumption (2) $na_n$ is bounded.
    Therefore Lemma~\ref{toeplitz_lemma} applies to $c_k$ and $z_k:= \frac{y_k}{k}$, since by assumption $\lim_{k\to\infty} \frac{y_k}{k} = L$, and hence it follows that
    \begin{equation*}
        \sum_{k=1}^n \frac{y_k}{k}\cdot k(a_{k-1}-a_k) = L (b_{n-1}-na_n)+o(b_{n-1}-na_n)
    \end{equation*} 
    as $n\to\infty$.
    Thus,
    \begin{align*}
        \sum_{k=0}^n a_kx_k  & = \frac{y_{n+1}}{n}\cdot na_n + L(b_{n-1}-na_n)+o(b_{n-1}-na_n)  \\
           & = L b_{n-1}+o(b_{n-1}-na_n),
    \end{align*}
    where in the last equality we have absorbed the vanishing term $\brs{\frac{y_{n+1}}{n} - L} na_n$ into $o(b_{n-1}-na_n).$ 
    Since by assumption (2) the sequence $\{na_n\}_{n=1}^\infty$ is bounded, it follows that:
    \begin{equation*}
        \sum_{k=0}^n a_kx_k = L b_n+o(b_n).\qedhere
    \end{equation*}
\end{proof}

We could also write the result of Lemma~\ref{toeplitz_corollary} as:
\begin{equation*}
    \lim_{n\to \infty} \frac{\sum_{k=0}^n x_k}{\sum_{k=0}^n 1} = \lim_{n\to \infty} \frac{\sum_{k=0}^n a_kx_k}{\sum_{k=0}^n a_k}
\end{equation*}
whenever the left hand side exists and $\{a_k\}_{k=0}^\infty$ satisfies the stated assumptions.

\bigskip
Our next aim is to prove Lemma~\ref{expectation_values_lemma}, which is the crux of the proof of Theorem~\ref{T: Main}.
The proof is based on the notion of a $V$-modulated operator from~\cite{KaltonLord2013} or \cite[Section 7.3]{LSZVol1}, see also Section~\ref{S:IntroPreliminaries}. 
Recall that, since $0 < W \in \mathcal{L}_{1,\infty}$, we have that  $TW$ is $W$-modulated for all $T \in B(\Hc)$.
Applying Theorem~\ref{T: Modulated}, we have as $n\to\infty$,
\begin{equation}\label{expectation_values_sum}
    \sum_{k=0}^n \lambda(k,TW) = \sum_{k=0}^n \langle e_k,Te_k\rangle \mu(k,W) + O(1).
\end{equation}

\begin{lem}\label{expectation_values_lemma}
    Let $0 < W \in \mathcal{L}_{1,\infty}(\Hc) \setminus \mathcal{L}_1(\Hc),$ and let $\{e_k\}_{k=0}^\infty$ be an orthonormal basis such that $We_k= \mu(k,W)e_k$ for all $k\geq 0$. If $T$ is a bounded operator such that
    \begin{equation*}
        \sum_{k=0}^n \langle e_k, Te_k\rangle = Ln+o(n),\quad n\to\infty
    \end{equation*}
    where $L \in \Cplx$
    then:
    \begin{equation} \label{F: the first assertion}
        \sum_{k=0}^n \lambda(k,TW) = L\left(\sum_{k=0}^n \mu(k,W)\right)+o\left(\sum_{k=0}^n \mu(k,W)\right).
    \end{equation}
    Moreover,
    \begin{equation*}
        \Tr_\omega(TW) = \Tr_\omega(W)L
    \end{equation*}
    for all extended limits $\omega$.
\end{lem}
\begin{proof}
By the assumption $0 < W \in \mathcal{L}_{1,\infty} \setminus \mathcal{L}_1,$
    Lemma~\ref{toeplitz_corollary} applies with $a_k = \mu(k,W)$ and $x_k = \langle e_k,Te_k\rangle$, so
    \begin{equation*}
        \sum_{k=0}^n \langle e_k,Te_k\rangle\mu(k,W) = L\left(\sum_{k=0}^n \mu(k,W)\right)+o\left(\sum_{k=0}^n \mu(k,W)\right),\, n\to\infty.
    \end{equation*}
    Thus \eqref{expectation_values_sum} yields the first equality \eqref{F: the first assertion}.

    To obtain the result concerning Dixmier traces, we divide both sides of \eqref{F: the first assertion} by $\log(n+2)$ to get:
    \begin{equation*}
        \frac{1}{\log(n+2)}\sum_{k=0}^n \lambda(k,TW)  = L\left(\frac{1}{\log(n+2)}\sum_{k=0}^n \mu(k,W)\right)+o(1).
    \end{equation*}
    Thus if $\omega$ is an extended limit,
    \begin{equation*}
        \omega\left(\left\{\frac{1}{\log(n+2)}\sum_{k=0}^n \lambda(k,TW)\right\}_{n=0}^\infty\right) = L\Tr_\omega(W).
    \end{equation*}
    The left hand side is exactly the Dixmier trace $\Tr_\omega(TW)$.
\end{proof}

\begin{rem}
The result of Lemma~\ref{expectation_values_lemma} can be written in a different way. 
We could say that:
\begin{equation}\label{feq}
    \Tr_\omega(TW) = \Tr_\omega(W) \lim_{n\to \infty} \frac{1}{n+1}\sum_{k=0}^n \langle e_k,Te_k\rangle
\end{equation}
whenever the right hand side exists.
\end{rem}

Now Lemma~\ref{L: subsequence summation} combined with \eqref{feq} immediately implies the following proposition.

\begin{prop}\label{general_DOS_proposition}
    If $T$ is a bounded linear operator and $0 < W \in \mathcal{L}_{1,\infty}$ such that \[\lim_{n\to\infty} \frac{|\{ k\geq 0 : \mu(k,W) \geq \varepsilon_{n+1} \}|}{|\{ k\geq 0 : \mu(k,W) \geq \varepsilon_n \}|} = 1.\] for some decreasing, strictly positive sequence $\varepsilon_n \to 0$, then for all extended limits $\omega$ we have 
    \begin{equation}\label{abstract_DOS}
        \Tr_\omega(TW) = \Tr_\omega(W)\lim_{n\to \infty} \frac{\Tr(T\chi_{[\varepsilon_n,\infty)}(W))}{\Tr(\chi_{[\varepsilon_n,\infty)}(W))},
    \end{equation}
    whenever the limit on the right hand side exists.
\end{prop}
This proposition gives us the density of states formula for the discrete case. The idea is that in $\ell_2(X)$ we consider the basis $\{\delta_v\}_{v \in X}$ of indicator functions
of points $p \in X$, and $W$ is an operator of pointwise multiplication by a function which is \emph{radially decreasing} with respect to some point $x_0 \in X$, so that $\chi_{[\varepsilon_n,\infty)}(W)$ is the indicator function of a ball, and the limit $n\to \infty$ is equivalent to taking a limit over balls with radius going to infinity. This is where property~\eqref{E: Condition} as discussed in Section~\ref{S: Metric Condition} comes in, as this ensures that the premise of Proposition \ref{general_DOS_proposition} is satisfied.

\begin{proof}[Proof of Theorem~\ref{T: Main}]
    Let $W = M_w$, where $w:X\to \C$ is a positive, radially strictly decreasing function with $M_w \in \mathcal{L}_{1,\infty}$. If $W$ is trace-class, the theorem is trivial as both sides are zero. So, assume that $0 < W \in \mathcal{L}_{1,\infty}(\Hc) \setminus \mathcal{L}_1(\Hc)$. Then $\{\delta_v\}_{v \in X}$ is a basis of normalised eigenvectors for $W$, with eigenvalue corresponding to $\delta_v$ equal to $w(v)$. By assumption $w(v)$
    is a strictly decreasing function of $d_X(x_0,v)$, and hence the sets
       $\{v\; \colon \;w(v)\geq \delta\}$
    are balls. In fact, if we define $\varepsilon_n = \tilde{\mu}(n, W)$, where $\tilde{\mu}(n, W)$ is the $n$th largest singular value of $W$ counted \textit{without} multiplicities, we have that \[|\{ k\geq 0 : \mu(k,W) \geq \varepsilon_n \}|=|\{v\; \colon \;w(v)\geq \varepsilon_n\}| = |B(x_0, r_n)|,\] and we have assumed that \[\lim_{n\rightarrow \infty}\frac{|B(x_0, r_{k+1})|}{|B(x_0, r_k)|} = 1.\] Hence due to Lemma~\ref{expectation_values_lemma} and Lemma~\ref{L: subsequence summation}, we can then conclude that
    \begin{align*}
        \Tr_\omega(TW) &= \Tr_\omega(W)\lim_{k\to \infty} \frac{1}{|B(x_0,r_k)|}\sum_{v \in B(x_0,r_k)} \langle \delta_v,T\delta_v\rangle
    \end{align*}
    since we have assumed that the limit on the right exists.
    Observing that the sum on the right side is equal to $\Tr(TM_{\chi_{B(0,r_k)}})$
    gives \eqref{E: general DOS formula}. \end{proof}

\begin{rem}
If we again denote $\tilde{\mu}(n, M_w)$ as the $n$th largest singular value of $M_w$ counted \textit{without} multiplicities, define $m_n$ as the multiplicity corresponding to this singular value and also define $M_n := \sum_{k=1}^n m_k$, then we have effectively imposed \[\lim_{n\rightarrow \infty} \frac{M_{n+1}}{M_n}=1.\] Now compare this to the paper by Cipriani and Sauvageot~\cite{CiprianiSauvageot2021} also referenced in Section~\ref{S: Metric Condition}. They study densely defined, nonnegative, unbounded, self-adjoint operators with exactly such a property for the multiplicities of its eigenvalues, and our $(M_w)^{-1}$ would fit such a description. The link between the DOS as considered in this chapter and the spectral weight those authors define remains unclear as of yet.
\end{rem}

\section{Problems of measurability}
\label{S: measurability}
This section is mostly the work of Edward McDonald. Going back to the Toeplitz lemma~\ref{toeplitz_lemma}, it is possible to get better behaviour of the convergence $\frac{1}{b_n}\sum_{k=0}^n a_kx_k\to L$ 
by assuming faster convergence of $x_k\to L$. For example, we have the following lemma with an obvious proof. 
\begin{lem} \label{L: unused lemma}
    Let $a_n$ and $b_n$ satisfy the same assumptions as Lemma~\ref{toeplitz_lemma}. 
    Let $x_n \in \mathbb{C}, \ n=0,1,\ldots.$ 
    If $x_k\rightarrow L$ sufficiently fast such that $\{a_k|x_k - L|\}_{k=0}^\infty \in \ell_1$, then:
    \begin{equation*}
        \sum_{k=0}^n a_kx_k = L b_n + O(1).
    \end{equation*}
\end{lem}

\begin{lem} \label{second unused lemma}
    Let $a_n$ and $b_n$ satisfy the same assumptions as Lemma~\ref{toeplitz_corollary}. 
    Let $x_n \in \mathbb{C}, \ n=0,1,\ldots$ and $\sigma_n = \frac 1{n+1} \sum_{k=0}^n x_k.$
    If $\sigma_n \rightarrow L$ sufficiently fast such that $\{a_k| \sigma_k - L|\}_{k=0}^\infty \in \ell_1$, then:
    \begin{equation*}
        \sum_{k=0}^n a_kx_k = L b_n + O(1).
    \end{equation*}
\end{lem}
\begin{proof}
    Let $y_n = \sum_{k=0}^{n-1} x_k$ with $y_{0}=0$. From the proof of Lemma~\ref{toeplitz_corollary} we have: 
    \begin{align*}
        \sum_{k=0}^n a_kx_k 
                            = y_{n+1}a_n+\sum_{k=1}^n \sigma_k \cdot k(a_{k-1}-a_{k}),
    \end{align*}
    the sequence $a_k' := k(a_{k-1}-a_{k})$ is non-negative, and as $n\to\infty$,
    \begin{equation*}
        \sum_{k=1}^n a_k' = b_{n-1}-na_n\rightarrow \infty,
    \end{equation*}
    Therefore, since by assumption $\lim_{k\to\infty} \sigma_k = L$
and $\{a_k| \sigma_k - L|\}_{k=0}^\infty \in \ell_1,$
     Lemma~\ref{L: unused lemma} applies with $a'$ in place of $a$ and $\sigma_k$ in place of $x_k,$ which gives 
    \begin{equation*}
        \sum_{k=1}^n \sigma_k\cdot k(a_{k-1}-a_k) = L (b_{n-1}-na_n) + O(1)
    \end{equation*} 
    as $n\to\infty$.
    Thus,
    \begin{align*}
        \sum_{k=0}^n a_kx_k  & = \frac{y_{n+1}}{n}\cdot na_n + L(b_{n-1}-na_n) + O(1)  \\
           & = L b_{n-1} + O(1) = L b_{n} + O(1).
    \qedhere\end{align*}
\end{proof}

\medskip
For an operator $A \in \mathcal{L}_{1,\infty},$ there are various criteria relating the behaviour of the sequence $\sum_{k=0}^n \lambda(k,A)$ to the measurability of $A$. For example,~\cite[Theorem 5.1.5]{LSZVol1} implies that
\begin{equation} \label{F: what a beauty}
    \sum_{k=0}^n \lambda(k,A) - c\log(2+n)=O(1),\quad n\to \infty
\end{equation}
if and only if $\varphi(A) = c$ for all normalised traces $\varphi$ on $\mathcal{L}_{1,\infty}$ (c.f.~\cite[Theorem 9.1.2]{LSZVol1}). For different classes of traces, different criteria are available, see~\cite{SemenovSukochev2015, UsachevThesis}.

\begin{thm}
    Let $(X,d_X)$, $T$ and $w$ satisfy the assumptions of Theorem~\ref{T: Main}.
    Let~$e_n$ be an orthonormal basis such that $M_w e_k = \mu(k,w) e_k$, where $\{\mu(k,w)\}_{k=0}^\infty$ is the decreasing rearrangement of $w$. 
     If $$\frac{1}{n+1}\sum_{k=0}^n \langle e_k,Te_k\rangle \rightarrow L \in \Cplx$$
    so fast that
    \begin{equation*}
        \sum_{n=0}^\infty \mu(n,w)\left| \frac{1}{n+1}\sum_{k=0}^n \langle e_k,Te_k\rangle - L\right| < \infty
    \end{equation*}
    and there exists $C > 0$ such that:
    \begin{equation*}
        \sum_{k=0}^n \mu(k,w) = C\log(2+n)+O(1),
    \end{equation*}
    then $TM_w$ is measurable in the sense of Connes, specifically
    $$
       \varphi(TM_w) = \Tr_\omega(M_w) \lim_{n \to \infty} \frac 1{n+1} \sum_{k=0}^n \langle e_k,Te_k\rangle
    $$
    for all traces $\varphi$ on $\mathcal{L}_{1,\infty}.$
\end{thm}
\begin{proof} 
Since $TM_w$ is $M_w$-modulated (see Definition~\ref{D: V-modulated} and Theorem~\ref{T: Modulated}), from \eqref{expectation_values_sum} we have
$$
   \sum_{k=0}^n \lambda(k,TM_w) = \sum_{k=0}^n \Scal{e_k,Te_k} \mu(k,w) + O(1).
$$
Hence, in view of \eqref{F: what a beauty} to prove the claim it suffices to show that 
$$
   \sum_{k=0}^n \Scal{e_k,Te_k} \mu(k,w) = CL \log(2+n) + O(1).
$$
By the second condition this is equivalent to 
$$
   \sum_{k=0}^n \Scal{e_k,Te_k} \mu(k,w) = L \sum_{k=0}^n \mu(k,w) + O(1),
$$
so it suffices to prove this. This follows from Lemma~\ref{second unused lemma} applied to $x_k = \Scal{e_k,Te_k}$ and $a_k = \mu(k,w).$ 
\end{proof}

We could also replace the assumption with the slightly stronger assertion:
\begin{equation*}
    \left\{\frac{1}{n+1}\sum_{k=0}^n \langle e_k,Te_k\rangle-L\right\}_{k=0}^\infty \in \Lambda_{\log},
\end{equation*}
where $\Lambda_{\log}$ is the space of sequences $x$ such that
\begin{equation*}
    \sum_{k=0}^\infty \frac{\mu(k,x)}{k+1} < \infty.
\end{equation*}

\section{Equivariance of the DOS under translations of the Hamiltonian}
\label{S: Translation}
In this section, mostly the work of Edward McDonald, we will provide a straightforward application of the Dixmier formula for the DOS put forward in Theorem~\ref{T: Main}. Namely, we provide a new and original proof of the equivariance of the DOS on lattice graphs $X$ under translations of the Hamiltionian. By this we mean that if $U$ denotes a shift operator on $\ell_2(X)$ and the DOS exists for both a Hamiltonian $H$ and the shifted $UHU^*$, then the DOS is equal for $H$ and $UHU^*$. This fact is not hard to prove without Theorem~\ref{T: Main}, but it does provide a different perspective on the claim.

Afterwards, we will discuss some consequences of this translation equivariance.
    
\subsection{Translation equivariance on lattice graphs}
    We will consider the example where $X = \mathbb{Z}^d$, embedded as a subset of $\Rl^d$ with the Euclidean metric. Precisely the same reasoning
    applies to other lattices $X\subset \Rl^d$ (recall that a lattice in $\Rl^d$ is the $\mathbb{Z}$-linear span of $d$ linearly independent
    vectors).
    
    We take 
    \begin{equation} \label{F: w(x) = (x)}
       w(x) := (1+\norm{x}_2)^{-d}
    \end{equation}
    where $x \in \mathbb{Z}^d$ like before.
    
    \begin{lem}\label{translation_difference_L1}
        For all $n\in \mathbb{Z}^d$, we have:
        \begin{equation*}
            \{w(x)-w(x-n)\}_{x \in \mathbb{Z}^d} \in \ell_{\frac{d}{d+1},\infty}.
        \end{equation*}
    \end{lem}
    \begin{proof}
        The difference $w(x)-w(x-n)$ is provided by the formula:
        \begin{equation*}
            w(x)-w(x-n) = \int_0^1 \langle \nabla w(x-(1-\theta)n),n\rangle\,d\theta.
        \end{equation*}
        The gradient $\nabla w$ of $w$ is easily computed as:
        \begin{equation*}
            \frac{\partial}{\partial x_j}w(x) = -\frac{dx_j}{\norm{x}_2 (1+\norm{x}_2)^{d+1}}.
        \end{equation*}
        Thus,
        \begin{equation*}
            \norm{\nabla w(x)}_2 = \frac{d}{(1+\norm{x}_2)^{d+1}}.
        \end{equation*}
        Therefore for $x$ with $\norm{x}_2 > \norm{n}$ we have
        \begin{align*}
            \abs{w(x)-w(x-n)} &\leq d\norm{n}_2 \max_{0\leq \theta\leq 1}(1+\norm{x-(1-\theta) n)}_2)^{-d-1} \\
            &\leq  \frac{d \norm{n}_2}{(1+\norm{x}_2 - \norm{n}_2)^{d+1}},
        \end{align*}
       hence $\abs{w(x) - w(x-n)}_{x\in \mathbb{Z}^d} $ is an element of $\ell_{\frac{d}{d+1}, \infty}$.
    \end{proof} 
    
    \begin{thm}
        For $n \in \mathbb{Z}^d$, let $U_n$ denote the operator on $\ell_2(\mathbb{Z}^d)$ of translation by~$n$. Assume that $H = H_0+M_V$ is a Hamiltonian operator
        such that the density of states exists for both $H$ and $U_nHU_n^*.$ Then both measures are equal.
    \end{thm}
    \begin{proof}
        For any $f \in C_c(\Rl)$, we have
            $f(U_nHU_n^*) = U_nf(H)U_n^*.$ Combining this with 
 the tracial property of the Dixmier trace gives 
        \begin{equation*}
            \Tr_\omega(f(U_nHU_n^*)M_w) = \Tr_\omega(f(H)U_n^*M_wU_n).
        \end{equation*}
        By Lemma~\ref{translation_difference_L1}, we have
            $U_n^*M_wU_n-M_w \in \mathcal{L}_{\frac{d}{d+1},\infty}\subset \mathcal{L}_1.$
Since the Dixmier trace vanishes on trace-class operators, it follows that 
            $\Tr_\omega(f(U_nHU_n^*)M_w) = \Tr_\omega(f(H)M_w).$
 Combining this with the Dixmier trace formula for the density of states, Theorem~\ref{T: Main}, completes the proof. 
    \end{proof}
    
    Via identical reasoning, we also have the following abstract assertion:
    \begin{thm}\label{general_translation_invariance}
        Let $(X,d_X)$ be an infinite discrete metric space and $w$ be a function such that these satisfy the assumptions of Theorem~\ref{T: Main}, let $\gamma$ be an isometry of $X$, and let $U_{\gamma}\delta_p = \delta_{\gamma(p)}$ be the corresponding unitary operator on $\ell_2(X)$. Assume that
            $w-w\circ\gamma \in (\ell_{1,\infty})_0(X).$
        Then 
        \begin{equation*}
            \lim_{R\to \infty} \frac{1}{|B(x_0,R)|}\sum_{d_X(p,x_0)\leq R} \langle \delta_p,T\delta_p\rangle = \lim_{R\to\infty} \frac{1}{|B(x_0,R)|}\sum_{d_X(p,x_0)\leq R} \langle \delta_p,U_{\gamma}^*TU_{\gamma}\delta_p\rangle,
        \end{equation*}        
        provided \emph{both} limits exist.
    \end{thm}
    
\subsection{Ergodic operators}
The following results are direct consequences of the translation equivariance of the DOS measure and therefore could be derived without help of Theorem~\ref{T: Main}. However, the Dixmier trace formula provides a different approach.

Let $(\Omega,\Sigma,\mathbb{P})$ be a probability space, and let $\Gamma$ be a discrete amenable group of isometries of the metric space $X$
 from Theorem~\ref{T: Main}. We assume that there
    is a representation of $\Gamma$ as automorphisms of $\Omega$:
    \begin{equation*}
        \gamma\in \Gamma\mapsto \alpha_\gamma\in \mathrm{Aut}(\Omega).
    \end{equation*}
    It is assumed that the action $\alpha$ is ergodic, in the sense that:
    \begin{enumerate}
        \item For every $\gamma\in \Gamma$, the automorphism $\alpha_\gamma$ is measure preserving;
        \item If $E\subseteq \Omega$ is invariant under every $\alpha_\gamma$, then $\mathbb{P}(E)=0$ or $\mathbb{P}(\Omega\setminus E) = 0$.
    \end{enumerate}
    
    We will use a generalisation of Birkhoff's ergodic theorem, obtained by Lindenstrauss~\cite[Theorem 1.3]{Lindenstrauss2001}. This uses the concept of a F\o lner sequence, we give the definition as it is used for discrete groups.

\begin{defn}
Let $\Gamma$ be a discrete group, and let $\{F_n\}_{n=0}^\infty$ be a sequence of subsets of $\Gamma$. 
    \begin{enumerate}
        \item If for every finite subset $K \subseteq \Gamma$ and every $\delta>0$, there exists $N$ sufficiently large such that if $n>N$, we have for all $k\in K$
        $$|F_n \, \Delta \, kF_n| \leq \delta|F_n|,$$ then $\{F_n\}_{n=0}^\infty$ is called a F\o lner sequence.
        \item If $\{F_n\}_{n=0}^\infty$ satisfies (1) and furthermore for some $C \geq 1$ and for every $n\geq 0$, we have:
        $$\left|\bigcup_{k\leq n} F_k^{-1}F_{n+1}\right| \leq C|F_{n+1}|,$$ then $\{F_n\}_{n=0}^\infty$ is called a tempered F\o lner sequence. 
    \end{enumerate}
\end{defn}

The existence of a F\o lner sequence in this sense is equivalent with the condition of $\Gamma$ being discrete and amenable~\cite[p.~23]{Lubotzky1994}. Also note that any F\o lner sequence has a tempered subsequence~\cite[Proposition~1.4]{Lindenstrauss2001}.

    Lindenstrauss' pointwise ergodic theorem~\cite[Theorem 1.3]{Lindenstrauss2001} implies that if $\{F_n\}_{n=0}^\infty$
    is a tempered F\o lner sequence, then for all $f \in L_1(\Omega)$ we have:
    \begin{equation} \label{Lindenstrauss' Thm}
        \lim_{n\to\infty} \frac{1}{|F_n|} \sum_{\gamma\in F_n} f(\alpha_\gamma\omega) = \mathbb{E}(f).
    \end{equation}
    for almost every $\omega \in \Omega.$ 
    
    For $\gamma\in \Gamma$, let $U_{\gamma}$ denote the induced unitary operator acting on $\ell_2(X)$ by:
    \begin{equation*}
        U_\gamma\delta_p := \delta_{\gamma(p)},\quad p\in X,\,\gamma\in \Gamma.
    \end{equation*}
    
    We will consider strongly measurable random operators $T \in L_1(\Omega,B(\ell_2(X))$ which are compatible with $\alpha$ in the sense that:
    \begin{equation}\label{group_compatibility}
        U_{\gamma}T(\omega)U_\gamma^* = T(\alpha_\gamma \omega),\quad \gamma\in \Gamma
    \end{equation}
    for almost all $\omega \in \Omega.$ 
        
    \begin{prop}
        Let $T \in L_1(\Omega,B(\ell_2(X))$ be a random operator satisfying \eqref{group_compatibility} with respect to a group of isometries $\Gamma$ of $X,$
        which admits a tempered F\o lner sequence $\set{F_n}_{n=0}^\infty$ of finite subsets,  
         and with respect to an ergodic action $\alpha$
        of $\Gamma$ on $\Omega$. If there exists a function $w \colon X \to \mathbb{R}_+$ satisfying the assumptions of Theorem~\ref{T: Main} such that
        \begin{equation*}
            w\circ \gamma-w \in (\ell_{1,\infty})_0(X)
        \end{equation*} 
        for every $\gamma\in \Gamma$ then the density of states of $T(\xi)$ is non-random, in the sense that if the limit:
        \begin{equation*}
            \lim_{R\to\infty} \frac{\Tr(T(\xi) M_{\chi_{B(x_0,R)}})}{|B(x_0,R)|}
        \end{equation*}
        exists for almost every $\xi$, then the limit is almost surely constant in $\xi$.        
    \end{prop}
    \begin{proof}
        This is an application of the Lindenstrauss version of Birkhoff's ergodic theorem. 
        The assumption on $w$ and Theorem~\ref{general_translation_invariance} imply that:
        \begin{equation} \label{alpha invariant}
            \Tr_\omega(T(\xi) M_w) = \Tr_\omega(T(\alpha_\gamma\xi)M_w),\quad \gamma\in \Gamma.
        \end{equation}  
        Therefore for every $n\geq 0$ we have:
        \begin{equation*}
            \Tr_\omega(T(\xi) M_w) = \frac{1}{|F_n|}\sum_{\gamma\in F_n} \Tr_\omega(T(\alpha_\gamma\xi)M_w).
        \end{equation*}
        Note that:
        \begin{equation*}
            |\Tr_\omega(T(\xi)M_w)| \leq \|T(\xi)\|_\infty\|w\|_{1,\infty}.
        \end{equation*}
        Hence the function $\xi\mapsto \Tr_\omega(T(\xi)M_w)$ is integrable, due to our assumption that $T\in L_1(\Omega,B(\ell_2(X)))$, and the measurability
        of $\xi\mapsto \Tr_\omega(T(\xi)M_w)$ follows from the strong measurability of $\xi\mapsto T(\xi)$ and the norm continuity of $T\mapsto \Tr_\omega(TM_w)$.
        Hence, Lindenstrauss' ergodic theorem \eqref{Lindenstrauss' Thm}
        applies to this function, and hence for almost every $\xi\in \Omega$ we have:
        \begin{equation*}
            \lim_{n\to\infty} \frac{1}{|F_n|} \sum_{\gamma \in F_n} \Tr_\omega(T(\alpha_\gamma\xi)M_w) = \mathbb{E}(\Tr_\omega(TM_w)).
        \end{equation*}
        The right hand side has no dependence on $\xi\in \Omega$, and hence the limit is almost surely constant in $\xi$.
  Due to \eqref{alpha invariant}, this implies that $\Tr_\omega(T(\xi) M_w)$ is almost surely constant in $\xi.$ Alluding to Theorem~\ref{T: Main}, 
  we conclude that the density of states of $T(\xi)$ is almost surely constant in $\xi.$
    \end{proof}
    
    In an alternative direction of inquiry, the condition \eqref{group_compatibility} can be used in some circumstances to imply the existence of the density of states.
    For simplicity, we state the following condition for $X = \mathbb{Z}^d$.
    \begin{thm} \label{T: on existence of DOS}
        Let $T\in L_1(\Omega,B(\ell_2(\mathbb{Z}^d)))$ be a linear operator which satisfies \eqref{group_compatibility} with respect to the action
        of $\mathbb{Z}^d$ on itself by translations and an ergodic action $\alpha$ of~$\mathbb{Z}^d$ on $\Omega.$ Then for almost every $\xi \in \Omega$ there exists the limit:
        \begin{equation*}
            \lim_{R\to\infty} \frac{1}{|B(0,R)|}\sum_{|n|\leq R} \langle \delta_n,T(\xi)\delta_n\rangle = \mathbb{E}(\langle \delta_0,T\delta_0\rangle).
        \end{equation*}
    \end{thm}
    \begin{proof}
        We have that $U_n\delta_0 = \delta_n$, and therefore:
        \begin{equation*}
            \langle \delta_n,T(\xi)\delta_n\rangle = \langle \delta_0,U_n^*T(\xi)U_n\delta_0\rangle = \langle \delta_0,T(\alpha_{-n}\xi)\delta_0\rangle.
        \end{equation*}
        It follows that:
        \begin{equation*}
            \frac{1}{|B(0,R)|}\sum_{|n|\leq R} \langle \delta_n,T(\xi)\delta_n\rangle = \frac{1}{|B(0,R)|}\sum_{|n|\leq R}\langle \delta_0,T(\alpha_n\xi)\delta_0\rangle.
        \end{equation*}
        By our assumption on $T$, the function $\xi \mapsto \langle \delta_0,T(\xi)\delta_0\rangle$ belongs to $L_1(\Omega)$.
        Note that the sequence $F_N := B(0,N)$ is a tempered F\o lner sequence in $\mathbb{Z}^d$, and hence Lindenstrauss' ergodic theorem \eqref{Lindenstrauss' Thm}  implies that
        for almost every $\xi\in \Omega$ there exists the limit
        \begin{equation*}
            \lim_{N\to\infty} \frac{1}{|B(0,N)|} \sum_{n\in B(0,N)} \langle \delta_0,T(\alpha_n \xi)\delta_0\rangle = \mathbb{E}(\langle \delta_0,T(\xi)\delta_0\rangle).\qedhere
        \end{equation*}
    \end{proof}
    Note that the result also holds if the limit over balls $\{B(0,N)\}_{N\geq 0}$ is replaced with any other tempered F\o lner sequence, such as cubes $\{[-N,N]^d\}_{N\geq 0}$.
    The limit in every case is $\mathbb{E}(\langle \delta_0,T(\xi)\delta_0\rangle)$, and hence does not depend on the choice of sequence of sets.

    \begin{thm}
      Let $H(\xi) = H_0 + V_\xi(x)$ be a random operator on $\ell_2(\mathbb{Z}^d),$
       where $H_0$ is a $\mathbb{Z}^d$-translation invariant difference operator and $V_\xi,$ $\xi \in \Omega,$ 
      an iid random bounded function. Then there exists a set $\Omega_0 \subset \Omega$ of probability $1,$
      such that for any $f \in C_c(\mathbb{R})$ and for any $\xi \in \Omega_0$ 
 there exists the limit:
        \begin{equation} \label{F: what we need}
            \lim_{R\to\infty} \frac{1}{|B(0,R)|}\sum_{|n|\leq R} \langle \delta_n, f(H_\xi)\delta_n\rangle = \mathbb{E}(\langle \delta_0,f(H_\xi)\delta_0\rangle).
        \end{equation}
    \end{thm}
\begin{proof} Proof follows a standard argument, see e.g.~\cite[Chapter 3]{AizenmanWarzel2015}.
Let $\Sigma$ be a countable dense subset  of $C_c(\mathbb{R}).$ 
The random operator $H(\xi)$ is ergodic and obeys \eqref{group_compatibility} so Theorem~\ref{T: on existence of DOS} is applicable. 
By this theorem, for every $f \in \Sigma$ 
there exists a full set $\Omega_f \subset \Omega$ such that \eqref{F: what we need} holds for all $\xi \in \Omega_f.$ Define a full set $\Omega_0 = \bigcap_{f \in \Sigma} \Omega_f,$
so for every $f \in \Sigma$ and every $\xi \in \Omega_0$ the equality \eqref{F: what we need} holds. Choose any $g \in C_c(\mathbb{R})$ and let $f_1,f_2, \ldots \in \Sigma$ be such that 
$f_n \to g$ in uniform topology. Let $\varepsilon>0.$ Further we proceed by a standard $\varepsilon/3$-trick. Let $N \in \mathbb{N}$ be such that for all $n\geq N$ 
$\norm{f_n-g}_\infty < \varepsilon/3.$ For $f_N$ the equality \eqref{F: what we need} holds for any $\xi \in \Omega _0.$ Let $R_0 > 0$ be such that for all $R > R_0$ 
and all $\xi \in \Omega _0$
$$
    \abs{ \frac{1}{|B(0,R)|}\sum_{|n|\leq R} \langle \delta_n, f_N(H_\xi)\delta_n\rangle - \mathbb{E}(\langle \delta_0,f_N(H_\xi)\delta_0\rangle) }  < \varepsilon/3.
$$
Then for all $R> R_0$ and $\xi \in \Omega_0$ we have 
\begin{align*}
    \Big| \frac{1}{|B(0,R)|}  &  \sum_{|n|\leq R} \langle \delta_n, g(H_\xi)\delta_n\rangle - \mathbb{E}(\langle \delta_0, g(H_\xi)\delta_0\rangle) \Big|  \\
      & \leq  \abs{ \frac{1}{|B(0,R)|}\sum_{|n|\leq R} \langle \delta_n, [g(H_\xi) - f_N(\xi)]   \delta_n\rangle  } \\ 
     & \qquad  +  \Big| \frac{1}{|B(0,R)|}  \sum_{|n|\leq R} \langle \delta_n, f_N(H_\xi)\delta_n\rangle - \mathbb{E}(\langle \delta_0, f_N(H_\xi)\delta_0\rangle) \Big|  \\
     & \qquad  + \Big|   \mathbb{E}(\langle \delta_0, [ f_N(H_\xi) - g(H_\xi) ]  \delta_0\rangle)   \Big|  \\
     & < \varepsilon,
\end{align*}
where the last inequality follows from the triangle, Schwartz and $\norm{f(H) - g(H)} \leq \norm{f-g}_\infty < \varepsilon/3$ inequalities. 
\end{proof}

    \chapter{The density of states on manifolds}\label{Ch:DOSManifolds}
{\setlength{\epigraphwidth}{0.5\textwidth}
\epigraph{I wholeheartedly congratulate you and wish you to get better and work much harder.}{Fedor Sukochev}}
This chapter is an adaptation of~\cite{HekkelmanMcDonald2024}, joint work with Edward McDonald. Section~\ref{S:MfsToDisc} appeared in the preprint version of that paper~\cite{HekkelmanMcDonaldv1}, but was omitted from the published article. For the content of this chapter, we are grateful to Teun van Nuland, Fedor Sukochev, and the anonymous referees of~\cite{HekkelmanMcDonald2024} for helpful comments. The main results in this chapter are a Dixmier trace formula for the density of states on manifolds in Theorem~\ref{T: main manifold thm} and a Dixmier trace formula for Roe's index on open manifolds in Theorem~\ref{roe_index_formula_theorem}.

Like Chapter~\ref{Ch:DOSDiscrete}, this chapter revolves around a Dixmier trace formula for the density of states (DOS). The general form of this formula can be stated for a metric space $(X,d)$ with a Borel measure, a weight $w: X \to \C$ such that $M_w \in \mathcal{L}_{1,\infty}$, and a basepoint $x_0 \in X$. As explained in the introduction of Chapter~\ref{Ch:DOSDiscrete} --- see also Section~\ref{S:DOS} --- the Dixmier trace formula for the DOS that we study is the equality of the two Borel measures associated with a self-adjoint operator $H$ on $L_2(X)$ and an extended limit  $\omega \in \ell_\infty^*$, 
\begin{equation}\label{eq:Measure1.2}
\Tr_\omega(f(H)M_w)  = \int_\R f \, d\nu_1, \quad f \in C_c(\R),
\end{equation}
and
\begin{equation}\label{eq:Measure2.2}
\lim_{R \to \infty} \frac{1}{|B(x_0, R)|} \Tr(f(H)M_{\chi_{B(x_0, R)}}) = \int_\R f\, d\nu_2, \quad f \in C_c(\R),
\end{equation}
where the existence of the measure $\nu_2$, the density of states, has to be assumed. The equality of these measures up to a constant implies that when the DOS exists, the measure $\nu_1$ does not depend on $\omega$.

Whereas Chapter~\ref{Ch:DOSDiscrete} (based on~\cite{AHMSZ}) proves this result for the case where $(X,d)$ is a discrete metric space with certain additional properties, here we will stay closer to the setting in the paper~\cite{AMSZ} where this formula was first proven for $X = \R^d$, $d\geq 2$, $w(x) = (1+|x|^2)^{-\frac d2}$, and $H = -\Delta + M_V$ with $V \in L_\infty(\R^d)$ real-valued.

Namely, we will consider (oriented, connected) non-compact Riemannian manifolds of bounded geometry (the definition follows below in Section~\ref{S: Prelims}). On these manifolds, the operators $H$ that we will take into account are self-adjoint lower-bounded uniformly elliptic differential operators of order two. Additionally, like in the discrete case in Chapter~\ref{Ch:DOSDiscrete}, we require the volume of the balls $|B(x_0, R)|$ to grow in a sub-exponential and regular manner, specified in Definition~\ref{D:Property_D}.

Let us discuss these conditions. The DOS has been studied before on manifolds, usually in the setting where for the non-compact Riemannian manifold $M$ one picks a discrete, finitely generated group $\Gamma$ of isometries of $(M,g)$ which acts freely and properly discontinuously on $M$ such that the quotient $M/\Gamma$ is compact, and the operator studied is a random Schr\"odinger operator~\cite{AdachiSunada1993, LenzPeyerimhoff2008, LenzPeyerimhoff2004, PeyerimhoffVeselic2002, Veselic2008}. It is not difficult to see that the existence of such a group is a stronger condition than requiring bounded geometry, this will be explicitly proven in Proposition~\ref{P:CoCompactBddGeom}. 

Furthermore, in each of the cited papers a recurring assumption is that the group of isometries $\Gamma$ is amenable. This is equivalent with the existence of an expanding family of bounded domains $D_j\subset M$ such that
\[
\lim_{j \to \infty} \frac{|\partial_h D_j|}{|D_j|} = 0, \quad \forall h>0,
\]
where $\partial_h D_j := \{x \in D_j : d(x, \partial D_j) \leq h \}$~\cite{AdachiSunada1993}. 
This is a \textit{weaker} condition than what we will require of the growth of the balls $B(x_0, R)$ specified in Definition~\ref{D:Property_D} below. In fact, for these manifolds to satisfy our assumptions, it is necessary that $\Gamma$ has subexponential growth which implies amenability. In Section~\ref{S: Example} a more detailed analysis of this co-compact setting with a random Schr\"odinger operator will follow.

Let $(M,g)$ be a $d$-dimensional Riemannian manifold, and let $d_g$ be the distance function on $M$ induced by the Riemannian metric. The Riemannian volume of the closed ball $B(x_0, r)$ is denoted by $|B(x_0, r)|$. Its boundary, $\partial B(x_0, r)$, is a $(d-1)$-dimensional Hausdorff-measurable subset of $M$, and as such we can talk about its volume, calculated with respect to its inherited Riemannian metric. This $d-1$-dimensional volume we will also denote as $|\partial B(x_0, r)|$.  In fact, it then holds that (see Section~\ref{S: Manifolds}) \[\frac{d}{dr}|B(x_0, r)|=|\partial B(x_0, r)|.\]

What we require of our manifolds is that both $|B(x_0,r)|$ and $|\partial B(x_0,r)|$ grow sufficiently slowly and regularly. Namely, we will ask that the ratios $\frac{|\partial B(x_0,r)|}{|B(x_0,r)|}$ and $\frac{\frac{d}{dr}|\partial B(x_0,r)|}{|\partial B(x_0,r)|}$ vanish as $R\to\infty$ in the following way.

\begin{defn}[Property (D)]\label{D:Property_D}
    Let $(M,g)$ be a non-compact Riemannian manifold of bounded geometry. It is said to have Property (D) if $r\mapsto V(r):=\abs{B(x_0, r)}$ is in $C^2(\R_{\geq0})$,
\begin{equation}\label{Condition1b}
\frac{V'(r)}{V(r)} \in L_2(\R_{\geq1}),
\end{equation}
and
\begin{equation}\label{Condition2}
    \lim_{r\to \infty} \frac{V''(r)}{V'(r)} = 0.
\end{equation}
\end{defn}

Lemma~\ref{L:NewPropD} below gives in particular that Property (D) implies that $\frac{V'(r)}{V(r)} \to 0$ as $r \to \infty$ (which is not immediate from the requirement that it be in $L_2(\R_{\geq 1})$). If a function $f\in C^1(\mathbb{R})$ satisfies $\lim_{x\to \infty}\frac{f'(x)}{f(x)}= 0$, then $\log f(x) = o(x)$ and hence $f(x) = e^{o(x)}$. Therefore, if the manifold $M$ satisfies Property (D), then necessarily both $|B(x_0,r)| = e^{o(r)}$ and $|\partial B(x_0, r)| = e^{o(r)}$. A quick calculation shows that Property (D) still admits volume growth of the order $|B(x_0, r)| = \exp (r^{\frac{1}{2}-\varepsilon})$, and it is not difficult to see that Property (D) is satisfied for Euclidean spaces. The conditions listed mostly serve to prevent erratic behaviour of the growth. Observe the similarity in this sense to Property (C) in Chapter~\ref{Ch:DOSDiscrete} which required for discrete metric spaces $X$
\[
\frac{|B(x_0, r_{k+1})|}{|B(x_0, r_k)|}\to 1, \quad k \to \infty,
\]
where $\{r_k\}_{k=0}^\infty$ is the set $\{d(y,x_0) \; : \; y \in X\}$ ordered in increasing manner.

Finally, let us specify the class of operators for which the main theorem is formulated. The following definition is essentially the same as~\cite{Kordyukov1991} and the $C^\infty$-bounded differential operators defined in \cite[Appendix 1]{Shubin1992}.
\begin{defn}\label{D:BD}
A differential operator $P$ on a $d$-dimensional Riemannian manifold $M$ of bounded geometry is called a uniform differential operator (of order $m$) if it can be expressed in normal coordinates in a neighbourhood of each point $x \in M$ as
\[
P = \sum_{|\alpha|\leq m} a_{\alpha, x}(y) \partial_y^\alpha,
\]
and for all multi-indices $\beta$ we have
\[
\abs{\partial_y^\beta a_{\alpha, x}(0)} \leq C_{\alpha, \beta}, \quad \abs{\alpha}\leq m.
\]
Following the notation of~\cite{Kordyukov1991}, we denote this by $P \in BD^m(M)$.

Furthermore, let $\sigma_x(y, \xi) = \sum_{\abs{\alpha}=m} a_{\alpha, x}(y) (i\xi)^\alpha$ be the principal symbol of $P$ near $x.$ We say that $P \in BD^m(M)$ is uniformly elliptic, denoted $P\in EBD^m(M)$, if there exists $\varepsilon>0$ such that \[\abs{\sigma_x(0,\xi)}\geq \varepsilon \abs{\xi}^m,\quad \xi\in \mathbb{R}^d, x\in M.\]
\end{defn}
A similar definition applies to operators acting between sections of vector bundles of bounded geometry; see \cite{Shubin1992}.

Having discussed all necessary details, we can now formulate the main theorem of this chapter.

\begin{thm}\label{T: main manifold thm}
    Let $(M,g)$ be a non-compact Riemannian manifold of bounded geometry with Property (D). Let $P \in EBD^2(M)$ be self-adjoint and lower-bounded, and let $w$ be the function on $M$ defined by 
    \[
    w(x) = \frac{1}{1+\abs{B(x_0,d_g(x,x_0))}},\quad x \in M.
    \]
    Then $f(P)M_w$ is an element of $\mathcal{L}_{1,\infty}$ for all compactly supported functions $f\in C_c(\mathbb{R})$. If $P$ admits a density of states $\nu_P$, we have for all extended limits $\omega \in \ell_\infty^*$ 
    \[
    \mathrm{Tr}_{\omega}(f(P)M_w) = \int_{\mathbb{R}} f \, d\nu_P, \quad f\in C_c(\mathbb{R}).
    \]
\end{thm}

Observe that Euclidean space is a manifold of bounded geometry satisfying Property (D), and Schr\"odinger operators $H=-\Delta + M_V$ with smooth bounded potential $V$ are operators of the required class. Therefore, besides the smoothness assumption on $V$, the main result in this paper is a generalisation of~\cite{AMSZ}.

Whereas the proofs in~\cite{AMSZ} are based on delicate singular value estimates particular to Euclidean space, and the ones in~\cite{AHMSZ} (i.e. Chapter~\ref{Ch:DOSDiscrete}) are based
on heavy real analysis, here we will prove a statement in abstract operator theory. It is a significant generalisation of~\cite[Theorem~5.7]{AMSZ}. In Section~\ref{S:MfsToDisc}, we will make a comparison between this method and the proof in Chapter~\ref{Ch:DOSDiscrete}.

\begin{thm}\label{MAIN_THEOREM}
    Let $W$ and $P$ be linear operators, such that $P$ is self-adjoint and lower-bounded and $W$ is positive and bounded. Assume that for every $t>0$ we have
    \begin{enumerate}
        \item\label{cwikelcond1} $\exp(-tP)W \in \mathcal{L}_{1,\infty},$
        \item\label{cwikelcond2} $\exp(-tP)[P,W] \in \mathcal{L}_1$.
    \end{enumerate}
    Then, for every extended limit $\omega$,
    \[
        \mathrm{Tr}_{\omega}(e^{-tP}W) = \lim_{\varepsilon\to 0} \varepsilon \mathrm{Tr}(e^{-tP}\chi_{[\varepsilon,\infty)}(W)), \quad t >0.
    \]
    whenever the limit on the right hand side exists.
\end{thm}

As an application of the Dixmier trace formula for the DOS, Theorem~\ref{T: main manifold thm}, we will look at Roe's index theorem on open manifolds~\cite{Roe1988a}. Roe's index theorem is one approach of many that extends Atiyah--Singer's index theorem~\cite{AtiyahSinger1963} to non-compact manifolds, namely non-compact manifolds of bounded geometry that admit a \textit{regular exhaustion}. For the precise definition we refer to~\cite[Section~6]{Roe1988a}, but for this introduction it suffices to know that Property (D) is a stronger assumption.

Given a \textit{compact} manifold $M$ with two vector bundles $E,F \to M$ and an elliptic differential operator $D: \Gamma(E) \to \Gamma(F)$, the local index formula equates the Fredholm index of $D$ to the integral of a differential form given by topological data denoted here simply by $\mathbf{I}(D)$,
\begin{equation}\label{eq:Atiyah-Singer}
\operatorname{Ind}(D) = \int_M \mathbf{I}(D).
\end{equation}
A special case of the index formula can be proved with the McKean--Singer formula, which in this case states that
\[
    \operatorname{Ind}(D) = \mathrm{Tr}(\eta e^{-t\mathcal{D}^2}), \quad t >0,
\]
where $\eta$ is the grading operator on the bundle $E\oplus F\to M,$ and $\mathcal{D}$ is the self-adjoint operator acting on $\Gamma(E\oplus F)$ by the formula
\[
    \mathcal{D} = \begin{pmatrix} 0 & D^* \\ D & 0\end{pmatrix}.
\]

If $M$ is not compact, neither side of \eqref{eq:Atiyah-Singer} is well-defined in general. In the setting of a non-compact Riemannian manifold $M$ of bounded geometry with a regular exhaustion, and with a graded Clifford bundle $S \to M$ also of bounded geometry (see Section~\ref{S:Roe}), which comes with a natural first-order elliptic differential operator, the Dirac operator $D: \Gamma(S_+) \to \Gamma(S_-)$, Roe modifies both sides of equation~\eqref{eq:Atiyah-Singer} as follows. Defining a linear functional $m$ on bounded $d$-forms via an averaging procedure, the right-hand side simply becomes $m(\mathbf{I}(D))$. For the left-hand side, Roe defines an algebra of uniformly smoothing operators $\mathcal{U}_{-\infty}$, and shows that elliptic operators are invertible modulo $\mathcal{U}_{-\infty}$. Recall that for operators that are invertible modulo compact operators (Fredholm operators), the Fredholm index is an element of the $K$-theory group $K_0(K((\mathcal{H})) = \mathbb{Z}$~\cite{Wegge-Olsen1993} (recall that $K(\Hc)$ are the compact operators). In this case, we can similarly define an abstract index of an elliptic operator $D$ as an element of $K_0(\mathcal{U}_{-\infty})$, by observing that $\mathcal{U}_{-\infty}$ forms an ideal in what Roe defines as uniform operators $\mathcal{U}$. Furthermore, using the functional $m$ one can define a trace $\tau$ on $\mathcal{U}_{-\infty}$. This trace can be extended to a trace on the matrix algebras $M_n(\mathcal{U}_{-\infty}^+)$ (with $\mathcal{U}_{-\infty}^+$ denoting the unitisation of $\mathcal{U}_{-\infty}$) by putting $\tau(1)=0$ when passing to the unitisation, and then tensoring with the usual trace on $M_n(\mathbb{C})$. The tracial property gives that this map descends to a map called the dimension-homomorphism $\dim_\tau: K_0(\mathcal{U}_{-\infty}) \to \mathbb{R}$. These ingredients give Roe's index theorem:
\begin{equation}
\dim_\tau(\operatorname{Ind}(D)) = m(\mathbf{I}(D)).
\end{equation}

The nature of the averaging procedure that Roe develops is such that if $D^2$ admits a density of states, Theorem~\ref{T: main manifold thm} leads to a Dixmier trace reformulation of the analytical index $\dim_\tau(\operatorname{Ind}(D))$. In Section~\ref{S:Roe} we prove
\[
\dim_\tau(\operatorname{Ind}(D)) = \mathrm{Tr}_\omega(\eta \exp(-tD^2)M_w), \quad t>0,
\]
where $\eta$ is the grading on $S$.

\section{Preliminaries}
\label{S: Prelims}

\subsection{Operator theory}
We recall some facts about sums of left-disjoint families of operators, and prove an estimate for their $\mathcal{L}_{p}$-norm and $\mathcal{L}_{p,\infty}$-norm. A family $\{T_j\}_{j=0}^\infty$ of bounded linear operators on a Hilbert space $\mathcal{H}$ is \emph{left-disjoint} if $T_j^*T_k = 0$ for all $j\neq k.$

Given a sequence $\{T_j\}_{j=0}^\infty$ of bounded linear operators on $\mathcal{H}$, let
\[
    \bigoplus_{j=0}^\infty T_j
\]
denote the operator on $\mathcal{H}\otimes \ell_2(\mathbb{N})$ given by
\[
    \sum_{j=0}^\infty T_j\otimes e_je_j^*
\]
where $e_je_j^*$ is the rank $1$ projection onto the orthonormal basis element $e_j \in \ell_2.$

Note that
\[
    \mu\left(\bigoplus_{j=0}^\infty T_j\right) = \mu\left(\bigoplus_{j=0}^\infty \mathrm{diag}(\mu(T_j))\right).
\]
To put it differently, the singular value sequence of the direct sum $\bigoplus_{j=0}^\infty T_j$ is the decreasing rearrangement of the sequence indexed by $\mathbb{N}^2$ given by
\[
    \{\mu(k,T_j)\}_{j,k= 0}^\infty.
\]
By definition, we have $\|T\|_{p_1,p_2} = \|\mu(T)\|_{\ell_{p_1,p_2}}.$ Therefore,
\[
    \left\|\bigoplus_{j=0}^\infty T_j\right\|_{p_1, p_2} =  \|\{\mu(k,T_j)\}_{j,k\geq 0}\|_{\ell_{p_1, p_2}(\mathbb{N}^2)}.
\]
Now let $q>0.$ We have for each $j$ that $\mu(k,T_j)\leq (k+1)^{-\frac1q}\|T_j\|_{q,\infty}.$ Therefore,
\[
    \left\|\bigoplus_{j=0}^\infty T_j\right\|_{p_1,p_2} \leq \|\{(1+k)^{-\frac1q}\|T_j\|_{q,\infty}\}_{j,k\geq 0}\|_{\ell_{p_1,p_2}(\mathbb{N}^2)}.
\]
Lemma 4.3 of \cite{LevitinaSukochev2020} implies that if $q$ is sufficiently small, then there exists a constant $C_{p_1,p_2,q}$ such that for any sequence $\{x_j\}_{j=0}^\infty$ we have
\[
    \|\{(1+k)^{-\frac1q}x_j\}_{j,k\geq 0}\|_{\ell_{p_1,p_2}(\mathbb{N}^2)} \leq C_{p_1,p_2,q}\|\{x_j\}_{j=0}^\infty\|_{\ell_{p_1,p_2}}.
\]
Taking $x_j=\|T_j\|_{q,\infty}$ and sufficiently small $q$ (depending on $p_1, p_2$) we have
\begin{equation}\label{direct_sum_quasinorm_bound}
    \left\|\bigoplus_{j=0}^\infty T_j\right\|_{p_1,p_2} \leq C_{p_1,p_2,q} \|\{\|T_j\|_{q,\infty}\}_{j=0}^\infty \|_{p_1,p_2}.
\end{equation}

Combining \cite[Proposition 2.7, Lemma 2.9]{LevitinaSukochev2020} gives the following result.
\begin{lem}\label{disjointification}
    Let $0<p<2,$ and let $\{T_j\}_{j=0}^\infty$ be a left-disjoint family of bounded linear operators. There exist constants $C_p$, $C_p'$ such that
    \begin{align*}
    \left\|\sum_{j=0}^\infty T_j\right\|_{p,\infty}&\leq C_p\left\|\bigoplus_{j=0}^\infty T_j\right\|_{p,\infty};\\
        \left\|\sum_{j=0}^\infty T_j\right\|_{p}&\leq C_p'\left\|\bigoplus_{j=0}^\infty T_j\right\|_{p}.
    \end{align*}
\end{lem}

A combination of \eqref{direct_sum_quasinorm_bound} and Lemma \ref{disjointification} immediately yields the following:
\begin{cor}\label{disjointification_corollary}
    Let $0<p<2,$ and let $\{T_j\}_{j=0}^\infty$ be a left-disjoint family of operators. There exist $q, q'>0$ (depending on $p$) and constants $C_{p,q}, C'_{p,q'}>0$ such that
    \begin{align*}
    \left\|\sum_{j=0}^\infty T_j\right\|_{p,\infty}&\leq C_{p,q}\left\|\{\|T_j\|_{q,\infty}\}_{j=0}^\infty\right\|_{\ell_{p,\infty}};\\
    \left\|\sum_{j=0}^\infty T_j\right\|_{p}&\leq C'_{p,q'}\left\|\{\|T_j\|_{q',\infty}\}_{j=0}^\infty\right\|_{\ell_{p}}.
    \end{align*}
\end{cor}

\begin{rem}
    For $p=1,$ we can take any $0<q,q'<1.$
\end{rem}
\subsection{Preliminaries on manifolds}
\label{S: Manifolds}
All manifolds in this chapter
are smooth, oriented, non-compact and connected unless stated otherwise.

\begin{defn}
Let $(M,g)$ be a Riemannian manifold. The injectivity radius $i(x)$ at a point $x \in M$ is defined as \[i(x):= \sup \{ R \in \mathbb{R}_{\geq 0} : \exp_x|_{B(0,R)} \textup{ is a diffeomorphism}\},\] where $\exp_x$ is the exponential map at $x$ and $B(0,R) \subseteq T_xM$ is the metric ball with radius $R$ centered around the origin. The injectivity radius of the manifold $M$ is defined as \[i_g := \inf_{x \in M} i(x).\]
\end{defn}

It is a theorem that the injectivity radius map \begin{align*}
    i : M &\to (0, \infty]\\
    x &\mapsto i(x)
\end{align*}
is continuous~\cite[Proposition~III.4.13]{Sakai1996}.

\begin{defn}\label{D:bounded_geometry}
A Riemannian manifold $(M,g)$ has \textbf{bounded geometry} if the injectivity radius $i_g$ satisfies
\[i_g >0,\] and the Riemannian curvature tensor $R$ and all its covariant derivatives are uniformly bounded.
\end{defn}

This is the definition as in~\cite[Definition 1.1]{Kordyukov1991}, \cite[Chapter II]{Eichhorn2008} and \cite[Chapter 3]{Eichhorn2007}. Every open manifold admits a metric of bounded geometry \cite{Greene1978}.

As is well-known (see e.g. \cite[Lemma 2.4]{Kordyukov1991}), bounded geometry implies that there exists $r_0>0$ and a countable set $\{x_j\}_{j=0}^\infty$ of points in $M$ such that:
\begin{enumerate}
    \item{} $M = \bigcup_{j=0}^\infty B(x_j,r_0)$;
    \item{} each ball $B(x_j,r_0)$ is a chart for the exponential normal coordinates based at $x_j$;
    \item{} the covering $\{B(x_j,r_0)\}_{j=0}^\infty$ has finite order, meaning that there exists $N$ such that each ball intersects at most $N$ other balls;
    \item{} $\sup_j |B(x_j,r_0)| < \infty$, recall that $|\cdot|$ indicates the volume; 
    \item{} there exists a partition of unity $\{\psi_j\}_{j=0}^\infty$ subordinate to $\{B(x_j,r_0)\}$ such that for every $\alpha$ we have $\sup_{j,x} |\partial^{\alpha}\psi_j(x)|<\infty,$ where $\partial^{\alpha}$ is taken in the exponential normal coordinates of $B(x_j,r_0).$
\end{enumerate}
Without loss of generality, $r_0$ can be taken smaller or equal to $1$ (see~\cite[Lemma~2.3]{Kordyukov1991}).

We will refer to the scale of Sobolev spaces $\{\Hc^s(M)\}_{s\in \mathbb{R}}$ defined over $M$ as
\[
\Hc^s := \overline{\dom (1-\Delta_g)^{\frac{s}{2}}}^{\| \cdot \|_s},
\]
with $\| \xi \|_s := \| (1-\Delta_g)^{\frac{s}{2}
}\xi\|_{L_2(M)}$, as in Section~\ref{S:PSDOs} and Chapter~\ref{Ch:FunctCalc}, see also \cite[Section~3]{Kordyukov1991} or \cite{GrosseSchneider2013}. 
We will make use of the fact that if $P \in BD^m(M),$ then $P \in \op^{m}(1-\Delta_g)^{\frac
12}$, i.e. $P$ defines a bounded linear operator from $\Hc^{s+m}(M)$ to $\Hc^s(M)$ for every $s\in \mathbb{R}.$

\begin{rem}
    It follows from the bounded geometry assumption that there exist constants $c,C>0$ such that for every $x\in M$ and $R>0$ we have
    \[
        |B(x,R)| \leq C\exp(cR).
    \]
\end{rem}

Note that for almost every $r>0,$ we have
\[
    \frac{d}{dr}|B(x,r)| = |\partial B(x,r)|.
\]
See \cite[Proposition III.3.2 \& Proposition III.5.1]{Chavel2006}.

A standard example of a manifold of bounded geometry is a covering space of a compact manifold. The following is well-known, but lacking a reference we supply a proof for the reader's convenience.
\begin{prop}\label{P:CoCompactBddGeom}
Let $(M,g)$ be a complete $d$-dimensional Riemannian manifold. Let $\Gamma$ be a discrete, finitely generated subgroup of the isometries of $(M,g)$ which acts freely and properly discontinuously on $M$ such that the quotient $M/\Gamma$ is a compact ($d$-dimensional) Riemannian manifold. Then $M$ has bounded geometry.
\begin{proof}
Since $\Gamma$ acts cocompactly on $M$, there exists a compact set $L \subseteq M$ such that $\bigcup_{\gamma \in \Gamma} \gamma L = M$. Indeed, since the open balls $B(x, 1), x\in M$ project onto an open cover of $M/\Gamma$, there exists a finite collection $\{ B(x_i,1)\}_{i=1}^N$ that projects onto $M/\Gamma$. Defining $L := \bigcup_{i=1}^N \overline{B(x_i,1)}$ gives the claimed compact set $L$.

Set \[i_L := \inf_{x \in L} i(x), \]
where $i(x)$ is the injectivity radius at the point $x$. Since the injectivity radius is a continuous function on $M$, $i(x)>0$ for all $x\in M$, and $L$ is compact, it follows that $i_L > 0$. Since $\Gamma$ acts by isometries, for any $\gamma \in \Gamma$ we have $i(\gamma x) = i(x)$, see for example the proof of~\cite[Theorem~III.5.4]{Sakai1996}. Therefore, $\inf_{x\in M} i(x) = i_L >0$.

Next, since the curvature tensor $R$ is smooth, clearly $R$ and all its covariant derivatives are bounded on $L$. Let $\Phi: M \to M$ be the isometry by which $\gamma \in \Gamma$ acts. Then for all $x\in M$ $\Phi_x^*:T_xM \to T_{\Phi(x)}M$ is an isomorphism, and by~\cite[p.~41]{Sakai1996}, \begin{align*}
    \Phi_x^*(\nabla_{\xi}\eta) &= \nabla_{\Phi_x^*\xi} \Phi_x^*\eta,\\
    \Phi_x^*(R(\eta, \xi)\zeta) &= R(\Phi_x^*\eta, \Phi_x^*\xi)\Phi_x^*\zeta,
\end{align*} where $\xi, \eta, \zeta \in T_xM$. Combine the facts that $R$ and its covariant derivatives are bounded on $L$, that $\bigcup_{\gamma \in \Gamma} \gamma L = M$, and that isometries preserve $R$ and taking covariant derivatives in the above manner, and we can conclude that $R$ and its covariant derivatives are uniformly bounded on $M$.
\end{proof}
\end{prop}

\begin{rem}\label{R: metric balls lower bound}
A non-compact Riemannian manifold of bounded geometry has infinite volume. This can be checked easily via the covering $M = \bigcup_{j=0}^\infty B(x_j,r_0)$ below Definition~\ref{D:bounded_geometry}, and the observation that $\inf_{j} |B(x_j, r_0)| > 0$~\cite{Kodani1988}.
\end{rem}

\section{Proof of Theorem~\ref{MAIN_THEOREM}}
\label{S: Main thm}
Part of the proof in \cite{AMSZ} used the identity
\begin{equation}\label{eq: Abelian Rd}
    \lim_{s \downarrow 1} (s-1)\int_{\mathbb{R}^d} F(x)(1+|x|^2)^{-\frac{s}{2}}\,dx = \lim_{R\to\infty} R^{-d}\int_{B(0,R)}F(x)\,dx
\end{equation}
for any bounded measurable function $F$ on $\mathbb{R}^d$ such that the right hand side exists.

Equation~\eqref{eq: Abelian Rd} can be generalised in the following manner. The proof is essentially the same as \cite[Lemma 6.1]{AMSZ}.

\begin{lem}\label{L: trace formula}
    Let $A$ and $B$ be bounded linear operators, $B\geq 0$, such that $AB^s \in \mathcal{L}_1$ for every $s>1.$ Then
    \[
        \lim_{s\downarrow 1}(s-1)\mathrm{Tr}(AB^s) = \lim_{\varepsilon\to 0} \varepsilon \mathrm{Tr}(A\chi_{[\varepsilon,\infty)}(B))
    \]
    whenever the limit on the right exists.
\end{lem}
\begin{proof}
    Writing $B^s = \int_0^{\|B\|_{\infty}} \lambda^s\,dE_B(\lambda),$ $\lambda^s = s\int_0^\lambda r^{s-1}\,dr$
    and applying Fubini's theorem yields
    \[
        \mathrm{Tr}(AB^s) = s\int_0^{\|B\|_{\infty}} r^{s-1} \mathrm{Tr}(A\chi_{[r,\infty)}(B))\,dr.
    \]
    Assume without loss of generality that $\|B\|_{\infty}  = 1.$
    Writing $F(r) := \mathrm{Tr}(A\chi_{[r,\infty)}(B))$ our assumption is that
    \[
        F(r) \sim \frac{c}{r},\quad r\to 0,
    \]
    and
    \[
        \mathrm{Tr}(AB^s) = s\int_0^1 r^{s-1}F(r)\,dr.
    \]
    Let $\varepsilon>0,$ and choose $R>0$ sufficiently small such that if $0<r<R$ then
    \[
        |rF(r)-c|<\varepsilon.
    \]
    We write $\mathrm{Tr}(AB^s)-\frac{c}{s-1}$ as
    \[
        \mathrm{Tr}(AB^s)-\frac{c}{s-1} = c+s\int_0^1 r^{s-1}F(r) - c r^{s-2}\,dr.
    \]
    Therefore
    \[
        |\mathrm{Tr}(AB^s)-\frac{c}{s-1}| \leq |c|+s\int_0^1 r^{s-2}|rF(r)-c|\,dr\leq |c|+s\int_R^1 r^{s-2}|rF(r)-c|\,dr + \frac{s\varepsilon}{s-1}.
    \]
    It follows that
    \[
        |(s-1)\mathrm{Tr}(AB^s)-c| = O(s-1)+s\varepsilon,\quad s\downarrow 1.
    \]
    Since $\varepsilon$ is arbitrary, this completes the proof.
\end{proof}

We will make use of the following theorem, which is \cite[Theorem~1.3.20]{LMSZVol2}.
\begin{thm}\label{LSZ2ethm}
    Let $A$ and $B$ be positive bounded linear operators such that $[B,A^{\frac12}] \in \mathcal{L}_1$ and $AB\in \mathcal{L}_{1,\infty}$. If
    \[
        \|A^{\frac12}B^s\|_1 = o((s-1)^{-2}),\quad s\downarrow 1
    \]
    then for every extended limit $\omega$ we have
    \[
        \mathrm{Tr}_{\omega}(AB) = \lim_{s\downarrow 1} (s-1)\mathrm{Tr}(AB^s)
    \]
    if the limit exists.
\end{thm}
Hence, if $[B,A^{\frac12}] \in \mathcal{L}_1$ and $\|A^{\frac12}B^s\|_1 = o((s-1)^{-2}),$ then by Lemma~\ref{L: trace formula}
\[
    \mathrm{Tr}_{\omega}(AB) = \lim_{\varepsilon\to 0} \varepsilon\mathrm{Tr}(A\chi_{[\varepsilon,\infty)}(B))
\]
whenever the limit on the right exists.

Recall the Araki--Lieb--Thirring inequality~\cite{Kosaki1992}
\begin{equation}\label{alt_inequality}
    \|AB\|_{r,\infty}^r \leq e\|A^rB^r\|_{1,\infty},\quad r > 1
\end{equation}
and the numerical inequality
\begin{equation}\label{zeta_inequality}
    \|X\|_{\infty,1} = \sum_{k=0}^\infty \frac{\mu(k,X)}{k+1} \leq \|X\|_{q,\infty}\zeta(1+\frac1q),
\end{equation}
obtained by simply writing out the definitions. Here $\zeta$ is the Riemann zeta function, and it obeys $\zeta(1+\frac1q) \sim q$ as $q\to\infty.$

\begin{cor}\label{C:main_result}
    Let $A$ and $B$ be positive bounded linear operators such that
    \begin{enumerate}
        \item{}\label{cond1} $[A^{\frac12},B] \in \mathcal{L}_1,$
        \item{}\label{cond2} $A^{\frac14}B \in \mathcal{L}_{1,\infty}.$
    \end{enumerate}
    Then for every extended limit $\omega$ we have
    \[
        \mathrm{Tr}_{\omega}(AB) = \lim_{s\downarrow 1} (s-1)\mathrm{Tr}(AB^s)
    \]
    if the limit exists.
\end{cor}
\begin{proof}
    Let $1<s<2.$ We have
    \[
        \|A^{\frac12}B^s\|_1 \leq \|[A^{\frac12},B]B^{s-1}\|_1+\|BA^{\frac12}B^{s-1}\|_1 \leq \|[A^{\frac12},B]\|_1\|B\|_{\infty}^{s-1}+\|BA^{\frac14}\|_{1,\infty}\|A^{\frac14}B^{s-1}\|_{\infty,1},
    \]
    where we have used the inequality
\[
\|TS\|_{1} \leq   2\|T\|_{1,\infty} \|S\|_{\infty,1}, \quad T \in \mathcal{L}_{1,\infty}, S \in \mathcal{L}_{\infty, 1},
\]
which can easily be checked via the definitions of these quasi-norms and the inequality $\mu(2k,TS)\leq \mu(k,T)\mu(k,S).$

    By the numerical inequality \eqref{zeta_inequality} above
    \[
        \|A^{\frac14}B^{s-1}\|_{\infty,1} \lesssim (s-1)^{-1}\|A^{\frac14}B^{s-1}\|_{\frac{1}{s-1},\infty}.
    \]
    Since $s<2,$ we have $\frac{1}{s-1}>1,$ and hence the Araki--Lieb--Thirring inequality \eqref{alt_inequality} delivers
    \[
        \|A^{\frac14}B^{s-1}\|_{\frac{1}{s-1},\infty}^{\frac{1}{s-1}} \leq e\|A^{\frac{1}{4(s-1)}}B\|_{1,\infty}\leq e\|A^{\frac14}\|_{\infty}^{\frac{2-s}{s-1}}\|A^{\frac14}B\|_{1,\infty}.
    \]
    This verifies the assumptions of Theorem \ref{LSZ2ethm}.
\end{proof}

In our case, we have an operator $P$ which is self-adjoint and lower-bounded, which means $\exp(-tP)$ is positive and bounded for all $t>0$, and we take $B=W$. Then the assumptions become
\begin{equation}\label{eq:original_conditions}
    [\exp(-tP),W] \in \mathcal{L}_1,\; \exp(-tP)W \in \mathcal{L}_{1,\infty}
\end{equation}
for every $t>0.$

The former condition can be modified to one which is easier to verify in geometric examples.
\begin{lem}[Duhamel's formula]\label{Duhamel formula}
Let $P$ be a lower-bounded self-adjoint operator on a Hilbert space $\Hc$, and let $W$ be a bounded operator. Then
\[
[\exp(-tP),W] = -\int_0^t \exp(-sP)[P,W]\exp(-(t-s)P)\,ds.
\]
\begin{proof}
The method of proof is standard, see for example~\cite[Lemma~5.2]{ACDS}. Define $F: [0,t] \to B(\Hc)$ by $F(s) = \exp(-sP)W\exp(-(t-s)P)$. Since $P$ is lower-bounded, $\exp(-sP)$ is bounded for all $s \in [0,t]$. Hence the derivative of $F(s)$ in the strong operator topology is
\begin{align*}
F'(s) &=  -P\exp(-sP)W\exp(-(t-s)P) +  \exp(-sP)W P \exp(-(t-s)P) \\
&= -\exp(-sP)[P,W]\exp(-(t-s)P),
\end{align*} in the sense that 
\[
\lim_{h\to 0} \frac{1}{h}(F(s+h)-F(s))\xi = F'(s)\xi, \quad \xi \in \Hc.
\]
Therefore, by the fundamental theorem of calculus for Banach space-valued functions, we can conclude that for $\xi \in \Hc$, \begin{align*}
[\exp(-tP),W]\xi &= (F(t)-F(0))\xi\\
&= \int_{0}^t F'(s) \xi\, ds\\
&= -\int_0^t \exp(-sP)[P,W]\exp(-(t-s)P) \xi\, ds.\qedhere
\end{align*}
\end{proof}
\end{lem}

\begin{lem}
    Let $P$ be a self-adjoint lower bounded linear operator
    and let $W$ be a bounded self-adjoint operator such that $[P,W]$ makes sense, and
    \[
        \exp(-tP)[P,W] \in \mathcal{L}_1,\quad t > 0.
    \]
    Then
    \[
        [\exp(-tP),W]\in \mathcal{L}_1,\quad t > 0.
    \]
\end{lem}
\begin{proof}
    By the Duhamel formula (Lemma~\ref{Duhamel formula}),
    \begin{align*}
        [\exp(-tP),W] &= - \int_0^t \exp(-sP)[P,W]\exp(-(t-s)P)\,ds\\
                      &= -\int_0^{\frac{t}{2}}\exp(-sP)[P,W]\exp(-(t-s)P)\,ds\\
                      &\quad -\int_{\frac{t}{2}}^t \exp(-sP)[P,W]\exp(-(t-s)P)\,ds.
    \end{align*}
    For $0<s<\frac{t}{2},$ we have
    \[
        \|\exp(-sP)[P,W]\exp(-(t-s)P)\|_1 \leq \|[P,W]\exp(-\frac{t}{2}P)\|_1
    \]
    while for $\frac{t}{2}<s<t,$
    \[
        \|\exp(-sP)[P,W]\exp(-(t-s)P)\|_1 \leq \|\exp(-\frac{t}{2})[P,W]\|_1.
    \]
    Hence, the triangle inequality for weak integrals we have
    \[
        \|[\exp(-tP),W]\|_1 \leq t\|[P,W]\|_1.\qedhere
    \]
\end{proof}

Combining the results of this section completes the proof of Theorem~\ref{MAIN_THEOREM}.

\section{Manifolds of bounded geometry}\label{S:Proof_Manifolds}
Let us now shift our attention to the case where we have a Riemannian manifold of bounded geometry $M$ and a self-adjoint, lower-bounded operator $P \in EBD^2(M)$.

Estimates of the form
\[
    \|M_fg(-i\nabla)\|_{p,\infty} \leq c_p\|f\|_p\|g\|_{p,\infty}
\]
for $p>2$ are sometimes called Cwikel-type estimates. Here, $f$ and $g$ are function on $\mathbb{R}^d,$
see \cite[Chapter 4]{Simon2005}. Similar estimates with $p<2$ were obtained earlier by Birman and Solomyak \cite{BirmanSolomyak1969}.
In particular, it follows from \cite[Theorem 4.5]{Simon2005} that if $f$ is a measurable function on $\mathbb{R}^d$ then for every $t>0$ and $0<p<2$ that we have
\[
    \|M_fe^{t\Delta}\|_{p} \lesssim_{t,p} \left(\sum_{k\in \mathbb{Z}^d} \|f\|_{L_\infty(k+[0,1)^d)}^p\right)^{\frac1p}
\]
and similarly
\[
    \|M_fe^{t\Delta}\|_{p,\infty} \lesssim_{t,p} \left\|\{\|f\|_{L_{\infty}(k+[0,1)^d)}\right\|_{p,\infty}
\]
We seek similar estimates for functions $f$ on manifolds of bounded geometry. In place of the decomposition of $\mathbb{R}^d$ into cubes, $\mathbb{R}^d = \bigcup_{k\in \mathbb{Z}}^d [0,1)^d+k,$ we will use the partition of unity constructed according to Section~\ref{S: Manifolds}.

Namely, we will show that for $0<p<2$, we have $\exp(-tP)M_f \in \mathcal{L}_{p,\infty}$ whenever $f\in \ell_{p,\infty}(L_\infty)$, and $\exp(-tP)M_f \in \mathcal{L}_{p}$ whenever $f\in \ell_{p}(L_\infty),$
where $\ell_{p,\infty}(L_\infty)$ and $\ell_p(L_{\infty})$ are certain function spaces on $M.$ 

\begin{defn}
    Let $\{x_j\}_{j=0}^\infty$ and $r_0$ be as in Section \ref{S: Manifolds}. Given a bounded measurable function $f$ on $M,$ define
    \[
    \|f\|_{\ell_{p,\infty}(L_\infty)} := \left\| \{\|f\|_{L_{\infty}(B(x_j,r_0))}\}_{j=0}^\infty\right\|_{\ell_{p,\infty}}.
    \]
    and
    \[
    \|f\|_{\ell_{p}(L_\infty)} := \left\| \{\|f\|_{L_{\infty}(B(x_j,r_0))}\}_{j=0}^\infty\right\|_{\ell_{p}}.
    \]
\end{defn}

Let $B$ be an open ball in $\mathbb{R}^d$ such that at every point $x\in M$ we have a normal coordinate system $\exp_x \circ e: B \to B(x, r_0)$, where $e$ is the identification of $\mathbb{R}^d$ with the tangent space $T_xM$ and $\exp_x$ is the Riemannian exponential map~\cite[Proposition~1.2]{Kordyukov1991}. Via these maps we can, for each $x\in M$, pullback the metric $g$ restricted to $B(x, r_0)$ to a metric $g^x$ on $B$. 

\begin{defn}\label{d:Riem}
A Riemannian metric $g$ on $B$ can be represented uniquely by the $d^2$ smooth functions
\[
g_{ij}:= g(\partial_i, \partial_j) \in C^\infty(B).
\]
Define $\mathrm{Riem}_b(B)$ as the set of Riemannian metrics for which the functions $g_{ij}$ extend to smooth functions in $C^\infty(\overline{B})$. We equip $\mathrm{Riem}_b(B)$ with the topology induced by the embedding $\mathrm{Riem}_b(B) \subseteq (C^\infty(\overline{B}))^{d^2}$, where we we take the usual topology on $C^\infty(\overline{B})$ defined by the seminorms
\[
p_N(f) :=  \max_{x\in \overline{B}} \{|\partial^\alpha f(x)|: |\alpha|\leq N \}.
\]
\end{defn}

The following is essentially \cite{Eichhorn1991}, see in particular the discussion below \cite[Theorem~A]{Eichhorn1991}. See also the related statement \cite[Proposition~2.4]{Roe1988a}.
\begin{prop}\label{p:bounded_metric}
Let $(M,g)$ be a manifold of bounded geometry. The functions $g^x_{ij}$ considered as a family of smooth functions parametrized by $i,j$ and by a point $x\in M$, can be extended to $\overline{B}$ and then lie in a bounded subset of $C^\infty(\overline{B}).$
\end{prop}
In principle the identification $e$ of $\mathbb{R}^d$ with the tangent space $T_xM$ can vary from point to point, and therefore
the matrix elements $g^x_{ij}$ are not uniquely defined. However this does not change the fact that for any given identification, the result of Proposition \ref{p:bounded_metric} holds.

\begin{cor}\label{c:compactmetrics}
Let $(M, g)$ be a manifold with bounded geometry, and let $\{B(x_j, r_0)\}_{j \in \mathbb{N}}$ be an open cover of $M$ as in Section~\ref{S: Manifolds}. The set
\[
\{ g^{x_j} : j \in \mathbb{N}\} \subset \mathrm{Riem}_b(B)
\]
is a pre-compact subset of $C^\infty(\overline{B}).$
\end{cor}
\begin{proof}
By Proposition~\ref{p:bounded_metric}, the functions $g^{x_k}_{ij}$ lie in a bounded set in $C^\infty(\overline{B})$. Since the closure of a bounded set is also bounded~\cite[Section~IV.2]{Conway1990}, and since $C^\infty(\overline{B})$ has the Heine--Borel property~\cite[Section~1.9]{GrandpaRudin}, the assertion follows.
\end{proof}

Given $g \in \mathrm{Riem}_b(B),$ we denote $\Delta_g^D$ the self-adjoint realisation of $\Delta_g$ on $B$ with Dirichlet boundary conditions. Explicitly, $\Delta_g^D$ is defined as the operator associated with the closure of the quadratic form
\[
    q_g(u,v) = \int_{B} \sqrt{|g(x)|}\sum_{i,j} g^{ij}(x)\partial_i u(x)\overline{ \partial_j v(x)}\,dx,\quad u,v\in C^\infty_c(B).
\]
\begin{lem}\label{l:uniformlaplacianbound}
Let $x_j$ and $g^{x_j} \in \mathrm{Riem}_b(B)$ be as in Corollary~\ref{c:compactmetrics}.
We have
\[
\sup_{j \in \mathbb{N}} \|(1-\Delta_{g^{x_j}}^D)^{-1}\|_{\mathcal{L}_{\frac{d}{2},\infty}(L_2(B)) }<\infty.
\]
\end{lem}
\begin{proof}
Let $g$ be a metric on the closed unit ball $\overline{B},$ and let $c_g$ and $c_G$ be positive constants such that
\[
    c_g\left(\sum_{k=1}^d |\xi_k|^2\right) \leq \sum_{k,l=1}^d \sqrt{|g(x)|}g^{kl}(x)\xi_k\overline{\xi_l} \leq C_g\left(\sum_{k=1}^d |\xi_k|^2\right)
\]
for all $\xi \in \mathbb{C}^d.$
We will prove that there is a constant $k_d$ such that
\begin{equation}\label{e:metric_uniformity}
\|(1-\Delta_{g}^D)^{-1}\|_{\mathcal{L}_{\frac{d}{2},\infty}(L_2(B)) } \leq  k_dc_g^{-\frac12}.
\end{equation}
Since
\[
    \inf_{j} c_{g^{x_j}} > 0,
\]
\eqref{e:metric_uniformity} implies the result.

Let $q_0$ denote the Dirichlet quadratic form on $B.$ That is,
\[
    q_0(u,v) := \sum_{j} \int_B \partial_j u(x)\overline{\partial_j v}(x)\,dx,\quad u,v \in C^\infty_c(B).
\]
The Dirichlet Laplacian $\Delta_0^D$ on $B$ is defined as the operator associated with the closure of the quadratic form $q_0$ (see e.g. \cite[Example 7.5.26]{Simon-course-IV}).
By the definitions of $q_g,$ $c_g$ and $C_g$ we have
\[
    c_gq_0(u,u) \leq q_g(u,u) \leq C_gq_0(u,u),\quad u\in C^\infty_c(B).
\]
It follows that the form domains of the closures of $q_0$ and $q_g$ coincide, we denote this space $H^1_0(B).$
The preceding inequality implies in particular that
\[
    c_g\|(1-\Delta_0^D)^{\frac12}u\|_{L_2(B)}^2 \leq \|(1-\Delta_g^D)^{\frac12}u\|_{L_2(B)}^2,\quad u\in H^1_0(B).
\]
By standard results in quadratic form theory (see e.g. \cite[Equation 7.5.29]{Simon-course-IV}), $1-\Delta_g^D$ defines a topological linear isomorphism from $H^1_0(B)$ to its dual $(H^1_0(B))^{^*}.$ Therefore, replacing $u$ with $(1-\Delta_g^D)^{-1}v$ for $v\in (H^1_0)^*,$ we arrive at
\[
    c_g\|(1-\Delta_0^D)^{\frac12}(1-\Delta_g^D)^{-1}v\|_{L_2(B)}^2 \leq \|(1-\Delta_g^D)^{-\frac12}v\|_{L_2(B)}^2,\quad v \in (H^1_0(B))^*.
\]
Replacing $v$ with $(1-\Delta_0^D)^{\frac12}w$ for $w \in L_2(B)$ gives
\[
    \|(1-\Delta_0^D)^{\frac12}(1-\Delta_g^D)^{-1}(1-\Delta_0)^{\frac12}w\|_{L_2(B)}^2 \leq c_g^{-1}\|w\|_{L_2(B)}^2.
\]
Recall that $\|(1-\Delta_0^D)^{-\frac12}\|_{\mathcal{L}_{d,\infty}(L_2(B)) }<\infty$ by standard Weyl asymptotics. Therefore,
\begin{align*}
    \|(1-\Delta_g^D)^{-1}\|_{\mathcal{L}_{\frac{d}{2},\infty}(L_2(B))} &\leq \|(1-\Delta_0^D)^{-\frac12}\|_{\mathcal{L}_{d,\infty}(L_2(B)) }^2\|(1-\Delta_0^D)^{\frac12}(1-\Delta_g^D)^{-1}(1-\Delta_0)^{\frac12}\|_{B(L_2(B))}\\
    &\leq c_g^{-\frac12}\|(1-\Delta_0^D)^{-\frac12}\|_{\mathcal{L}_{d,\infty}(L_2(B)) }^2.
\end{align*}
Defining $k_d = \|(1-\Delta_0^D)^{-\frac12}\|_{\mathcal{L}_{d,\infty}(L_2(B)) }^2$
completes the proof of \eqref{e:metric_uniformity},
and hence of the Lemma.
\end{proof}

\begin{prop}\label{P:UniformBound}
For all $q>0,$ there exists $K>0$ such that
\[
    \sup_{j} \|M_{\psi_j}(1-\Delta)^{-K}\|_{q,\infty} < \infty,
\]
where $\psi_j$ is the partition of unity subordinate to $\{B(x_j, r_0)\}_{j \in \mathbb{N}}$ mentioned in Section~\ref{S: Manifolds}.
\end{prop}
\begin{proof}
The proof is inspired by the proof of~\cite[Lemma~3.4.8]{SukochevZanin2018}. Let $V_j: L_2(M) \to L_2(B, g^{x_j})$ be the partial isometry mapping $\xi \in  L_2(M)$ to $V\xi (z) = f|_{B(x_j, r_0)} \circ \exp_x \circ e \in L_2(B, g^{x_j})$. Denote $V_j\psi_j := \phi_j$. Then, by construction,
\[
M_{\psi_j} = V_j^* M_{\phi_j} V_j
\]
and
\[
(1-\Delta)^K M_{\psi_j} = V_j^* (1-\Delta_{g^{x_j}})^K M_{\phi_j} V_j.
\]
It follows that
\begin{align*}
    (1-\Delta)^K M_{\psi_j} V_j^* (1-\Delta_{g^{x_j}}^D)^{-K} V_j &= V_j^* (1-\Delta_{g^{x_j}})^K M_{\phi_j} (1-\Delta_{g^{x_j}}^D)^{-K} V_j\\
    &=V_j^* [(1-\Delta_{g^{x_j}})^K, M_{\phi_j}] (1-\Delta_{g^{x_j}}^D)^{-K} V_j + V_j^*  M_{\phi_j}  V_j\\
    &= [(1-\Delta)^K, M_{\psi_j}]  V_j^* (1-\Delta_{g^{x_j}}^D)^{-K} V_j +  M_{\psi_j}.
\end{align*}
Multiplying both sides by $(1-\Delta)^{-K}$ and rearranging gives
\[
 (1-\Delta)^{-K}M_{\psi_j}  = M_{\psi_j} V_j^* (1-\Delta_{g^{x_j}}^D)^{-K} V_j - (1-\Delta)^{-K}[(1-\Delta)^K, M_{\psi_j}]  V_j^* (1-\Delta_{g^{x_j}}^D)^{-K} V_j.
\]
We claim that
\[
\sup_{j \in \mathbb{N}} \| (1-\Delta)^{-K}[(1-\Delta)^K, M_{\psi_j}] \|_\infty < \infty.
\]
For every $\alpha$ we have $\sup_{j,x} |\partial^{\alpha}\psi_j(x)|<\infty,$ where $\partial^{\alpha}$ is taken in the exponential normal coordinates of $B(x_j,r_0),$ and therefore $[(1-\Delta)^K, M_{\psi_j}]$ is a uniform differential operator of order $2K-1$ with coefficients that are uniform in $j$. Using~\cite[Theorem~3.9]{Kordyukov1991}, it follows that $[(1-\Delta)^K, M_{\psi_j}]$ is a bounded operator from $L_2(M)$ to the Sobolev space $\Hc^{1-2K}(M)$ with norm uniform in $j$. By~\cite[Proposition~4.4]{Kordyukov1991}, $(1-\Delta)^{-K}$ is a bounded operator from that space into $L_2(M)$, and so the claim holds.

Since the norm of $V_j$ is equal to one, via Lemma~\ref{l:uniformlaplacianbound} we get for $K$ large enough
\begin{align*}
    \sup_{j \in \mathbb{N}} \|(1-\Delta)^{-K}M_{\psi_j}\|_{q,\infty} &\leq  \sup_{j \in \mathbb{N}} \left(   \| M_{\psi_j} \|_\infty \cdot  \|(1-\Delta_{g^{x_j}})^{-K}\|_{q,\infty} \right)\\
    & \quad + \sup_{j \in \mathbb{N}} \left(\| (1-\Delta)^{-K}[(1-\Delta)^K, M_{\psi_j}] \|_\infty \cdot \|(1-\Delta_{g^{x_j}})^{-K}\|_{q,\infty} \right) \\
    &<\infty. \qedhere
\end{align*}
\end{proof}

It follows from this proposition that for every $q>0$ and every $j,$ there exists $K>0$ and a constant $C_K$ independent of $j$ such that
\begin{equation}\label{weak_bounded_cwikel_estimate}
    \|M_{f\psi_j}(1-\Delta)^{-K}\|_{q,\infty} \leq C_K\|f\|_{L_{\infty}(B(x_j,r_0))}.
\end{equation}

Let $\{\psi_{j}^{(0)}\}_{j=0}^\infty,$ $\{\psi_{j}^{(2)}\}_{j=0}^\infty$,$\ldots$, $\{\psi_{j}^{(N)}\}_{j=0}^\infty$ be a partition of $\{\psi_{j}\}_{j=0}^\infty$ into disjointly supported subfamilies. That is, for all $0\leq k\leq N,$ the functions $\{\psi_{j}^{(k)}\}_{j=0}^\infty$ are disjointly supported, and for every $j\geq 0$ there exists
a unique $0\leq k\leq N$ such that $\psi_{j} \in \{\psi_l^{(k)}\}_{l=0}^\infty.$

\begin{thm}\label{T: Cwikel estimate}
    Let $0<p<2.$ For sufficiently large $K,$ we have
    \begin{align*}
    \|M_f(1-\Delta)^{-K}\|_{p,\infty} &\leq C_{p,K,N}\|f\|_{\ell_{p,\infty}(L_\infty)};\\
        \|M_f(1-\Delta)^{-K}\|_{p} &\leq C_{p,K,N}'\|f\|_{\ell_{p}(L_\infty)}.
    \end{align*}
\end{thm}
\begin{proof}
    We prove the first inequality, the second can be proved analogously. Let $f \in \ell_{p,\infty}(L_\infty).$
    Since $\{\psi_j\}_{j=1}^\infty$ is a partition of unity, we have
    \[
        f = \sum_{j=0}^\infty \psi_jf = \sum_{k=0}^N \sum_{j=0}^\infty \psi_{j}^{(k)}f.
    \]
    By the quasi-triangle inequality, there exists $C_{N,p}$ such that
    \[
        \|M_f(1-\Delta)^{-K}\|_{p,\infty} \leq C_{N,p}\sum_{k=0}^N \left\| \sum_{j=0}^\infty M_{\psi_j^{(k)}f} (1-\Delta)^{-K}\right\|_{p,\infty}.
    \]
    The operators $\{M_{\psi_j^{(k)}f}(1-\Delta)^{-K}\}_{j=0}^\infty$ are left-disjoint. Hence, Corollary \ref{disjointification_corollary} implies that
    there exists $q>0$ such that
    \[
        \|M_f(1-\Delta)^{-K}\|_{p,\infty} \leq C_{N,p,q}\sum_{k=0}^N \left\| \{\|M_{f\psi_j^{(k)}}(1-\Delta)^{-K}\|_{q,\infty}\}_{j=0}^\infty\right\|_{\ell_{p,\infty}}
    \]
    From \eqref{weak_bounded_cwikel_estimate}, it follows that if $K$ is sufficiently large (depending on $q$), we have
    \[
        \|M_f(1-\Delta)^{-K}\|_{p,\infty}\leq C_{N,p,q}\left\| \{\|f\|_{L_{\infty}(B(x_j,r_0))}\}_{j=0}^\infty\right\|_{\ell_{p,\infty}}.
    \]
    The latter is the definition of the $\ell_{p,\infty}(L_\infty)$ quasinorm.
\end{proof}

\begin{cor} \label{C:Cwikel} Let $0<p<2$. Let $P$ be a self-adjoint, lower-bounded $P \in EBD^m(M)$. We have $\exp(-tP)M_f \in \mathcal{L}_{p,\infty}$ whenever $f\in \ell_{p,\infty}(L_\infty)$, and $\exp(-tP)M_f \in \mathcal{L}_{p}$ whenever $f\in \ell_{p}(L_\infty)$.
\end{cor}
\begin{proof}
By the preceding theorem, we already have that for $f\in \ell_{p,\infty}(L_\infty)$ and sufficiently large $K$, $M_f(1-\Delta)^{-K} \in \mathcal{L}_{p,\infty}$. Noting that $P+C>0$ and is therefore an invertible operator on $L_2(M)$ for some $C\in \mathbb{R}$, by Proposition~4.4 in~\cite{Kordyukov1991} it follows that $(P+C)^{-1}$ maps $\Hc^s(M)$ boundedly into $\Hc^{s+m}(M)$. By Theorem~3.9 in~\cite{Kordyukov1991}, $(1-\Delta)^K$ maps $\Hc^s(M)$ boundedly into $\Hc^{s-2K}(M)$. Therefore, we can find $N\in \mathbb{N}$ large enough such that $(P+C)^{-N}(1-\Delta)^K$ is a bounded operator on $L_2(M)$. Noting that $\exp(-tP) (P+C)^N$ is a bounded operator for any $N$ by the functional calculus on unbounded operators, the claim follows. The case for $f \in \ell_p(L_\infty)$ is proven analogously.
\end{proof}

This corollary will eventually make it possible to apply Theorem~\ref{MAIN_THEOREM} on $P$ and $W = M_w$ for $w \in \ell_{1,\infty}(L_\infty)$ defined by $w(x) = (1+|B(x_0,d_g(x, x_0))|)^{-1}$. We will also need to show that $\exp(-tP)[P,M_w] \in \mathcal{L}_1$, but this will require geometric conditions on the manifold $M$.

\begin{lem}\label{Grimaldi}
Let $(M,g)$ be a complete connected Riemannian manifold of bounded geometry. Then for any fixed $r\in \mathbb{R}$,
\[
\liminf_{R \to \infty} \frac{ \abs{B(x_0, R-r)}}{ \abs{B(x_0, R+r)}} > 0.
\]
\begin{proof}
The paper \cite{GrimaldiPansu2011} proves that for manifolds of bounded geometry, we have for fixed $R\geq 1$,
\[
\frac{1}{L} \leq \abs{B(x_0, R+2)}-\abs{B(x_0, R+1)} \leq L(\abs{B(x_0, R+1)} - \abs{B(x_0, R))}
\]
for some constant $L>0$ independent of $R$.

Now with this inequality, one can show by induction that
\[
\frac{\abs{B(x_0,R+k)}}{\abs{B(x_0,R)}} \leq (1+L)^k, \quad R \geq 1.
\] For $k=1$,
\begin{align*}
    \frac{\abs{B(x_0,R+1)}}{\abs{B(x_0,R)}} 
    &\leq 1 + L \frac{\abs{B(x_0,R)} - \abs{B(x_0,R-1)}}{\abs{B(x_0,R)}}\\
    & \leq 1 + L.
\end{align*}
Suppose that $\frac{\abs{B(x_0,R+k)}}{\abs{B(x_0,R)}} \leq (1+L)^k$ for some $k$, then \begin{align*}
     \frac{\abs{B(x_0,R+k+1)}}{\abs{B(x_0,R)}} &\leq \frac{\abs{B(x_0,R+k)}}{\abs{B(x_0,R)}} + L \frac{\abs{B(x_0,R+k)} - \abs{B(x_0,R+k-1)}}{\abs{B(x_0,R)}}\\
     &\leq (1+L)^k + L (1+L)^k = (1+L)^{k+1}.
\end{align*}

Therefore,
\begin{align*}
    \liminf_{R \to \infty} \frac{ \abs{B(x_0, R-r)}}{ \abs{B(x_0, R+r)}} &=  \liminf_{R \to \infty} \frac{ \abs{B(x_0, R)}}{ \abs{B(x_0, R+2r)}}\\
    &> \frac{1}{(1+L)^K},
\end{align*}
where $K$ is some integer greater than $2r$.
\end{proof}
\end{lem}

\begin{lem}\label{wl1infty}
Let $(M,g)$ be a complete connected Riemannian manifold of bounded geometry. Then the function
\[
    w(x) = (1+\abs{B(x_0,d_g(x,x_0))})^{-1},\quad x \in M,
\]
is an element of $\ell_{1,\infty}(L_\infty)(M).$
\begin{proof}
We denote $d_g(x,x_0)$ by $r(x)$. Note that for any $x \in B(x_k, r_0)$ we have $r(x) \geq \abs{x_k}-r_0$ by the triangle inequality, and hence
\[
\|w\|_{L_\infty(B(x_k,r_0))} \leq (1+\abs{B(x_0,r(x_k)-r_0)})^{-1}.
\]
Without loss of generality, order the $x_j$ such that $r(x_1) \leq r(x_2) \leq \dots$. 
Note that as in Remark \ref{R: metric balls lower bound}, we have that $\inf_{j \in \mathbb{N}} \abs{B(x_j, r_0)} > 0$ and hence
\[
    \abs{B(x_0,\abs{x_k}-r_0)} = \inf_{j \in \mathbb{N}} \abs{B(x_j, r_0)} \cdot \frac{ \abs{B(x_0, \abs{x_k}-r_0)}}{ \abs{B(x_0, \abs{x_k}+r_0)}} \cdot \frac{\abs{B(x_0, \abs{x_k}+r_0)}}{\inf_{j \in \mathbb{N}} \abs{B(x_j, r_0)}}.
\]
By the ordering of the $x_j$, we know that all the balls $B(x_j, r_0)$ from $j=1$ up to and including $j=k$ are contained in $B(x_0, \abs{x_k}+r_0)$.
Hence \begin{align*}
    k \cdot \inf_{m\in \mathbb{N}} \abs{B(x_m, r_0)} & \leq \sum_{j=1}^k \abs{B(x_j, k)}\\
    &\leq (N+1) \abs{B(x_0, \abs{x_k}+r_0)},
\end{align*}
since any ball can only intersect at most $N$ other balls. We thus get
\[
\abs{B(x_0, \abs{x_k}+r_0)} \geq \frac{k}{N+1}\inf_{j \in \mathbb{N}} \abs{B(x_j, r_0)}.
\]
We will now show that for $k$ large enough, $\frac{ \abs{B(x_0, \abs{x_k}-r_0)}}{ \abs{B(x_0, \abs{x_k}+r_0)}}$ is bounded below away from zero. By Lemma~\ref{Grimaldi} we have
\[
\liminf_{R \to \infty} \frac{ \abs{B(x_0, R-r_0)}}{ \abs{B(x_0, R+r_0)}} > 0,
\]
and so there must be some $R_0$ such that for $R \geq R_0$ we have $\frac{ \abs{B(x_0, R-r_0)}}{ \abs{B(x_0, R+r_0)}} > \delta > 0.$ We claim that we can take $K$ large enough such that $\abs{x_k} \geq R_0$ for $k \geq K$. Indeed, by analogous reasoning as before, at most $K:= (N+1)\frac{\abs{B(x_0, R_0+r_0)}}{\inf_j \abs{ B(x_j, r_0)}}$ points $x_j$ can be inside the ball $\abs{B(x_0, R_0)}$, thus $\abs{x_k} > R_0$ for $k > K$.

Gathering the results above, we get the existence of some constant $C$ such that for $k \geq K$ we have
\[
\abs{B(x_0,\abs{x_k}-r_0)} \geq C k.
\]
This means that
\[
\|w\|_{L_\infty(B(x_k,r_0))} \leq (1+C k)^{-1},
\]
proving that $w \in \ell_{1,\infty}(L_\infty)(M).$
\end{proof}
\end{lem}

We will now mold Property (D) into the form that we will apply it in.
\begin{lem}\label{L:NewPropD}
Let $(M,g)$ be a complete connected Riemannian manifold of bounded geometry. If $M$ has Property (D), that is if $V(r):=|B(x_0,r)| \in C^2(\R_{\geq 0})$ satisfies
\begin{equation}\label{eq3:Condition1}
    \frac{V'(r)}{V(r)} \in L_2(\R_{\geq 1}),
\end{equation}
and
\begin{equation}\label{eq3:Condition2}
\frac{V''(r)}{V'(r)} \to 0, \quad r \to \infty,
\end{equation}
then for every $h > 0$,
\begin{equation}
    \label{eq:Condition1Mod}
\left \{ \frac{\sup_{s \in [0, h]}\abs{\partial B(x_0, k+1+s)}}{\abs{B(x_0,k+1)}}\right\}_{k\in \mathbb{N}} \in \ell_2(\mathbb{N}).
\end{equation}
\end{lem}
\begin{proof}
    Using condition~\eqref{eq3:Condition2}, let $\varepsilon >0$ and choose $R$ large enough so that $r \geq R$ implies that
    \[
    V''(r) \leq \varepsilon V'(r).
    \]
    For $r \geq R$, $\delta > 0$, we have
    \begin{align*}
        V'(r+\delta) &= V'(r) + \int_r^{r+\delta} V''(s)\,ds\\
        &\leq V'(r) + \varepsilon \delta \sup_{s\in [0,\delta]} V'(r+s).
    \end{align*}
    Taking the supremum over $\delta \in [0,h]$ on both sides and rearranging gives
    \[
     (1-\varepsilon h) \sup_{s\in [0,h]} V'(r+s) \leq V'(r).
    \]
    In particular, this implies that
    \begin{align*}
        \inf_{s\in [0,h]} V'(k+s) \geq  \inf_{s\in [0,h]} (1-\varepsilon h)\sup_{t\in[0,h]} V'(k+s+t) \geq (1-\varepsilon h) V'(k+h).
    \end{align*}
    Now choose $\varepsilon < \min(1,\frac{1}{h})$, and we can estimate the $L_2$-norm of $\frac{V'(r)}{V(r)}$ from below by a Riemann sum, 
    \begin{align*}
        \int_1^\infty \bigg(\frac{V'(r)}{V(r)} \bigg)^2\,dr & \geq \sum_{k=1}^\infty \frac{\big(\inf_{s\in [0,1]}V'(k+s)\big)^2}{\big(\sup_{s\in [0,1]} V(k+s)\big)^2}\\
        &\geq C + \sum_{k=1}^\infty(1-\varepsilon)^2  \frac{V'(k+1)^2}{V(k+1)^2}\\
        &\geq C + (1-\varepsilon)^2(1-\varepsilon h)^2\sum_{k=1}^\infty \frac{\sup_{s\in[0,h]}V'(k+1+h)^2}{V(k+1)^2}.
    \end{align*}
    This shows that~\eqref{eq3:Condition1} and~\eqref{eq3:Condition2} together imply~\eqref{eq:Condition1Mod}
\end{proof}

\begin{lem}\label{L:Property_D}
    Let $(M,g)$ be a non-compact Riemannian manifold with bounded geometry and Property (D). Then \begin{equation}\label{Condition1}
 \sum_{k=0}^\infty  \frac{\sup_{s \in [-1, 2]}\abs{\partial B(x_0, (k+s)r_0)}^2}{(1+\abs{B(x_0,(k-1)r_0)})^2} <\infty.
\end{equation}
\end{lem}
\begin{proof}
    Since $M$ has Property (D), Lemma~\ref{L:NewPropD} gives that
    \[
    \sum_{k=1}^\infty \frac{\sup_{s \in [0, 3]}\abs{\partial B(x_0, k+s)}^2}{\abs{B(x_0,k)}^2} < \infty.
    \]
    Recall that we can assume $r_0 \leq 1$. Then,
    \begin{align*}
        \sum_{k=0}^\infty  &\frac{\sup_{s \in [-1, 2]}\abs{\partial B(x_0, (k+s)r_0)}^2}{(1+\abs{B(x_0,(k-1)r_0)})^2} \\
        & = \sum_{N=0}^\infty \sum_{N \leq kr_0 \leq N+1}  \frac{\sup_{s \in [-1, 2]}\abs{\partial B(x_0, (k+s)r_0)}^2}{(1+\abs{B(x_0,(k-1)r_0)})^2}\\
        &\leq \sum_{N=0}^\infty \sum_{N \leq kr_0 \leq N+1}  \frac{\sup_{s \in [-1, 2]}\abs{\partial B(x_0, N+s)}^2}{(1+\abs{B(x_0,N-1)})^2}\\
        &\leq \left\lceil \frac{1}{r_0}\right\rceil \left(2\cdot \sup_{s \in [0, 3]}\abs{\partial B(x_0, s)}^2+ \sum_{k=1}^\infty \frac{\sup_{s \in [0, 3]}\abs{\partial B(x_0, k+s)}^2}{\abs{B(x_0,k)}^2}\right)\\
        &<\infty. \qedhere
    \end{align*}
\end{proof}

Note that in the next lemma, we do not need uniform ellipticity of the differential operators considered, although we will consider operators $L \in BD^2(M)$ which lack a constant term. The meaning of this condition
is that in a system of normal coordinates $(U,\phi),$ where $U = B(x,r_0)$ we have
\[
    L = \sum_{0<|\alpha|\leq 2} a_{\alpha,x}(y)D^{\alpha}.
\]
Equivalently, $L1 = 0.$ The important feature of these operators is that if $P \in BD^2(M)$ and $f$ is a smooth function with uniformly bounded derivatives, then $[P,M_f] = [L, M_f]$ for a differential operator $L\in BD^2(M)$ with no constant term.

\begin{lem}\label{L:Dwl1}
Let $(M,g)$ be a non-compact Riemannian manifold with bounded geometry satisfying Property (D) (Definition~\ref{D:Property_D}), and take $w:M \to \R$ as in Lemma~\ref{wl1infty}. Then $Lw \in \ell_1(L_\infty)(M)$ for any $L \in BD^2(M)$ that lacks a constant term.
\end{lem}
\begin{proof}
For any $x \in M$ we can take a neighbourhood of normal coordinates $(U, \phi)$  in which $L$ takes the following form by definition:
\[
\sum_{|\alpha|= 1,2} a_{\alpha, x}(y) \partial_y^\alpha,
\]
where for any multi-index $\beta$
\[
\abs{\partial_y^\beta a_{\alpha, x}(0)} \leq C_{\alpha, \beta}, \quad x \in M.
\]
Now denote $r(x) := d_g(x_0,x)$, $\tilde{w}(r) = (1+|B(x_0,r)|)^{-1}$ so that $w = \tilde{w} \circ r$. Note that $\tilde{w}$ is a $C^2$ function on $[0,\infty)$ by assumption (Property (D)), and the combination of~\cite[Lemma~III.4.4]{Sakai1996} and~\cite[Proposition~III.4.8]{Sakai1996} gives that $r$ is smooth almost everywhere with $\norm{\nabla r} \leq 1$ almost everywhere. We therefore have almost everywhere
\begin{align*}
\abs{Lw(x)}  & \leq \Big | \sum_{|\alpha|= 1} a_{\alpha, x}(0) (\partial^\alpha (\tilde{w}\circ r \circ \phi^{-1}))(0) \Big | + \Big | \sum_{\substack{|\alpha|=2\\\alpha = \beta + \gamma\\|\beta|=|\gamma|= 1}} a_{\alpha, x}(0) (\partial^\beta \partial^\gamma (\tilde{w}\circ r \circ \phi^{-1}))(0) \Big |\\
&= \Big | \sum_{|\alpha|= 1} a_{\alpha, x}(0) (\partial^\alpha (r \circ \phi^{-1}))(0) \tilde{w}'(r(x)) \Big |\\
&\quad + \Big | \smash{\sum_{\substack{|\alpha|=2\\\alpha = \beta + \gamma\\|\beta|=|\gamma|= 1}} a_{\alpha, x}(0) \big( \tilde{w}''(r(x))  \cdot \partial^\beta (r \circ \phi^{-1})(0) \cdot \partial^\gamma (r \circ \phi^{-1})(0)} \\
&\quad \quad \quad \quad \quad \quad \quad \quad \quad \quad \quad \quad \quad \quad \quad \quad \quad \quad \quad \quad  + \tilde{w}'(r(x)) \cdot \partial^\beta \partial^\gamma (r \circ \phi^{-1})(0)\big) \Big |\\
&\leq \Big | \sum_{|\alpha|= 1} a_{\alpha, x}(0) \Big | \cdot |\tilde{w}'(r(x))| + \Big |\sum_{\substack{|\alpha|=2\\\alpha = \beta + \gamma\\|\beta|=|\gamma|= 1}} a_{\alpha,x}(0) \Big|\cdot |\tilde{w}''(r(x))| \\
&\quad + C\Big |\sum_{\substack{|\alpha|=2\\\alpha = \beta + \gamma\\|\beta|=|\gamma|= 1}} a_{\alpha,x}(0) \Big|\cdot |\tilde{w}'(r(x))| \\
&\leq (1+C)\cdot\big( |\tilde{w}'(r(x))|+|\tilde{w}''(r(x))|\big),
\end{align*}
where we have used the chain rule and
\[
|(\partial^\alpha (r \circ \phi^{-1}))(0)| = \abs{\frac{\partial}{\partial x^\alpha}\Big|_x (r)} \leq \norm{\nabla r(x)} \leq 1
\]
(because  $(\nabla r)_\alpha = \sum_{\beta}g_{\alpha \beta} \partial^\beta r$ and in normal coordinates at zero, $g_{\alpha,\beta}(0) = \delta_{\alpha,\beta}$)
in addition to
\[
|\partial^\beta \partial^\gamma (r \circ \phi^{-1})(0)| \leq \| \mathrm{Hess}\  r \| \leq C
\]
for some constant $C$, since the Hessian of $r$ is uniformly bounded~\cite[Theorem~6.5.27]{Petersen2006}.

By continuity, we therefore have everywhere
\[
\abs{Lw(x)} \leq  (1+C)\cdot(|\tilde{w}'(r(x))| + |\tilde{w}''(r(x))|).
\]
Therefore
\begin{align*}
    \norm{Lw}_{L_\infty(B(x_j, r_0))} &\leq (1+C)\cdot \big(\sup_{s \in [-r_0, r_0]} \abs{\tilde{w}'(r(x_j)+s)} + \abs{\tilde{w}''(r(x_j)+s)} \big)\\
    &\leq (1+C)\cdot \Bigg(\sup_{s \in [-r_0, r_0]} \frac{\abs{\partial B(x_0, r(x_j)+s)}}{(1+\abs{B(x_0,r(x_j)+s)})^2} \\
    & \quad + \sup_{s \in [-r_0, r_0]} \frac{\frac{d}{dR}\Big|_{R= r(x_j)+s}\abs{\partial B(x_0, R)}}{(1+\abs{B(x_0,r(x_j)+s)})^2} \\
    & \quad + 2 \sup_{s \in [-r_0, r_0]} \frac{\abs{\partial B(x_0, r(x_j)+s)}^2}{(1+\abs{B(x_0,r(x_j)+s)})^3}\Bigg).
\end{align*}

Next, we calculate
\begin{align*}
    \sum_{j\in \mathbb{N}} &\sup_{s \in [-r_0, r_0]} \frac{\abs{\partial B(x_0, r(x_j)+s)}}{(1+\abs{B(x_0,r(x_j)+s)})^2} \\
    &=  \sum_{k=0}^\infty \sum_{\{j: kr_0 \leq r(x_j)< (k+1)r_0\}} \sup_{s \in [-r_0, r_0]} \frac{\abs{\partial B(x_0, r(x_j)+s)}}{(1+\abs{B(x_0,r(x_j)+s)})^2}\\
    &\leq  \sum_{k=0}^\infty \abs{\{j: k r_0 \leq r(x_j)< (k+1)r_0\}}\cdot \sup_{l \in [-1, 2]} \frac{\abs{\partial B(x_0, (k+l)r_0}}{(1+\abs{B(x_0,(k+l)r_0)})^2}.
\end{align*}

With a similar calculation as in the proof of Lemma~\ref{wl1infty}, we have
\begin{align*}
    \inf_{m \in \mathbb{N}} \abs{B(x_m, r_0)} &\cdot \abs{\{j: k r_0 \leq r(x_j)< (k+1)r_0\}}\\
    &= \sum_{\{j : k r_0 \leq r(x_j)< (k+1)r_0\}} \inf_{m \in \mathbb{N}}\abs{B(x_m, r_0)}\\
    &\leq   \sum_{\{j : k r_0 \leq r(x_j)< (k+1)r_0\}} \abs{B(x_j, r_0)}\\
    &\leq (N+1) \big(\abs{B(x_0, (k+2)r_0)} - \abs{B(x_0, (k-1)r_0)}\big)\\
    &\leq 3(N+1)r_0 \sup_{l\in [-1,2]} \abs{\partial B(x_0, (k+l)r_0)},
\end{align*}
since all the balls $B(x_j, r_0)$ with $kr_0 \leq r(x_j)\leq (k+1)r_0$ are contained in the annulus $B(x_0, (k+2)r_0)\setminus B(x_0, (k-1)r_0)$, and balls can intersect at most $N$ other balls.

Using Lemma~\ref{L:Property_D} we can infer that the expression \eqref{Condition1} is finite, and so
\begin{align*}
    \sum_{j\in \mathbb{N}} \sup_{s \in [-r_0, r_0]} \frac{\abs{\partial B(x_0, r(x_j)+s)}}{(1+\abs{B(x_0,r(x_j)+s)})^2}
    &\leq C' \sum_{k=0}^\infty  \frac{\sup_{l \in [-1, 2]}\abs{\partial B(x_0, (k+l)r_0}^2}{(1+\abs{B(x_0,(k-1)r_0)})^2} < \infty.
\end{align*}

With analogous calculations,
we also have that
\[
\sum_{j\in \mathbb{N}}\sup_{s \in [-r_0, r_0]} \frac{\frac{d}{dR}\Big|_{R= r(x_j)+s}\abs{\partial B(x_0, R)}}{(1+\abs{B(x_0,r(x_j)+s)})^2} <\infty
\]
and
\[
\sum_{j \in \mathbb{N}}\sup_{s \in [-r_0, r_0]} \frac{\abs{\partial B(x_0, r(x_j)+s)}^2}{(1+\abs{B(x_0,r(x_j)+s)})^3} < \infty.
\]
We conclude that \[\|Lw\|_{\ell_1(L_\infty)} = \sum_{j\in \mathbb{N}} \| Lw \|_{L_{\infty}(B(x_j,r_0))} < \infty.\qedhere\]
\end{proof}

\begin{cor}\label{C:CwikelL1}
Let $P\in EBD^2(M)$ be a self-adjoint lower-bounded operator. Let $M$ be a non-compact Riemannian manifold with Property (D), and take $w:M \to \R$ as in Lemma~\ref{wl1infty}. Then $\exp(-tP)[P, M_w] \in \mathcal{L}_1$.
\end{cor}
\begin{proof}
Take $P \in EBD^2(M).$ Similarly to the proof of Corollary~\ref{C:Cwikel}, since $\exp(-tP)(1-\Delta)^K$ is a bounded operator for all $K$, it suffices to prove that $(1-\Delta)^{-K-1}[P,M_w] \in \mathcal{L}_1$ for $K$ large enough.

First note that
\begin{align*}
\|(1-\Delta)^{-K-1}[P,M_w]\|_1 &\leq \sum_{j \in \mathbb{N}} \|(1-\Delta)^{-K-1}[P,M_w] M_{\psi_j}\|_1.
\end{align*}
Using the expression $P = \sum_{|\alpha|\leq 2} a_{\alpha,x_j}(y)\partial_y^{\alpha}$ in the coordinate chart $B_j,$ we have for $K \geq 0$
\begin{align*}
\|(1-\Delta)^{-K-1}[P,M_w] M_{\psi_j}\|_1 &\leq \|(1-\Delta)^{-K-1}\sum_{|\alpha|\leq 2} [M_{a_{\alpha,x_j}} \partial^\alpha,M_w] M_{\psi_j}\|_1\\
&\leq \|(1-\Delta)^{-K-1} M_{\psi_j} M_{D w}\|_1\\
& \quad + \Bigg\|(1-\Delta)^{-K-1}\sum_{\substack{|\alpha|=2\\\alpha = \beta + \gamma\\|\beta|=|\gamma|= 1}} M_{a_{\alpha,x_j}} (\partial^\gamma M_{\partial^\beta w}+\partial^\beta M_{\partial^\gamma w}) M_{\psi_j}\Bigg\|_1\\
&\leq \|(1-\Delta)^{-K-1} M_{\psi_j} M_{D w}\|_1\\
&\quad + \Bigg\|(1-\Delta)^{-K-1}\sum_{|\beta|=1}\sum_{|\gamma|=1} M_{a_{\beta+\gamma,x_j}} \partial^\gamma M_{\partial^\beta w} M_{\psi_j}\Bigg\|_1\\
&\quad + \Bigg\|(1-\Delta)^{-K-1}\sum_{|\beta|=1} M_{a_{2\beta,x_j}} \partial^\beta M_{\partial^\beta w} M_{\psi_j}\Bigg\|_1\\
&=: I + II + III,
\end{align*}
where we have denoted $D$ for the differential operator given near $x_j$ by
\[
    D = \sum_{1\leq|\alpha|\leq 2} a_{\alpha,x_j}(y)\partial_y^{\alpha}.
\]
That is, $D$ is equal to $P$ without constant terms. Since $\psi_j$ is by definition supported in $B_j,$ we have
\begin{align*}
I &= \|(1-\Delta)^{-K-1} M_{\psi_j} M_{D w}\|_1\\
&= \|(1-\Delta)^{-K-1} M_{\psi_j} M_{\chi_{B_j}}M_{D w}\|_1\\
&\leq \|(1-\Delta)^{-K-1} M_{\psi_j}\|_{1} \| M_{\chi_{B_j}}M_{D w} \|_\infty\\
&=\|(1-\Delta)^{-K-1} M_{\psi_j}\|_{1} \| Dw \|_{L_\infty (B_j)}.
\end{align*}
By Proposition~\ref{P:UniformBound}, $\|(1-\Delta)^{-K-1} M_{\psi_j}\|_1$ is uniformly bounded in $j$ for $K$ large enough, and by Lemma~\ref{L:Dwl1}, $\{\| Dw \|_{L_\infty (B_j)} \}_{j\in \mathbb{N}}\in \ell_1$.

\begin{align*}
II &= \|(1-\Delta)^{-K-1}\sum_{|\beta|=1}\sum_{|\gamma|=1} M_{a_{\beta+\gamma,x}} \partial^\gamma M_{\partial^\beta w} M_{\psi_j}\|_1\\
&\leq \|\sum_{|\beta|=1}\sum_{|\gamma|=1} (1-\Delta)^{-K-1} [M_{a_{\beta+\gamma,x}}, \partial^\gamma] M_{\chi_{B_j}} M_{\partial^\beta w} M_{\psi_j}\|_1\\
& \quad + \|\sum_{|\beta|=1}\sum_{|\gamma|=1} (1-\Delta)^{-K-1} \partial^\gamma M_{\chi_{B_j}}M_{a_{\beta+\gamma,x}}  M_{\partial^\beta w} M_{\psi_j}\|_1\\
&=  \|\sum_{|\beta|=1}\sum_{|\gamma|=1} (1-\Delta)^{-K-1} M_{\partial^\gamma a_{\beta+\gamma,x}} M_{\chi_{B_j}} M_{\partial^\beta w} M_{\psi_j}\|_1\\
& \quad + \|\sum_{|\beta|=1}\sum_{|\gamma|=1} (1-\Delta)^{-K-1} \partial^\gamma M_{\chi_{B_j}}M_{a_{\beta+\gamma,x}}  M_{\partial^\beta w} M_{\psi_j}\|_1\\
&\leq \sum_{|\beta|=1}\sum_{|\gamma|=1} \| (1-\Delta)^{-K-1} M_{\psi_j}\|_{1} \|\partial^\gamma a_{\beta+\gamma,x}\|_{L_\infty (B_j)} \| \partial^\beta w \|_{L_\infty (B_j)}\\
& \quad + \sum_{|\beta|=1}\sum_{|\gamma|=1} \| (1-\Delta)^{-1} \partial^\gamma M_{\chi_{B_j}} \|_{\infty} \| (1-\Delta)^{-K}M_{\psi_j} \|_1 \| a_{\beta+\gamma,x}\|_{L_\infty (B_j)} \| \partial^\beta w\|_{L_\infty(B_j)}.
\end{align*}
Suppose $y$ is in the normal neighbourhood $(U_x,\phi_x)$ of $x$. Then, $D$ being a local operator, at $y$ (taking $\tilde{y} = \phi_x(y)$), there are two different expressions of $D$ in normal coordinates:
\begin{align*}
D &= \sum_{|\alpha| \leq 2} a_{\alpha, x}(\tilde{y}) \partial_x^\alpha\\
&= \sum_{|\beta|\leq 2} a_{\beta, y}(0) \partial_y^\beta.
\end{align*}
Hence we can express $a_{\alpha, x}(\tilde{y})$ in terms of $a_{\beta, y}(0)$ and transition functions $\frac{\partial^\alpha y}{\partial x^\alpha}$. These are uniformly bounded by the definition of a differential operator with bounded coefficients and~\cite[Proposition~1.3]{Kordyukov1991}. Therefore $\|a_{\beta+\gamma,x}\|_{L_\infty (B_j)} $ and also $\|\partial^\gamma a_{\beta+\gamma,x}\|_{L_\infty (B_j)} $ are uniformly bounded in $j$. Likewise (taking $|\beta| = 1$),

\begin{align*}
|\partial^\beta_x w (y)| &= \Big|\sum_{|\alpha|=1} \frac{\partial^\alpha y}{\partial x^\alpha} \partial^\alpha_y w(y) \Big|\\
&\leq C_d \| \nabla r\| |\tilde{w}'(y)|,
\end{align*}
and hence $\{\| \partial^\beta w \|_{L_\infty(B_j)}\}_{j \in \mathbb{N}} \in \ell_1$ by the arguments in the proof of Lemma~\ref{L:Dwl1}.

Finally,
\begin{align*}
\| (1-\Delta)^{-1} \partial^\gamma M_{\chi_{B_j}} \|_{\infty} &\leq \sup_{ |\alpha|=1} \|\frac{\partial^\alpha y}{\partial x^\alpha}\|_{L_{\infty}(B_j)} \| (1-\Delta)^{-1} |\nabla| \|_\infty < \infty,
\end{align*}
uniformly in $j$. The same estimates hold for $III$.

Combining everything, we have
\begin{align*}
\|(1-\Delta)^{-K-1}[P,M_w]\|_1 &\leq \sum_{j \in \mathbb{N}} \|(1-\Delta)^{-K-1}[P,M_w] M_{\psi_j}\|_1 < \infty. \qedhere
\end{align*}
\end{proof}

Gathering all results in this section, let $M$ be a non-compact manifold of bounded geometry with Property (D). Let $P\in EBD^2(M)$ be self-adjoint and lower-bounded. Then $\exp(-tP)M_w \in \mathcal{L}_{1,\infty}$ by the Cwikel estimate in Corollary~\ref{C:Cwikel} and Lemma~\ref{wl1infty}. Corollary~\ref{C:CwikelL1} states that $\exp(-tP)[P,M_w] \in \mathcal{L}_1$. Theorem~\ref{MAIN_THEOREM} then gives that \[
        \mathrm{Tr}_{\omega}(e^{-tP}M_w) = \lim_{\varepsilon\to 0} \varepsilon \mathrm{Tr}(e^{-tP}\chi_{[\varepsilon,\infty)}(M_w)),
    \]
if the limit on the right hand side exists. If we assume that $P$ admits a density of states, we do in fact get the existence of the limit
\begin{align*}
\int_{\mathbb{R}} e^{-t\lambda} \, d\nu_P(\lambda)&=\lim_{R\to\infty} \frac{1}{|B(x_0,R)|}\mathrm{Tr}(e^{-tP}M_{\chi_{B(x_0,R)}})\\
 & = \lim_{\varepsilon\to 0} \varepsilon \mathrm{Tr}(e^{-tP}\chi_{[\varepsilon,\infty)}(M_w)).
\end{align*}
Note that the above calculation assumes that the volume of $M$ is infinite, which is equivalent with $M$ being non-compact (see Remark~\ref{R: metric balls lower bound}).

Hence, 
\[
\mathrm{Tr}_{\omega}(e^{-tP}M_w) = \int_{\mathbb{R}} e^{-t\lambda} \, d\nu_P(\lambda),\quad t>0.
\]
From this we can easily deduce by a density argument that
\[
\mathrm{Tr}_{\omega}(f(P)M_w) = \int_{\mathbb{R}} f \, d\nu_P, \quad f \in C_c(\mathbb{R}).
\]
For details, see \cite[Remark 6.3]{AMSZ}.
This concludes the proof of Theorem~\ref{T: main manifold thm}.

\section{Roe's index theorem}\label{S:Roe}

The results of the preceding section were stated for operators $P$ acting on scalar valued functions on $M.$ Identical results, with the same proofs, apply to operators acting between sections of vector bundles of bounded geometry. In the terminology of Shubin \cite{Shubin1992}, a rank $N$ vector bundle $\pi:S\to M$ is said to have bounded geometry if in every coordinate chart $\{B(x_j,r_0)\}_{j=0}^\infty$ (as defined in Section \ref{S: Manifolds}) $E$ has a trivialisation \[
    \pi^{-1}(B(x_j,r_0)) \approx B(x_j,r_0)\times \mathbb{R}^N
\]
such that the transition functions
\[
    t_{j,k}:B(x_j,r_0)\times \mathbb{R}^N\cap B(x_k,r_0)\times \mathbb{R}^N \to B(x_j,r_0)\times \mathbb{R}^N\cap B(x_k,r_0)\times \mathbb{R}^N
\]
have uniformly bounded derivatives in the exponential normal coordinates around $x_j$ or $x_k.$ See also~\cite[Page 65]{Eichhorn2008}.

For our purposes we will assume that $S\to M$ is equipped with a Hermitian metric $h,$ which is assumed to be a $C^\infty$-bounded section of the bundle $\overline{S}\otimes S$ in the terminology of Shubin~\cite[Appendix 1]{Shubin1992}. We denote $L_2(S)$ for the Hilbert space of square integrable sections of $S$ with respect to the volume form of $M$ and the Hermitian metric~$h$.

Shubin defines elliptic differential operators
acting on sections of vector bundles of bounded geometry.
Given a vector bundle $S$ of bounded geometry, define $EBD^m(M,S)$ as the space of differential operators $P$ such that in the exponential normal coordinates $y$ around $x\in M,$ we have
\[
    P = \sum_{|\alpha|\leq m} a_{\alpha,x}(y)\partial_y^{\alpha}
\]
where $a_{\alpha,x}(y)$ are $N\times N$ matrices, identified with sections of $\mathrm{End}(S)$ in the local trivialisation of $S$ and a synchronous frame, and
\[
    \|\partial_y^{\beta} a_{\alpha,x}(0)\| \leq C_{\alpha,\beta},\quad |\alpha|\leq m
\]
where $\|\cdot\|$ is the norm on the fibre $\mathrm{End}(S)_x$ defined by the Hermitian metric $h.$

Given such an operator $P \in EBD^2(M,S),$ we say by analogy with the scalar-valued case that $P$ has a density of states $\nu_P$ if for every $t$ there exists the limit
\[
    \lim_{R\to\infty} \frac{1}{|B(0,R)|}\mathrm{Tr}(e^{-tP}M_{\chi_{B(0,R)}}) = \int_{\mathbb{R}} e^{-t\lambda}\,d\nu_P(\lambda).
\]
Here, the trace is now with respect to the Hilbert space $L_2(S)$ rather than $L_2(M).$ A verbatim repetition of the proof of Theorem~\ref{T: main manifold thm} shows the following.
\begin{thm}\label{T: vector bundle dos thm}
    Let $S\to M$ be a vector bundle of bounded geometry over a non-compact Riemannian manifold of bounded geometry with Property (D). If $P \in EBD^2(M,S)$ is a self-adjoint lower bounded operator having a density of states $\nu_P,$ then for $f\in C_c(\mathbb{R})$ we have
    \[
        \mathrm{Tr}_{\omega}(f(P)M_w) = \int_\mathbb{R} f(\lambda)\,d\nu_P(\lambda)
    \]
    where $w(x) = (1+|B(x_0,d(x,x_0))|)^{-1}.$
    Similarly, for all $t>0$ we have
    \[
        \mathrm{Tr}_{\omega}(\exp(-tP)M_w) = \int_\mathbb{R} \exp(-t\lambda)\,d\nu_P(\lambda).
    \]
\end{thm}

In~\cite{Roe1988a}, Roe considers (orientable) manifolds of bounded geometry that have a regular exhaustion. In this section, we will only consider manifolds that satisfy the assumptions of Theorem~\ref{T: main manifold thm}. In particular, the assumptions imply that $\lim_{R\to \infty}\frac{|\partial B(x_0, R)|}{|B(x_0, R)|} = 0$, i.e. any increasing sequence of metric balls $\{B(x_0, R_i)\}_{i=0}^\infty$ where $R\to\infty$ forms a regular exhaustion.

Denote the Banach space of $C^1$ uniformly bounded $n$-forms on $M$ by $\Omega^n_{\beta}(M)$. An element $m$ in the dual space of $\Omega^n_{\beta}(M)$ is said to be associated with the regular exhaustion $\{B(x_0, R_i)\}$ if for each bounded $n$-form~$\alpha$ \[\liminf_{i\to \infty} \bigg| \langle \alpha, m \rangle - \frac{1}{|B(x_0, R_i)|} \int_{B(x_0, R_i)} \alpha\bigg| = 0.\]
The algebra $\mathcal{U}_{-\infty}(M)$ consists of operators $A: C_c^\infty(M) \to C_c^\infty(M)$ such that for each $s,k\in \mathbb{R},$ $A$ has a continuous extension to a quasilocal operator from $\Hc^s(M) \to \Hc^{s-k}(M)$. In Roe's terminology, an operator $A: \Hc^s(M) \to \Hc^{s-k}(M)$ is quasilocal if for each $K \subset M$ and each $u\in \Hc^k(M)$ supported within $K$, \[\| Au\|_{\Hc^{s-k}(M\setminus \text{Pen}^+(K,r))} \leq \mu(r) \| u \|_{\Hc^s(M)},\] where $\mu: \mathbb{R}^+ \to \mathbb{R}^+$ is a function such that $\mu(r) \to 0$ as $r \to \infty$, and $\text{Pen}^+(K,r)$ is the closure of $\bigcup \{B(x,r) : x\in K\}$.

Operators $A \in \mathcal{U}_{-\infty}(M)$ are represented by uniformly bounded smoothing kernels 
\[
Au(x) = \int k_A(x,y)u(y) \, d\nu_g(y),
\]
where $\nu_g$ is the Riemannian volume form (recall that $M$ is assumed to be oriented). 
Roe then defines traces on $\mathcal{U}_{-\infty}(M)$ coming from functionals associated with our regular exhaustion by 
\[
\tau(A) = \langle \alpha_A, m\rangle,
\] where $\alpha_A$ is the bounded $n$-form defined by $x\to k_A(x,x)\, d\nu_g$.

Recall from the introduction of this chapter that the trace $\tau$ extends to a trace on $M_n(\mathcal{U}_{-\infty}^+)$, and hence descends to a dimension-homomorphism
\[
\dim_\tau: K_0(\mathcal{U}_{-\infty}) \to \mathbb{R}.
\]
Furthermore, since Roe showed that elliptic differential operators are invertible modulo $\mathcal{U}_{-\infty}$~\cite{Roe1988a}, one can define an abstract index of an elliptic differential operator acting on sections of a Clifford bundle as an element of $K_0(\mathcal{U}_{-\infty})$ via standard $K$-theory constructions. The Roe index theorem then states the following.
\begin{thm}[Roe index theorem]
Let $M$ be a Riemannian manifold, $S$ a graded Clifford bundle on $M$, both with bounded geometry. Let $D$ be the Dirac operator of $S$. Let $m$ and $\tau$ be defined as above. Then
\[
    \dim_\tau (\operatorname{Ind}(D)) = m(\mathbf{I}(D)),
\]
where $\mathbf{I}(D)$ is the integrand in the Atiyah--Singer index theorem.
\end{thm}
The basis for the proof of the theorem is the following McKean--Singer formula
\[
    \dim_{\tau}(\operatorname{Ind}(D)) = \tau(\eta e^{-tD^2})
\]
where $\eta$ is the grading operator on $S,$ and $\tau$ is now defined on operators on acting on sections of $S.$

The following lemma relates Roe's $\tau$ functional to the density of states. We will use the fact that if $P$ is elliptic and self-adjoint, then the mapping
\[
    f\mapsto \tau(f(P))
\]
is continuous on $f \in C_c(\mathbb{R}).$ Indeed, by the Sobolev inequality the uniform norm of the integral kernel of $f(P)$ is bounded above by the norm of $f(P)$ as an operator from a Sobolev space of sufficiently negative smoothness to a Sobolev space of sufficiently positive smoothness. Since $f$ is compactly supported, by functional calculus, the operator $f(P)(1+P)^N$ is bounded on $L_2(M,S)$ for every $N,$ with norm depending on the width of the support of $f$ and the uniform norm of $f.$ Since $P$ is elliptic, it follows from these arguments that if $f$
is supported in $[-K,K]$ then there is a constant $C_K$ such that
\[
    |\tau(f(P))| \leq C_K\|f\|_{\infty}.
\]
See the related arguments in \cite[Proposition 2.9, Proposition 2.10]{Roe1988a}.

\begin{lem}\label{L: tau is dos}
    Let $S\to M$ be a vector bundle of bounded geometry,
    and let $P \in EBD^m(M,S)$ for some $m>0.$
    Assume that $P$ has a density of states $\nu_P$
    with respect to the base-point $x_0\in M.$ If $\tau$ is associated with the exhaustion $\{B(x_0,R_i)\}_{i=0}^\infty$ for some sequence $R_i\to\infty,$ then
    \[
        \tau(f(P)) = \int_{\mathbb{R}} f\,d\nu_P,\quad f \in C_c(\mathbb{R}).
    \]
    Similarly,
    \[
        \tau(\exp(-tP)) = \int_{\mathbb{R}} \exp(-t\lambda)\,d\nu_P(\lambda),\quad t > 0.
    \]
\end{lem}
\begin{proof}
By Theorem \ref{T: vector bundle dos thm}, we have
\begin{align*}
\mathrm{Tr}_\omega (\exp(-tP) M_w)&= \int_{\mathbb{R}} e^{-t\lambda} \, d\nu_P(\lambda)\\
&= \lim_{R\to \infty} \frac{1}{|B(x_0, R)|} \mathrm{Tr}(\exp(-tP)M_{\chi_{B(x_0,R)}})\\
&= \lim_{R\to \infty} \frac{1}{|B(x_0, R)|} \int_{B(x_0,R)}\mathrm{tr}_{\mathrm{End}(S_x)}(K_{\exp(-tP)}(x,x)) \,d\nu_g(x)\\
&= \tau(\exp(-tP)),
\end{align*}
Since
\[
    f\mapsto \tau(f(P)),\quad f \in C_c(\mathbb{R})
\]
is continuous in the sense described in the paragraph preceding the theorem, it follows from the Riesz theorem that there exists a measure $\mu_{\tau,P}$ on $\mathbb{R}$ such that
\[
    \tau(f(P)) = \int_{\mathbb{R}} f\,d\mu_{\tau,P},\quad f \in C_c(\mathbb{R}).
\]
Since $\mu_{\tau,P}$ and $\nu_P$ have identical Laplace transform, it follows that $\mu_{\tau,P} = \nu_P.$
\end{proof}

A combination of Lemma \ref{L: tau is dos} and Theorem \ref{T: vector bundle dos thm} immediately yields the following:
\begin{thm}
Let $M$ be a manifold that satisfies the assumptions of Theorem~\ref{T: main manifold thm}, and let $S\to M$ be a vector bundle of bounded geometry. Let $P \in EBD^2(M,S)$, be self-adjoint and lower-bounded, and assume it admits a density of states $\nu_P$ at $x_0$. Let $w$ be the function on $M$ defined by \[ w(x) = (1+\abs{B(x_0,d_g(x,x_0))})^{-1},\quad x \in M.\] Then for any $f\in C_c(\mathbb{R})$
we have \[\tau(f(P)) = \mathrm{Tr}_\omega(f(P)M_w)\]
for any $\tau$ associated with the regular exhaustion $\{B(x_0,R_i)\}_{i\in \mathbb{N}}$ where $R_i\to\infty,$ and for any extended limit $\omega.$
Similarly,
\[
    \tau(\exp(-tP)) = \mathrm{Tr}_\omega(\exp(-tP)M_w)
\]
\end{thm}

The preceding theorem is proved under the strong assumption that $P$ admits a density of states, which in particular implies that $\tau(\exp(-tP))$ is independent of the choice of functional $m$ used to define $\tau.$ Addressing the question of determining which traces $\tau$ and which extended limits $\omega$ are related in this way in general is beyond the scope of this chapter.

Roe \cite{Roe1988a} defines an algebra $\mathcal{U}_{-\infty}(E)$ of operators acting on sections of a vector bundle $E,$ and $\tau$ is extended to $\mathcal{U}_{-\infty}(E)$ essentially by composing $\tau$ with the pointwise trace on $\mathrm{End}(E),$ see \cite{Roe1988a} for details.

\begin{thm}\label{roe_index_formula_theorem}
Let $M$ be a non-compact Riemannian manifold of bounded geometry with Property (D), with a graded Clifford bundle $S \to M$ of bounded geometry. Let $D$ be a Dirac operator on $M$ associated with the Dirac complex \[C^\infty(S^+) \xrightarrow{D_+} C^\infty(S^-),\] where $D_+$ is the restriction of $D$ to the sections of $S^+$ and $D_- = D_+^*$ is its adjoint (cf.~\cite[Chapter~11]{Roe1998}). Let $D^2$ admit a density of states both when considered as an operator restricted to $L_2(S^+)$ and $L_2(S^-)$, in the sense that \[\lim_{R\to \infty} \frac{1}{|B(x_0,R)|} \mathrm{Tr}(\exp(-tD_-D_+)M_{\chi_{B(x_0,R)}}) = \int_{\mathbb{R}} e^{-t\lambda}\, d\nu_{D_-D_+}(\lambda), \quad t>0,\] for a Borel measure $\nu_{D_-D_+}$ and similarly for $D_+D_-$. Then for any $f\in C_c(\mathbb{R})$ such that $f(0)=1$, we have
\[
\dim_\tau(\mathrm{Ind} D) = \mathrm{Tr}_\omega(\eta f(D^2)M_w).
\]
\end{thm}
\begin{proof}
By~\cite[Proposition~8.1]{Roe1988a}, we have that \[\dim_\tau(\text{Ind} D) = \tau(\eta \exp(-tD^2)), \quad t>0\]
where $\eta$ is the grading operator
\[
    \eta = \begin{pmatrix} 1 & 0 \\ 0 & -1\end{pmatrix}
\]
with respect to the orthogonal direct sum $L_2(S) = L_2(S^+)\oplus L_2(S^-).$

The proof of the theorem amounts to showing that
\[
    \tau(\eta e^{-tD^2}) = \mathrm{Tr}_{\omega}(\eta e^{-tD^2}M_w).
\]
The left-hand side is the same as
\[
    \tau(e^{-tD_+D_-})-\tau(e^{-tD_-D_+})
\]
while the right hand side is
\[
    \mathrm{Tr}_{\omega}(e^{-tD_+D_-}M_w)-\mathrm{Tr}_{\omega}(e^{-tD_-D_+}M_w).
\]
Applying Theorem~\ref{T: vector bundle dos thm} to $D_+D_-$ and $D_-D_+$ individually proves the result.
\end{proof}

\begin{rem}
The index $\dim_{\tau}(\operatorname{Ind}(D))$ is computed by a version of the McKean--Singer formula \cite[Proposition 8.1]{Roe1988a}
\[
    \mathrm{dim}_{\tau}(\mathrm{Ind}(\mathrm{D})) = \tau(\eta \exp(-tD^2))
\]
for arbitrary $t>0.$ One of the motivations in developing the present theorem was
to give a new explanation of why the function
\[
    t\mapsto\tau(\eta\exp(-tD^2))
\]
is independent of $t.$ If the assumptions of Theorem \ref{roe_index_formula_theorem} hold, then
\[
    \tau(\eta \exp(-tD^2)) = \mathrm{Tr}_\omega(\eta e^{-tD^2}M_w).
\]
Formally differentiating the right hand side with respect to $t$ and using the tracial property of $\mathrm{Tr}_\omega$ yields
\[
    \frac{d}{dt}\mathrm{Tr}_\omega(\eta e^{-tD^2}M_w) = -\frac12\mathrm{Tr}_{\omega}(\eta e^{-tD^2}D[D,M_w]).
\]
Our conditions ensure that $e^{-tD^2}D[D,M_w]$ is trace-class, and hence that
\[
    \frac{d}{dt}\mathrm{Tr}_{\omega}(\eta e^{-tD^2}M_w) = 0.
\]
It is interesting that $\mathrm{Tr}_{\omega}(\eta e^{-tD^2}M_w)$ and the traditional heat supertrace $\mathrm{Tr}(\eta e^{-tD^2})$ on a compact manifold are both independent of $t$ for apparently different reasons.
\end{rem}

\section{An example with a random operator}\label{S: Example}
This section has mostly been the work of Edward McDonald.

The assumptions in Theorem~\ref{T: main manifold thm} appear quite strong, especially the existence of the density of states. The following example of a random operator on a non-compact manifold where the density of states exists was given by Lenz, Peyerimhoff and Veseli\'{c}  \cite{LenzPeyerimhoffVeselic2007}, generalising earlier examples in \cite{LenzPeyerimhoff2004,PeyerimhoffVeselic2002}.

\begin{ex}\cite[Example (RSM)]{LenzPeyerimhoffVeselic2007}
    Let $(M,g_0)$ be the connected Riemannian covering of a compact Riemannian manifold $X = M/\Gamma,$ where $\Gamma$ is an infinite  group acting freely and properly discontinuously on $M$ by isometries. Assume that there is an ergodic action $\alpha$ of $\Gamma$ on a probability space $(\Xi,\Sigma,\mathbb{P}),$ and let $\{g_\xi\}_{\xi\in \Xi}$ be a measurable family of metrics on $M$ which are uniformly comparable with $g_0,$ in the sense that there exists $A>0$ such that
    \[
        \frac1A g_0(v,v) \leq g_\xi(v,v)\leq Ag_0(v,v).
    \]
    for all tangent vectors $v$ to $M.$ Assume that the action $\alpha$ of $\Gamma$ on $\Xi$ is compatible with the action on $\Gamma$ on $M$ in the sense that $g_{\alpha_{\gamma}^{-1}\xi}$ is the pullback of $g_{\xi}$ under the automorphism defined by $\gamma.$

    Similarly, it is assumed that there is a measurable family $\{V_{\xi}\}_{\xi\in \Xi}$ of smooth functions on $M$ such that
    \[
        V_{\xi}\circ \gamma = V_{\alpha^{-1}_\gamma\xi}.
    \]
    Let $\nu^\xi$ denote the Riemannian volume form on $M$ corresponding to $g_{\xi},$ and let $\Delta_{\xi}$ be the Laplace-Beltrami operator on $(M,g_{\xi}).$ Then \cite{LenzPeyerimhoffVeselic2007} consider the operator on $L_2(M,\nu^{\xi})$ given by
    \[
        H_{\xi} = -\Delta_{\xi}+M_{V_\xi}.
    \]
    What is shown in~\cite[Equation (27)]{LenzPeyerimhoffVeselic2007} is that if $\Gamma$ is amenable, then for every tempered F\o lner sequence $\{A_n\}_{n=0}^\infty$ in $\Gamma,$ the limit
    \[
        \lim_{n\to\infty} \frac{1}{|A_n|}\mathrm{Tr}_{L_2(M,\nu^{\xi}}(M_{\chi_{A_nF}}e^{-tH_{\xi}})
    \]
    exists for almost every $\xi,$ where $F$ is a fundamental domain for $\Gamma.$

    Recall that a we say that a F\o lner sequence $\{A_n\}_{n=0}^\infty$ is tempered if there is a constant $C>0$ such that for every $n\geq 1$ we have
    \[
        \left|\bigcup_{k<n} A_k^{-1}A_{n}\right| \leq C|A_{n}|.
    \]
\end{ex}

We will make one further assumption: that the measure $\nu^{0}$ associated with $g_0$ has a doubling property. That is, there exists a constant $C$ such that
\[
    \nu^0(B_{g_{0}}(x_0,2R)) \leq C\nu^{0}(B_{g_{0}}(x_0,R)),\quad R>0.
\]
Note that since the identity function is a bi-Lipschitz continuous map from $(M,g_0)$ to $(M,g_{\xi}),$ the same holds for the measures $\nu^{\xi}$ associated with the metrics $g_{\xi}.$

\begin{prop}
    Let $(M, g_0)$ be as above, and let
    \[
    w_\xi(x) = \frac{1}{1+\nu^\xi({B_{g_\xi}(x_0,d_{g_\xi}(x,x_0))})},\quad x \in M.
    \]
    If $M$ satisfies Property (D), the density for $H_\xi$ exists almost surely, and for almost all $\xi \in \Xi$,
    \[
    \mathrm{Tr}_{\omega}(f(H_\xi)M_{w_\xi}) = \int_{\mathbb{R}} f(\lambda) \, d\nu_{H_\xi}(\lambda), \quad f\in C_c(\mathbb{R}),
    \]
    for every extended limit $\omega \in \ell_\infty^*.$ 
\end{prop}
\begin{proof}
    For brevity, we will denote the measure $\nu^{\xi}$ by $|\cdot|$
    and $B(x_0,R)$ for $B_{g_{\xi}}(x_0,R).$

    We will show that Property (D) implies that $\Gamma$ admits a tempered F\o lner sequence $\{A_n\}_{n=1}^\infty$, and that for any bounded measurable function $g$ on $M$ we have
    \begin{equation}\label{equivalence_of_limits}
        \lim_{k\to\infty} \frac{1}{|A_kF|} \int_{A_kF}g\,d\nu^{\xi} = \lim_{R\to\infty} \frac{1}{|B(x_0,R)|}\int_{B(x_0,R)} g\,d\nu^{\xi}
    \end{equation}
    if either limit exists. Recall that $F$ is a fundamental domain for the action of $\Gamma.$
    Together with the results of \cite{LenzPeyerimhoffVeselic2007}, this implies that the limit
    \[
        \lim_{R\to\infty} \frac{1}{|B(x_0,R)|}\mathrm{Tr}(M_{\chi_{B(x_0,R)}}e^{-tH_{\xi}})
    \]
    exists for every $t>0$, and hence the assumptions of Theorem \ref{T: main manifold thm} are satisfied.

    Let $h>2\mathrm{diam}(F).$ For $k\geq h,$ let
    \[
        A_k = \{\gamma \in \Gamma\;:\; \mathrm{dist}(\gamma x_0,x_0) < k-h \}
    \]
    We claim that $A_k$ is a tempered F\o lner sequence.

    Note that $\mathrm{dist}(\gamma x_0,x_0) = \mathrm{dist}(\gamma^{-1}x_0,x_0),$ and so automatically $A_k=A_k^{-1}.$
    Define
    \[
        B_k := A_kF = \bigcup_{\gamma\in A_k} \gamma F
    \]
    First we show that $B(x_0,k-2h)\subseteq B_k.$ Indeed, if $p \in B(x_0,k-2h)$ there exists some $\gamma\in  \Gamma$ such that $p \in \gamma F,$ so $\mathrm{dist}(p,x_0) \leq \mathrm{diam}(F) < h,$ and thus $\mathrm{dist}(\gamma x_0,x_0) < k-2k+h =k-h.$ On the other hand, since $F$ has diameter smaller than $\frac{h}{2},$ if $p\in B_k$ then $\mathrm{dist}(p,x_0)\leq \frac{h}{2}+k-h < k.$
    That is, for all $k\geq 2h$ we have
    \[
        B(x_0,k-2h)\subset B_k \subset B(x_0,k).
    \]
    Since the action of $\Gamma$ is free, the union of the translates of $F$ is disjoint, and
    \[
        |B_k| = |A_k||F|,\quad k\geq 0
    \]
    and hence
    \[
        |B(x_0,k-2h)| \leq |F||A_k| \leq |B(x_0,k)|.
    \]
    As in the proof of Lemma~\ref{Grimaldi}, we know that for each $h>0$ there is a constant $C_h$ so that we have
    \[
    \frac{|B(x_0, k)|}{|B(x_0, k-2h)|} \leq C_h, \quad k > 1+2h
    \]
    Therefore,
    \[
        |A_k| \approx |B_k| \approx |B(x_0,k)|.
    \]
    uniformly in $k>1+2h.$

    To see that $A_k$ is F\o lner, let $\gamma\in \Gamma.$ By the triangle inequality, there exists $N>0$ such that
    \[
        \gamma A_k\subset A_{k+N}
    \]
    and, for $k$ sufficiently large,
    \[
        A_k\subset \gamma A_{k-N}
    \]
    and therefore the symmetric difference of $A_k$ and $\gamma A_k$ satisfies
    \[
        (\gamma A_k\setminus A_k)\cup (A_k\setminus \gamma A_k) \subset A_{k+N}\setminus A_{k-N}.
    \]
    It follows that, as $k\to\infty,$
    \[
        \frac{|(\gamma A_k\setminus A_k)\cup (A_k\setminus \gamma A_k)|}{|A_k|} \approx \frac{|B(x_0,k+N)|-|B(0,k-N)|}{|B(x_0,k)|}
    \]
    which is vanishing as $k\to\infty.$ Hence, $\{A_k\}_{k=1}^\infty$ is F\o lner.

    To see that $\{A_k\}_{k=1}^\infty$ is tempered, it suffices to show that there is a constant $C$ such that
    \[
        |A_k\cdot A_k| \leq C|A_k|.
    \]
    By the triangle inequality and the fact that $\Gamma$ acts isometrically we see that
    \[
        A_k\cdot A_k \subseteq A_{2k}.
    \]
    Therefore, by the doubling condition,
    \[
        |A_k\cdot A_k| \leq |A_{2k}| = \frac{1}{|F|}|B(x_0,2k)| \lesssim |B(x_0,k)| \approx |A_k|.
    \]
    uniformly in $k$ for sufficiently large $k.$
    Hence, $\{A_k\}_{k=1}^\infty$ is tempered.

Finally we prove \eqref{equivalence_of_limits}. Using the fact that $B_k\subseteq B(x_0,k),$ we write
\begin{align*}
    \frac{1}{|B(x_0,k)|}\int_{B(x_0,k)} g\,d\nu^{\xi} &= \frac{1}{|B_k|}\int_{B_k} g\,d\nu^{\xi} + \frac{|B_k|-|B(x_0,k)|}{|B(x_0,k)||B_k|}\int_{B_k} g\,d\nu^{\xi} \\
    &\quad + \frac{1}{|B(x_0,k)|}\int_{B(x_0,k)\setminus B_k} g\,d\nu^{\xi}
\end{align*}
Therefore,
\begin{align*}
    \left|\frac{1}{|B(x_0,k)|}\int_{B(x_0,k)} g\,d\nu^{\xi} - \frac{1}{|B_k|}\int_{B_k} g\,d\nu^{\xi}\right|
    &\leq 2\|g\|_{\infty}\frac{|B(x_0,k)|-|B_k|}{|B(x_0,k)|}\\
    &\leq 2\|g\|_{\infty} \frac{|B(x_0,k)|-|B(x_0,k-2h)|}{|B(x_0,k)|}
\end{align*}
and this vanishes as $k\to\infty.$ From here one easily deduces \eqref{equivalence_of_limits}.
\end{proof}

\section{Discrete metric spaces revisited}\label{S:MfsToDisc}
The abstract operator theoretical result in Theorem~\ref{MAIN_THEOREM} also recovers a weaker version of the main result from Chapter~\ref{Ch:DOSDiscrete}. This section is adapted from~\cite[Section~6]{HekkelmanMcDonaldv1}, the first preprint version of~\cite{HekkelmanMcDonald2024}, which did not appear in the published version.

As in Chapter~\ref{Ch:DOSDiscrete}, take a countably infinite discrete metric space $(X,d)$ such that every ball contains at most finitely many points, and let $x_0 \in X$. Write $\{r_k\}_{k \in \mathbb{N}} \subseteq \mathbb{R}$ for the set of values $\{d(x,x_0) \; : \; x\in X \}$ ordered in increasing fashion. Take a positive radially symmetric strictly increasing $w : X \to \C$ such that $M_w \in \mathcal{L}_{1,\infty}$.

Now take a self-adjoint bounded operator $T \in B(\ell_2(X))$ and $C >0$ such that $T + C$ is strictly positive. Writing $P = \log(T+C)$, we have that $\exp(-tP) = (T+C)^t$. Then Theorem~\ref{MAIN_THEOREM} --- using the weaker conditions~\eqref{eq:original_conditions} ---   gives us that if for all $t >0$
\begin{enumerate}
    \item $(T+C)^t M_w \in \mathcal{L}_{1,\infty}$;
    \item $[(T+C)^t, M_w] \in \mathcal{L}_1$,\label{WeakerCondition2}
\end{enumerate}
then we have for every extended limit $\omega \in \ell_\infty^*$,
\begin{equation}\label{eq:disc2}
\Tr_\omega((T+C) M_w) = \lim_{\varepsilon \to 0} \varepsilon \Tr((T+C) \chi_{[\varepsilon, \infty)}(M_w)),
\end{equation}
whenever the limit on the right hand exists. On the other hand, the main result in Chapter~\ref{Ch:DOSDiscrete}, Theorem~\ref{T: Main}, would let us conclude for all bounded $T \in B(\ell_2(X))$ (not just self-adjoint $T$)
\begin{equation}\label{eq:disc1}
    \mathrm{Tr}_\omega(TM_w) = \mathrm{Tr}_\omega(M_w)\lim_{k\to \infty} \frac{\mathrm{Tr}(TM_{\chi_{B(x_0,r_k)}})}{|B(x_0,r_k)|},
\end{equation}
whenever the limit on the right-hand side exists. These statements are not equivalent, namely, the existence of the limit~\eqref{eq:disc2} is a formally stronger requirement. 

Recall that $w$ is assumed to be radially strictly decreasing, so for $\varepsilon < 0$ we can pick $k \in \N$ such that 
\[
\tilde{w}(r_{k+1}) < \varepsilon \leq \tilde{w}(r_{k}),
\]
where $w(\cdot) = \tilde{w} \circ d(x_0, \cdot)$. Then we have for positive $0 < \delta 1 < T \in B(\ell_2(X))$,
\begin{align*}
    \varepsilon \Tr(T \chi_{[\varepsilon, \infty)}(M_w)) &-\mathrm{Tr}_\omega(M_w)\frac{\mathrm{Tr}(TM_{\chi_{B(x_0,r_k)}})}{|B(x_0,r_k)|}\\
    &\leq \bigg(\tilde{w}(r_{k}) - \mathrm{Tr}_\omega(M_w) \frac
    {1}{|B(x_0,r_k)|}\bigg) \mathrm{Tr}(TM_{\chi_{B(x_0,r_k)}})\\
    &\quad + \tilde{w}(r_{k}) \Tr(TM_{\chi_{B(x_0, r_{k+1} ) \setminus B(x_0, r_k)}})\\
    &\leq  \Big|\tilde{w}(r_{k})|B(x_0,r_k)| - \mathrm{Tr}_\omega(M_w)\Big| \|T\|_\infty\\
    &\quad + \tilde{w}(r_k) |S(x_0, r_{k+1})|\|T\|_\infty,
\end{align*}
where $S(x_0,r_{k+1}) = B(x_0, r_{k+1})\setminus B(x_0, r_k)$. Assuming that $X$ satisfies Property $(C)$, we have that
\[
\tilde{w}(r_k) |S(x_0, r_{k+1})| = \tilde{w}(r_k)|B(x_0, r_{k})| \frac{|S(x_0, r_{k+1})|}{|B(x_0, r_{k})|} \to 0, \quad k \to \infty,
\]
since $M_w \in \mathcal{L}_{1,\infty}$ implies that
\[
\tilde{w}(r_k)|B(x_0, r_{k})| \leq \sup_{k \in \N} k\mu(k,w) < \infty,
\]
where $\{\mu(k,w)\}_{k=0}^\infty$ is the decreasing rearrangement of $w$.

On the other hand we have
\begin{align*}
    \varepsilon \Tr(T \chi_{[\varepsilon, \infty)}(M_w)) &-\mathrm{Tr}_\omega(M_w)\frac{\mathrm{Tr}(TM_{\chi_{B(x_0,r_k)}})}{|B(x_0,r_k)|}\\
    &\geq \bigg(\tilde{w}(r_{k+1}) - \mathrm{Tr}_\omega(M_w)\frac{1}{|B(x_0,r_k)|}\bigg)\Tr(TM_{\chi_{B(x_0,r_k)}})\\
    &\geq -\delta \Big|\tilde{w}(r_{k+1}) |B(x_0,r_k)| - \Tr_\omega(M_w)\Big|.
\end{align*}
Hence, using Property $(C)$, the condition
\begin{equation}\label{eq:discrequirement1}
    \lim_{k \to \infty} \tilde{w}(k) |B(x_0, r_k)| = \mathrm{Tr}_\omega(M_w)
\end{equation}
allows going from equation~\eqref{eq:disc2} to~\eqref{eq:disc1}. 
This does not hold for general $w \in \ell_{1,\infty}(X)$, however. Equation~\eqref{eq:discrequirement1} requires $M_w$ to be Dixmier measurable, but it is an even stronger condition than this~\cite[Theorem~9.1.6]{LSZVol1}.

At least for the function $w(x) = (1+|B(x_0, d(x,x_0))|)^{-1}$ condition~\eqref{eq:discrequirement1} is true due to Corollary~\ref{C:MwTrace}, and we can re-derive a weaker version of Theorem~\ref{T: Main} using Theorem \ref{MAIN_THEOREM}.

\begin{prop}\label{P:HelpfulPowert}
Suppose $T \geq \varepsilon 1 >0, V \in B(H)$ are bounded operators such that $[T, V] \in \mathcal{L}_1$. Then  \[[T^t, V] \in \mathcal{L}_1, \quad t >0.\]
\end{prop}
\begin{proof}
Let $0 < t < 1.$ We have
\[
    T^t = \frac{\sin(\pi t)}{\pi} \int_0^\infty \lambda^{t-1}T(\lambda+T)^{-1}\,d\lambda.
\]
Then
\begin{align*}
    [T^t,V] &= \frac{\sin(\pi t)}{\pi} \int_0^\infty \lambda^{t-1}[T(\lambda+T)^{-1},V]\,d\lambda\\
              &= \frac{\sin(\pi t)}{\pi}\int_0^\infty \lambda^{t-1}[1-\lambda(\lambda+T)^{-1},V]\,d\lambda\\
              &= \frac{\sin(\pi t)}{\pi}\int_0^\infty \lambda^{t}(\lambda+T)^{-1}[T,V](\lambda+T)^{-1}\,d\lambda.
\end{align*}
Since $T\geq \varepsilon1,$ we have
\[
    \|(\lambda+T)^{-1}\| \leq (\lambda+\varepsilon)^{-1},
\]
and hence
\[
    \|[T^t,V]\|_1 \leq \frac{\sin(\pi t)}{\pi} \int_0^\infty \lambda^t (\lambda+\varepsilon)^{-2}\,d\lambda \cdot \|[T,V]\|_1<\infty.
\]
To conclude, for $t > 1$ we can inductively apply the equality
\[
[T^t, V] = T[T^{t-1}, V] + [T, V]T^{t-1}
\]
to obtain the result.
\end{proof}

\begin{thm}\label{T:WeakerDiscrete}
    Let $(X,d)$ be a countably infinite discrete metric space such that every ball contains at most finitely many points, and let $x_0 \in X$. Write by $\{r_k\}_{k \in \mathbb{N}}$ for the set $d(\cdot, x_0): X \rightarrow \mathbb{R}_{\geq 0}$ ordered in increasing manner. Suppose that \begin{equation}\tag{C}
        \lim_{k\rightarrow \infty}\frac{|B(x_0, r_{k+1})|}{|B(x_0, r_k)|} = 1.
    \end{equation} Define the function $w(x) = (1+|B(x_0, d(x,x_0))|)^{-1}$. If $T$ is a bounded self-adjoint operator on  $\ell_2(X)$ for which $[T,M_w] \in \mathcal{L}_1$, we have that for every extended limit $\omega$,
    \begin{equation}
        \mathrm{Tr}_\omega(TM_w) = \lim_{k\to \infty} \frac{\mathrm{Tr}(TM_{\chi_{B(x_0,r_k)}})}{|B(x_0,r_k)|},
    \end{equation}
    if the limit on the right-hand side exists.

    In particular, if $H$ is a self-adjoint, possibly unbounded, operator on $\ell_2(X)$ with density of states measure $\nu_H$, and $f \in C_c(\mathbb{R})$ and such that $[f(H), M_w] \in \mathcal{L}_1$, then for all extended limits $\omega$
\begin{equation}
    \mathrm{Tr}_\omega(f(H)M_w) = \int_{\mathbb{R}} f\, d\nu_H.
\end{equation}
\end{thm}
\begin{proof}
    Let $C>0$ such that $T+C > \delta 1 > 0$. The fact that $M_w \in \mathcal{L}_{1,\infty}$, Proposition~\ref{P:HelpfulPowert} and the assumption that $[T, M_w] \in \mathcal{L}_1$ give that for all $t >0$,
    \begin{enumerate}
        \item $(T+C)^tM_w \in \mathcal{L}_{1,\infty}$;
        \item  $[(T+C)^t, M_w] \in \mathcal{L}_1$.
    \end{enumerate}
    Hence, Theorem~\ref{MAIN_THEOREM} gives that 
\[
\Tr_\omega((T+C) M_w) = \lim_{\varepsilon \to 0} \varepsilon \Tr((T+C) \chi_{[\varepsilon, \infty)}(M_w)),
\]
whenever the limit on the right-hand side exists. Let $\varepsilon > 0$ small and take $k\in \N$ such that $(1+|B(x_0, r_{k+1})|)^{-1} \leq \varepsilon < (1+|B(x_0, r_k)|)^{-1}$. Then,
\begin{align*}
     \frac{|B(x_0, r_{k})|}{1+|B(x_0,r_{k+1})|} \leq \varepsilon \Tr(\chi_{[\varepsilon, \infty)}(M_w)) \leq \frac{|B(x_0, r_{k+1})|}{1+|B(x_0,r_{k})|}.
\end{align*}
Due to Property $(C)$, we therefore have
\[
\lim_{\varepsilon \to 0}  \varepsilon \Tr(\chi_{[\varepsilon, \infty)}(M_w)) = 1 = \Tr_\omega(M_w),
\]
using Corollary~\ref{C:MwTrace}. Therefore,
\[
\Tr_\omega(TM_w) = \lim_{\varepsilon \to 0} \varepsilon \Tr(T \chi_{[\varepsilon, \infty)}(M_w)),
\]
whenever the limit on the right-hand side exists. With the computations given previously, we can conclude that 
\[
        \mathrm{Tr}_\omega(TM_w) = \lim_{k\to \infty} \frac{\mathrm{Tr}(TM_{\chi_{B(x_0,r_k)}})}{|B(x_0,r_k)|},
\]
    if the limit on the right-hand side exists.

Finally, if $H$ is a self-adjoint operator on $\ell_2(X)$ for which the density of states $\nu_H$ exists, we have the existence of all limits
\[
 \lim_{k\to \infty} \frac{\mathrm{Tr}(f(H)M_{\chi_{B(x_0,r_k)}})}{|B(x_0,r_k)|}, \quad f \in C_c(\R).
\]
If $f \in C_c(\R)$ is such that $[f(H), M_w] \in \mathcal{L}_1$, then also  the real part $\Re([f(H), M_w]) = [\Re (f)(H), M_w]$ is trace-class, and likewise the imaginary part. Since $\Re(f)(H)$ and $i\Im(f)(H)$ are self-adjoint, by the above we therefore have
\[
\mathrm{Tr}_\omega(f(H)M_w) = \int_\R f\, d\nu_H.\qedhere
\]
\end{proof}

We remark that, in the case that $H$ is self-adjoint, bounded, and $[H,M_w]$ is trace-class, we get that $[f(H),M_w]$ is trace-class if $f$ is operator-Lipschitz. In particular, this is the case when $f$ is in the homogeneous Besov space $\dot{B}^1_{\infty,1}(\mathbb{R})$ \cite[Theorem 2]{Peller1990}.

The condition that $[T,M_w]\in \mathcal{L}_1$ is relatively weak. Since the operator $[T,M_w]$ already belongs to the commutator subspace \[\mathrm{Com}(\mathcal{L}_{1,\infty}):=[B(H), \mathcal{L}_{1,\infty}],\] it follows that~\cite[Theorem~5.1.4]{LSZVol1} 
\[
\sum_{k=0}^n \lambda(k, [T,M_w]) = O(1).
\]
However, that is not quite enough for it to be trace-class. The following example has been provided by Teun van Nuland.

\begin{ex}
Let $X = \mathbb{N}^+ = \{1,2,\ldots\}$, so that $M_w: e_n \mapsto \frac{1}{n} e_n$. Define a bijection $\phi: \mathbb{N} \to \mathbb{N}$ by \[2n+1 \mapsto 2^n,\] and recursively mapping each even integer $2n$ to the least integer that doesn't appear in $\{\phi(k) : k<2n\}$. Then define the operator $T: e_n \mapsto e_{\phi(n)}$. Observe that $T^*: e_n \mapsto e_{\phi^{-1}(n)}$. By a simple calculation, \begin{align*}
    [T,M_w][T,M_w]^* e_n &=  \left(\frac{1}{\phi^{-1}(n)}-\frac{1}{n} \right) [T,M_w] e_{\phi^{-1}(n)}\\
    &= \left(\frac{1}{\phi^{-1}(n)} - \frac{1}{n}\right)^2 e_n.
\end{align*}
Hence, $[T,M_w]$ is trace-class if and only if \[\sum_{n\in \mathbb{N}} \abs{\frac{1}{\phi^{-1}(n)} - \frac{1}{n}} = \sum_{n\in \mathbb{N}} \abs{\frac{1}{n}- \frac{1}{\phi(n)}} < \infty.\]
However, \[\sum_{n\in \mathbb{N}} \abs{\frac{1}{n}- \frac{1}{\phi(n)}} > \sum_{n\in \mathbb{N}} \left(\frac{1}{2n+1} - \frac{1}{2^n}\right),\] which diverges. Hence $[T,M_w]$ is not trace-class.
\end{ex}

In conclusion, Theorem~\ref{T:WeakerDiscrete} is significantly weaker than Theorem~\ref{T: Main}, justifying the method of proof in Chapter~\ref{Ch:DOSDiscrete}.
    \clearpage
   
   
    \clearpage
    
    \backmatter
    
    \pagestyle{noHeading}

\printbibliography

\end{document}